\documentclass[11pt]{article}
\usepackage[english]{babel}
\usepackage{amsfonts}
\usepackage[utf8]{inputenc}
\usepackage{amsthm}
\usepackage{ mathrsfs }
\usepackage{amsmath}
\usepackage{mathrsfs}
\usepackage{enumitem}
\usepackage{mathtools}
\usepackage{bbm}
\usepackage{makeidx}
\usepackage{graphicx}
\usepackage{frontespizio}
\usepackage{color}
\usepackage[makeroom]{cancel}
\usepackage[normalem]{ulem}
\usepackage{ amssymb }
\usepackage{a4wide}
\usepackage[hidelinks = true]{hyperref}
\usepackage{mathtools}
\usepackage{enumitem}
\usepackage[titletoc]{appendix}
\usepackage{esint,soul}
\mathtoolsset{showonlyrefs=true}

\numberwithin{equation}{section}

\newtheorem{theorem}{Theorem}[section]
\newtheorem{corollary}[theorem]{Corollary}
\newtheorem{lemma}[theorem]{Lemma}
\newtheorem{proposition}[theorem]{Proposition}
\newtheorem{definition}[theorem]{Definition}

\newtheorem{example}[theorem]{Example}
\newtheorem{remark}[theorem]{Remark}

\newtheorem{open}[theorem]{Open problem}
\newtheorem{thmdef}[theorem]{Theorem/Definition}

\newenvironment{idea}{\removelastskip\par\medskip   
\noindent{\em Idea of the proof} \rm}{\penalty-20\null\hfill$\square$\par\medbreak}

\newcommand{\N}{\mathbb{N}}
\newcommand{\Q}{\mathbb{Q}}
\newcommand{\R}{\mathbb{R}}
\newcommand{\Z}{\mathbb{Z}}

\newcommand{\aalpha}{{\mbox{\boldmath$\alpha$}}}

\newcommand{\ggamma}{{\mbox{\boldmath$\gamma$}}}

\newcommand{\ppi}{{\mbox{\boldmath$\pi$}}}

\newcommand{\sggamma}{{\mbox{\scriptsize\boldmath$\gamma$}}}

\newcommand{\sppi}{{\mbox{\scriptsize\boldmath$\pi$}}}
\newcommand{\restr}[1]{\lower3pt\hbox{$|_{#1}$}}
\newcommand{\pr}{\mathscr P}
\newcommand{\prd}{\mathscr P_2}

\newcommand{\e}{{\rm e}}

\newcommand{\Restr}[2]{{\rm Restr}_{#1}^{#2}}

\newcommand{\X}{{\rm X}}
\newcommand{\Y}{{\rm Y}}
\renewcommand{\Z}{{\rm Z}}

\newcommand{\sfd}{{\sf d}}
\newcommand{\diam}{{\sf diam}}

\newcommand{\mm}{{\mathfrak m}}

\newcommand{\vol}{{\rm vol}}								

\newcommand{\BV}{{\rm BV}}

\newcommand{\LIP}{{\rm Lip}}

\newcommand{\CD}{{\sf CD}}

\newcommand{\RCD}{{\sf RCD}}

\newcommand{\ch}{{\sf Ch}}

\newcommand{\ent}{{\rm Ent}_\mm}
\newcommand{\entm}[1]{{\rm Ent}_{#1}}

\renewcommand{\d}{{\rm d}}
\newcommand{\D}{{\rm D}}		

\renewcommand{\div}{{\rm div}}
\newcommand{\lip}{{\rm lip}}
\newcommand{\Lip}{{\rm Lip}}
\newcommand{\lipa}{{\rm lip}_a}
\newcommand{\Hess}{{\rm Hess}}

\newcommand{\lims}{\varlimsup}
\newcommand{\limi}{\varliminf}

\newcommand{\bs}{{\rm bs}}
\renewcommand{\b}{{\rm b}}

\newcommand{\bc}{{\rm bc}}

\newcommand{\weakto}{\rightharpoonup}
\newcommand{\nchi}{{\raise.3ex\hbox{$\chi$}}}
\newcommand{\supp}{{\rm supp}}
\newcommand{\eps}{\varepsilon}
\newcommand{\Id}{{\rm Id}}
\newcommand{\B}{\mathcal B}                            

\newcommand{\comp}{{\sf Comp}} 

\newcommand{\h}{{\sf h}} 

\newcommand{\E}{{\sf E}} 

\newcommand{\fr}{\penalty-20\null\hfill$\blacksquare$}         

\DeclareMathOperator*{\esssup}{\rm ess-sup}
\DeclareMathOperator*{\essinf}{\rm ess-inf}

\newcommand\la{\langle}
\newcommand\ra{\rangle}
\newcommand\KE{{\sf KE}}
\newcommand{\T}{\mathbb T}
\newcommand{\HS}{\sf HS}
\newcommand{\M}{\mathscr M}
\newcommand{\Hi}{\mathscr H}
\newcommand{\WW}{\mathcal W}
\newcommand{\DD}{\mathbf D}
\newcommand{\test}{{\rm Test}}
\newcommand{\Comp}{{\sf Comp}}
\newcommand{\Ric}{{\rm Ric}}
\newcommand{\EDE}{{\sf EDE}}
\newcommand{\EVI}{{\sf EVI}}
\newcommand{\AVR}{{\sf AVR}}
\newcommand{\Ggamma}{{\mathbf \Gamma}}
\DeclareMathOperator*{\glimi}{{\Gamma-\limi}}
\DeclareMathOperator*{\glims}{\Gamma-\lims}
\DeclareMathOperator*{\glim}{\Gamma-\lim}

\DeclareMathOperator*{\wlimi}{{\rm w}-\limi}

\newcommand{\ssggamma}{{\mbox{\tiny\boldmath$\gamma$}}}

\DeclareRobustCommand{\hsout}[1]{\texorpdfstring{\sout{#1}}{#1}}

\title{De Giorgi and Gromov working together \\[1ex] \large  A survey about stability results in connection with lower Ricci bounds}

\begin{document}

\author{Nicola Gigli\ \thanks{SISSA, ngigli@sissa.it}  }

\maketitle	
	
\begin{abstract}
The title is meant as way to honor two great mathematicians that, although never actually worked together, introduced concepts of convergence that perfectly match each other and very fruitfully interact: De Giorgi's $\Gamma$-convergence of lower semicontinuous functions and Gromov's  convergence of geometric structures.

\end{abstract}

\tableofcontents

\section{Introduction}
A central theme in geometry is the understanding of how curvature bounds affect the shape and properties of the space in consideration. In pondering this question, a major intuition by Gromov \cite{Gromov07} is that it is convenient to pass from the category of smooth Riemannian manifolds to a suitable compactification of it. Elements of such compactification are non-smooth objects for which, in the appropriate weak sense, curvature bounds are still in place and the idea is that their study can be fruitful  to answer questions formulated in the smooth world. This is very much in line with what modern analysis has been doing for more than a century: for instance, the moment one realizes that the Dirichlet energy $\int_U|\d f|^2$ is a relevant functional on the space $C^\infty(U)$, she/he is naturally lead to consider limits of smooth functions with a uniform bound on such energy, and thus to the notion of Sobolev space. 

\medskip

This note is about uniform lower bounds on the Ricci curvature.

\subsection{Brief history of the subject}

Non-smooth spaces with lower Ricci curvature bounds have been first studied in the nineties by  Cheeger and Colding (see in particular \cite{Colding96b}, \cite{Colding97}, \cite{Cheeger-Colding96}, \cite{Cheeger-Colding97I}, \cite{Cheeger-Colding97II}, \cite{Cheeger-Colding97III} and the survey \cite{Cheegersurvey}) from the `extrinsic' point of view, i.e.\ the main focus was on the structure of Ricci-limit spaces, namely those spaces arising as limits of Riemannian manifolds with a uniform lower Ricci and upper dimension bound. Their work served as strong foundation for all the subsequent research on the field.

 Later, Lott-Villani \cite{Lott-Villani09} and Sturm \cite{Sturm06II,Sturm06I} independently realized that optimal transport could be used to identify lower Ricci bounds `intrinsically' by studying convexity properties of entropy functionals in the Wasserstein space. Their approach takes root in the fundamental paper by McCann \cite{McCann97}, where he introduced the notion of `displacement interpolation' and proved that, when working in the Euclidean setting, certain entropy functionals are convex along this form of interpolation. Subsequent investigations in the Riemannian framework (\cite{OV00}, \cite{CEMCS01}, \cite{vRS05}, \cite{CEMS06}) showed the connection between such displacement convexity and lower Ricci bounds: these  ultimately led Lott-Sturm-Villani to the formalization of the notion of Curvature-Dimension $\CD(K,N)$ for metric measure spaces. Here the parameter $K$ is the lower bound on the Ricci curvature and $N$ the upper bound on the dimension. In the simple(st) case  $N=\infty$, their definition reads as: a metric measure space $(\X,\sfd,\mm)$ is $\CD(K,\infty)$ provided for any $\mu_0,\mu_1$ Borel probability measures on it, there is a $W_2$-geodesic $(\mu_t)$ connecting them such that
 \begin{equation}
\label{eq:cdintro}
\ent(\mu_t)\leq (1-t)\ent(\mu_0)+t\ent(\mu_1)-\tfrac K2t(1-t)W_2^2(\mu_0,\mu_1)\qquad\forall t\in[0,1].
\end{equation}
Notice the interplay between the distance and the measure: the former is used to define the $W_2$-distance between probability measures, the latter to define the entropy functional.  Here the key word is `convexity': being this a concept that can be defined without relying on differential calculus, it can be  used in the non-smooth setting and shown to be stable under weak notions of convergence (in our setting, the relevant notion is that of mGH-convergence, see Section \ref{se:mGH}).

 The $\CD(K,N)$ notion  has been refined into the $\RCD(K,N)$ condition (`${\sf R}$' standing for `Riemannian') by me in \cite{Gigli12} with the addition of the functional analytic notion of `infinitesimal Hilbertianity', i.e.\ the requirement that the  space $W^{1,2}(\X)$ of real valued Sobolev functions on $\X$, that in general is a Banach space, is in fact Hilbert. Notice that unlike the $\CD(K,N)$ condition, infinitesimal Hilbertianity is not stable under mGH-convergence, but it becomes so when coupled with $\CD(K,N)$. A  way of thinking at this is to notice that infinitesimal Hilbertianity is a first order condition, and as such one cannot reasonably expect it to be stable under a zeroth order convergence such as mGH-convergence; stability is restored once a uniform second order bound is imposed (the curvature bound, in our case). In practice, stability is ultimately based on the stability of the Heat Flow observed in \cite{Gigli10}, see also below. The need for the refinement given by the $\RCD$ notion  is the fact that Finsler structures are included in the $\CD$ category, but can have geometries very different from the Riemannian ones and, moreover, are known to be excluded by the class of Ricci-limit spaces (by the almost-splitting theorem of Cheeger-Colding). While having a stable and unified Riemannian/Finslerian approach to lower Ricci bounds is certainly   useful, if one wants to use non-smooth spaces to derive informations about the geometry of smooth Riemannian manifolds, it is better to rule out Finsler spaces. The idea behind the concept of infinitesimal Hilbertianity is that, in some sense, if we were able to make analysis on non-smooth spaces as we are on smooth Riemannian manifolds, then in principle we would be able to deduce the same geometric properties. It is therefore natural to try to develop a Riemannian-like calculus, as opposed to a Finsler-like one. Then, notice that  a  non-smooth calculus needs, sooner or later, to integrate by parts and `integration by parts' is a concept tightly linked to that of `Sobolev functions': since the property `$W^{1,2}$ is Hilbert' is easily seen to identify Riemannian manifolds among Finsler ones, it is reasonable to focus on spaces having it. 
 
As we shall discuss, the class of $\RCD(K,N)$ spaces is the smallest known one that:
 \begin{enumerate}
  \item Contains the Riemannian manifolds with ${\rm Ric}\geq K$ and ${\rm dim}\leq N$,
  \item Is stable under (measured-Gromov-Hausdorff) convergence,
  \item Is stable under natural geometric operations, such as (warped) products, and quotients by group actions.  
  \end{enumerate}
Here good examples of the geometric stability mentioned in the last point are the papers \cite{AmbrosioGigliSavare12},  \cite{Erbar-Kuwada-Sturm13}, \cite{Ketterer13}, \cite{GKMS17} (see the more detailed discussion in Section \ref{se:pitagora}).
 
 Before all this, lower Ricci and upper dimension bounds were studied by Bakry-\'Emery in the abstract setting of Dirichlet forms (\cite{BakryEmery85}, see also the monograph \cite{BakryGentilLedoux14}). Roughly said, they started from the observation that the Bochner inequality 
 \begin{equation}
\label{eq:BEintro}
\tfrac12\Delta|\d f|^2\geq \la\d f,\d\Delta f\ra+K|\d f|^2\qquad\forall f\in C^\infty(M),
\end{equation}
that on a smooth Riemannian manifold $M$ characterizes $\Ric\geq K$, can be recast solely in terms of suitable algebraic properties of the Laplacian (notice that $|\d f|^2=\Delta(\tfrac{f^2}2)-f\Delta f$). With this idea in mind they used \eqref{eq:BEintro} to define Ricci curvature lower bounds for the infinitesimal generator of an abstract Dirichlet form. Even more, they have been the first to introduce the concept of Curvature-Dimension bound, i.e.\ to enforce the lower Ricci with  an upper dimension bound $N$   (in terms of \eqref{eq:BEintro} it  means to add the non-negative term $\tfrac{(\Delta f)^2}N$ at the right hand side). The main data of this approach is not a metric measure space, but rather some kind of differential structure encoded in the given form $\mathcal E$: as such it seems not so viable to study mGH-limits (but see the recent \cite{CMTkatoI} for a potentially game-changing approach - I will make some comments about this in Section \ref{se:bibliorcd}).

Inequalities \eqref{eq:cdintro} and \eqref{eq:BEintro} should be seen as `Lagrangian' and `Eulerian' version of the same concept, namely ${\rm Ric}\geq K$, with the former being the `integrated' version of the latter. Notice also that \eqref{eq:cdintro} speaks about displacement interpolation and $W_2$-geometry, while \eqref{eq:BEintro} is more related to classical affine interpolation and $L^2$-geometry. It is a non-trivial fact that in the $\RCD$ category both these viewpoints are available, the link between the two being the Heat Flow. Recall that it is a classical fact that the heat flow can be seen as the Gradient Flow of the Dirichlet energy w.r.t.\ the distance $L^2$ and that in the groundbreaking paper \cite{JKO98} (see also \cite{Otto01}) it has been observed that it is also the Gradient Flow of the Boltzmann-Shannon entropy w.r.t.\ the distance $W_2$. This fact  can be made rigorous using De Giorgi's concept of metric gradient flow  \cite{DeGiorgiMarinoTosques80},  \cite{DeGiorgi92}, \cite{DeGiorgi93} - see \cite{AmbrosioGigliSavare08} for a more modern presentation and refined results. Understanding whether this identification can be pushed up to the metric level was one of the reasons, another being the stimulus of conversations with Sturm, for which I started stydying the Heat Flow on $\CD(K,\infty)$ spaces. In the first paper on the topic (\cite{Gigli10}) it is showed that the $W_2$-Gradient Flow of the entropy is uniquely defined on $\CD(K,\infty)$ spaces and stable under convergence. All the stability results of differentiation operators in this setting  are ultimately a consequence of the stability of the heat flow. Later, in collaboration \cite{Gigli-Kuwada-Ohta10} with Kuwada and Ohta we obtained the  metric identification of such $W_2$-G.F. of the entropy with the $L^2$-G.F.\ of the Dirichlet energy, and used it to obtain the first proof of the Bochner inequality in a non-smooth setting, which in our case was that of compact, finite dimensional Alexandrov spaces.  I refer to the sections about bibliographical notes for more detailed references, here I just conclude pointing out that in the subsequent collaboration \cite{AmbrosioGigliSavare11-2} with Ambrosio and Savar\'e  an $\RCD$-type condition made the first appearance: in such paper we focused on the infinite-dimensional case  and the emphasis was on properties of the heat flow,  in line with the studies just mentioned.

\subsection{The focus of this survey}
The theory of $\CD/\RCD$ spaces is now a mature topic, so wide that it would be impossible to cover it in any reasonable level of detail with a limited number of pages. Over time, various surveys have been written, see \cite{AmbrosioGigliSavare-compact}, \cite{savareEMS}, \cite{AmbrosioICM}, \cite{Gigli13over}, \cite{Gigli17}, \cite{Villani2017}, \cite{Cavalletti17}, \cite{PozzettaSurvey} covering different aspects of the theory.

In this one, I focus on stability properties, that form a crucial reason for this whole theory to exist. In particular, this manuscript should not be taken as a survey over the whole theory of $\RCD$ spaces, as in order to keep it somehow concise  I had to make strong choices about the content to present. For instance, I left out all the discussions about variable/time dependent/distributional lower Ricci bounds, the localization technique  and all the recent  progresses about the fine structural properties of $\RCD$ spaces, just to mention a few.

There are four goals I aim to achieve, the first being:
\begin{quote}
1) To illustrate how $\Gamma$-convergence is `the correct notion to use' (whatever this means), when passing to the limit on functionals defined on a converging sequence of geometric structures.
\end{quote}
This might be obvious in some communities (and realized already in \cite{KS03conv}), but perhaps it is less so in others. Also, there are few interesting similarities between the two concepts of convergence, one being that, as I'll try to argue in Section \ref{se:weak}, they can be seen as the natural weak topologies in the class of lower semicontinuous functions and metric measure spaces, respectively. In this direction, it is worth to mention that from the historical point of view they both arose when De Giorgi on one side and Gromov on the other drastically lowered the regularity of the objects under consideration to focus on the sole key relevant property: lower semicontinuity in place of convexity-like properties for $\Gamma$-convergence and metric structure in place of the differentiable one for  convergence of spaces.

As an example of the point above, notice that the main stability result in this setting, namely Lott-Villani-Sturm's convergence of entropies along a mGH-converging sequence of spaces, can be cast, as done in \cite{Gigli10}, in terms of $\Gamma$-convergence of such entropies. Even more so, such $\Gamma$-convergence is \emph{equivalent} to the mGH-convergence of the underlying spaces, see Theorem \ref{thm:gammamgh} for the precise statement.

My second goal is:
\begin{quote}
2) To convince the reader that virtually \emph{any} lower semicontinuous functional defined on a (${\sf R}$)$\CD$ space remains so when defined  on a sequence of spaces satisfying a uniform lower Ricci bound.
\end{quote}
Let me try to illustrate this with an example. One functional we will consider is the Dirichlet energy $\ch:L^2(\X)\to[0,\infty]$, that in this setting takes the name of Cheeger energy and is the non-smooth analogue of $f\mapsto \tfrac12\int|\d f|^2\,\d\mm$. Such functional can be defined on an arbitrary metric measure space and is always lower semicontinuous. Also, once one has a mGH-converging sequence $\X_n\to\X_\infty$ it is possible to give a meaning to the concept of a sequence $n\mapsto f_n\in L^2(\X_n)$ converging to $f_\infty\in L^2(\X_\infty)$ in the `$L^2$-topology' (see Definition \ref{def:convl2var}). It turns out that if the $\X_n$ are all $\CD(K,\infty)$ spaces, then $f_n\stackrel{L^2}\to f_\infty$ implies $\ch(f_\infty)\leq\limi_n\ch(f_n)$. In the jargon of $\Gamma$-convergence, this means that the sequence of Cheeger energies satisfies a $\glimi$ inequality.

To put it differently, starting from $f\mapsto \ch(f)$ we can define a functional $(\X,f)\mapsto \widehat\ch(\X,f)$ taking spaces $\X$ and $L^2$ functions on them and returning the Cheeger energy of that function on that space. Then not only $\ch$ is $L^2(\X)$-lower semicontinuous on any space $\X$, but $\widehat\ch$ is lower semicontinuous on the space of $\{\CD(K,\infty)\, spaces,\, functions\}$ for any $K\in\R$. Notably, this makes it possible to apply the direct method of calculus of variations in situations where also the underlying space is varying.

This would fail without a lower Ricci bound: passage to the limit from discrete to continuum and homogeneization results provide standard counterexamples\footnotemark (and in the latter case uniform doubling and Poincar\'e inequalities are in place).\footnotetext{speaking of connections between $\Gamma$-convergence and convergence of metric (measure) spaces, it is worth to notice that the first results showing that `homogenising a periodic Riemannian metric we get  a Finsler norm' have been produced in the language of $\Gamma$-convergence of suitable kinetic energy functionals, see \cite{homofinsl}. This sort of convergence is in fact equivalent to Gromov-Hausdorff convergence of the induced metrics, see Theorem \ref{thm:ghgamma}.}

The phenomenon above is not limited to functionals defined on functions, but works as well for functionals on measures and tensors. Here linearity of the underlying differential operator plays no role: only the variational structure matters (in line with  general principles of $\Gamma$-convergence). For example, the principle $(2)$  is related to the convergence of flows of vector fields (see \cite{AST17} and Section \ref{se:convflow}) and to the lower semicontinuity of the Willmore energy functional, which is built from the non-linear second order operator $\div(\tfrac{\nabla u}{|\nabla u|})$ representing the mean curvature of the level sets of $u$ (see \cite{GV23} and Section \ref{se:willmore}).

It is hard to make rigorous the claim $(2)$ above, and thus to provide an actual proof of it. Still, the principle works in any example I'm aware of. Informally, the reason is that the differential operators, at least most of those considered in this setting, are all built starting from the differential $f\mapsto \d f$ of Sobolev functions (by taking adjoints, iterating etc..) and 
\begin{quote}
\emph{the differential of Sobolev functions is a closed operator, even along a converging sequence of $\RCD(K,\infty)$ spaces},
\end{quote} 
see Section \ref{se:defconvtens} for the rigorous meaning. Notice that even if the origin of this research field is in metric geometry, this concept of closure comes from functional analysis.  Such closure of the differential is in turn consequence (and in some sense equivalent) to the stability of the heat flow  mentioned earlier, for which the lower Ricci bound plays a crucial role.

My third, and perhaps most ambitious, goal is:
\begin{quote}
3) To give an outline of how the theory is built.
\end{quote}
It is a quite remarkable fact that starting only from the $W_2$-convexity of entropy and infinitesimal Hilbertianity, a rich and robust calculus can be built, to the extent that one can get the aforementioned stability results. In fact, even the possibility of speaking about, say, flows of vector fields or Hessian of a function in a fixed space is already non trivial. In order to give an account on how this is achieved, I had to make some  choices. First and foremost, I never present full proofs, but only give the main idea (often under additional simplifying assumptions, such as compactness, $K=0$, etc..). Hopefully, this should be sufficient for the expert reader to get convinced about the validity of the claim, while the less mature one might use the presentation here as a guide on what to focus when studying the actual source material. My hope is that with this choice it will be clear which concepts play a role and when, in which situation a certain technical lemma is useful, why a given definition appears etc... Another choice has been to present all the results in the setting of normalized (i.e.\ reference measure of unit mass) infinite-dimensional spaces. With this, by no means I'm implying that the extension to the case of finite dimensional spaces/spaces with infinite mass (that from the point of view of geometric applications are typically the most relevant ones) has minor importance: I made this choice because it was the only one that allowed me to stay in a reasonable amount of pages. Few details about how to extend the given results are given in the sections about bibliographical references.

Finally, since generalizing for the sole sake of it can be dangerous, my final goal is 
\begin{quote}
4) To present some applications of the theory to the study of smooth Riemannian manifolds.
\end{quote}
Some of these, such as the use of the splitting in $\RCD$ spaces to get informations about the almost quotient space, are expected  as in some sense `built in' the theory. Others, such as the sharp concavity  of the isoperimetric profile  or the variant of Bonnet-Myers estimate for positive and pinched Ricci curvature are, in my opinion, more surprising.

\medskip

\noindent{{\bf Acknowledgment}}
It would have been impossible to write this note without frequent conversations with a number of colleagues. In particular, I'd like to thank L. Benatti, M. Braun, E. Bru\`e, E. Caputo, F. Cavalletti, G. Dal Maso, M. Fogagnolo, S.\ Honda, J. Lott, A. Malchiodi, L. Mazzieri, I. Mondello, A. Mondino, F. Nobili, E. Pasqualetto, M. Pozzetta, G. Savar\'e, F. Schulze, D. Semola,  K.T. Sturm,  I.Y. Violo.

I wish also to thank L. Ambrosio  for a careful reading of a preliminary version of this manuscript.

Part of this work has been carried out while I was visiting the Fields Institute in Toronto: I wish to thank it for the warm hospitality and R.J. McCann for the neverending support and all the stimulating interactions.

\section{Two key non-linear weak topologies}\label{se:weak}
Here I introduce the two crucial notions of convergence we are going to deal with: $\Gamma$-convergence of lower semicontinuous functionals and measured-Gromov-Hausdorff convergence of metric measure spaces. In the attempt of highlighting the analogies between the two, I am calling  them weak topologies, this being related to the fact that compactness naturally arises in both cases, possibly after imposing suitable  `bounds on the size'. 

Of course it is easy to produce a compact topology on any set: just take the trivial one, or more generally a topology with `few' open sets. Choices of this kind have limited value for the working analyst, as  uniqueness of limits is typically relevant. In turn, such uniqueness is the same as requiring the topology to be Hausdorff\footnote{at least if we speak of limits of nets. If one only cares about sequences, equivalence holds if  the topology is first countable (which will always be in cases of interest for us).}, a condition that needs `many' open sets. These basic considerations should be coupled  with the classical topological rigidity:
\begin{equation}
\label{eq:rigtop}
\begin{split}
&\text{Given a compact and Hausdorff topology on a set,}\\
&\text{any  weaker topology is not Hausdorff and any  stronger topology is not compact}\footnotemark
\end{split}
\end{equation}
(`weaker' and `stronger' being intended in the strict sense),
\footnotetext{this is typically formulated in the equivalent way: a continuous bijection  from a compact space to a Hausdorff  is a homeomorphism.}to conclude that if we happen to have a compact and Hausdorff topology $\tau$ on a given set $\mathcal X$ which is compatible in some natural way with some additional structure on $\mathcal X$, then we might regard $\tau$ as `the correct' weak topology on $\mathcal X$.

\subsection{Weak convergence of lower semicontinuous functionals}

Let $(\X,\sfd)$ be a metric space. A functional $\E:\X\to\bar\R:=\R\cup\{\pm\infty\}$ is called \emph{lower semicontinuous} provided $\E(x)\leq \limi_{y\to x}\E(y)$ for every $x\in\X$. Since $``\sup\limi\leq\limi\sup"$, the supremum of an arbitrary family of lower semicontinuous functions is still lower semicontinuous and thus  any functional $\E$ admits a unique largest lower semicontinuous functional $\E^*$ among those $\leq \E$: it is called lower semicontinuous envelope of $\E$ and satisfes $\E^*(x)=\limi_{y\to x}\E(y)$.

\begin{definition}[$\Gamma$-convergence]\label{def:gammaconv}
Given    $\E_n:\X\to\bar\R$, $n\in\N$, we define the functionals ${\glimi_n} \E_n$ and $\glims_n \E_n$ on $\X$ as
\begin{equation}
\label{eq:defgamma}
\glimi_n\E_n(x):=\inf\limi_n\E_n(x_n)\qquad\text{ and }\qquad\glims_n\E_n(x):=\inf\lims_n\E_n(x_n),
\end{equation}
where in both cases the $\inf$ is taken among all sequences $(x_n)$ converging to $x$. 

If these coincide, calling $\E$ the common value  we say that $(\E_n)$ $\Gamma$-converges to $\E$ w.r.t.\ the topology induced by $\sfd$.
\end{definition}
In other words, $\E$ is the $\Gamma$-limit of the $\E_n$'s iff
\begin{subequations}\label{eq:glimeq}
\begin{align}\label{eq:glimi}
\E(x)&\leq\limi_n\E_n(x_n)&&\qquad\text{for any sequence $(x_n)$ converging to $x$,}\\
\label{eq:glims}
\E(x)&\geq \lims_n\E_n(x_n) &&\qquad\text{for some sequence $(x_n)$ converging to $x$.}
\end{align}
\end{subequations}
Inequalities \eqref{eq:glimi} and \eqref{eq:glims} are known as $\glimi$ and $\glims$ inequality, respectively, and a sequence $(x_n)$ as in \eqref{eq:glims} is called \emph{recovery sequence} for $x$.

A diagonalization argument shows that $\glimi_n\E_n(x)$ and $\glims_n\E_n(x)$ are always lower semicontinuous and that 
\[
\glimi_n\E_n(x)=\glimi_n\E_n^*(x)\qquad\text{ and }\qquad \glims_n\E_n(x)=\glims_n\E_n^*(x).
\]
In other words, that of $\Gamma$-convergence is a notion strictly related to lower semicontinuous functions.

%
Given $\E:\X\to\bar\R$, we define the \emph{effective domain}  $D(\E)\subset\X$ of $\E$ as $D(\E):=\{\E<+\infty\}$ (in applications, the functionals rarely take the value $-\infty$). Then \eqref{eq:glims} implies
\begin{equation}
\label{eq:linkdomini}
\E=\glim_n\E_n,\quad x\in D(\E)\qquad\Rightarrow\qquad \text{there is $n\mapsto x_n\in D(\E_n)$ converging to $x$}.
\end{equation}
There are no other direct links between $D(\E)$ and $D(\E_n)$. In particular we can have $D(\E)\cap D(\E_n)=\varnothing$ for any $n$ (consider $x_n\to x_\infty$ and  $\E_n$ to be equal to 0 on $x_n$ and to $+\infty$ elsewhere: these $\Gamma$-converge to the analogously defined functional $\E_\infty$).

Notice also that there is nothing linear in the concept of $\Gamma$-convergence - in particular it does not preserve sums\footnote{Still, if $\E_n\stackrel\Gamma\to\E_\infty$ then  $\E_n+{\sf F}\stackrel\Gamma\to\E_\infty+{\sf F}$ for ${\sf F}$  continuous. This tells that lower order perturbations do not affect $\Gamma$-convergence.}:
it is instead based on variational principles and is very much related 
to \emph{convergence of infima} (that under suitable compactness assumptions becomes \emph{convergence of minima}). For instance, with little work starting from the definition we see that
\begin{equation}
\label{eq:gvar}
\glimi_n\E_n(x)=\sup_{r>0}\limi_n\inf_{B_r(x)}\E_n\qquad\text{ and }\qquad \glims_n\E_n(x)=\sup_{r>0}\lims_n\inf_{B_r(x)}\E_n.
\end{equation}
On separable spaces, this characterization directly implies the \emph{compactness in $\Gamma$-convergence}:
\begin{equation}
\label{eq:compgamma}
\text{for any sequence $(\E_n)$ of functionals on $\X$ there is $(n_k)$ so that $(\E_{n_k})$ has a $\Gamma$-limit}.
\end{equation}
To see this, by diagonalization  select $(n_k)$ so that the limit of $\inf_U\E_{n_k}$ exists for any $U$ belonging to some countable base of $\X$ and then use \eqref{eq:gvar}.

\begin{remark}[The topology of $\Gamma$-convergence]\label{re:topgamma}{\rm Recalling the rigidity statement  \eqref{eq:rigtop} we can interpret the compactness in \eqref{eq:compgamma} above and the uniqueness of the $\Gamma$-limit as an abstract indication of the fact that the concept of $\Gamma$-convergence  is `the most natural one' on the space ${\sf LSC}(\X)$ of lower semicontinuous functions from $\X$ to $\bar\R$.

To corroborate this observation, we should also couple it with some link between $\Gamma$-convergence and the structure of lower semicontinuous functions: this is the convergence of infima alluded above.  More rigorously (and referring to \cite{DalMaso93} for all the proofs): there is a topology on ${\sf LSC}(\X)$ whose converging sequences are precisely those $\Gamma$-converging and it is the weakest one such that
\begin{equation}
\label{eq:Gammatop}
\begin{split}
&{\sf LSC}(\X)\ni f\quad\mapsto\quad \inf_Uf\in\bar \R\qquad\text{is upper semicontinuous for any $U\subset\X$ open},\\
&{\sf LSC}(\X)\ni  f\quad\mapsto\quad \inf_Kf\in\bar \R\qquad\text{is lower semicontinuous for any $K\subset\X$ compact}.
\end{split}
\end{equation}
Such topology is always compact. It is Hausdorff  iff $\X$ is locally compact. More generally, for $(\X,\sfd)$ arbitrary its restriction to any equicoercive subset $\mathcal C\subset{\sf LSC}(\X)$  is Hausdorff, and in fact even metrizable. Here $\mathcal C$ being  equicoercive means that $B\cap( \cup_{f\in\mathcal C}\{f\leq \alpha\})$ is relatively compact for any $\alpha\in\R$ and $B\subset\X$ bounded. If functions in $\mathcal C$ are also uniformly bounded from below, the topology of $\Gamma$-convergence on $\mathcal C$ can also be characterized as the weakest one such that 
\[
\mathcal C\ni f\quad\mapsto\quad Q_n f\in C_\b(\X) \qquad\text{is continuous for every $n\in\N$}
\]
where the space $C_\b(\X)$ of bounded continuous functions on $\X$ is equipped with the topology of uniform convergence on compact sets\footnote{in the literature it is often considered pointwise convergence of the $Q_nf$'s. Given that the functions $\{Q_nf:f\in\mathcal C\}$ are easily seen to be equicontinuous on compact subsets of $\X$ for any fixed $n$, pointwise convergence and uniform convergence on compact sets coincide. One might also notice that  by the rigidity  \eqref{eq:rigtop} and Ascoli-Arzel\`a's theorem, this topology is `the most natural one'  on    $\{Q_nf:f\in\mathcal C\}$ for any $n\in\N$.}, i.e.\ the compact-open topology, and $Q_nf$ is the  Moreau-Yosida approximation of $f$ defined as 
\[
\qquad Q_nf(x):=\inf_{y\in\X}f(y)+n\sfd(x,y),\qquad\forall x\in\X.
\]
This is in line with the general mantra that wants that under a uniform control on the size - here interpreted as size of  sublevels - a sequence weakly converges iff after regularization it strongly converges.  

Notice that uniform bounds from below are  quite common in applications\footnote{and in any case, for the purpose of the discussion just made one can always reduce to this case by post-composing functions in $\mathcal C$ with an increasing homemorphism $\psi:\bar \R\to[0,1]$: by the definition of $\Gamma$-topology given in \eqref{eq:Gammatop}, this map is a homeomorphism of ${\sf LSC}(\X)$ with its image} and so is  dealing with equicoercive family of functionals, where possibly compactness is intended w.r.t.\ some suitable weak topology (see also the discussion below). 
} \fr\end{remark}
It is worth to emphasize that the role of the distance $\sfd$ on the space $\X$ is limited to producing the concept of converging sequences of points: as soon as we have such a notion of convergence, the definition of $\Gamma$-convergence carries over.  In this note, the sort of convergences we shall encounter will either come from metrizable topologies (in the case of weak/$W_2$-convergence of measures or $L^2$-convergence of functions and tensors) or from topologies that are `metrizable un bounded sets' (in the case of weak $L^2$-convergence of functions and tensors).

Speaking of weak and strong convergences, it will happen to deal with both at the same time. In other words, in some circumstances we will have a \emph{strong} convergence, typically denoted by $x_n\to x_\infty$, and a \emph{weak} one, typically denoted $x_n\weakto x_\infty$. Much like in the familiar Banach setting,  strong convergence is related to `the correct' geometry at hand and implies weak convergence, while the latter is useful for its compactness properties. In this situation one can speak of $\Gamma$-limits  w.r.t.\ both strong and weak convergence : if the same functional $\E$ is both the strong $\Gamma$-limit and the weak $\Gamma$-limit of a sequence $\E_n$, one  says that the $\E_n$'s \emph{Mosco-converge} to $\E$.

Notice that since $x_n\to x_\infty$ implies $ x_n\weakto x_\infty$, we always have 
\[
{\rm weak-}\glimi_n\E_n\leq {\rm strong-}\glimi_n\E_n\qquad\text{ and }\qquad{\rm weak-}\glims_n\E_n\leq {\rm strong-}\glims_n\E_n
\]
and thus $(\E_n)$ Mosco-converge to $\E$ iff
\begin{subequations}
\begin{align}
\limi_n\E_n(x_n)&\geq \E(x)&&\qquad\text{for any sequence $(x_n)$ weakly converging to $x$,}\\
\lims_n\E_n(x_n)&\leq \E(x)&&\qquad\text{for some sequence $(x_n)$ strongly converging to $x$.}
\end{align}
\end{subequations}
Notice that the compactness statement \eqref{eq:compgamma} does not imply a similar statement for Mosco convergence, because the weak and strong $\Gamma$-limits may differ.

%
%
%
%
%
%
%
%
%
%
%
%
%
%

\subsection{Weak convergence of geometric structures}\label{se:mGH}
We shall be interested in convergence of \emph{normalized} metric measure spaces, i.e.\ triples $(\X,\sfd,\mm)$ where $(\X,\sfd)$ is a complete and separable metric space and $\mm$ is a Borel probability measure on it. We restrict to the case $\mm(\X)=1$ just for simplicity: the definitions can be adapted - without deep conceptual work - to pointed spaces equipped with measures giving finite mass to bounded sets. 

Given a  metric space $\Y$, by $C_\b(\Y)$ we denote the space of bounded continuous functions on $\Y$ and by $\pr(\Y)$ that of Borel probability measures. We say that $(\mu_n)\subset\pr(\Y)$ weakly converges to $\mu\in\pr(\Y)$, and write $\mu_n\weakto\mu$, provided
\[
\int\varphi\,\d\mu_n\ \qquad\to\ \qquad\int \varphi\,\d\mu\ \qquad\forall \varphi\in C_\b(\Y).
\]
I shall frequently use the fact that
\[
\mu_n\weakto\mu_\infty\text{ and $f$ lower semicontinuous bounded from below }\Rightarrow\limi_n\int f\,\d\mu_n\geq\int f\,\d\mu_\infty
\]
and analogously for upper semicontinuous functions bounded from above. To see why this holds, notice that any $f$ as in the claim is the supremum of an increasing sequence $(f_n)$ of bounded continuous functions, thus by monotone convergence we have $\int f\,\d\mu=\sup_n\int f_n\,d\mu_n$ and the claim follows as `the sup of a family of continuous functionals is lower semicontinuous'.
\begin{definition}[measured-Gromov-Hausdorff convergence]\label{def:mgh}
We say that the sequence $n\mapsto(\X_n,\sfd_n,\mm_n)$ of normalized spaces converges to the space $(\X_\infty,\sfd_\infty,\mm_\infty)$ in the measured-Gromov-Hausdorff sense, and write $\X_n\stackrel{mGH}\to\X_\infty$, if there is a complete and separable metric space $(\Y,\sfd_\Y)$ and isometric embeddings $\iota_n:\X_n\to\Y$, $n\in\N\cup\{\infty\}$, such that $(\iota_n)_*\mm_n\weakto(\iota_\infty)_*\mm_\infty$, i.e.\ such that
\[
\int\varphi\circ\iota_n\,\d\mm_n\qquad\to\qquad\int \varphi\circ\iota_\infty\,\d\mm_\infty\qquad\forall \varphi\in C_\b(\Y).
\]
The collection $(\Y,\sfd_\Y,(\iota_n))$ is called \emph{realization} of the convergence.
\end{definition}
The fact that the spaces $\X_n$ are approaching $\X_\infty$ can be visualized noticing that
\begin{equation}
\label{eq:vismgh}
\text{$\forall x\in\supp(\mm_\infty)$ there is $n\mapsto x_n\in\X_n$ such that $\iota_n(x_n)\to\iota_\infty(x_\infty)$,}
\end{equation}
where   $\supp(\mm)$ is the support of the measure $\mm$ i.e.\ the smallest closed set on which it is concentrated. Picking $\varphi\equiv1$ we see that $\X_\infty$ is normalized as well. It is worth pointing out that implicit in the definition above there is the choice of declaring $(\X,\sfd,\mm)$ and $(\X',\sfd',\mm')$ {\bf isomorphic} provided there is an isometry $\iota:\supp(\mm)\to\X'$ such that $\iota_*\mm=\mm'$\footnote{alternatively, one can declare two spaces isomorphic if there is a measure preserving isometry everywhere defined, in which case it is natural to impose also Hausdorff convergence of $\iota_n(\X_n)$ to $\iota_\infty(\X_\infty)$ in the definition of convergence. In spaces with a lower Ricci and an upper dimension bound, there is little difference between these two notions, basically thanks to the doubling property that is implied by such assumption (that in turn gives compactness in the Hausdorff sense). In the general case Definition \ref{def:mgh}  and the related concept of isomorphism seem more natural because on one hand imposing Hausdorff convergence is unnatural on non-compact spaces and on the other hand `what happens outside the support of the reference measure is irrelevant'.}. In particular, $(\X,\sfd,\mm)$ is always isomorphic to the set $\supp(\mm)$ equipped with the (restrictions of) $\sfd,\mm$.  It is not difficult to check that the above notion of convergence is invariant if we replace the $\X_n$'s, $n\in\N$, with isomorphic spaces (this might require enlarging the target space $\Y$ in order to extend the given isometries). 

More generally, in giving definitions concerning mm spaces, we should pay attention to their invariance under this notion of isomorphism: this will be self evident in every circumstance with the notable exception of Definition \ref{def:ch} (but see Lemma \ref{le:locmcsh}).

Since weak convergence of measures on metric spaces is metrizable, so is mGH-convergence. A natural choice is to start considering the Kantorovich distance 
\[
\mathbb K(\mu_1,\mu_2):=\inf_{\sggamma}\int 1\wedge\sfd(x_1,x_2)\,\d\ggamma,
\]
where the $\inf$ is taken among all \emph{transport plans} $\ggamma$ from $\mu_1$ to $\mu_2$, i.e.\ among $\ggamma\in\pr(\X^2)$ such that $\pi^1_*\ggamma=\mu_1$ and $\pi^2_*\ggamma=\mu_2$, where $\pi^i:\X^2\to\X $ is the projection onto the $i$-th coordinate: $\mathbb K$ metrizes the weak convergence on $\pr(\X)$ (see e.g.\ \cite{AmbrosioGigliSavare08}, \cite{Villani09} or the survey \cite{G11}). Then we put
\begin{equation}
\label{eq:dtheo}
\mathbb D\big((\X_1,\sfd_1,\mm_1),(\X_2,\sfd_2,\mm_2)\big):=\inf \mathbb K\big((\iota_1)_*\mm_1,(\iota_2)_*\mm_2),
\end{equation}
the $\inf$ being taken among all $(\Y,\sfd_\Y)$ and isometric immersions $\iota_i:\X_i\to\Y$.  It is then easy to see that $\mathbb D$ metrizes\footnote{It is clear  that $\mathbb D$ is symmetric, satisfies the triangle inequality and that $\mathbb D(\X,\X)=0$, but less so that  if $\mathbb D(\X_1,\X_2)=0$ then the two spaces are isomorphic. To see why, notice that by definition and  a `gluing' argument we can find a space $(\Y,\sfd_\Y)$ and isometric embeddings $\iota_1:\X_1\to\Y$ and $\iota_{2,n}:\X_2\to\Y$, $n\in\N$, with $(\iota_{2,n})_*\mm_2\weakto(\iota_1)_*\mm_1$.  To conclude it is sufficient to show that some subsequence $(\iota_{2,n_k})$ has a pointwise limit $\iota_2$, as then clearly $\iota_2$ is an isometry with $(\iota_2)_*\mm_2=(\iota_1)_*\mm_1$.  By equicontinuity, it suffices to find $(n_k)$ so that $k\mapsto \iota_{2,n_k}(x)$ has a limit for any $x$ belonging to some dense subset of $\X_2$. Fix $\eps>0$ and  use the tightness of $\{(\iota_{2,n})_*\mm_2:n\in\N\}$ to find $K\subset\Y$ compact with $((\iota_{2,n})_*\mm_2)(K)>1-\eps$ for every $n$. Then the set $A:=\cap_n\cup_{j\geq n}\iota_{2,j}^{-1}(K)$ of points $x$ so that $\iota_{2,n}(x)\in K$ for infinitely many $n$'s is so that $\mm_2(A)=\lim_n\mm_2(\cup_{j\geq n}\iota_{2,j}^{-1}(K))\geq \lims_n \mm_2(\iota_{2,n}^{-1}(K))=1-\eps$ (this is the Borel-Cantelli lemma) and the claim follows after suitable diagunalizations.}  mGH-convergence
 and, from the analogous properties of $\mathbb K$, that it is complete and separable. 

To understand  compactness we introduce the notion of $n$-space: $(\X,\sfd,\mm)$ is an $n$-space, $n\in\N$, provided both the cardinality and the diameter of $\X$ are $\leq n$. Since $\mathbb D$ is complete, finitely supported measures are weakly dense in $\pr(\X)$ and for fixed $n$ the collection of $n$-spaces is - rather trivially  - mGH-compact, we have the following (simple variant of) \emph{Gromov's precompactness theorem}:
\[
\begin{split}
&\text{a collection $\mathcal X$ of normalized mm-spaces is mGH-relatively sequentially compact iff} \\
&\text{$\forall \eps>0$ there is $n\in\N$ such that $\forall \X\in \mathcal X$ there is an $n$-space $\X'$ with $\mathbb D(\X,\X')<\eps$}.
\end{split}
\]
A sufficient condition for the above to hold is that $\mathcal X$ is made by spaces that are uniformly bounded and locally uniformly metrically doubling, the latter meaning that there are $R,N>0$ such that $B_{2r}(x)$ can be covered by $N$ balls of radius $r$ for any $r\in(0,R)$, $x\in\supp(\mm)\subset\X$ and $\X\in\mathcal X$. In turn, this latter condition is implied by local uniform doubling in the sense of measures, i.e.\ if there are $R,C>0$ such that 
\begin{equation}
\label{eq:doubmeas}
\text{there are $R,C>0\quad$ so that $\quad \mm(B_{2r}(x))\leq C\mm(B_r(x))\qquad \forall r\in(0,R)$, $x\in\X$}
\end{equation}
holds for every  $\X\in\mathcal X$.
It is then relevant to recall that the celebrated Bishop-Gromov inequality provides such uniform doubling on the class of Riemannian manifolds with uniform lower bound on the Ricci and upper bound on the dimension, as it implies that
\[
\frac{\vol(B_r(x))}{\vol(B_R(x))}\geq \frac{{\rm v}_{K,N}(r)}{{\rm v}_{K,N}(R)}\qquad\forall x\in M, \ 0< r\leq R
\]
provided the Riemannian manifold $M$ has ${\rm Ric}\geq K$ and $\dim\leq N$. Here ${\rm v}_{K,N}(r)$ is equal to $\omega_N\int_0^r({\sf s}_{K,N}(t))^{N-1}\,\d t$,  with $\omega_N=\frac{\pi^{ N/2}}{\int_0^\infty t^{N/2}e^{-t}\,\d t}$ and
\[
{\sf s}_{K,N}(t):=\left\{\begin{array}{ll}
\sqrt{\tfrac{N-1}K}\sin(t\sqrt{\tfrac K{N-1}}),&\qquad\text{ if }K>0,\\
t,&\qquad\text{ if }K=0,\\
\sqrt{\tfrac{N-1}{|K}}\sinh(t\sqrt{\tfrac {|K|}{N-1}}),&\qquad\text{ if }K<0.
\end{array}\right.
\]
In particular, for  $N\in\N$ we have that $\omega_N$ is the volume of the unit ball in $\R^N$ and ${\rm v}_{K,N}(r)$  the one of the ball of radius $r$ in the space form (=sphere/$\R^N$/hyperbolic space depending on the sign of $K$) of dimension $N$ and  Ricci curvature $K$.

All this discussion tells us that the collection of Riemannian manifolds with   ${\rm Ric}\geq K$ and $\dim\leq N$ is precompact w.r.t.\ mGH-convergence.

\begin{remark}{\rm In line with Remark \ref{re:topgamma}, these considerations about compactness and the rigidity \eqref{eq:rigtop} lead to the conclusion that the topology of mGH-convergence is the most natural weak topology on uniformly bounded and doubling metric measure spaces, for any given doubling constants and diameter bound.

It is then worth to recall that if one is not willing to impose any restriction on the class of spaces beside that of being normalized, then a different convergence, still introduced by Gromov, exists. It is called convergence in concentration and is very much related to the phenomenon of concentration of measures: the typical example of sequence converging in this sense but not in mGH-sense is that of the spheres $S^n$ with normalized volume, that converges in concentration to the one point space\footnote{heuristically, this is justified by the follows fact: any sequence $f_n:S^n\to\R$ of 1-Lipschitz maps with zero mean satisfies $(f_n)_*\mm_n\weakto\delta_0$, where $\mm_n$ is the normalized volume measure. Thus in some sense 1-Lipschitz functions on $S^n$ with $n\gg1$ are almost constant, much like Lipschitz functions on the one point space.}, but has no mGH-limit.

Convergence in concentration is induced by a distance, is weaker than mGH-convergence and coincides with this one on any given class of uniformly bounded and doubling metric measure spaces. The collection of all normalized mm spaces is relatively compact w.r.t.\ convergence in concentration: the compactification is described in terms of so-called `pyramids of mm spaces' and include `extended' spaces, i.e.\ spaces where the  distance can take the value $+\infty$ (see e.g.\ \cite{Shi16}, \cite{Shi17}). In principle, nothing excludes that the results collected here about mGH-convergence of ({\sf R})$\CD$ spaces can be extended to convergence in concentration: understanding up to what extent this is actually possible is a currently active research area (see e.g.\ \cite{OzaShi14}, \cite{OzaYok19}, \cite{KazShi20}, \cite{NakShi21}, \cite{NakShi21-2}, \cite{KazYok21}, \cite{KazYok22}).
}\fr\end{remark}
\begin{quote}
{\bf Why metric \emph{measure} spaces?}
\end{quote}

The reason for considering metric measure spaces, as opposed to  metric spaces, that are the natural category where to study lower/upper sectional curvature bounds, has to do with the properties of Ricci curvature. Informally, but without cheating, we can observe that Ricci curvature is defined as the trace of the sectional curvature and the `trace' operation is conceptually close to that of `averaging/integrating', which therefore involves some underlying measure. Also, many key inequalities one can deduce from lower Ricci bound involve not only the distance, but also the measure (e.g.\ Bishop-Gromov, Sobolev/Poincar\'e inequalities etc.). One might in principle think the relevant measure to be a function of the distance (just let it to be the Hausdorff measure of appropriate dimension), but in doing so one might run into troubles when studying stability properties. The problem is that if we pick a sequence of Riemannian manifolds with Ricci curvature uniformly $\geq K$, regard them as metric spaces and look at their limit in the Gromov-Hausdorff sense, then the limit, even if smooth, might fail to have Ricci curvature  $\geq K$. On the other hand if we regard them as metric-measure spaces and consider their limit in the measured-Gromov-Hausdorff sense, then the limit space, if smooth, is a weighted manifold (i.e.\ a manifold equipped with a measure $\mm:=e^{-V}\vol$ possibly different from the volume one) and its weighted Ricci curvature
\begin{equation}
\label{eq:BEricci}
\Ric_\mm:=\Ric+\Hess\, V
\end{equation}
is  $\geq K$\footnote{for an explicit example of this phenomenon consider the sequence of spherical suspensions with cross section $\tfrac1n S^2$. These all have $\Ric\geq 2$ (at least far from the tips - at the tips one can either regularize the manifolds or use the results in \cite{Ketterer13} to directly deal with singular spaces) and converge as metric spaces to the flat interval $[0,\pi]$, which certainly has not $\Ric\geq 2$. On the other hand, keeping track of the measures we see that after normalization these spaces converge to the interval equipped with the measure   $\sin^2(t)\d t=e^{-(-2\log\sin(t))}\d t$ and by direct computation we see that $-2\log\sin(t)''\geq 2$, so that the weighted Ricci curvature of the limit interval is $\geq 2$.}. To see why the Ricci curvature should be affected by the choice of the measure, notice, beside the above, that it can be defined as the tensor for which the Bochner identity
\[
\tfrac12\Delta {|\d f|^2}=|\Hess f|^2_{\HS}+\la\d f,\d\Delta f\ra+\Ric(\nabla f,\nabla f)
\]
holds and in turn the choice of the measure affects the notion of Laplacian because $\Delta f$ is (or can be) defined via integration by parts as $\int g\Delta f\,\d\mm=-\int\la\d f,\d g\ra\,\d\mm$ for all $g$'s sufficiently smooth.

\subsection{How these interact} Say that we have a sequence $(\X_n,\sfd_n,\mm_n)$ that is mGH-converging to $(\X_\infty,\sfd_\infty,\mm_n)$, with $(\Y,\sfd_\Y,(\iota_n))$ being a realization of such convergence (in principle the choice of the realization affects the notions of convergences we are going to discuss, but in reality this is not really the case, see Remark \ref{re:nondipende}). Then we can say that a sequence $n\mapsto x_n\in\X_n $, that in principle is made of points belonging to different spaces, is converging to $x_\infty\in\X_\infty$ provided $\iota_n(x_n)\to \iota_\infty(x_\infty)$ in the space $\Y$. Now suppose we are also given functionals $\E_n:\X_n\to\bar\R$, $n\in\N\cup\{\infty\}$. Then we can say that $(\E_n)$ $\Gamma$-converges to $\E_\infty$ provided the relations \eqref{eq:glimeq} hold w.r.t.\ this notion of convergence of points.

This is the same as saying that $\hat\E_n\stackrel{\Gamma}\to\hat\E_\infty$ as functionals on $\Y$, where $\hat\E_n:\Y\to\bar\R$ is defined as $\hat \E_n(y)=\E_n(x)$ if $y=\iota_n(x)$ and $\hat\E_n(y)=+\infty$ if $y\notin\iota_n(\X_n)$.

It is worth to emphasize that since in general the sets $\iota_n(\X_n)$ are all disjoint, what makes this concept non-trivial is the combination  of the property \eqref{eq:vismgh} of mGH-convergence and \eqref{eq:linkdomini} of $\Gamma$-convergence (for the sake  of comparison: one could in principle study the pointwise limit of the $\hat\E_n$'s, but this will typically be identically $+\infty$).

\medskip

In practice, we will never consider convergence of functionals defined on the base spaces $\X_n$, but rather on some space built over them,  such as that of probability measures, of $L^2$ functions or $L^2$ tensors: this will make sense as soon as we shall use the concept of mGH-convergence of the $\X_n$'s to give a meaning to convergence of  measures/functions/tensors defined on different spaces.

Let us start with the case of measures:
\begin{definition}[Weak convergence of measures in varying spaces] Let $(\X_n,\sfd_n,\mm_n)\stackrel{mGH}\to(\X_\infty,\sfd_\infty,\mm_\infty)$ and $(\Y,\sfd_{\Y},(\iota_n))$ be a realization of it. Also, let $\mu_n\in\pr(\X_n)$, $n\in \N\cup\{\infty\}$. 

We  say that $\mu_n\weakto \mu_\infty$ if  $(\iota_n)_*\mu_n\weakto(\iota_\infty)_*\mu_\infty$, that is to say  if
\[
\int \varphi\circ\iota_n\,\d\mu_n\quad\to\quad\int \varphi\circ\iota_\infty\,\d\mu_\infty\qquad\forall \varphi\in C_\b(\Y).
\]
\end{definition}
With this concept at hand we can define $\Gamma$-convergence of functionals  $\E_n:\pr(\X_n)\to\bar \R$ w.r.t.\ weak convergence of measures: we say that $\E_n\stackrel{\Gamma}\to\E_\infty$ if the relations \eqref{eq:glimeq} hold, being understood that the convergence of points (i.e.\ measures in the current case) is the weak convergence just defined. Notice that, as before, this is the same as requiring the $\Gamma$-convergence of $(\hat\E_n)$ to $\hat\E_\infty$ in $\pr(\Y)$, where $\hat\E_n(\mu):=\E_n(\nu)$ if $\mu=(\iota_n)_*\nu$ for some $\nu\in\pr(\X_n)$ and $\hat\E_n(\mu):=+\infty$ if $\mu\notin(\iota_n)_*\pr(\X_n)$.

A typical functional on measures we shall consider is the Boltzmann-Shannon relative entropy: given $(\X,\sfd,\mm)$ we define $\ent:\pr(\X)\to[0,+\infty]$ as
\[
\ent(\mu):=\left\{\begin{array}{ll}
\displaystyle{\int\rho\log\rho\,\d\mm},&\qquad\text{ if }\mu=\rho\mm,\\
+\infty,&\qquad\text{ if }\mu\not\ll\mm.
\end{array}
\right.
\]
We then have the following result:
\begin{theorem}\label{thm:gammamgh}  $(\X_n,\sfd_n,\mm_n)\stackrel{mGH}\to(\X_\infty,\sfd_\infty,\mm_\infty)$ if and only if $\entm{\mm_n}\stackrel{\Gamma}\to \entm{\mm_\infty} $.

More precisely, given normalized spaces $(\X_n,\sfd_n,\mm_n)$, $n\in\N$, and isometric embeddings $\iota_n$ of these in a common complete and separable space $(\Y,\sfd_\Y)$ we have:
%
\[
\text{ $(\iota_n)_*\mm_n\weakto (\iota_\infty)_*\mm_\infty\qquad$ if and only if $\qquad\entm{(\iota_n)_*\mm_n}\stackrel{\Gamma}\to \entm{(\iota_\infty)_*\mm_\infty}$,}
\]
the $\Gamma$-convergence being w.r.t.\ the weak convergence of measures in $\pr(\Y)$.
\end{theorem}
\begin{idea} Identify the spaces $\X_n$ with their isometric image in $\Y$. \ \\
\noindent\underline{mGH$\Rightarrow\Gamma$} Let $u(z):=z\log z$ and $u^*(w):=\sup_zzw-u(z)=e^{w-1}$ be its Fenchel-Legendre transform. Then one  proves 
\begin{equation}
\label{eq:dualent}
\entm\mm(\mu)=\sup_{\varphi\in C_\b(\X)}\int \varphi\,\d\mu-\int u^*\circ\varphi\,\d\mm,
\end{equation}
where the inequality $\geq$ follows by the dual identity $u(z)=\sup_wzw-u^*(w)$ and the inequality $\leq $ by approximating the optimal choice $\varphi:=\log\rho+1$ with bounded continuous functions. Then one notices that for $\varphi\in C_\b(\X)$ we have $ u^*\circ\varphi\in C_\b(\X)$ as well, so that the RHS of the above is weakly continuous in both $\mu$ and $\mm$. This implies  the joint lower semicontinuity in $\mu,\mm$ of $\entm\mm(\mu)$, which is the $\Gamma-\limi$ condition.

For the  $\Gamma-\lims$ it is sufficient to consider the case of $\mu_\infty\in \pr(\X_\infty)$ with finite entropy, say in particular $\mu_\infty=\rho\mm_\infty$. With an approximation argument we can find $\rho_k\in C_\b(\Y)$ non-negative so that $\rho_k\mm_\infty$ are probability measures with $\rho_k\mm_\infty\weakto\mu_\infty$ and $\entm{\mm_\infty}(\rho_k\mm_\infty)\to\entm{\mm_\infty}(\mu)$. Then for every $k$ the measures $c_{k,n}\rho_k\mm_n$, with $c_{k,n}$ normalizing constant, are a recovery sequence for  $\rho_k\mm_\infty$. The conclusion follows by diagonalization.

\noindent\underline{$\Gamma\Rightarrow$mGH} Let $\mu_\infty=\rho\mm_\infty$ be so that both $\rho$ and $\varphi:=\log\rho+1$ are in  $C_\b(\X_\infty)$. Then   $\entm{\mm_\infty}(\mu_\infty)=\int \varphi\,\d\mu_\infty-\int \rho\,\d\mm_\infty$ and for  $(\mu_n)$  recovery sequence, by  \eqref{eq:dualent} we get
\[
\lims_n\int \varphi\,\d\mu_n-\int \rho\,\d\mm_n\leq\lims_n\entm{\mm_n}(\mu_n)\leq\entm{\mm_\infty}(\mu_\infty)=\int \varphi\,\d\mu_\infty-\int \rho\,\d\mm_\infty.
\]
Since $\mu_n\weakto\mu_\infty$ it follows that $\limi_n\int \rho\,\d\mm_n\geq \int \rho\,\d\mm_\infty$. Replacing $\rho$ with (the normalization of) $\|\rho\|_{\infty}+1-\rho$ and using the arbitrariness of $\rho$ we conclude.
\end{idea}
There is nothing special about the relative entropy, and thus about the function $u(z)=z\log z$ in the above result: analogue statements hold for other choices of strictly convex  $u$.
\begin{remark}\label{re:nondipende}{\rm The choice of the realization $(\Y,\sfd_\Y,(\iota_n))$ of the convergence affects the notion of `converging sequence of points' and the subsequent one  of `converging sequence of measures'. Thus in principle it might affect the $\Gamma$-limit one is considering. In practice, however, this does not happen and the situation as in Theorem \ref{thm:gammamgh} above is typical: the $\Gamma$-limit is independent on the realization. It is so because in presence of two different realizations and two different $\Gamma$-limits $\E$ and $\E'$, there typically  is an automorphism of the limit space $(\X_\infty,\sfd_\infty,\mm_\infty)$ sending $\E$ in $\E'$.
}\fr\end{remark}
I shall discuss later on convergence of functions and tensors in varying spaces. In the rest of this section I present a purely metric analogue of the result above which, although irrelevant for our purposes, is interesting on its own.

Recall that a curve $\gamma:[0,1]\to\X$ on a metric space $\X$ is absolutely continuous if there is $f\in L^1([0,1])$ such that
\begin{equation}
\label{eq:defac}
\sfd(\gamma_t,\gamma_s)\leq\int_t^sf(r)\,\d r\qquad\forall t,s\in[0,1],\ t<s.
\end{equation}
Clearly, for any such $f$ we have $\lims_{h\to0}\frac{\sfd(\gamma_{t+h},\gamma_t)}{|h|}\leq f(t)$ for a.e.\ $t$. On the other hand, letting $(x_n)\subset\X$ be dense and putting $g_n(t):=\sfd(\gamma_t,x_n)$, we see that the $g_n$'s are equi-absolutely continuous and the identity $\sfd(\gamma_t,\gamma_s)=\sup_n g_n(s)-g_n(t)\leq \int_t^s\sup_ng_n'(r)\,\d r$ shows that the choice $\sup_ng_n'(t)$ is admissible in \eqref{eq:defac} (and clearly $\leq f$ for a.e.\ $t$). Since $\limi_{h\to 0}\frac{\sfd(\gamma_{t+h},\gamma_t)}{|h|}\geq\limi_{h\to 0}\frac{|g_n(t+h)-g_n(t)|}{|h|}=|g_n'|(t)$ for any $n$ and a.e.\ $t$, we conclude that 
\begin{equation}
\label{eq:metrsp}
\text{the  \emph{metric speed of $\gamma$}, defined as the limit }\  |\dot\gamma_t|:=\lim_{h\to 0}\frac{\sfd(\gamma_{t+h},\gamma_t)}{|h|},\ \text{ exists for a.e.\ $t$}
\end{equation}
and is the least, in the a.e.\ sense, admissible function $f$ in \eqref{eq:defac}. 

Having a notion of metric speed allows to introduce the kinetic energy functional. Let $C([0,1],\X)$ be the space of continuous curves on $\X$ equipped with the `sup' distance and notice that it is complete and separable as soon as $\X$ is so. Then the kinetic energy ${\sf KE}:C([0,1],\X)\to[0,\infty]$ is defined  as
\[
{\sf KE}(\gamma):=\left\{\begin{array}{ll}
\displaystyle{\tfrac12\int_0^1 |\dot\gamma_t|^2\,\d t},&\qquad\text{ if $\gamma$ is absolutely continuous},\\
+\infty,&\qquad\text{ otherwise.}
\end{array}
\right.
\]
A \emph{geodesic} is a curve $\gamma$ such that $\frac12\sfd^2(\gamma_0,\gamma_1)={\sf KE}_{\X}(\gamma)$ (notice that the inequality $\leq$ always holds) or, equivalently, such that
\begin{equation}
\label{eq:defgeod}
\sfd(\gamma_t,\gamma_s)=|s-t|\sfd(\gamma_0,\gamma_1)\qquad\forall t,s\in[0,1].
\end{equation}
$(\X,\sfd)$ is called geodesic if for any two given points there is a geodesic connecting them.

\begin{definition}[Gromov-Hausdorff distance] The Hausdorff distance $\sfd_{H}(A,B)$ between two compact subsets $A,B$ of a metric space $\Y$ is defined as
\[
\sfd_H(A,B):=\inf\{r>0\ :\ A\subset B^r\quad\text{ and }\quad B\subset A^r\},
\] 
where $A^r:=\{y\in \Y:\sfd(y,A)<r\}$. 

Then the Gromov-Hausdorff distance between two compact metric spaces $(\X_0,\sfd_0)$, $(\X_1,\sfd_1)$ is defined as
\[
\sfd_{GH}(\X_0,\X_1):=\inf \sfd_H(\iota_0(\X_0),\iota_1(\X_1))
\]
where the infimum is taken among all compact spaces $(\Y,\sfd_\Y)$ and isometric embeddings $\iota_i$ of $\X_i$ in $\Y$, $i=0,1$.
\end{definition}
With a suitable `gluing' argument one can see that if $(X_n,\sfd_n)\stackrel{GH}\to (\X_\infty,\sfd_{\infty})$, then there is a compact space $(\Y,\sfd_\Y)$ and isometric embeddings $\iota_n:\X_n\to\Y$, $n\in \N\cup\{\infty\}$ so that $\sfd_H(\iota_n(X_n),\iota_\infty(X_\infty))\to 0$.

We have mentioned that $\Gamma$-convergence is related to convergence of infima and noticed that in geodesic spaces, by definition, the infimum of the kinetic energy among curves with fixed endpoints is the  squared distance. Because of this and the similarity with Theorem \ref{thm:gammamgh} above, the following result is perhaps not surprising:
\begin{theorem}\label{thm:ghgamma} Let $(\X_n,\sfd_n)$, $n\in \N\cup\{\infty\}$, be compact and geodesic. 

Then $(\X_n,\sfd_n)\stackrel{GH}\to(\X_\infty,\sfd_\infty)$ if and only if ${\sf KE}_{\X_n}\stackrel{\Gamma}\to {\sf KE}_{\X_\infty}$ w.r.t.\ uniform convergence.

More precisely, given $(\X_n,\sfd_n)$ as above, $(\Y,\sfd_\Y)$ compact and $\iota_n:\X_n\to\Y$ isometric embeddings, $n\in\N$, we have 
\[
\text{$\sfd_H(\iota_n(X_n),\iota_\infty(X_\infty))\to 0\qquad$ if and only if $\qquad\hat{\sf KE}_{\X_n}\stackrel{\Gamma}\to \hat{\sf KE}_{\X_\infty}$ in $C([0,1],\Y)$}
\]
(here $\hat{\sf KE}_{\X_n}(\gamma)$ is set to be ${\sf KE}_{\X_n}(\eta)$ if $\gamma=\iota_n\circ\eta$ for some $\eta\in C([0,1],\X_n)$ and $+\infty$ otherwise).
\end{theorem}
\begin{idea}Identify the spaces $\X_n$ with their isometric image in $\Y$. \ \\
\noindent\underline{GH$\Rightarrow\Gamma$} The key is the formula
\begin{equation}
\label{eq:keequiv}
{\sf KE}(\gamma)=\sup\sum_{i=0}^{n-1}\frac{\sfd^2(\gamma_{t_i},\gamma_{t_{i+1}})}{2(t_{i+1}-t_i)},
\end{equation}
the sup being taken among all $n$ and among all choices $0=t_0<\cdots<t_n=1$. Here $\geq$ is obvious, while $\leq$ follows noticing that if the RHS is finite, the functions $f_n:[0,1]\to\R^+$ defined as $n\sfd(\gamma_{\frac in},\gamma_{\frac {i+1}n})$ on $[\tfrac in,\tfrac{i+1}n]$ are uniformly bounded in $L^2$ and any weak limit $f$ is admissible in \eqref{eq:defac}. 

For any partition, the sum in the RHS of \eqref{eq:keequiv} is continuous w.r.t.\ uniform (or even pointwise) convergence of $\gamma$, thus ${\sf KE}$ is lower semicontinuous on $C([0,1],\Y)$ and the $\glimi$ inequality follows. For the recovery sequence, let $\gamma$ be such that $\hat{\sf KE}_{\X_\infty}(\gamma)<\infty$, pick a partition $(t_i)$ almost maximizing for the RHS of \eqref{eq:keequiv} and use the GH-convergence to find points $x_{i,n}\in \X_n$ converging to $\gamma_{t_i}$. Since the $\X_n$'s are geodesic we can find a piecewise geodesic curve connecting them and it is clear that such curve has kinetic energy equal to $\sum_{i=0}^{n-1}\frac{\sfd^2(x_{i,n},x_{i+1,n})}{2(t_{i+1}-t_i)}$. The claim follows.

\noindent\underline{$\Gamma\Rightarrow$GH} Via a compactness argument one can reduce to proving that
\[
\begin{split}
&\text{if $n\mapsto x_n\in\X_n$ converges to some $x_\infty$, we must have $x_\infty\in\X_\infty$,}\\
&\text{for any $x_\infty\in\X_\infty$ there is $n\mapsto x_n\in\X_n$ converging to it.}
\end{split}
\]
Both are trivial: for the first we pick the curves constantly equal to $x_n$, that have 0 energy and converge uniformly to the curve constantly $x_\infty$. Thus this limit curve has 0, hence finite, energy which forces $x_\infty\in\X_\infty$. For the second, for any recovery sequence $(\gamma_n)$ of the curve constantly $x_\infty$   we must eventually have  $\gamma_n([0,1])\subset\X_n$.
\end{idea}
\subsection{Bibliographical notes}
{\footnotesize{The concept of $\Gamma$-convergence has been introduced in \cite{DeGiorgiGamma}. Standard references for this topic are \cite{DalMaso93}, \cite{Braides02}, \cite{Foc12}.

The concept of convergence of metric spaces has been introduced in \cite{Gro81} (see also the more recent edition \cite{Gromov07}).  A standard reference for this topic is \cite{BBI01}. The relevance of the measure, in connection with spectral convergence under lower Ricci curvature bounds, has been realized in \cite{fukaya}. 

Notice that in the literature there are two \emph{different} competing concepts of `measured-Gromov-Hausdorff convergence', depending on weather one also insists on Hausdorff convergence of $\iota_n(\X_n)$ to $\iota_\infty(\X_\infty)$ inside the large space $\Y$, or one only cares about weak convergence of $(\iota_n)_*\mm_n$ to $(\iota_\infty)_*\mm_\infty$. The former approach has been proposed in \cite{fukaya}, but  in this presentation I chose the latter one, that comes from  \cite{Sturm06I}, where also the distance \eqref{eq:dtheo} has been introduced. Heuristic arguments in favour of this choice are:
\begin{itemize}
\item[-] What happens in regions with no mass should be irrelevant.
\item[-]  On non-compact setting  it seems unnatural to insist on Hausdorff  convergence (even if the spaces are  bounded).
\item[-]  If $(\mu_n)\subset\pr(\X)$ is a sequence converging to a Dirac delta $\delta_{\bar x}$, it seems natural to declare $(\X,\sfd,\mu_n)$ to be converging to the one point space, as the mass far from the point is vanishing. Not doing so would require considering the spaces $(\X,\sfd,\delta_{\bar x})$ and $(\X,\sfd,\delta_{\bar y})$ different for $\bar x\neq\bar y$, which  would be in contrast with the first point above.
\end{itemize}
Prior to \cite{Sturm06I}, the standard approach was instead to also insist on Hausdorff convergence and this has also been the choice adopted in \cite{Lott-Villani09} and \cite{Villani09}. Notice in any case that on locally doubling spaces (as in \eqref{eq:doubmeas}) the measure is always of full support and that if such doubling condition holds uniformly along a given sequence, then the two notions of convergences agree. This somehow explains why in connection with lower Ricci bounds and with geometric applications in mind, where finite dimensionality and thus the Bishop-Gromov inequality are present, the two concepts are rarely actively distinguished.

Even though $\Gamma$-convergence was not explicitly mentioned, the implication ``mGH-convergence implies $\Gamma$-convergence of the entropies" is THE key convergence statement proved in \cite{Lott-Villani09}  and  \cite{Sturm06I}. All the convergence results presented in this note (with the notable exception of Theorem \ref{thm:glimsch}) depend on this one. The converse implication appears here for the first time:   I doubt it will ever be useful, but I think it provides an interesting link between these two seemingly different concepts of convergence.  

The concept of metric speed and all the properties recalled around the claim \eqref{eq:metrsp} come from \cite{Ambr90}. Theorem \ref{thm:ghgamma}  appears here for the first time, but at least the $\glimi$ inequality, i.e.\ the lower semicontinuity of the kinetic energy, is a well established fact in metric geometry (and the $\glims$ is trivial).

Theorem \ref{thm:gammamgh} is reminiscent of an analogous result obtained by Mariani in \cite{mariani} linking ``$\Gamma$-convergence of the rescaled entropies $\frac1n\entm{\mu_n}$ to $\int I\,\d\cdot$" to ``the measures $(\mu_n)$ obey a Large Deviation Principle with rate function $I$" (compare also \eqref{eq:Gammatop} with the definition of Large Deviation Principle).
}
}

\section{The role of optimal transport in the stability issue}\label{ch:stab1}
\subsection{The Curvature Dimension condition}
I start recalling few basic facts about Optimal Transport.

For $(\X,\sfd)$ complete and separable we denote by $\prd(\X)\subset\pr(\X)$ the collection of measures $\mu$ with finite second moment, i.e.\ such that $\int\sfd^2(x,\cdot)\,\d\mu<\infty$ for some, and thus any, $x\in\X$. For $\mu,\nu\in\prd(\X)$ the quadratic Kantorovich-Wasserstein distance $W_2$ is defined as
\begin{equation}
\label{eq:defw2}
W_2^2(\mu,\nu):=\min\int \sfd^2(x,y)\,\d\ggamma(x,y),
\end{equation}
the minimum being among all $\ggamma\in\pr(\X^2)$ with $\pi^1_*\ggamma=\mu$ and $\pi^2_*\ggamma=\nu$. $(\prd(\X),W_2)$ is complete and separable (and geodesic) if $(\X,\sfd)$ is complete and separable (and geodesic). If $\X$ is bounded, $W_2$ metrizes the weak convergence.  The distance $W_2$ admits the  dual formulation
\begin{equation}
\label{eq:dualw2}
\tfrac12W_2^2(\mu,\nu)=\sup\int\varphi\,\d\mu+\int\varphi^c\,\d\nu\qquad\text{ where }\qquad \varphi^c(y):=\inf_{x}\tfrac{\sfd^2(x,y)}2-\varphi(x)
\end{equation}
the sup being taken among all $\varphi:\X\to\R$ bounded and Lipschitz\footnote{a maximizer $\varphi$ exists in the class of unbounded and non-Lipschitz functions. Replacing such $\varphi$ with $\varphi\wedge n$ for $n\gg1$ we get a maximizing sequence (because $(\varphi\wedge n)^c\geq\varphi^ c$ and $\int \varphi\wedge n\,\d\mu\uparrow\int \varphi\,\d\mu$ by monotone convergence) made of functions bounded from above. Then it is easy to see that $\varphi_n:=(n\wedge (\varphi\wedge n)^c)^c$ is still a maximizing sequence and that each $\varphi_n$ is bounded and Lipschitz.}. Finally, an admissible plan $\ggamma$ in \eqref{eq:defw2} is optimal if and only if there is a maximizer $\varphi$ in \eqref{eq:dualw2} with $\varphi=\varphi^{cc}$ so that $\supp(\ggamma)\subset\partial^c\varphi$, where $\partial^c\varphi:=\{(x,y):\varphi(x)+\varphi^c(y)=\tfrac12\sfd^2(x,y)\}$. Such maximizers $\varphi$ are called \emph{Kantorovich potentials} for $(\mu,\nu)$.
\begin{definition}\label{def:cdki}
A functional $\E:\X\to\R\cup\{+\infty\}$ is called $K$-geodesically convex (or simply $K$-convex) if for any $x,y\in D(\E)$ there is a geodesic $\gamma$ connecting them with 
\begin{equation}
\label{eq:defkconv}
\E(\gamma_t)\leq(1-t)\E(\gamma_0)+t\E(\gamma_1)-\tfrac12Kt(1-t)\sfd^2(\gamma_0,\gamma_1)\qquad\forall t\in[0,1].
\end{equation}
\end{definition}
The following is a crucial result for this note, as it justifies all what comes next. It has been proved in \cite{vRS05} (see also \cite{OV00} and \cite{CEMCS01}); I omit the proof, that is based on Jacobi fields analysis and a  precise description of $W_2$-geodesics.
\begin{theorem}\label{thm:sturmvonrenesse} Let $M$ be a smooth, connected, complete Riemannian manifold without boundary, $\sfd$ the induced distance and $\vol$ the volume measure. Then the following are equivalent:
\begin{itemize}
\item[i)] The Ricci curvature of $M$ is  $\geq K$ (i.e.\ $\Ric(v,v)\geq K|v|^2$ for any $v\in T_xM$ and $x\in M$)
\item[ii)] The relative entropy $\entm\vol$ is $K$-geodesically convex on $(\prd(M),W_2)$\footnote{I haven't discussed the definition of entropy for reference measures with infinite mass: this is potentially troubling as both the integral of both the positive and negative parts of $\rho\log\rho$ could be $+\infty$. A way to solve the problem is to define the entropy only on measures with bounded support, that form a convex set in $(\prd(M),W_2)$ and for which the contribution of the negative part is always finite. Then if $(ii)$ holds one obtains suitable volume growth estimates on $\vol$ (see \cite[Theorem 4.26]{Sturm06I}) that in turn ensure that the entropy is well defined and lower semicontinuous on the whole $(\prd(M),W_2)$ (see e.g.\ \cite[Section 4.1.1]{GMS15}). Still, in this note I will restrict the attention to normalized spaces for simplicity.}.
\end{itemize}
\end{theorem}
This theorem characterizes lower Ricci bounds by means that do not call into play any derivative. As such, this characterization can be used to define lower Ricci curvature bounds on low regularity structures: the following definition has been independently proposed in \cite{Lott-Villani09} and \cite{Sturm06I}.
\begin{definition}[The Curvature Dimension condition]\label{def:cd}
A (normalized) space $(\X,\sfd,\mm)$ is called a $\CD(K,\infty)$ space, $K\in\R$, if the relative entropy functional $\ent$ is $K$-convex on $(\prd(\X),W_2)$. 
\end{definition}

A direct consequence of Theorem \ref{thm:gammamgh} is the stability of the $\CD(K,\infty)$ condition. To handle the proof in the case of $\Y$ non-compact, we shall need two remarks. The first is the tightness criterium
\begin{equation}
\label{eq:tightcr}
\mm_n\weakto\mm_\infty\qquad\Rightarrow\qquad  \bigcup_n\{\entm{\mm_n}\leq C\}\quad\text{is tight for every }C>0.
\end{equation}
This can be proved noticing that $u(z)=z\log(z)\geq-1$, so Jensen's inequality yields $\ent(\mu)\geq-1+\mu(E)\log(\tfrac{\mu(E)}{\mm(E)})$ for any $\mu,\mm\in\pr(\Y)$ and $E$ Borel. Thus a bound on the entropy and a smallness condition on $\mm(E)$ forces smallness of $\mu(E)$; the claim then follows from the tightness of the weakly converging sequence $(\mm_n)$.

The second is a slight improvement over Theorem \ref{thm:stabcd}:
\begin{equation}
\label{eq:recw2}
\text{in Theorem \ref{thm:stabcd}, if $\mu_\infty\in\pr_2(\Y)$, there is a recovery sequence $(\mu_n)$  with $W_2(\mu_n,\mu_\infty)\to0$.}
\end{equation}
To see this, just notice that  the densities $\rho_k$  in the proof of the $\glims$ in  Theorem \ref{thm:gammamgh} can be chosen with bounded support and recall that weak- and $W_2$- convergences agree on measures with uniformly bounded supports.
\begin{theorem}\label{thm:stabcd}
The $\CD(K,\infty)$ condition is stable under mGH-convergence.
\end{theorem}
\begin{idea} Let $\X_n\stackrel{mGH}\to \X_\infty$ be with realization $\Y,\iota_n$, say $\Y$ compact. Pick $\mu^0_\infty,\mu^1_\infty\in D(\entm{\mm_\infty})$ and use the $\Gamma-\lims$ inequality given by Theorem \ref{thm:gammamgh} to find $\mu^i_n\in\prd{(\X_n)}$ with
\begin{equation}
\label{eq:glscd}
\mu^i_n\weakto\mu^i_\infty\qquad\text{ and }\qquad\lims_n\entm{\mm_n}(\mu^i_n)\leq\entm{\mm_\infty}(\mu^i_{\infty})\qquad i=0,1.
\end{equation}
Since the $\X_n$'s are $\CD(K,\infty)$ there are geodesic $(\mu_{n,t})\subset\prd(\X_n)$ for which \eqref{eq:defkconv} holds. The length of these geodesics is uniformly bounded (e.g.\ by the diameter of $\Y$), hence up to a non-relabeled subsequence they converge $W_2$-uniformly to a limit geodesic $(\mu_{\infty,t})$ that, clearly, joins $\mu^0_\infty$ to $\mu^1_\infty$. In particular, we have $\mu_{n,t}\weakto\mu_{\infty,t}$ for every $t\in[0,1]$ and thus the $\Gamma-\limi$ inequality given by Theorem \ref{thm:gammamgh} ensures that
\[
\limi_n\entm{\mm_n}(\mu_{n,t})\leq\entm{\mm_\infty}(\mu^i_{\infty,t})\qquad\forall t\in[0,1].
\]
This bound and \eqref{eq:glscd} allow to pass to the limit in \eqref{eq:defkconv} and conclude.

 If $\Y$ is not compact, we first find recovery sequences as in \eqref{eq:recw2}, then we observe that \eqref{eq:glscd} and the $\CD(K,\infty)$ condition yield $\sup_{n,t}\entm{\mm_n}(\mu_{n,t})<\infty$ and get compactness from  the tightness criterium \eqref{eq:tightcr}.
\end{idea}

Notice that this last results is telling something non-trivial even if all the $\X_n$'s are smooth Riemannian manifolds and so is their mGH-limit: proving the same statement without optimal transport would be extremely complicated\footnote{In the celebrated series of papers \cite{Cheeger-Colding97I,Cheeger-Colding97II,Cheeger-Colding97III} Cheeger-Colding do not state such `stability of lower Ricci bounds for smooth limit spaces' but it is worth to notice that their results were not far from this, either, at least under a uniform upper bound on the dimension. Specifically,  the arguments used in \cite{Cheeger-Colding97III} for the convergence results for the spectrum of the Laplacian might be used to pass to the limit in the sharp Laplacian bounds for the squared distance, that in turn are equivalent - in the smooth Riemannian category - to lower Ricci bounds (see e.g.\ \cite[Theorem 3.6]{Wei07} for one implication, for the other differentiate \cite[Inequalities (3.4),(3.5)]{Wei07} at $r=0$).}.

\subsection{The Heat Flow}\label{se:hf}
In the celebrated paper \cite{JKO98} it is argued that the heat flow can be seen as the gradient flow of the Boltzmann-Shannon entropy w.r.t.\ the distance $W_2$. From classical smooth geometric analysis we also know that the heat flow is well-behaved on manifolds with a lower Ricci curvature bounds (for instance in terms of mass preservations, or estimates like the Li-Yau one). On the other hand, the information we have from the $\CD(K,\infty)$ condition, in fact the only information, is that the entropy is convex along $W_2$-geodesics. It is therefore only natural to study the $W_2$-gradient flow of the entropy and name the resulting flow `heat flow'. Recall also that in the classical smooth setting, the subdifferential of convex functions passes to the limit under 0-th order convergence (pointwise/uniform/$\Gamma$...) of the functions themselves: this suggests that there is room for a stability result of the heat flow in our setting.

\medskip

To turn this idea into action, we need some definitions. For $(\X,\sfd)$ metric space and $\E:\X\to\R\cup\{+\infty\}$ we define   the \emph{descending slope}, or simply slope, functional $|\partial^-\E|:\X\to[0,+\infty]$ as  $|\partial^-\E|(x)=\infty$ if $x\notin D(\E)$ and
\begin{equation}
\label{eq:defsl}
|\partial^-\E|(x):=\lims_{y\to x}\frac{(\E(y)-\E(x))^-}{\sfd(x,y)}\qquad\text{ where }\quad z^-:=\max\{0,-z\}
\end{equation}
otherwise. If $\E$ is $K$-geodesically convex, then the slope admits the representation
\begin{equation}
\label{eq:reprsl}
|\partial^-\E|(x)=\sup_{y\neq x}\frac{(\E(x)-\E(y)+\tfrac K2\sfd^2(x,y))^+}{\sfd(x,y)}\qquad\text{ where }\quad z^+:=\max\{0,z\},
\end{equation}
hence $\E(x)-\E(y)\leq\sfd(x,y)|\partial^-\E|(x)-\tfrac K2\sfd^2(x,y)$ and thus
\[
\E(\gamma_0)-\E(\gamma_t)=\sum_i\E(\gamma_{t_i})-\E(\gamma_{t_{i+1}})\leq\sum_i\tfrac{t_{i+1}-t_i}2|\partial^-\E|^2(\gamma_{t_i})+\tfrac{\sfd^2(\gamma_{t_{i+1}},\gamma_{t_i})}{2(t_{i+1}-t_i)}-\tfrac K2\sfd^2(\gamma_{t_{i+1}},\gamma_{t_i})
\]
holds for any curve $\gamma$ and partition of $[0,t]$. It is then easy to see that
\begin{equation}
\label{eq:uppgr}
\E(\gamma_0)\leq\E(\gamma_t)+\tfrac12\int_0^t|\dot\gamma_s|^2+|\partial^-\E|^2(\gamma_s)\,\d s\qquad\forall t\geq 0
\end{equation}
holds for any  curve $\gamma$ in $D(\E)$ (here and below $\int_0^t|\dot\gamma_s|^2\,\d s$ is intended $+\infty$ if $\gamma$ is not absolutely continuous). In the smooth setting, equality holds   iff $\gamma_t'=-\nabla E(\gamma_t)$; this motivates the definition:
\begin{definition}[Gradient flow]\label{def:gfede}
Let $\E$ be $K$-convex and lower semicontinuous. We say that a continuous curve $\gamma:[0,+\infty)\to\X$ is a gradient flow trajectory in the sense of the Energy Dissipation Equality (or $\EDE$-gradient flow trajectory) for $\E$ starting at $\gamma_0$ if   $\E(\gamma_0)<\infty$ and
\begin{equation}
\label{eq:defgf}
\E(\gamma_0)=\E(\gamma_t)+\tfrac12\int_0^t|\dot\gamma_s|^2+|\partial^-\E|^2(\gamma_s)\,\d s\qquad\forall t\geq 0.
\end{equation}
\end{definition}
Even though in what comes next I won't need it, I  state the following (suboptimal) result that ensures that the above definition is non-void, see \cite[Section 2.4]{AmbrosioGigliSavare08} for sharper statements and the proof:
\begin{theorem}[Existence of gradient flow trajectories]
Let $(\X,\sfd)$ be compact and $\E:\X\to\R\cup\{+\infty\}$ be a $K$-convex and lower semicontinuous function. Then for every $x\in D(\E)$ there is a gradient flow trajectory starting at $x$. 
\end{theorem}
It is also worth to notice the lack of uniqueness in normed settings:
\begin{example}\label{ex:nounique}{\rm
Equip  $\R^2$ with the distance induced by the `sup' norm and let $\E(x_1,x_2):=x_1$. Then $\E$ is convex and it is easy to see that uniqueness of gradient flow trajectories for any given initial datum fails.
}\fr\end{example}
On $\R^d$, for a converging sequence of convex functions the subdifferentials converge. A metric analogue  is the following. Notice that under suitable coercivity assumptions, arguing as for Theorem \ref{thm:stabcd} we see that the limit functional is automatically $K$-convex.
\begin{theorem}[$\Gamma-\limi$ for the slopes]\label{thm:glimisl}
Let   $(\E_n)$ be a sequence of $K$-convex and lower semicontinuous functionals on $\X$ that is $\Gamma$-converging to a limit $K$-convex functional $\E_\infty$. 

Then $|\partial^-\E_\infty|\leq \Gamma-\limi |\partial^-\E_n|$ on $D(\E_\infty)$.
\end{theorem}
\begin{idea} Let $x_n\to x_\infty\in D(\E_\infty)$ (so that $\limi_n\E_n(x_n)\geq\E_\infty(x_\infty)$), pick $y_\infty\in\X$ and let $(y_n)$ be a recovery sequence. Then 
\[
\limi_n\frac{(\E_n(x_n)-\E_n(y_n)+\frac K2\sfd^2(x_n,y_n))^+}{\sfd(x_n,y_n)}\geq \frac{(\E_\infty(x_\infty)-\E_\infty(y_\infty)+\frac K2\sfd^2(x_\infty,y_\infty))^+}{\sfd(x_\infty,y_\infty)}.
\] 
The conclusion follows from  \eqref{eq:reprsl}.
\end{idea}
A direct consequence of the above is the following stability result:
\begin{theorem}[Stability]\label{thm:stabede} Let $(\X,\sfd)$ be compact and $(\E_n)$ a sequence of non-negative, $K$-convex and lower semicontinuous functionals on it $\Gamma$-converging to $\E_\infty$. Let furthermore $x_\infty\in D(\E_\infty)$, $(x_n)$ be a recovery sequence and $\gamma_n$ be a gradient flow trajectory for $\E_n$ starting from $x_n$ for every $n\in\N$. Then:
\begin{itemize}
\item[i)] The sequence $(\gamma_n)$ is relatively compact w.r.t.\ local uniform convergence.
\item[ii)] Any limit curve is a gradient flow trajectory for $\E_\infty$ starting from $x_\infty$.
\end{itemize}
\end{theorem}
\begin{idea} By \eqref{eq:defgf} and the assumptions it follows that $\lims\int_0^\infty|\dot\gamma_{n,s}|^2\,\d s<\infty$  that in conjunction with the trivial estimate $\sfd^2(\gamma_t,\gamma_s)\leq (\int_t^s|\dot\gamma_r|\,\d r)^2\leq |s-t|\int_0^\infty|\dot\gamma_r|^2\,\d r$ and Ascoli-Arzel\`a's theorem gives $(i)$. For $(ii)$ we use the inequality $\int_t^s|\dot\gamma_{\infty,r}|^2\,\d r\leq \limi\int_t^s|\dot\gamma_{n,r}|^2\,\d r$ valid for any converging sequence of curves (recall  Theorem \ref{thm:ghgamma}) together  with Theorem  \ref{thm:glimisl} and Fatou's lemma to pass to the limit in \eqref{eq:defgf} and conclude.
\end{idea}
We have seen in Example \ref{ex:nounique} that even for convex energies uniqueness may fail. Our next task is to show that  for the specific case of $W_2$-gradient flow of the entropy uniqueness is in place. 

To this aim, for $\mu\in\pr (\X)$ and $\ggamma\in\pr(\X^2)$ with $\mu\ll\pi^1_*\ggamma$ we define $\ggamma_*\mu\in\pr(\X)$ as
\begin{equation}
\label{eq:gammapf}
\ggamma_*\mu:=\pi^2_*(\ggamma_\mu)\qquad\text{where}\qquad \ggamma_\mu:=\big(\frac{\d\mu}{\d(\pi^1_*\ggamma)}\circ\pi^1\big)\ggamma.
\end{equation}
Notice that if $\ggamma=(\Id,T)_*\mu$, then $\ggamma_*\mu=T_*\mu$, so this construction can be seen as a generalization of the concept of push-forward via a map. Also, $\mu\mapsto\ggamma_\mu$ is clearly linear, hence $\mu\mapsto \ent(\ggamma_*\mu)$ is convex. The next result tells that it is `less convex' than $\mu\mapsto\ent(\mu)$:
\begin{lemma}\label{le:convDE} Let $\ggamma\in\pr(\X^2)$. Then 
\[
D(\ent)\cap\{\mu:\mu\ll\pi^1_*\ggamma\}\ni \mu\qquad\mapsto\qquad {\rm DE}_\sggamma(\mu):= \ent(\mu)-\ent (\ggamma_*\mu)\quad\text{ is convex.} 
\]
\end{lemma} 
\begin{idea}
Let $\mu_0,\mu_1\in D(\ent)\cap\{\mu:\mu\ll\pi^1_*\ggamma\}$, put $\mu:=\frac{\mu_0+\mu_1}{2}$ and then $\nu_i:=\ggamma_*\mu_i$, $\nu:=\ggamma_*\mu$. By direct computation we have $
\ent(\mu_0)+\ent(\mu_1)=\entm\mu(\mu_0)+\entm\mu(\mu_1)+2\ent(\mu)$.
Subtracting the analogous identity written for then $\nu$'s we get
\[
{\rm DE}_\sggamma(\mu_0)+{\rm DE}_\sggamma(\mu_1)=2{\rm DE}_\sggamma(\mu)+\sum_{i=0}^1\big(\entm\mu(\mu_i)-\entm\nu(\nu_i)\big).
\]
Now notice that $\frac{\d\mu_i}{\d\mu}\circ\pi^1=\frac{\d \sggamma_{\mu_i}}{\d\sggamma_\mu}$ and thus
\[
\begin{split}
\entm{\mu}(\mu_i)=\entm{\sggamma_\mu}(\ggamma_{\mu_i})&\stackrel{\eqref{eq:dualent}}=\sup_{\varphi\in C_\b(\X^2)}\int\varphi\,\d\ggamma_{\mu_i}-\int u^*\circ\varphi\,\d\ggamma_\mu
\\
\text{(pick $\varphi=\psi\circ\pi^2$)}\qquad&\stackrel{\phantom{\eqref{eq:dualent}}}\geq \sup_{ \psi\in C_\b(\X)}\int\psi\,\d\pi^2_*\ggamma_{\mu_i}-\int u^*\circ\psi\,\d\pi^2_*\ggamma_\mu\stackrel{\eqref{eq:dualent}}=\entm{\nu}(\nu_i)
\end{split}
\]
for $i=0,1$, concluding the proof.
\end{idea}
Such convexity has the following useful consequence:
\begin{proposition}\label{prop:convsl}
Let $(\X,\sfd,\mm)$ be a normalized $\CD(K,\infty)$ space. Then 
\[
D(\ent)\ni \mu\quad\mapsto\quad|\partial^-\ent|^2(\mu)\qquad\text{ is convex (w.r.t.\ affine interpolation).}
\]
\end{proposition}
\begin{idea} Say that $\X$ is compact (so $\sfd^2$ is bounded on $\X^2$). I claim that for any $\mu\in D(\ent)$ we have
\begin{equation}
\label{eq:altraslope}
\sup_{\nu\neq\mu}\frac{\big(\big(\ent(\mu)-\ent(\nu)-\tfrac{K^-}2W_2^2(\mu,\nu)\big)^+\big)^2}{W_2^2(\mu,\nu)}=\!\!\sup_{\sggamma\in\pr(\X^2)\atop \mm\ll\pi^1_*\sggamma}\!\!\frac{\big(\big({\rm DE}_\sggamma(\mu)-\tfrac{K^-}2\int\sfd^2\,\d\ggamma_\mu\big)^+\big)^2}{\int\sfd^2\,\d\ggamma_\mu}.
\end{equation}
Indeed the inequality $\geq$ follows noticing that the requirement on $\ggamma$ ensure that $\ggamma_*\mu$ is well defined and  from the trivial bound $\int\sfd^2\,\d\ggamma_\mu\geq W_2^2(\mu,\ggamma_*\mu)$ plus little algebraic manipulation. For $\leq$ one argues by approximation: for given $\mu,\nu\ll\mm$ pick an optimal plan and then add $\eps(\Id,\Id)_*\mm $ and normalize to  meet the requirement $\mm\ll\pi^1_*\ggamma$.

By \eqref{eq:altraslope} to conclude it is enough to show  that $\mu\mapsto\frac{\big(\big({\rm DE}_\ssggamma(\mu)-\frac{K^-}2\int\sfd^2\,\d\sggamma_\mu\big)^+\big)^2}{\int\sfd^2\,\d\sggamma_\mu} $ is convex for any $\ggamma$ as in \eqref{eq:altraslope}.  By Lemma \ref{le:convDE} above and the linearity of $\mu\mapsto \int\sfd^2\,\d\ggamma_\mu$ we have that $\mu\mapsto \big({\rm DE}_\sggamma(\mu)-\tfrac{K^-}2\int\sfd^2\,\d\sggamma_\mu\big)^+$ is convex and non-negative. The conclusion then follows from the fact that $(\R^+)^2\ni(a,b)\mapsto\frac{a^2}b$ is convex and increasing in $a$.
\end{idea}
Notice also that a trivial consequence of the definition \eqref{eq:defw2} is that
\begin{equation}
\label{eq:w2conv}
\prd(\X)\ni \mu,\nu\qquad\mapsto\qquad W_2^2(\mu,\nu)\qquad\text{ is jointly convex (w.r.t.\ affine interpolation)}
\end{equation}
We then have:\begin{theorem}[Uniqueness of heat flow]\label{thm:uniqueheat}
Let $(\X,\sfd,\mm)$ be a normalized $\CD(K,\infty)$ space and $\mu\in D(\ent)$. Then there is at most one gradient flow trajectory of $\ent$ starting from $\mu$.
\end{theorem}
\begin{idea} Let $(\mu_t),(\nu_t)$ be two gradient flow trajectories for $\ent$ starting from $\mu$ and such that $\mu_t\neq \nu_t$ for some $t>0$. Let $\sigma_t:=\tfrac12(\mu_t+\nu_t)$ and notice that  taking convex combinations of \eqref{eq:defgf} written for $(\mu_t),(\nu_t)$, using the convexity in \eqref{eq:w2conv}, the one given by Proposition \ref{prop:convsl} above and the \emph{strict} convexity of $\ent(\cdot)$ (w.r.t.\ affine interpolation) we conclude that $(\sigma_t)$ is locally absolutely continuous, starts from $\mu$ and satisfies
\[
\ent(\sigma)>\ent(\sigma_t)+\tfrac12\int_0^t|\dot\sigma_s|^2+|\partial^-\ent|^2(\sigma_s)\,\d s.
\]
This contradicts \eqref{eq:uppgr} and proves the statement.
\end{idea}
We thus call `heat flow from $\mu$' the only gradient flow trajectory of $\ent$ starting from $\mu$. 
\begin{theorem}[Stability of the heat flow]\label{thm:stabheat}
Assume that:
\begin{itemize}
\item[i)] $(\X_n,\sfd_n,\mm_n)\stackrel{mGH}\to(\X_\infty,\sfd_\infty,\mm_\infty)$, the spaces being normalized and $\CD(K,\infty)$.
\item[ii)] The mGH-convergence above is realized by some compact $(\Y,\sfd_\Y)$.
\item[iii)] $\prd(\X_n)\ni \mu_n\weakto\mu_\infty\in\prd(\X_\infty)$  be with $\lims_n\entm{\mm_n}(\mu_n)\leq\entm{\mm_\infty}(\mu_\infty)<\infty$.
\item[iv)] $t\mapsto\mu_{n,t}\in \prd(\X_n)$ is the heat flow starting from $\mu_n$.
\end{itemize} 
Then $(\mu_{n,t})$ weakly converge to the heat flow $(\mu_{\infty,t})$ starting  from $\mu_\infty$ and the convergence is $W_2$-locally uniform in any realization of the mGH-convergence.
\end{theorem}
\begin{idea}
 Follows directly from Theorem \ref{thm:stabede} above.
 \end{idea}
Much like Theorem \ref{thm:stabcd}, this last result is telling something non-trivial even if all the $\X_n$'s and their limit are smooth manifolds: proving the same result without relying on the optimal transport interpretation of the heat flow would be more complicated\footnote{as for Theorem \ref{thm:stabcd}, one can notice that Cheeger-Colding in \cite{Cheeger-Colding97I,Cheeger-Colding97II,Cheeger-Colding97III} were not too far from a statement of this sort either, at least under a uniform upper bound on the dimension. A possible line of thought based on their argument would start noticing that a lower Ricci and upper dimension bounds implies estimates on the heat flow (e.g. a parabolic Harnack inequality from the Li-Yau estimate) that in turns gives compactness. Then one is left to prove that the limit flow is the heat flow in the limit manifold, which can be done if one is able to pass to the limit in the Laplacian.

Interestingly, a proof of this sort has been made in \cite{CMTkatoI} for Kato-lower bounds on the Ricci, see also the comments in Section \ref{se:bibliorcd}.}.

\subsection{Bibliographical notes}

{\footnotesize{I haven't discussed at all the geometric meaning of Ricci curvature in the smooth setting. For this, see for instance \cite{Chavel06}, \cite{Li12}, \cite{Petersen16}, \cite{Jost17}. For discussions more oriented toward comparison geometry and limit spaces, and thus to the topic of this note, see \cite{Gromov07}, \cite{ChEb08}, \cite{Cheegersurvey}, \cite{Wei07}, see also the more recent \cite{Villani09} for a presentation that emphasizes the role of Optimal Transport in this matter.

Standard references for Optimal Transport are  \cite{AmbrosioGigliSavare08}, \cite{Villani09}, \cite{G11}, \cite{Santambrogio15}, \cite{FigGla21}, \cite{AmbBruSem21}.

The concept of displacement interpolation, i.e.\ interpolation of probability measures along $W_2$-geodesics, goes back to the seminal paper by McCann    \cite{McCann97}. The key theorem \ref{thm:sturmvonrenesse} comes from \cite{vRS05}; relevant  precursors of this result have been obtained in \cite{OV00}, where a formal argument in favour of the the implication $(i)\Rightarrow(ii)$ was proposed, and \cite{CEMCS01}, where this implication has been made rigorous in the case $K=0$.

The Curvature-Dimension condition $\CD(K,\infty)$ for metric measure spaces via convexity of the entropy has been introduced in \cite{Lott-Villani09}, \cite{Sturm06I}, where also the first stability result have been proved. In \cite{Lott-Villani09} the spaces where assumed compact and in \cite{Sturm06I} normalized; the generalization of the stability results to spaces with infinite mass has been obtained in \cite{Villani09}.

The concept of metric gradient flow  has been introduced by De Giorgi and collaborators in \cite{DeGiorgiMarinoTosques80} (precisely to study evolutions problems under $\Gamma$-convergence) and further studied in \cite{DeGiorgi92}, \cite{DeGiorgi93}. The topic has been throughout investigated in \cite{AmbrosioGigliSavare08}  also in connection with the convexity assumption, leading in particular to inequality \eqref{eq:uppgr} and Definition \ref{def:gfede}.

As mentioned, the interpretation of heat flow as $W_2$-gradient flow of the entropy comes from \cite{JKO98}, see also \cite{Otto01} where this topic has been further investigated.

The first studies of the heat flow on $\CD$ spaces is in my paper \cite{Gigli10}, where the stability results  in Theorems \ref{thm:stabede}, \ref{thm:stabheat} and the uniqueness statement in Theorem \ref{thm:uniqueheat} have been obtained. Stability of gradient flows under $\Gamma$-convergence were also studied in \cite{SandSerf}. The concept of `push forward via a plan' was introduced by Sturm in \cite{Sturm06I}. I learnt about Lemma \ref{le:convDE} from  Savar\'e (personal communication). The generalization of the stability of the heat flow to non-compact spaces, possibly with infinite mass, has been obtained in \cite{GMS15}; notice that such generalization requires tools from Sobolev calculus that we are going to see in the next chapter (see also \cite[Remark 5.13]{GMS15}).

\medskip

The finite dimensional analogue  of  $\CD(K,\infty)$ cannot be written as easily as in Definition \ref{def:cdki} and partly for this reason there is a more flourishing taxonomy, mostly investigated by Sturm and collaborators: the conditions $\CD/\CD^*/\CD^e(K,N)$ have been introduced in \cite{Lott-Villani09}-\cite{Sturm06II}, \cite{BacherSturm10}, \cite{Erbar-Kuwada-Sturm13}  respectively  (more precisely, in \cite{Lott-Villani09} for $N<\infty$ only the case $K=0$ was studied, see also the bibliographical notes in \cite{Villani09}). Of these, the simplest to describe is $\CD^e(K,N)$,  I thus focus on this one. Roughly said, it amounts at replacing the inequality $\partial_{tt}\ent(\mu_t)\geq KW_2^2(\mu_0,\mu_1)$ which is, morally, \eqref{eq:defkconv}, with the stronger 
\begin{equation}
\label{eq:diffcdekn}
\partial_{tt}\ent(\mu_t)\geq \tfrac1N(\partial_t\ent(\mu_t))^2+KW_2^2(\mu_0,\mu_1).
\end{equation}
 The corresponding `integrated' inequality, which replaces \eqref{eq:defkconv} and is the actual definition of the $\CD^e(K,N)$ condition, reads as
\begin{equation}
\label{eq:defcdekn}
\E_N(\mu_t)\leq \sigma^{(1-t)}_{K,N}(W_2(\mu_0,\mu_1))\E_N(\mu_0)+\sigma^{(t)}_{K,N}(W_2(\mu_0,\mu_1))\E_N(\mu_1)\qquad\forall t\in[0,1],
\end{equation}
where $\E_N(\mu)=-\exp(-\tfrac1N\ent(\mu))$ and for given parameters $K\in\R$, $N\geq1$, $d\geq 0$ with $d^2 \tfrac KN<\pi^2$ the function $t\mapsto \sigma^{(t)}_{K,N}(d)$ is the only solution $g$ of $g''+d^2\tfrac KNg=0$ such that $g(0)=0$ and $g(1)=1$ (if $d^2 \tfrac KN\geq\pi^2$ we set $ \sigma^{(t)}_{K,N}(d)\equiv+\infty$).  To see the link between \eqref{eq:defcdekn} and \eqref{eq:diffcdekn}  notice that $f\in C^\infty([0,1])$ satisfies $f''\geq\tfrac1N(f')^2+Kd^2$ iff $F:=-\exp(-\tfrac1Nf)$ satisfies $F''+d^2\tfrac KNF\geq0$. The reason one asks for the integrated inequality \eqref{eq:defcdekn} in place of the differential one \eqref{eq:diffcdekn} has to do with the proof of stability, which is much simpler for  \eqref{eq:defcdekn} and follows as in Theorem \ref{thm:stabcd}; it is anyway worth to notice that the variational selection of geodesic made in the conclusion of the proof of Lemma \ref{le:tapio} grants that a definition based on \eqref{eq:diffcdekn} would also be stable (and equivalent to $\CD^e(K,N)$).

Concerning the relations among $\CD/\CD^*/\CD^e$, one should know that:
\begin{itemize}
\item[i)] $\CD(K,N)\Rightarrow\CD^*(K,N)$. Viceversa,  local $\CD^*(K,N)$ implies  local $\CD(K,N)$ (here \emph{local} means that any point has a neighbourhood $U$ such that for any two probability measures with support in $U$ there is a $W_2$-geodesic, possibly exiting $U$, along which the desired convexity inequality holds). See \cite{BacherSturm10}.
\item[ii)] on essentially non-branching spaces (see below) we have $\CD^*(K,N)\Leftrightarrow\CD^e(K,N)$. See \cite{Erbar-Kuwada-Sturm13}.
\item[iii)] a major result proved in \cite{CavMil16} is that on essentially non-branching and normalized spaces, $\CD^*(K,N)\Rightarrow\CD(K,N)$. It is widely believed (and I certainly do so) that the conclusion holds also on spaces with infinite mass, but a proof is currently missing.
\end{itemize}
As for the non-branching: a set $\Gamma$ of geodesics in $\X$ is said to be non-branching if 
\[
\gamma,\eta\in\Gamma,\quad \gamma\restr{[0,t]}=\eta\restr{[0,t]}\ \text{ for some $t\in(0,1]$}\qquad\Rightarrow\qquad \gamma=\eta.
\]
A space is non-branching if the set of all geodesics is non-branching, and essentially non-branching if any lifting $\ppi\in\pr(C([0,1],\X))$ (in the sense of Lemma \ref{le:superpp} below) of a $W_2$-geodesic with absolutely continuous endpoints is concentrated on a set of non-branching geodesics. 

These non-branching conditions are \emph{not} stable under convergence (think to the $L^p$-norm in $\R^2$ converging to the extremely branching $L^\infty$-norm as $p\uparrow\infty$). However, it has been proved in \cite{RajalaSturm12} that if the  $K$-convexity inequality holds along \emph{all} $W_2$-geodesics, then the space is essentially non-branching. Notably, if one adds  infinitesimal Hilbertianity (see Definition \ref{def:infhilb} and the next chapter) to the $\CD(K,\infty)$ condition  such stronger form of $K$-convexity holds (as a consequence of the so-called $\EVI_K$-property of the heat flow, see Proposition \ref{prop:EVIconv}) and the resulting class is also stable.

In summary, at least for normalized spaces we have
\begin{equation}
\label{eq:equivcd}
\CD(K,N)+\text{ inf. Hilb.}\qquad\Leftrightarrow\qquad \CD^*(K,N)+\text{ inf. Hilb.}
\qquad\Leftrightarrow\qquad \CD^e(K,N)+\text{ inf. Hilb.}
\end{equation}
and for this reason when speaking about $\RCD$ spaces one typically does not distinguish between the\linebreak $\RCD/\RCD^*/\RCD^e(K,N)$ conditions and simply speaks about $\RCD(K,N)$ spaces.

In terms of non-branching, notably in \cite{Deng20} it is proved that $\RCD(K,N)$ spaces, $N<\infty$, are non-branching. The proof, partly inspired by \cite{ColdingNaber12}, is based on a careful analysis of tangent spaces along geodesics.

}}

\section{Functional analysis enters into play}\label{ch:fa}
\subsection{First order Sobolev functions}
The space $W^{1,p}(\X)$ of real valued Sobolev functions on the metric measure space $(\X,\sfd,\mm)$ can be defined for any $p\in(1,\infty)$, and so does the space of BV functions. In this note I shall mostly focus on the case $p=2$, see Section \ref{se:gtv} for comments about different $p$'s. 
\subsubsection{A `vertical' approach and a stability result}\label{se:vertsob}
For $(\X,\sfd)$ metric space, we denote by $\Lip_\b(\X),\Lip_\bs(\X)$ the collection of real valued Lipschitz functions on $\X$ that are bounded, resp.\ with bounded support. Also, for $f:\X\to\R$ we define the \emph{asymptotic Lipschitz constant} $\lipa f:\X\to[0,+\infty]$ as
\[
\lipa f(x):=\lim_{r\downarrow0}\Lip(f\restr{B_r(x)})=\inf_{r>0}\Lip(f\restr{B_r(x)}).
\]
Notice that the latter writing shows that $\lipa f$ is  upper semicontinuous.
\begin{definition}\label{def:ch}
Let $(\X,\sfd,\mm)$ be a metric measure space. The Cheeger energy $\ch:L^2(\X)\to[0,\infty]$ is defined as
\[
\ch(f):=\inf\limi_n\tfrac12\int \lipa^2(f_n)\,\d\mm,
\]
the $\inf$ being taken among all sequences $(f_n)\subset\Lip_\bs(\X)$ converging to $f$ in $L^2$.

Then we define the Sobolev space $W^{1,2}_*(\X):=\{\ch<\infty\}\subset L^2(\X)$ and  the norm
\[
\|f\|_{W^{1,2}_*}:=(\|f\|_{L^2}+2\ch(f))^{\frac12}.
\]
\end{definition}
It it trivial to check that $\|\cdot\|_{W^{1,2}_*}$ is a norm. Completeness follows from the $L^2$-lower semicontinuity of $\ch$ (that is obvious by definition). Indeed, let $(f_n)$ be $W^{1,2}_*$-Cauchy, notice that it is also $L^2$-Cauchy and thus converges to some $f\in L^2$. Then the $L^2$-lower semicontinuity of $\ch$ gives the $L^2$-lower semicontinuity of $\|\cdot\|_{W^{1,2}_*}$ and thus $f\in W^{1,2}_*(\X)$ and
\begin{equation}
\label{eq:chlsc}
\lims_n\|f_n-f\|_{W^{1,2}_*}\leq\lims_n\limi_m\|f_n-f_m\|_{W^{1,2}_*}=0,
\end{equation}
showing that $f$ is also the $W^{1,2}_*$-limit of $(f_n)$. For $f\in W^{1,2}_*(\X)$ we say that $G$ is a \emph{relaxed slope} if for some $(f_n)\subset\Lip_\bs(\X)$ converging to $f$ in $L^2$ we have $\lip_a(f_n)\stackrel{L^2}\weakto G'$ for some $G'\leq G$ $\mm$-a.e.. In this case we say that $(f_n)$ \emph{realizes} $G$. It is easy to check the set of relaxed slopes is  a convex closed set, and thus it admits a unique element $|\D f|_*$ of minimal $L^2$-norm. Let  $f_n\to f$ be a sequence in $\Lip_\bs(\X)$ such that $\lim_n\tfrac12\int\lipa^2(f_n)\,\d\mm=\ch(f)<\infty$; then an application of Mazur's lemma\footnote{it tells: if $v_n\weakto v$ in some Banach space $B$, then there are convex combinations of the $v_n$'s strongly converging to $v$.} to a weakly converging subsequence of $(\lipa(f_n))$ shows that 
\begin{equation}
\label{eq:perdenslip}
\lipa(f_n)\to |\D f|_*\quad\text{ in $L^2\qquad$ and }\qquad \ch(f)=\tfrac12\int|\D f|^2_*\,\d\mm.
\end{equation}
 Moreover, one can prove that $|\D f|_*$ is minimal also in the (stronger) $\mm$-a.e.\ sense:
\[
|\D f|_*\leq G\quad\mm-a.e.\qquad\forall G\text{ relaxed slope of $f$.}
\]
To see this, by approximation it  suffices to prove that if $G_1,G_2$ are relaxed slopes and $\eta\in\Lip(\X)$ has values in $[0,1]$, then also $(1-\eta)G_1+\eta G_2$ is a relaxed slope. In turn, this follows using the elementary inequality
\[
\lipa((1-\eta)f_1+\eta f_2)\leq (1-\eta)\lipa(f_1)+ \eta\lipa(  f_2)+|f_2-f_1|\,\lipa(\eta)
\]
for $f_1=f_{1,n}$ and $f_2=f_{2,n}$ sequences realizing $G_1,G_2$ respectively. In particular
\begin{equation}
\label{eq:dflip}
|\D f|_*\leq\lipa f\quad\mm-a.e.\qquad \forall  f\in\Lip_\b(\X)\cap W^{1,2}_*(\X),
\end{equation}
as it is obvious picking $f_n:=f$ in Definition \ref{def:ch} (this is possible at least if $f$ has also bounded support, otherwise  one uses a cut-off argument).

Notice that in principle it might seem that Definition \ref{def:ch} is \emph{not} invariant under isomorphism of metric measure spaces, because the asymptotic Lipschitz constant is affected by the behaviour of the function outside $\supp(\mm)$. The following lemma solves this problem:
\begin{lemma}[``Non-linear" McShane extension]\label{le:locmcsh}
Let $C\subset\X$ be a subset of a metric space and $g:C\to\R$ be Lipschitz with bounded support. Then there exists $f\in\Lip_\bs(\X)$  such that
\[
\lipa^\X(f)(x)=\lipa^C(g)(x)\qquad\forall x\in C.
\] 
In particular, Definition \ref{def:ch} is invariant under isomorphism of metric measure spaces.
\end{lemma}
\begin{idea}
For any given ball $B\subset\X$ it is easy to construct a Lipschitz extension $f$ of $g\restr{B\cap C}$ to the whole $\X$ with $\Lip(f\restr B)=\Lip(g\restr{B\cap C})$ by putting 
\[
f(x):=\inf_{y\in B\cap C}g(y)+\sfd(x,y)\Lip(g\restr{B\cap C}). 
\]Then to achieve the conclusion one  performs this construction for every $x\in C$ and ball $B_{r_i}(x)$ with $r_i\downarrow0$ and suitably interpolates across different scales. The details are rather lengthy, and a bit tricky. I shall omit them.

For the last claim it suffices to show that $W^{1,2}_*(\X,\sfd,\mm)=W^{1,2}_*(\supp(\mm),\sfd,\mm)$. Inequality $\lipa^\X f(x)\geq \lipa^{\supp(\mm)} f(x)$ trivially holds on $\supp(\mm)$ and yields one inclusion. For the other, start from $g\in\Lip_\bs(\supp(\mm))$ and extend it using the first part of the statement.
\end{idea}
One can now use the upper semicontinuity of $\lipa f$ to obtain a stability result that is independent on any curvature condition. For the rest of the section, $\X_n\stackrel{mGH}\to\X_\infty$ is a converging sequence of normalized spaces and $(\Y,\sfd_\Y,(\iota_n))$ a realization of the convergence. I identify the $\X_n$'s with their isometric images in $\Y$.

Let me first define  $L^2$-convergence in varying spaces:
\begin{definition}[Convergence in varying $L^2$ spaces]\label{def:convl2var} We say that $n\mapsto f_n\in L^2(\X_n,\mm_n)$ weakly $L^2$-converges to $f_\infty\in L^2(\X_\infty,\mm_\infty)$ provided $\sup_n\|f_n\|_{L^2(\mm_n)}<\infty$ and $f_n\mm_n\weakto f_\infty\mm_\infty$, i.e.
\[
\int \varphi f_n\,\d\mm_n\quad\to\quad\int \varphi f_\infty\,\d\mm_\infty\qquad\forall \varphi\in C_\b(\Y).
\]
The convergence is strong if moreover $\|f_\infty\|_{L^2(\mm_\infty)}=\lim_n\|f_n\|_{L^2(\mm_n)}$. 
\end{definition}
It is clear that in the case of a constant sequence of spaces, these convergences reduce to the standard weak/strong convergence in  $L^2$. Many of the usual properties of $L^2$ convergence are retained even in this more general framework. To see why, notice that if $\mm_\infty$ has not atoms (which is very common in our setting: the reference measure of a  $\CD(K,\infty)$ space is either a Dirac mass or non-atomic\footnote{this is a consequence of a suitable infinite-dimensional version of the Bishop-Gromov inequality obtained in \cite{Sturm06I}. In finite dimensional $\CD(K,N)$ space the more classical Bishop-Gromov inequality holds (see \cite{Sturm06II}) so the spaces are locally doubling and the claim is obvious.}), then from $\mm_n\weakto\mm_\infty$, the fact that weak convergence is metrized by the Wasserstein distance on bounded spaces and by general results about the fact that the $\inf$ of the cost of transport in the `Monge formulation' coincides with that in the `Kantorovich formulation' (see \cite{Pratelli07}) applied to the cost $c=1\wedge\sfd_\Y$, we can find maps $T_n:\Y\to\Y$ so that 
\begin{equation}
\label{eq:mappeTn}
(T_n)_*\mm_\infty=\mm_n\qquad \forall n\in\N,\qquad\text{and}\qquad \int1\wedge\sfd_\Y(x,T_n(x))\,\d\mm_\infty(x)\to0.
\end{equation}
Then since for any $\varphi\in C_\b(\Y)$ we have $\varphi\circ T_n\stackrel{L^2(\mm_\infty)}\to\varphi$ (by dominated convergence) we  get 
\begin{subequations}
\label{eq:equivl2conv}\begin{align}
\label{eq:equivl2convw}
f_n\stackrel{L^2}\weakto &f_\infty\qquad\Leftrightarrow\qquad  f_n\circ T_n\weakto f_\infty \text{ weakly in }L^2(\mm_\infty),\\
\label{eq:equivl2convs}
f_n\stackrel{L^2}\to& f_\infty\qquad\Leftrightarrow \qquad f_n\circ T_n\to f_\infty \text{ strongly in }L^2(\mm_\infty).
\end{align}
\end{subequations}
These\footnote{a characterization as in \eqref{eq:equivl2convw}, \eqref{eq:equivl2convs} (and all the consequences we derive from them) is in place regardless of any assumption on $\mm_\infty$, as in any case one can prove, via a gluing technique and Kolmogorov's theorem, the existence of $\aalpha\in\pr(\Y^{ \N\cup\{\infty\}})$ such that $\pi^n_*\aalpha=\mm_n$ for every $n\in\N\cup\{\infty\}$ and $\lim_n\int1\wedge\sfd_\Y(y_n,y_\infty)\,\d\aalpha(\ldots,y_n,\ldots,y_\infty)=0$. It is then easy to see that
\begin{subequations}
\begin{align}
\label{eq:strongequiv}
f_n\stackrel{L^2}\to& f_\infty\quad\Leftrightarrow \quad f_n\circ\pi^n\to f_\infty\circ\pi^\infty\text{ strongly in }L^2(\aalpha) \\
\label{eq:weakequiv}
f_n\stackrel{L^2}\weakto &f_\infty\quad\Leftrightarrow\quad \sup_n\|f_n\|_{L^2(\mm_n)}<\infty\quad\text{and}\quad f_\infty (x_\infty)=\int g\,\d\aalpha_{x_\infty} \ \mm_\infty\text{-a.e.}\ x_\infty
\end{align}
\end{subequations}
For any $L^2(\aalpha)$-weak limit  $g$ of some subsequence of $(f_n\circ\pi^n)$ and $\{\aalpha_{x_\infty}\}$ is the disintegration of $\aalpha$ w.r.t.\ $\pi^\infty$. 
}
 ensure that many familiar properties of $L^2$-convergence(s) hold in our modified setting, for instance we have:
\begin{equation}
\label{eq:exprl2}
\begin{split}
&\text{1) $\sup_n\|f_n\|_{L^2(\X_n)}<\infty\qquad\Rightarrow\qquad$ there is a weakly converging subsequence},\\
&\text{2) strong limits respect linear  combinations (for weak this was already obvious)}.
\end{split}
\end{equation}

I now come to the following general result that, notably, needs no curvature assumption:
\begin{theorem}[$\Gamma-\lims$ of Cheeger energies]\label{thm:glimsch}
Let $\X_n\stackrel{mGH}\to\X_\infty$.

Then   $\Gamma-\lims_n\ch_n \leq \ch_\infty$,
the underlying convergence being strong $L^2$-convergence.
\end{theorem}
\begin{idea}  Let $f\in W^{1,2}_*(\X_\infty)$,  and $(f_k)\subset\Lip_\bs(\Y)$ (recall  Lemma \ref{le:locmcsh}) be with $\lims_k \tfrac12\int\lipa^2(f_k)\,\d\mm_\infty\leq \ch(f)$.

Since each  $f_k$ is Lipschitz, each of the functions $\lipa(f_k)$ is globally bounded. Then  the weak convergence of the $\mm_n$'s in $\Y$ and the  upper semicontinuity of $\lipa(f_k)$  ensures that
\[
\lims_{n}\int\lipa^2(f_k)\,\d\mm_n\leq\int\lipa^2(f_k)\,\d\mm_\infty\qquad\forall k\in\N.
\]
For every $k\in\N$, the sequence $n\mapsto f_k\in L^2(\X_n)$ clearly strongly $ L^2$-converges to $f_k\in L^2(\X_\infty)$, thus we conclude with a diagonalization argument.
\end{idea}
A direct consequence of this last theorem is that virtually any `first order inequality' (such as isoperimetric inequality, local/global Poincar\'e inequality etc..) where some integral norm of the differential bounds the size/oscillation of the function, is stable under mGH-convergence. I illustrate this point with the following example, showing that  a log-Sobolev inequality is stable under mGH-limits, but the argument easily generalizes. 
\begin{corollary}\label{cor:stablogsob}
Let $c>0$. Then the class of spaces $(\X,\sfd,\mm)$ such that 
\begin{equation}
\label{eq:logsob}
\int f^2\log(f^2)\,\d\mm-\Big(\int f^2\,\d\mm\Big)\log\Big(\int f^2\,\d\mm\Big)\leq\frac2c\int|\D f|_*^2\,\d\mm\qquad\forall f\in W^{1,2}_*(\X)
\end{equation}
is stable under mGH-convergence.
\end{corollary}
\begin{idea} By truncation one sees that \eqref{eq:logsob} holds for any $f\in W^{1,2}_*(\X)$ iff it holds for any $f\in L^\infty\cap W^{1,2}_*(\X)$. Also,  for $f_\infty\in W^{1,2}_*(\X_\infty)$ bounded, a recovery sequence $(f_n)$ for the Cheeger energy can be chosen uniformly bounded. In this case from \eqref{eq:equivl2convs} we easily get that the LHS of \eqref{eq:logsob} passes to the limit, while by construction the RHS converges.
\end{idea}

\subsubsection{An `horizontal' approach}\label{se:horsob}
Recall that $C([0,1],\X)$ is the space of continuous curves on $\X$ equipped with the complete and separable `sup' distance. For $t\in[0,1]$ we denote by $\e_t:C([0,1],\X)\to\X$ the evaluation map sending $\gamma$ to $\gamma_t$.
\begin{definition}[Test plan]  $\ppi\in\pr(C([0,1],\X))$ is a test plan  provided for some $C>0$ 
\begin{subequations}
\begin{align}\label{eq:boundcompr}
(\e_t)_*\ppi&\leq C\mm\qquad\forall t\in[0,1]&\text{(bounded compression)}&\\
\label{eq:finiteenergy}
\KE(\ppi)&:=\iint_0^1|\dot\gamma_t|^2\,\d t\,\d\ppi(\gamma)<\infty&\text{(finite  kinetic energy)}&
\end{align}
\end{subequations}
The least  $C$ in \eqref{eq:boundcompr} is called compression constant of $\ppi$ and denoted $\comp(\ppi)$.
\end{definition}
Notice that if $\ppi(\Gamma)>0$ and $\ppi$ is a test plan, then so is the `restricted' plan $\ppi(\Gamma)^{-1}\ppi\restr\Gamma$. Also, given $t,s\in[0,1]$ we consider the restriction (and rescaling) map $\Restr ts:C([0,1],\X)\to C([0,1],\X)$ sending $\gamma$ to $r\mapsto\gamma_{(1-r)t+rs}$. Then it is easy to see that $(\Restr ts)_*\ppi$ is a test plan for any $t,s\in[0,1]$ if $\ppi$ is so.

\begin{definition}\label{def:defsobw} Let  $f\in L^2(\X)$. We say that $f\in W^{1,2}_w(\X)$ and that $G\in L^2(\X)$, $G\geq 0$, is a \emph{weak upper gradient} provided 
\begin{equation}
\label{eq:defsobw}
\int|f(\gamma_1)-f(\gamma_0)|\,\d\ppi(\gamma)\leq\iint_0^1 G(\gamma_r)|\dot\gamma_r|\,\d r\,\d\ppi(\gamma),\qquad\forall\,\ppi\text{ test plan}.
\end{equation}
\end{definition}
 Notice that
\begin{equation}
\label{eq:dappi}
\iint_0^1 G(\gamma_r)|\dot\gamma_r|\,\d r\,\d\ppi(\gamma)\leq \sqrt{ \int_0^1\!\!\! \int G^2\,\d(\e_r)_*\ppi\,\d r}\sqrt{\KE(\ppi)}\leq  \sqrt{\comp(\ppi)\,\KE(\ppi)}\|G\|_{L^2(\X)},
\end{equation}
so the RHS of \eqref{eq:defsobw} is finite under the stated assumptions. We also have:
\begin{lemma}\label{le:equivsobw} Let $f,G\in L^2(\X)$ be with $G\geq 0$.
Then the following are equivalent:
\begin{itemize}
\item[i)] $G$ is a weak upper gradient of $f$.
\item[ii)] For every test plan $\ppi$ we have: the curve $t\mapsto f\circ\e_t\in L^1(\ppi)$ is absolutely continuous, the strong $L^1(\ppi)$-limit $\partial_t(f\circ \e_t)$ of $\frac{f\circ\e_{t+h}-f\circ\e_t}{h}$ as $h\to 0$ exists for a.e.\ $t$ and satisfies   $|\partial_t(f\circ\e_t)|\leq G(\gamma_t)|\dot\gamma_t|$ $\ppi$-a.e.\ for a.e.\ $t$.
\item[iii)] For every test plan $\ppi$ we have: for $\ppi$-a.e.\ $\gamma$ the function $f\circ \gamma$ is in $W^{1,1}([0,1])$ with $|(f\circ \gamma)'|(t)\leq G(\gamma_t)|\dot\gamma_t|$ for a.e.\ $t$.
\end{itemize}
\end{lemma}
\begin{idea}$(i)\Rightarrow(ii)$ Given $\ppi$, an argument by contradiction based on the fact that   $(\Restr ts)_*\ppi$ and $\ppi(\Gamma)^{-1}\ppi\restr\Gamma$ are  test plans imply that
\begin{equation}
\label{eq:curvewise}
 f(\gamma_s)-f(\gamma_t)\leq \int_t^sG(\gamma_r)|\dot\gamma_r|\,\d r\qquad\ppi-a.e.\ \gamma,\quad\forall t,s\in[0,1],\ t<s.
\end{equation}
Since $(\gamma,r)\mapsto G(\gamma_r)|\dot\gamma_r|$ is in $L^1(\ppi\times\mathcal L^1)$ by \eqref{eq:dappi}, absolute continuity in $L^1(\ppi)$ follows. It is also clear that if the strong $L^1(\ppi)$-limit of $\frac{f\circ\e_{t+h}-f\circ\e_t}{h}$ as $h\to0$ exists, then it satisfies the stated inequality, so it remains to prove this. In general, absolute continuity in $L^1(\ppi)$ is not sufficient to get existence of strong derivative (because $L^1$ does not have the Radon-Nikodym property), but in this case we have the pointwise (i.e.\ curvewise) information coming from \eqref{eq:curvewise} that in turn implies uniform integrability of the difference quotients: this can be used to obtain existence of weak limit and once one has this and the identity $f\circ\e_s-f\circ\e_t=\int_t^s\partial_r(f\circ\e_r)\,\d r$ it is easy to improve weak convergence to strong convergence.

$(ii)\Leftrightarrow(iii)$ This is basically  the  statement $L^1(\ppi,W^{1,1}([0,1]))\sim W^{1,1}([0,1],L^1(\ppi))$ that in turn is a rather natural consequence of Fubini's theorem. 

$(iii)\Rightarrow(i)$ By integration we see that  $\int|f(\gamma_s)-f(\gamma_t)|\,\d\ppi(\gamma)\leq \iint_t^sG(\gamma_r)|\dot\gamma_r|\,\d r\,\d\ppi(\gamma)$ holds for a.e.\ $t,s$. The RHS is obviously continuous in $t,s$, so to conclude it is sufficient to show that the same is true for the LHS. If $f\in C_\b(\X)$ this is clear by dominated convergence. For the general case we observe that by $\|f\circ\e_r\|_{L^1(\sppi)}\leq\|f\circ\e_r\|_{L^2(\sppi)}\leq \sqrt{\comp(\ppi) }\|f\|_{L^2(\mm)}$
the operators $L^2(\X)\ni f\mapsto f\circ\e_r\in L^1(\ppi)$ are equicontinuous in $r$, so the conclusion follows by approximation using the density of $C_\b(\X)$ in $L^2(\X)$.
\end{idea}
A direct consequence of what just proved is that there exists a minimal, in the $\mm$-a.e.\ sense, weak upper gradient $|\D f|_w$. Indeed, it is clear from \eqref{eq:defsobw} and \eqref{eq:dappi} that 
\begin{equation}
\label{eq:wclosed}
\text{the collection of weak upper gradients is convex and $L^2$-closed,}
\end{equation}
so a unique one with minimal $L^2$-norm exists. Then we notice that by $(iii)$ in Lemma \ref{le:equivsobw} we see that if $G_1,G_2$ are weak upper gradients, so is $G_1\wedge G_2$, giving the claim.

We equip $W^{1,2}_w(\X)$  with the norm $\|f\|_{W^{1,2}_w}^2:=\|f\|_{L^2}^2+\||\D f|_w\|_{L^2}^2$. Then from \eqref{eq:defsobw} and \eqref{eq:dappi} it is easy to see that the $W^{1,2}_w$-norm is $L^2$-lower semicontinuous, thus the same arguments used in the previous section prove that  is a Banach space. Finally, for later use let me point out that Lemma  \ref{le:equivsobw} and the analogous properties of $W^{1,1}([0,1])$ functions (recall that these are a.e.\ equal to absolutely continuous functions, which are a.e.\ classically differentiable) ensure that 
\begin{equation}
\label{eq:chainweak}
f\in W^{1,2}_w(\X),\ \varphi\in C^1\cap \Lip(\R)\quad\Rightarrow\quad  \varphi\circ f\in  W^{1,2}_w(\X)\text{ with }|\D(\varphi\circ f)|_w=|\varphi'|\circ f|\D f|_w.
\end{equation}

\subsubsection{The two definitions coincide}

We are now going to see that the two definitions of Sobolev functions and of `norm of distributional differential' given in the previous sections coincide. Perhaps surprisingly, the study of the Hopf-Lax formula will play a role. Given $\varphi:\X\to\R$ we set $Q_0\varphi:=\varphi$ and
\begin{equation}
\label{eq:defHL}
Q_t\varphi(x):=\inf_yF(t,x;y)\qquad\text{where}\qquad  F(t,x;y):=\varphi(y)+\tfrac1{2t}\sfd^2(x,y).
\end{equation}
Then we have:
\begin{proposition}\label{prop:HL}
Let $(\X,\sfd)$ be a metric space and $\varphi\in\Lip_\b(\X)$. Then $[0,\infty)\ni t\mapsto Q_t\varphi\in C_\b(\X)$ is Lipschitz and for any $x\in\X$ we have
\begin{equation}
\label{eq:HJ}
\partial_tQ_t\varphi(x)+\tfrac12\lipa(Q_t\varphi)^2(x)\leq 0,\qquad a.e.\ t>0.
\end{equation}
\end{proposition}
\begin{idea} Say $\X$ compact and put ${\sf D}_t(x):=\max\sfd(x,y)$, the max being among minima $y$ of $F(t,x;\cdot)$. Since clearly $Q_t\varphi(x)\to\varphi(x)$ as $t\downarrow0$ and ${\sf D}_t(x)\leq 2t\Lip(\varphi)$ (from $F(t,x;y)\leq\varphi(x)$ for $y$ minimum), the proof will follow if we show that for any $x\in\X$
\begin{subequations}
\begin{align}\label{eq:HJ1}
&(0,\infty)\ni t\mapsto {\sf D}_t(x)&&\ \text{ is non-decreasing},\\
\label{eq:HJ2}
&(0,\infty)\ni t\mapsto Q_t\varphi&&\ \text{ is locally Lipschitz with }\partial_tQ_t\varphi(x)=-\tfrac{{\sf D}^2_t(x)}{2t^2}\ a.e.\ t,\\
\label{eq:HJ3}
&\lipa(Q_t\varphi)\leq \tfrac{{\sf D}_t(x)}t&&\  \text{ for any }t>0.
\end{align}
\end{subequations}
For \eqref{eq:HJ1} let $t<s$, pick $x_t,x_s$ minimizers of $F(t,x;\cdot), F(s,x;\cdot)$ realizing ${\sf D}_t(x),{\sf D}_s(x)$ respectively. The claim follows adding up the inequalities
\[
F(t,x;x_t)\leq F(t,x;x_s)\qquad\text{and}\qquad F(s,x;x_s)\leq F(s,x;x_t)
\]
and recalling the definition of $F$. Then \eqref{eq:HJ2} follows from \eqref{eq:HJ1} and
\[
\tfrac{{\sf D}_s^2(x)(t-s)}{2ts}=F(s,x;x_s)-F(t,x;x_s)\leq Q_s\varphi(x)- Q_t\varphi(x)\leq F(s,x;x_t)-F(t,x;x_t)\leq \tfrac{{\sf D}_t^2(x)(t-s)}{2ts}.
\]
For \eqref{eq:HJ3} notice that a diagonalization argument gives that $x\mapsto {\sf D}_t(x)$ is upper semicontinuous. Then let $y,z\in\X$, $y_t$ be a minimum of $F(t,y;y_t)$ and notice that
\[
\begin{split}
Q_t\varphi(z)-Q_t\varphi(y)\leq F(t,z;y_t)-F(t,y;y_t)=\tfrac{\sfd^2(z,y_t)-\sfd^2(y,y_t)}{2t}\leq\sfd(z,y)\tfrac{2{\sf D}_t(y)+\sfd(z,y)}{2t},
\end{split}
\]
then swap $z,y$, send $z,y$ to $x$ and conclude with the upper semicontinuity of ${\sf D}_t(\cdot)$.
\end{idea}
To produce test plans we shall often use the following general result:
\begin{lemma}[Superposition principle]\label{le:superpp}
Let $(\X,\sfd)$ be complete and separable and $(\mu_t)\subset \prd(\X)$ be $W_2$-absolutely continuous  with $t\mapsto|\dot\mu_t|$ in $L^2([0,1])$. Then there is $\ppi\in\pr(C([0,1],\X))$ with $(\e_t)_*\ppi=\mu_t$ for every $t\in[0,1]$ such that $\KE(\ppi)<\infty$ and
\begin{equation}
\label{eq:eqsp}
\int|\dot\gamma_t|^2\,\d\ppi(\gamma)=|\dot\mu_t|^2\qquad a.e.\ t\in[0,1].
\end{equation}
\end{lemma}
\noindent{(Note: a plan $\ppi$ as in the statement is called `lifting' of the curve $(\mu_t)$.)}
\begin{idea} 
Let  $\ppi\in\pr(C([0,1],\X))$ be  with $(\e_t)_*\ppi=\mu_t$ for every $t\in[0,1]$. Then from
\begin{equation}
\label{eq:superpfacile}
W_2^2(\mu_t,\mu_s)\leq\int\sfd^2(\gamma_t,\gamma_s)\,\d\ppi(\gamma)\leq|s-t|\iint_t^s|\dot\gamma_r|^2\,\d r\,\d\ppi(\gamma)
\end{equation}
we see that the inequality $\geq$ holds in \eqref{eq:eqsp} for a.e.\ $t$. 

For the converse assume $\X$ geodesic  (up to a Kuratowski embedding\footnote{Any separable space $(\X,\sfd)$ can be isometrically embedded in $\ell^\infty$ via the map sending $x$ to the sequence $n\mapsto \sfd(x,x_n)-\sfd(x_0,x_n)$, where $(x_n)\subset\X$ is dense. The closed convex hull of the image of $\X$ in $\ell^\infty$ is complete, separable and geodesic.} we can assume this). For $n\in\N$ use first a gluing argument and then Borel selection of geodesics to find $\ppi_n\in\pr(C([0,1],\X))$ such that  $(\e_t)_*\ppi=\mu_t$ for  $t=\frac i{2^n}$ for  $i=0,\ldots,2^n-1$ and $
\iint_{\frac i{2^n}}^{\frac {i+1}{2^n}}|\dot\gamma_r|^2\,\d r\,\d\ppi_n(\gamma)=2^nW_2^2(\mu_{\frac i{2^n}},\mu_{\frac {i+1}{2^n}})$. Then we have
\begin{equation}
\label{eq:pertight}
\iint_0^1|\dot\gamma_r|^2\,\d r\,\d\ppi_n(\gamma)= \sum_i2^n\,W_2^2(\mu_{\frac i{2^n}},\mu_{\frac {i+1}{2^n}})\leq \int_0^1|\dot\mu_r|^2\,\d r
\end{equation}
and thus by the lower semicontinuity of the kinetic energy, what previously proved and Prokorhov's theorem, to conclude it is sufficient to prove that the sequence $\{\ppi_n\}_n$ is tight. If $\X$ is compact this is trivial from the estimate \eqref{eq:pertight} and the fact that $\gamma\mapsto\int_0^1|\dot\gamma_r|^2\,\d r$ has compact sublevels (by Ascoli-Arzel\`a and the Holder continuity estimate $\sfd^2(\gamma_t,\gamma_s)\leq {|s-t|}{\int_0^1|\dot\gamma_r|^2\,\d r}$). In the general case one also uses the tightness of the weakly compact family $\{\mu_t\}_{t\in[0,1]}$. 
\end{idea}

Another ingredient in the proof of equivalence of the two approaches to Sobolev functions is the $L^2$-gradient flow of the Cheeger energy. Recall that if $\sf E$ is a convex and lower semicontinuous functional on a Hilbert space, then its subdifferential $\partial^-\E(x)$ at a point $x\in H$ with $\E(x)<\infty$ is defined as the collection of $v\in H$ such that
\[
\E(x)+\la v,y-x\ra\leq \E(y)\qquad\forall y\in H.
\]
It is clear that $\partial^-\E(x)$ is closed and convex and therefore admits a unique element of minimal norm, provided it is not empty.

We consider $H:=L^2(\X)$ and $\E:=\ch$ and denote by $D(\Delta)\subset L^2(\X)$ the collection of $f$'s with $\partial^-\ch(f)\neq\varnothing$. Also, for $f\in D(\Delta)$ we shall denote by $\Delta f\in L^2(\X)$ the opposite of the element of minimal norm in $\partial^-\ch(f)$. Notice that at this level of generality the functional $\ch$ is not a quadratic form, thus $D(\Delta)$ is not necessarily a vector space and $\Delta$ not necessarily linear. Still, for this  Laplacian the following weak integration by parts formulas hold:
\begin{lemma}\label{le:inpartI}
Let $(\X,\sfd,\mm)$ be a normalized metric measure space, $f\in D(\Delta)$, $g\in W^{1,2}(\X)$ and $\varphi\in C^1\cap \Lip(\R)$. Then 
\begin{subequations}
\begin{align}\label{eq:intpa1}
\Big|\int g\Delta f\,\d\mm\Big|&\leq \int |\D f|_*|\D g|_*\,\d\mm,\\
\label{eq:intpa2}
\int \varphi\circ f\,\Delta f\,\d\mm&=-\int \varphi'\circ f |\D f|_*^2\,\d\mm.
\end{align}
\end{subequations}
\end{lemma}
\begin{idea} 
For any $\eps\in\R$ we have
\[
\ch(f)-\eps \int g\Delta f\,\d\mm\leq \ch(f+\eps g).
\]
Then use the rather trivial bound $|\D(f+\eps g)|_*\leq |\D f|_*+|\eps ||\D g|_*$ (that follows from $\lipa(f+\eps g)\leq \lipa(f)+|\eps|\lipa( g)$ and a relaxation argument)  to bound from above the RHS and conclude that \eqref{eq:intpa1} holds by the arbitrariness of $\eps\in\R$. 

For \eqref{eq:intpa2} we start noticing that for $\psi\in C^1\cap\Lip(\R)$ we have $\psi\circ f\in W^{1,2}(\X)$ with $|\D (\psi\circ f)|_*\leq |\psi'|\circ f|\D f|_*$ (it follows via relaxation from the trivial $\lipa(\psi\circ f)\leq |\psi|'\circ f\lipa (f)$). Then we argue as above starting from $|\D(f+\eps \varphi\circ f)|_*=|\D( (\Id+\eps\varphi)\circ f)|_*\leq (1+\eps\varphi')\circ f|\D f|_*$, which is   valid for $|\eps|\leq \frac1{\Lip\varphi}$ (and in particular also for $\eps<0$).
\end{idea}

The standard theory of gradient flows of convex and lower semicontinuous functionals on Hilbert spaces, together with the fact that $D(\ch)$ is dense in $L^2(\X)$ (as it contains $\Lip_\bs(\X)$) ensures that: for any $f\in L^2(\X)$ there is a unique continuous curve $[0,\infty)\ni t\mapsto f_t\in L^2(\X)$, locally absolutely continuous in $(0,\infty)$ such that 
\begin{equation}
\label{eq:gfch}
\partial_tf_t=\Delta f_t\quad a.e.\ t>0\qquad\text{and}\qquad f_0=f.
\end{equation}
See \cite[Chapter 7]{Brezis11} or, for instance, \cite[Theorem 5.1.12]{GP19} for a proof of this fact. Some basic properties of this flow are collected in the next lemma. Notice in particular that item $(iv)$, known as `Kuwada's lemma', provides a crucial link between $L^2$ and $W_2$ geometries that is going to be key also in out interpretation of the heat flow as gradient flow in Theorem \ref{thm:hfgfagain}.
\begin{lemma}\label{le:basegfch}
Let $(\X,\sfd,\mm)$ be a normalized metric measure space, $f\in L^2(\X)$ and $(f_t)$ the gradient flow trajectory of $\ch$ starting from $f$. Then:
\begin{itemize}
\item[i)] If $f\leq c$ (resp. $f\geq c$) $\mm$-a.e., then $f_t\leq c$ (resp. $f_t\geq c$) $\mm$-a.e.\ for every $t\geq 0$.
\item[ii)] if $f$ is a probability density, then so is $f_t$ for every $t>0$.
\item[iii)] if $f$ is a probability density with $0<c\leq f\leq c^{-1} $ $\mm$-a.e.\ for some $c>0$, then $[0,\infty)\ni t \mapsto \ent{(f_t\mm)}$ is in $C([0,\infty))\cap AC((0,\infty))$ and
\[
\frac\d{\d t}\ent{(f_t\mm)}=-\int\frac{|\D f_t|^2_*}{f_t}\,\d\mm\qquad a.e.\ t>0.
\]
\item[iv)] Let $f$ be as in $(iii)$ and also so that $\mu:=f\mm$ is in $\prd(\X)$. Then $t\mapsto \mu_t:=f_t\mm\in(\prd(\X),W_2)$ is locally absolutely  continuous with $t\mapsto|\dot\mu_t|$ in $L^2([0,\infty))$ and
\begin{equation}
\label{eq:kuwlem}
|\dot\mu_t|^2\leq \int\frac{|\D f_t|_*^2}{f_t}\,\d\mm\qquad a.e.\ t>0.
\end{equation}
\end{itemize}
\end{lemma}
\begin{idea} For $\varphi\in C^\infty\cap\Lip(\R)$ it is easy to see that $t\mapsto\int\varphi\circ f_t\,\d\mm$ is absolutely continuous with 
\[
\partial_t\int\varphi\circ f_t\,\d\mm=\int\varphi'\circ f_t\Delta f_t\,\d\mm\stackrel{\eqref{eq:intpa2}}=-\int\varphi''\circ f_t|\D f_t|^2_*\,\d\mm.
\]
Picking $\varphi$ convex, non-negative, identically $ 0$ on $(-\infty,c]$ and positive on $(c,+\infty)$ we get $(i)$. Picking $\varphi\equiv1$ we see that $\int f_t\,\d\mm=\int f\,\d\mm$, thus also $(ii)$ follows (as non-negativity comes from $(i)$). For $(iii)$ pick $\varphi(z):=z\log(z)$ (that is Lipschitz on the image of the $f_t$'s).

For $(iv)$ fix $\varphi\in\Lip_\b(\X)$ and notice that the regularity of $t\mapsto f_t$ and that of $t\mapsto Q_t\varphi$ granted by Proposition \ref{prop:HL} give that $t\mapsto \int Q_t\varphi\,\d\mu_t$ is absolutely continuous. Thus for  $0<t\leq s $ and $\ell:=s-t$   it is easy to see that
\[
\begin{split}
\int Q_1\varphi\,\d\mu_s-\int \varphi\,\d\mu_t&=\iint_0^1\partial_r (Q_r\varphi f_{t+r\ell})\,\d r\,\d\mm\\
\text{(by \eqref{eq:HJ} and \eqref{eq:gfch})}\qquad&\leq \iint_0^1-\frac{\lipa Q_t\varphi ^2(x)}2 f_{t+r\ell}+\ell Q_t\varphi \Delta f_{t+r\ell}\,\d r\,\d\mm\\
\text{(by \eqref{eq:dflip} and \eqref{eq:intpa1})}\qquad&\leq \iint_0^1-\frac{|\D Q_t\varphi |_*^2(x)}2 f_{t+r\ell}+\ell |\D Q_t\varphi|_* |\D f_{t+r\ell}|_*\,\d r\,\d\mm\\
\text{(by Young's inequality)}\qquad&\leq\frac {\ell^2}2\iint_0^1\frac{|\D f_{t+r\ell}|_*^2}{f_{t+r\ell}}\,\d r\,\d\mm.
\end{split}
\]
It follows from \eqref{eq:dualw2} (and $Q_1\varphi=(-\varphi)^c$ and the arbitrariness of $\varphi$) that $\tfrac12W_2^2(\mu_t,\mu_s)\leq  \frac{\ell^2}2\iint_0^1\frac{|\D f_{t+r\ell}|_*^2}{f_{t+r\ell}}\,\d r\,\d\mm$ and with an argument about Lebesgue points \eqref{eq:kuwlem} is proved. Also, by $(iii)$ and the non-negativity of the entropy, the integral of the RHS of \eqref{eq:kuwlem} is bounded by $\int f\log(f)\,\d\mm\leq-\log c$.
\end{idea}
We can now prove the following (notice that no curvature condition is needed):
\begin{theorem}\label{thm:sobuguali}
Let $(\X,\sfd,\mm)$ be a normalized metric measure space. Then $W^{1,2}_*(\X)=W^{1,2}_w(\X)$ and for a function $f$ belonging to these spaces we have $|\D f|_*=|\D f|_w$ $\mm$-a.e..
\end{theorem}
\begin{idea} It is clear that for $f\in\Lip_\bs(\X)$ the function $\lipa(f)$ is a weak upper gradient. Then the approximation property \eqref{eq:perdenslip} and the closure property \eqref{eq:wclosed}  give $W^{1,2}_*(\X)\subset W^{1,2}_w(\X)$ with $|\D f|_*\geq |\D f|_w$. To conclude, with a density argument it suffices to prove that if  $f\in W^{1,2}_w(\X)$ is so that $\mu:=f^2\mm\in\prd(\X)$ and $0<c^{-1}\leq f\leq c<\infty $ $\mm$-a.e.\ for some $c>0$, then $\ch(f)\leq \tfrac12\int|\D f|^2_w\,\d\mm$. 

Let $(g_t)$ be the gradient flow trajectory of $\ch$ starting from $g:=f^2\in W^{1,2}_w(\X)$. The convexity of $u(z):=z\log z$ and item $(iii)$ in Lemma \ref{le:basegfch} give
\[
\int u'(g_0)(g_0-g_t)\,\d\mm\geq\int u(g_0)-u(g_t)\,\d\mm=\iint_0^t\frac{|\D g_s|_*^2}{g_s}\,\d\mm.
\]
We combine item $(iv)$ of Lemma \ref{le:basegfch}  and  Lemma \ref{le:superpp} to find $\ppi\in\pr(C([0,1],\X))$ with $(\e_t)_*\ppi=g_t\mm$ and $\int|\dot\gamma_t|^2\,\d\ppi\leq\int\frac{|\D g_t|^2_*}{g_t}\,\d\mm$ for a.e.\ $t$. Then by $(i)$ of Lemma \ref{le:basegfch} we know that $\ppi$ is a test plan, hence (recall \eqref{eq:chainweak}) we get
\[
\begin{split}
\int u'(g_0)(g_0-g_t)\,\d\mm&=\int u'(g_0)\circ\e_0-u'(g_0)\circ\e_t\,\d\ppi\\
&\leq \iint_0^t |\D (u'(g))|(\gamma_s)|\dot\gamma_s|\,\d s\,\d\ppi\leq\frac12 \iint_0^t\frac{|\D g|^2_w}{g^2}g_s+\frac{|\D g_s|^2_*}{g_s}\,\d s\,\d\mm.
\end{split}
\]
Combining these  two inequalities we get 
\[
4\iint_0^t{|\D f|^2_w}\frac {g_s}{g}\,\d s\,\d\mm=\iint_0^t\frac{|\D g|^2_w}{g^2}g_s\,\d s\,\d\mm\geq \iint_0^t\frac{|\D g_s|_*^2}{g_s}\,\d s\,\d\mm=8\int_0^t\ch(\sqrt g_s)\,\d s
\]  and dividing by $t$ and letting $t\downarrow0$, by the lower semicontinuity of $\ch$ we conclude.
\end{idea}
From now on I will drop the subscripts `$*$' and `$w$'. For later use I also collect here some basic calculus rule for $|\D f|$ (below we always have $f,g\in W^{1,2}(\X)$):
\begin{subequations}\label{eq:calcdf}
\begin{align}
\label{eq:dfconv}
|\D(\alpha f+\beta g)&\leq |\alpha||\D f|+\beta|\D g|&&\mm-a.e.\qquad\qquad\forall \alpha,\beta\in\R\\
\label{eq:dfloc}
|\D f|&=0&&\mm-a.e.\ on\ \{f=0\}\\
\label{eq:dfchain}
|\D(\varphi\circ f)|&=|\varphi'|\circ f|\D f|&&\mm-a.e.\qquad\qquad \forall   \varphi\in C^1\cap\Lip(\R),\\
\label{eq:dfleib}
|\D(fg)|&\leq |f||\D g|+|g||\D f|&&\mm-a.e.\qquad\qquad\forall f,g\in W^{1,2}\cap L^\infty(\X)\qquad.
\end{align}
\end{subequations}
here part of the claim is that $\alpha f+\beta g,\varphi\circ f,fg$ are in $W^{1,2}(\X)$. Properties  \eqref{eq:dfconv}, \eqref{eq:dfleib} and inequality $\leq$ in \eqref{eq:dfchain}  are obvious from item $(iii)$ in Lemma \eqref{le:equivsobw} and the analogous properties of $W^{1,1}([0,1])$ functions. Then  \eqref{eq:dfchain}  follows by first considering the inequality $\leq $ in \eqref{eq:dfchain} for $\psi(z):=z\Lip(\varphi)-\varphi$ and then noticing that
\[
\Lip(\varphi)|\D f|=|\D(\varphi \circ f+\psi\circ f)|\leq|\D (\varphi\circ f)|+|\D (\psi\circ f)|\leq (|\varphi'|\circ f+|\psi'|\circ f)|\D f|.
\]
On the set $f^{-1}(\{\varphi'\geq 0\})$ the rightmost side is equal to $\Lip(\varphi)|\D f|$, forcing the desired equality on $f^{-1}(\{\varphi'\geq 0\})$. The conclusion on $f^{-1}(\{\varphi'\leq 0\})$ follows analogously. For \eqref{eq:dfloc} we let $(\varphi_n)\subset C^1(\R)$ be 1-Lipschitz, converging to the identity and with $\varphi'_n\equiv0$ on some neighborhood of 0 (depending on $n$). Then \eqref{eq:dfchain} gives
\[
\int_{\{|\D f|\neq 0\}}|\D f|^2\,\d\mm\geq\int |\D(\varphi_n\circ f)|^2\,\d\mm\qquad\forall n
\]
and letting $n\to\infty$ using the lower semicontinuity of $\ch$ we get $\int_{\{|\D f|\neq 0\}}|\D f|^2\,\d\mm\geq \int|\D f|^2\,\d\mm$, which is the claim. 

 For later use I also point out that
\begin{equation}
\label{eq:w12reflsep}
\text{if $W^{1,2}(\X)$ is reflexive, then it is separable and $\Lip_\bs(\X)$ is dense.}
\end{equation}
To see this, let $D\subset \Lip_\bs(\X)$  be a countable $L^2$-dense subset of the unit ball $B$ of $W^{1,2}(\X)$. Then for $f\in B$ find $(f_n) \subset D$ converging to $f$ in $L^2(\X)$: since $(f_n)$ is $W^{1,2}$-bounded, by reflexivity up to subsequences it must have a weak limit in $W^{1,2}(\X)$ and this weak limit must be $f$. Hence the weak closure of $D$ is precisely $B$ and by Mazur’s lemma this is sufficient to conclude.

Finally,  if $(\X,\sfd,\mm)$ is a smooth Finsler manifold equipped with the induced distance and a measure that, read in charts, has smooth density, then little work shows that  the concept of Sobolev function as defined here coincides with that classically defined via charts and integration by parts. In particular, for $f\in C^\infty(\X)$ we have
\begin{equation}
\label{eq:dffinsl}
|\D f|(x)=\| \d f(x)\|_*\qquad\mm-a.e.\ x\in\X,
\end{equation}
where here $\| \cdot\|_*$ is the norm on $T^*_x\X$ dual to the Finsler one. Identity \eqref{eq:dffinsl} shows that even in the smooth Finsler setting, $W^{1,2}$ is in general {not} a Hilbert space: it is so if and only if the manifold is actually Riemannian.
\subsection{The Heat Flow as gradient flow (again)}
In Section \ref{se:hf} we defined the heat flow on $\CD(K,\infty)$ spaces as the $W_2$-gradient flow of the entropy and in the previous section we encountered the $L^2$-gradient flow of the Cheeger/Dirichlet energy, that certainly is also the heat flow in the smooth setting. 

Notably, these two flows coincide in the general setting of $\CD(K,\infty)$ spaces:
\begin{theorem}[The heat flow as gradient flow on $\CD(K,\infty)$ spaces]\label{thm:hfgfagain}
Let $(\X,\sfd,\mm)$ be a normalized $\CD(K,\infty)$ space, $f\in L^2(\X)$ be so that $\mu:=f\mm$ is in $\prd(\X)$. Then
\begin{equation}
\label{eq:slopech}
8\ch(\sqrt{f})=|\partial^-\ent|^2(\mu).
\end{equation}
Also, let $(f_t)$ (resp.\ $(\mu_t)$) be the gradient flow trajectory of $\ch$ in $L^2$ (resp.\ $\ent$ in $\prd(\X)$) starting from $f$ (resp.\ from $\mu$).  Then $\mu_t=f_t\mm$ for every $t\geq 0$.
\end{theorem}
\begin{idea} Say $K=0$, $0<c\leq f\leq c^{-1} $ $\mm$-a.e.\ for some $c>0$ and notice that $8\ch(\sqrt{f})=\int\frac{|\D f|^2}{f}\,\d\mm$. Writing inequality \eqref{eq:uppgr} for $\nu_t:=f_t\mm$ keeping in mind Lemma \ref{le:basegfch} we get, after rearranging, the bound
\[
\iint_0^t\frac{|\D f_s|^2}{f_s}\,\d s\,\d\mm\leq \ent(\nu_0)-\ent(\nu_t)\stackrel{\eqref{eq:reprsl}}\leq|\partial^-\ent|(\nu_0)W_2(\nu_0,\nu_t)-\tfrac K2W^2_2(\nu_0,\nu_t),
\]
thus keeping in mind \eqref{eq:kuwlem} and the lower semicontinuity of $t\mapsto\ch(\sqrt{f_t})$ the inequality  $\leq$ in \eqref{eq:slopech} follows.  The same line of thought shows that if we prove $\geq$ in \eqref{eq:slopech}, then $(\nu_t)$ satisfies \eqref{eq:defgf}, so that the conclusion $\nu_t=\mu_t$ would follow from the uniqueness Theorem \ref{thm:uniqueheat}.

Assume for a moment that $f$ is also Lipschitz, define $L:\X^2\to\R^+$ as
\[
L(x,y):=\left\{\begin{array}{ll}
\tfrac{|\log(f(x))-\log(f(y))|}{\sfd(x,y)},&\qquad\text{ if }x\neq y,\\
\lipa(\log(f))(x),&\qquad\text{ if }x=y.
\end{array}\right.
\]
Then from the convexity of $u(z)=z\log z$ we get $\ent(\mu)-\ent(\nu)\leq \int\log f\d(\mu-\nu)$ and thus for $\ggamma\in\pr(\X^2)$ optimal from $\mu$ to $\nu$  we have
\[
\begin{split}
\frac{\ent(\mu)-\ent(\nu)}{W_2(\mu,\nu)}&\leq\frac{\int \log f(x)-\log f(y)\,\d\ggamma}{W_2(\mu,\mu_n)}\leq \sqrt{\int L^2(x,y)\,\d\ggamma(x,y)}.
\end{split}
\]
Then for $\nu=\nu_n\to\mu$ the corresponding optimal plans weakly converge to  $(\Id,\Id)_*\mu$, thus the (trivial) upper semicontinuity of $L$ gives
\[
|\partial^-\ent|^2(\mu)\leq \int L^2(x,x)\,\d\mu(x)=\int \lipa(\log f)^2 f\,\d\mm=4\int \lipa(\sqrt f)^2\,\d\mm.
\]
Now the conclusion follows by approximation taking an optimal sequence $(g_n)\subset \Lip_\bs(\X)$ for the definition of $\ch(\sqrt f)$, noticing that, up to a cut-off and normalization, we have $g_n^2\mm\stackrel{W_2}\to\mu$  and using the $W_2$-lower semicontinuity of $|\partial^-\ent|$ granted by Theorem \ref{thm:glimisl} (pick $\E_n=\ent$ for every $n\in \N\cup\{\infty\}$).
\end{idea}
In Theorem \ref{thm:glimsch} we proved the $\glims$ inequality for the Cheeger energies and in Theorem \ref{thm:glimisl} the $\glimi$ inequality for the slopes. Thanks to the identity \eqref{eq:slopech} we can combine these and obtain the following result. Notice that an interesting intermediate result is a Rellich-Kondrachov type of result along a converging sequence of $\CD$ spaces:
\begin{theorem}[Convergence of slopes and  energies]\label{thm:convslen} Let $\X_n \stackrel{mGH}\to \X_\infty $, the spaces being normalized and $\CD(K,\infty)$.

Then $(|\partial^-\entm{\mm_n}|)$ $\Gamma$-converges to $|\partial^-\entm{\mm_\infty}|$ w.r.t.\ $W_2$-convergence\footnote{we say that $n\mapsto\mu_n\in \prd(\X_n)$ converges to $\mu_\infty\in\prd(\X_\infty)$ w.r.t.\ $W_2$ if $(\iota_n)_*\mu_n\stackrel{W_2}\to(\iota_\infty)_*\mu_\infty$} and $(\ch_n )$ Mosco-converges to $\ch_\infty$ w.r.t.\ weak/strong $L^2$-convergence.

Moreover, assume that $(f_n)$ is so that
\begin{equation}
\label{eq:perfortecomp}
\sup_n\|f_n\|_{W^{1,2}(\X_n)}<\infty\qquad\text{ and }\qquad \lim_{R\to\infty}\sup_n\int_{B_R^c(x_n)} |f_n|^2\,\d\mm_n=0\quad\text{for some $x_n\to x_\infty$. }
\end{equation}
Then $(f_n)$ has a subsequence strongly $L^2$-converging.
\end{theorem}
\begin{idea} The recovery sequence for the slopes can be built as in Theorem \ref{thm:glimsch} taking into account the identity \eqref{eq:slopech}.

For the claim about strong compactness  we can assume that the $f_n$'s are non-negative (by taking positive/negative parts),  have uniformly bounded support (by a cut-off argument based on $|\D(\eta f)|\leq \eta|\D f|+\nchi_{B^c}|f|$ for a 1-Lipschitz cut-off function $\eta$ identically 1 on $B$) and that $n\mapsto \|f_n\|_{L^2}$ has a limit. If such limit is 0 there is nothing to prove, otherwise by rescaling we can assume that $\|f_n\|_{L^2}=1$ for every $n$. Put $\mu_n:=f_n^2\mm_n$ and notice that \eqref{eq:slopech} gives ${\sf S}:=\sup_n|\partial^-\entm{\mm_n}|(\mu_n)<\infty$, while from \eqref{eq:reprsl} we get 
\begin{equation}
\label{eq:percompl2}
\entm{\mm_n}(\mu_n)\leq \entm{\mm_n}(\nu_n)+{\sf S}\,W_2(\mu_n,\nu_n)+\tfrac{K^-}2W_2^2(\mu_n,\nu_n)\qquad\forall\nu_n\in D(\entm{\mm_n}).
\end{equation}
Choosing the $\nu_n$'s so that $\sup_n\entm{\mm_n}(\nu_n)+W_2(\nu_n,\delta_{\bar x})<\infty$ we see that $\sup_n\entm{\mm_n}(\mu_n)<\infty$, which, together with $\mm_n\weakto\mm_\infty$, gives tightness of the $\mu_n$'s (recall \eqref{eq:tightcr}). Since the supports are uniformly bounded, up to subsequences we have  $W_2(\mu_n,\mu_\infty)\to 0$ for some $\mu_\infty=f^2_\infty\mm_\infty$. Since $\|f_n\|_{L^2}=\|f_\infty\|_{L^2}$ for every $n$, to conclude it suffices to  prove $f_n\stackrel{L^2}\weakto f_\infty$. 

Theorem \ref{thm:gammamgh} gives $\entm{\mm_\infty}(\mu_\infty)\leq\limi_n\entm{\mm_n}(\mu_n)$. On the other hand, letting $(\nu_n)$ be a recovery sequence for $ \entm{\mm_\infty}(\mu_\infty)$ with  $W_2(\nu_n,\mu_\infty)\to0$ (recall \eqref{eq:recw2}) in \eqref{eq:percompl2} we obtain $\lims_n\entm{\mm_n}(\mu_n)\leq \entm{\mm_\infty}(\mu_\infty)$. 

Thus the entropies converge and we can conclude with a classical argument based on Young's measures. We give some details: let $\tilde\mu_n:=(\Id,\rho_n)_*\mm_n$, with $\rho_n:=f_n^2$, notice that they are weakly compact in $\pr(\Y\times\R)$ and let $\d\nu(x,z)=\nu_x(z)\otimes\d\mm_\infty(x)$ be a limit. Testing the weak convergence with functions of the form $\Y\times\R\ni (x,z)\mapsto z\varphi(x)$ we deduce that $\int z\,\d\nu_x(z)=\rho_\infty(x):=f_\infty^2(x)$ for $\mm_\infty$-a.e.\ $x$. Also, testing the weak convergence with $(x,z)\mapsto u(z):=z\log z$ gives
\[
\lim_n\int u(z)\,\d\tilde\mu_n(x,z)=\iint u(z)\,\d\nu_x(z)\,\d\mm_\infty(x)\geq \int u\Big(\int z\,\d\nu_x(z)\Big)\,\d\mm_\infty(x)=\int u(\rho_\infty)\,\d\mm_\infty
\]
having used Jensen's inequality. Since $\int u(z)\,\d\tilde\mu_n(x,z)=\entm{\mm_n}(\mu_n)$, the convergence of the entropies forces the equality in the inequality above, which in turn, by the strict convexity of $u$, forces $\nu_x(z)$ to be a Dirac mass for $\mm_\infty$-a.e.\ $x$. In other words, $\tilde\mu_n\weakto (\Id,\rho_\infty)_*\mm_\infty$ and testing such weak convergence  with functions of the form $\Y\times\R\ni (x,z)\mapsto \sqrt{z}\varphi(x)$ we conclude that $\int f_n\varphi\,\d\mm_n\to\int f_\infty\varphi\,\d\mm_\infty$, as desired.

It remains to prove the $\glimi$ inequality for the Cheeger energies, thus let $f_n\weakto f_\infty$ and, with no loss of generality and a cut-off argument as above, assume that the $f_n$'s have uniformly bounded supports and $W^{1,2}$-norms.   Thus by what just proved the convergence to $f_\infty$ is $L^2$-strong, which in turn implies that $\mu_n:=f_n^2\mm_n$ converge to $\mu_\infty:=f_\infty^2\mm_\infty$ in  total variation. This and the uniform bound of the supports give $W_2$-convergence, so that the claim follows from \eqref{eq:slopech} and Theorem \ref{thm:glimisl}.
\end{idea}

\subsection{The Riemannian Curvature Dimension condition}

\subsubsection{Notion of calculus on $\CD(K,\infty)$ spaces}

We have seen that Sobolev calculus is (or can be) based upon the concept of plan with bounded compression, but the $\CD$ condition (together with the superposition principle in Lemma \ref{le:superpp}) a priori only grants existence of plans with marginals of finite entropy. This regularity gap is bridged by the following selection lemma. Here and below we say that $(\mu_t)$ has bounded compression if $\mu_t\leq C\mm$ for some $C>0$ and every $t\in[0,1]$.
\begin{lemma}\label{le:tapio}
Let $(\X,\sfd,\mm)$ be a $\CD(K,\infty)$ space and $\mu_0,\mu_1\in\prd(\X)$ be with bounded support and $\leq c\mm$ for some $c>0$. Then there is a $W_2$-geodesic $(\mu_t)$ connecting them satisfying \eqref{eq:defkconv}   and of bounded compression.
\end{lemma}
\begin{idea} Say $K=0$ and  notice that the class of $W_2$-midpoints of two measures  is  convex (w.r.t.\ affine interpolation) and weakly compact. \\
{\sc Step 1} We claim that 
\begin{equation}
\label{eq:pertapio}
\text{for any $\nu_0,\nu_1\leq c\mm$ there is a $W_2$-midpoint $\nu_{1/2}$ with $\nu_{1/2}\leq c\mm$.}
\end{equation}
Let $F_c:\pr(\X)\to[0,1]$ be defined as $F_{c}(\mu):=\int(\rho-c)^+\,\d\mm+\mu^s(\X)$ for  $\mu=\rho\mm+\mu^s$. Let $ \nu$  be a minimizer of $F_c$ among midpoints of $\nu_0,\nu_1$: if we prove that $F_{c}(\nu)=0$ we are done, as this means $\nu\leq c\mm$. To prove this we argue  by contradiction and suppose that $F_{c}(\nu)>0$; for simplicity we also assume that $\nu=\eta\mm\ll\mm$ and that $\nu(\{\eta=c\})=0$. Let $ \ppi$ be a lifting of a geodesic from $\nu_0$ to $\nu_1$ passing through $\nu$,  $A:=\{\eta> c\}$, $ m:= \nu(A)$ and $ \ppi':=\tfrac1{ m}\ppi\restr{\e_{1/2}^{-1}(A)}$. Then the measures $(\e_0)_*\ppi',(\e_1)_*\ppi'$ are $\leq \tfrac c{ m}\mm$, thus their entropies are $\leq \log(\tfrac c{ m})$ and the $\CD(0,\infty)$ assumption ensures that they have a midpoint $\sigma=\xi\mm$ with entropy $\leq \log(\tfrac c{ m})$.  Jensen's inequality gives  $\mm(\{\xi>0\})\geq \frac{ m}c$ and since  $\mm(A)<\frac1c\int_A\eta\,\d\mm=\frac{ m}{c}$ we have  that $\sigma(A)<1$. The convexity of $W_2^2(\cdot,\cdot)$ ensures that $\nu\restr{A^c}+ m\,\sigma$ is a $W_2$-midpoint of $\nu_0,\nu_1$, thus so is  $\sigma_\eps:=\nu+\eps( m\,\sigma-\nu\restr A )$ for $\eps\in(0,1)$ and, letting $u_c(z):=(z-c)^+$, by direct computation we see that 
\[
\lim_{\eps\downarrow0}\frac{F_{c}(\sigma_\eps)-F_{c}(\nu)}\eps=\int u_c'(\eta)( m\xi- \eta)\,\d\mm=\int_A( m\xi- \eta)\,\d\mm= m\,\sigma(A)-\underbrace{ \nu(A)}_{= m}<0,
\]
contradicting the minimality of $\nu$.\\
{\sc Step 2} Let $\mu_0,\mu_1\leq C\mm$  be given and $\mu_{1/2}=\rho_{1/2}\mm$ be the minimizer of the entropy among midpoints. We claim that $\mu_{1/2}\leq C\mm$ and argue by contradiction. If not, for $A:=\{\rho_{1/2}>C\}$ we have $m:=\mu_{1/2}(A)>0$. Let $\ppi$ be a lifting of a geodesic $(\mu_t)$ from $\mu_0$ to $\mu_1$ passing through $\mu_{1/2}$ and $\ppi':=\tfrac1m\ppi\restr{\e_{1/2}^{-1}(A)}$. Then $(\e_0)_*\ppi',(\e_1)_*\ppi'\leq \tfrac Cm\mm$ and thus by \eqref{eq:pertapio} there is a $W_2$-midpoint $\nu=\eta\mm$ of these measures with $\eta\leq \tfrac Cm$. Arguing as before, $\mu^\eps:=(1-\eps)\mu_{1/2}+\eps\tilde\mu$ is a  $W_2$-midpoint of $\mu_0,\mu_1$ for any $\eps\in(0,1)$. Then  by direct computation we have
\[
\begin{split}
\lim_{\eps\downarrow0}\frac{\ent(\mu^\eps)-\ent(\mu_{1/2})}\eps=\int_{A}(\log(\rho_{1/2})+1)(m\eta-\rho_{1/2})\,\d\mm.
\end{split}
\]
Then we observe that the negative part of $\nchi_A(m\eta-\rho_{1/2})$ is concentrated on $\{\rho_{1/2}>C\}$ and the positive part on $\{\rho_{1/2}\leq C\}$ (as $m\eta\leq C$ by construction), thus the monotonicity of the logarithm  tells that the above derivative is $<0$, contradicting the minimality of $\ent(\mu_{1/2})$. \\
{\sc Conclusion}  Let $\mu_0,\mu_1\leq C\mm$  be given. Recursively define $\mu_{1/2}$ as the midpoint minimizing the entropy, then  $\mu_{1/4},\mu_{3/4}$ as minimizers of the entropy among midpoints of $\mu_0,\mu_{1/2}$ and $\mu_{1/2},\mu_1$ respectively and so on. It is clear that these dyadic choices induce a geodesic satisfying \eqref{eq:defkconv} and, by what proved, with bounded compression.
\end{idea}
An important consequence of the above is:
\begin{proposition}[Sobolev-to-Lipschitz property]\label{prop:sobtolip} Let $(\X,\sfd,\mm)$ be $\CD(K,\infty)$,  $f\in W^{1,2}(\X)$ and $g\in C_\b(\X)$ be such that $|\D f|\leq g$ $\mm$-a.e.. 

Then $f$ has a $\|g\|_{L^\infty}$-Lipschitz representative $\tilde f$ and it satisfies $\lipa(\tilde f)\leq g$  on $\X$. 
\end{proposition}
\begin{proof} To extend $f$ outside $\supp(\mm)$ we use Lemma \ref{le:locmcsh}, if needed. 
For $x,y\in\X$ and $r>0$ let $\mu,\nu\in\pr(\X)$ be of bounded compression and concentrated on $B_r(x),B_r(y)$ respectively. Use first Lemma \ref{le:tapio} and then Lemma \ref{le:superpp} to find $\ppi$ with bounded compression such that $(\e_0)_*\ppi=\mu$, $(\e_1)_*\ppi=\nu$ and $\iint_0^1|\dot\gamma_t|^2\,\d t\,\d\ppi(\ggamma)=W_2^2(\mu,\nu)$. 

Thus $\ppi$ is a test plan, and since $f\in W^{1,2}(\X)$ with $|\D f|\leq g$ we have
\[
\begin{split}
\big|\int f\,\d\mu-\int f\,\d\nu\big|&\leq \iint_0^1g(\gamma_t)|\dot\gamma_t|\,\d\ppi(\gamma)\leq (\sfd(x,y)+2r)\iint^1_0g(\gamma_t)\,\d\ppi(\gamma).
\end{split}
\]
Roughly said, the conclusion follows picking $x,y$ Lebesgue points of $f$ and letting $r\downarrow0$. In practice, since we don't know if Lebesgue points exist, we    prove that for every $\eps>0$ and $\mm$-a.e.\ $x\in\X$ there is $\mu\subset \pr(\X)$  with bounded compression, concentrated on $B_\eps(x)$ and with $|\int f\,\d\mu-f(x)|\leq\eps$. To see this we rely on the separability of $\X$ to deduce that: for any $E\subset\X$ Borel, letting $E':=\{x\in E:\mm(U\cap E)>0\text{ for any neighbourhood $U$ of $x$}\}$ we have $\mm(E\setminus E')=0$ (because $E\setminus E'$ can be covered by a countable number of negligible neighbourhoods). We apply this principle to $E:=f^{-1}([a,a+\eps))$ for $a\in\Q$ to deduce that for $\mm$-a.e.\ $x\in E$ the measure $\mu:=\mm(E\cap B_\eps(x))^{-1}\mm\restr{E\cap B_\eps(x)}$ does the job.
\end{proof}

Directly by Definition \ref{def:defsobw} it is not hard to see  that for $f\in W^{1,2}(\X)$ and $\ppi$ test  we have
\begin{equation}
\label{eq:perreprgr}
\lims_{t\downarrow0}\int\frac{f(\gamma_t)-f(\gamma_0)}{t}\,\d\ppi(\gamma)\leq\tfrac12\int|\D f|^2(\gamma_0)\,\d\ppi(\gamma)+\lims_{t\downarrow0}\tfrac1{2t}\iint_0^t|\dot\gamma_r|^2\,\d r\,\d\ppi(\gamma),
\end{equation}
this being the analogue of the inequality $\partial_tf(\gamma_t)\restr{t=0}\leq \tfrac12|\d f|^2(\gamma_0)+\tfrac12|\dot\gamma_0|^2$ valid in the smooth context. In this latter case, equality holds iff $\gamma'_0=\nabla f(\gamma_0)$ (both in a Riemannian and Finslerian\footnote{Given a normed space $(V,\|\cdot\|)$ the duality map ${\rm Dual}$ takes an element $v\in V$ and returns all the elements $w\in V^*$ such that $w(v)=\|v\|^2=\|w\|_*^2$. Equivalently, ${\rm Dual}(v)$ is the subdifferential of $\tfrac{\|\cdot\|^2}2:V\to\R$ at $v$. Then for a smooth function $f$ on a Finsler manifold we put $\nabla f:={\rm Dual}(\d f)$, where $\d f$ is classically defined without the need of any norm on the cotangent space. Notice that the gradient  is uniquely defined iff the norm on the cotangent space(s) is differentiable, or equivalently iff the norm on the tangent space(s) is strictly convex. Thus in the general Finsler setting we should actually write $\gamma'_0\in\nabla f(\gamma_0)$.} context), therefore, also in analogy with Definition \ref{def:gfede}, we give the following 
\begin{definition}[Plan representing a gradient]\label{def:planrepgr}
Let $(\X,\sfd,\mm)$ be a metric measure space, $f\in W^{1,2}(\X)$ and $\ppi$ a test plan. We say that $\ppi$ represents the gradient of $ f$ provided  
\begin{equation}
\label{eq:derreprgr}
\limi_{t\downarrow0}\int\frac{f(\gamma_t)-f(\gamma_0)}{t}\,\d\ppi(\gamma)\geq\tfrac12\int|\D f|^2(\gamma_0)\,\d\ppi(\gamma)+\lims_{t\downarrow0}\tfrac1{2t}\iint_0^t|\dot\gamma_r|^2\,\d r\,\d\ppi(\gamma).
\end{equation}
\end{definition}
We shall use this concept mainly in connection with:
\begin{theorem}[Metric Brenier theorem]\label{thm:metbr}
Let $(\X,\sfd,\mm)$ be a metric measure space, $(\mu_t)$ a $W_2$-geodesic of bounded compression and  $\varphi$ a Kantorovich potential for $(\mu_0,\mu_1)$ that is in $W^{1,2} \cap \Lip_{loc}(\X)$.  Then any lifting  $\ppi$ of $(\mu_t)$ represents the gradient of $-\varphi$.
\end{theorem}
\begin{idea} It is clear that $\ppi$ is a test plan. By combining \eqref{eq:defgeod} and \eqref{eq:eqsp} we  get
\begin{equation}
\label{eq:br1}
\tfrac1{2t}\iint_0^t|\dot\gamma_r|^2\,\d r\,\d\ppi(\gamma)=\tfrac12W_2^2(\mu_0,\mu_1)=\tfrac1{2t}\int\sfd^2(\gamma_0,\gamma_t)\,\d\ppi(\gamma)\qquad\forall t\in(0,1],
\end{equation}
The second identity for $t=1$ and the fact that $\varphi$ is a Kantorovich potential  give  $\gamma_1\in\partial^c\varphi(\gamma_0)$ for $\ppi$-a.e.\ $\gamma$ (here $c=\tfrac{\sfd^2}2$), which  implies $\varphi (\gamma_0)-\varphi (\gamma_t)\geq \tfrac{\sfd^2(\gamma_0,\gamma_1)-\sfd^2(\gamma_t,\gamma_1)}2=(t-\tfrac{t^2}2)\sfd^2(\gamma_0,\gamma_1)$ and thus
\begin{equation}
\label{eq:br2}
\limi_{t\downarrow0}\int\frac{\varphi(\gamma_0)-\varphi(\gamma_t)}{t}\geq \int\sfd^2(\gamma_0,\gamma_1)\,\d\ppi(\gamma)=W_2^2(\mu_0,\mu_1).
\end{equation}
Now put $|\partial^+\varphi|(x):=\lims_{y\to x}\frac{(\varphi(y)-\varphi(x))^+}{\sfd(x,y)}$ and observe that since $\varphi$ is locally Lipschitz, for any curve $\gamma$  absolutely continuous $t\mapsto\varphi(\gamma_t)$ is Lipschitz with $|\partial_t\varphi(\gamma_t)|\leq|\partial^+\varphi|(\gamma_t)|\dot\gamma_t|$ for a.e. $t$. Hence $|\partial^+\varphi|$ is a weak upper gradient and  by minimality we get
\begin{equation}
\label{eq:slopewug}
|\D\varphi|\leq |\partial^+\varphi|\qquad\mm-a.e..
\end{equation}
Also, for $\ppi$-a.e.\ $\gamma$ the bound  $\varphi (y)-\varphi (\gamma_0)\leq \tfrac{\sfd^2(y,\gamma_1)-\sfd^2(\gamma_0,\gamma_1)}2\leq \sfd(y,\gamma_0)\tfrac{\sfd(y,\gamma_1)+\sfd(\gamma_0,\gamma_1)}2$ holds for any $y$, thus dividing by  $\sfd(y,\gamma_0)$, letting $y\to\gamma_0$ and using \eqref{eq:slopewug} we get
\begin{equation}
\label{eq:br3}
|\D \varphi|(\gamma_0)\leq |\partial^+\varphi|(\gamma_0)\leq\sfd(\gamma_0,\gamma_1)\qquad\ppi-a.e.\ \gamma.
\end{equation}
Combining the bounds \eqref{eq:br1}, \eqref{eq:br2}, \eqref{eq:br3} we conclude.
\end{idea}
Inequality \eqref{eq:dfconv} implies that $\R\ni \eps\mapsto |\D(g+\eps f)|\in L^2(\X)$ is convex (in the a.e.\ sense), thus so is $\R\ni \eps\mapsto \tfrac12|\D(g+\eps f)|^2\in L^1(\X)$. It follows that\footnote{given an arbitrary, possibly uncountable, family $(f_i)_{i\in I}$ of Borel functions on $\X$, the `essential infimum'  $\essinf_if_i$ is a Borel function $f:\X\to\bar\R$ which is $\leq f_i$ $\mm$-a.e.\ for any $i\in I$ and $\geq g$ $\mm$-a.e.\ for any $g$ with the same property. It exists and is unique up to $\mm$-a.e.\ equality. Similarly for the `essential supremum' $\esssup$. In other words, the space of equivalence classes up to $\mm$-a.e.\ equality of Borel functions from $\X$ to $\bar \R$ is a complete lattice.} 
\begin{equation}
\label{eq:defdpm}
\begin{split}
D^+f(\nabla g)&:=\essinf_{\eps>0}\frac{|\D (g+\eps f)|^2-|\D g|^2}{2\eps}=\lim_{\eps\downarrow0}\frac{|\D (g+\eps f)|^2-|\D g|^2}{2\eps},\\
 D^-f(\nabla g)&:=\esssup_{\eps<0}\frac{|\D (g+\eps f)|^2-|\D g|^2}{2\eps}=\lim_{\eps\uparrow0}\frac{|\D (g+\eps f)|^2-|\D g|^2}{2\eps},
\end{split}
\end{equation}
the limits being intended in $L^1(\X)$. The following is simple, yet crucial:
\begin{lemma}[Horizontal and vertical derivatives]\label{le:horver} Let $(\X,\sfd,\mm)$ be a metric measure space, $f,g\in W^{1,2}(\X)$ and $\ppi$ representing the gradient of $g$. Then
\[
\begin{split}
\int D^-f(\nabla g)(\gamma_0)\,\d \ppi(\gamma)&\leq \limi_{t\downarrow0}\int\frac{f(\gamma_t)-f(\gamma_0)}{t}\,\d\ppi(\gamma)\\
&\leq\lims_{t\downarrow0}\int\frac{f(\gamma_t)-f(\gamma_0)}{t}\,\d\ppi(\gamma)\leq \int D^+f(\nabla g)(\gamma_0)\,\d \ppi(\gamma)
\end{split}
\]\end{lemma}
\begin{idea} Write \eqref{eq:perreprgr} with $g+\eps f$ in place of $f$ and subtract the defining identity \eqref{eq:derreprgr} to obtain
\[
\lims_{t\downarrow0}\eps \int\frac{f(\gamma_t)-f(\gamma_0)}{t}\,\d\ppi(\gamma)\leq\tfrac12\int |\D (g+\eps f)|^2-|\D g|^2\,\d\ppi(\gamma).
\]
The conclusion follows from the arbitrariness of $\eps\in\R$ and \eqref{eq:defdpm}.
\end{idea}
We conclude with two useful formulas, that for simplicity we state for compact  spaces:
\begin{lemma}[Derivative of $W_2^2$ along the heat flow]\label{le:derw2}
Let $(\X,\sfd,\mm)$ be a compact $\CD(K,\infty)$ space, $(f_t\mm)\subset \prd(\X)$ a heat flow and $\nu\in\prd(\X)$. Then $t\mapsto W_2^2(f_t\mm,\nu)$ is absolutely continuous and for a.e.\ $t>0$ we have 
\begin{equation}
\label{eq:derw2}
\frac\d{\d t}\tfrac12W_2^2(f_t\mm,\nu)=\int \varphi_t\,\Delta f_t\,\d\mm\qquad\text{ for any $\varphi_t$ Kantorovich potential from $f_t\mm$ to $\nu$.}
\end{equation}
\end{lemma}
\begin{idea} By definition of heat flow (and Theorem \ref{thm:hfgfagain}) the curve $t\mapsto f_t\mm$ is $W_2$-absolutely continuous. Hence so is $t\mapsto W_2^2(f_t\mm,\nu)$. Let $t$ be a differentiability point and $\varphi_t$ as in \eqref{eq:derw2} (by compactness $\varphi_t$ is bounded, thus in $L^2$). Then by Kantorovich duality  
\[
\tfrac12W_2^2(f_s\mm,\nu)\geq \int\varphi_t\, f_s\,\d\mm+\int\varphi_t^c\,\d\nu
\]
for any $s>0$, with equality for $s=t$. The conclusion follows from \eqref{eq:gfch}.
\end{idea}
\begin{lemma}[Derivative of the entropy along a $W_2$-geodesic]\label{le:derent}
Let $(\X,\sfd,\mm)$ be a compact space and $(\mu_t)=(\rho_t\mm)\subset \prd(\X)$ a $W_2$-geodesic of bounded compression so that   $\rho_0\geq c>0$ $\mm$-a.e.\ for some $c$. Assume also $\rho_0\in W^{1,2}(\X)$. Then
\begin{equation}
\label{eq:derent}
\limi_{t\downarrow0}\frac{\ent(\mu_t)-\ent(\mu_0)}{t}\geq \int D^-(\log\rho_0)(\nabla(-\varphi))\,\rho_0\d\mm
\end{equation}
for any $\varphi$ Kantorovich potential from $\mu_0$ to $\mu_1$.
\end{lemma}
\begin{idea} Since the space is compact, $\varphi$ is Lipschitz and thus in $W^{1,2}$, so the claim makes sense. Now let $\ppi$ be a lifting of $(\mu_t)$ and notice that it is a test plan and that the convexity of $u(z)=z\log(z)$ gives
\[
\frac{\ent(\mu_t)-\ent(\mu_0)}{t}\geq \int \log(\rho_0)\frac{\rho_t-\rho_0}{t}\,\d\mm=\int\frac{\log(\rho_0)(\gamma_t)-\log(\rho_0)(\gamma_0)}t\,\d\ppi(\gamma).
\]
The assumptions and \eqref{eq:dfchain} give $\log(\rho_0)\in W^{1,2}(\X)$, then the  conclusion follows combining  Lemma \ref{thm:metbr} with Lemma \ref{le:horver}.
\end{idea}

\subsubsection{Infinitesimally Hilbertian spaces}
We noticed that $W^{1,2}(\X)$ is in general a Banach space but not necessarily Hilbert. From the compatibility property \eqref{eq:dffinsl} it is also clear that if $\X$ is a smooth Finsler manifold, then $W^{1,2}(\X)$ is Hilbert iff the manifold is in fact Riemannian. This motivates the following definition, proposed in \cite{Gigli12}:
\begin{definition}\label{def:infhilb}  $(\X,\sfd,\mm)$ is said infinitesimally Hilbertian if $W^{1,2}(\X)$ is a Hilbert space.
\end{definition}

\begin{proposition}[Calculus on infinitesimally Hilbertian spaces]\label{prop:infhilbcalc} Let $(\X,\sfd,\mm)$ be infinitesimally Hilbertian. Then for any $f,g\in W^{1,2}(\X)$ we have $D^+f(\nabla g)=D^-f(\nabla g)$. Call this common value $\la\d f,\d g\ra$. Then $[W^{1,2}(\X)]^2\ni f,g\mapsto \la\d f,\d g\ra$ is bilinear, symmetric and 
\begin{subequations}
\begin{align}\label{eq:chainh}
\la\d(\varphi\circ f),\d g\ra&=\varphi'\circ f\la\d f,\d g\ra,\\
\label{eq:leibh}
\la\d(f_1f_2)\,\d g\ra&=f_1\la\d f_2,\d g\ra+f_2\la\d f_1,\d g\ra\\
\label{eq:csh}
|\la \d f,\d g\ra|&\leq|\D f|\,|\D g|\\
\label{eq:loch}
\la \d f,\d g\ra&=\la \d f',\d g'\ra\qquad\mm-a.e.\ on\ \{f=f'\}\cap\{g=g'\}
\end{align}
\end{subequations}
for every $ f,f',g,g'\in W^{1,2}(\X)$, $f_1,f_2\in W^{1,2}\cap L^\infty(\X)$ and $ \varphi\in C^1\cap \Lip(\R)$. Moreover:
\begin{equation}
\label{eq:intlaph}
\int g\Delta f\,\d\mm=-\int\la\d g,\d f\ra\,\d\mm\qquad\forall f\in D(\Delta),\ g\in W^{1,2}(\X).
\end{equation}
\end{proposition}
\begin{idea} For $f,g$ as in \eqref{eq:intlaph} and $\eps>0$ we have $-\eps\int g\Delta f\,\d\mm\leq {\ch(f+\eps g)-\ch(f)}$ and thus
\[
-\int g\Delta f\,\d\mm\leq\lim_{\eps\downarrow0}\int \frac{|\D (f+\eps g)|^2-|\D f|^2}{2\eps}\,\d\mm=\int D^+g(\nabla f)\,\d\mm,
\]
hence \eqref{eq:intlaph} follows from the other claims (also replacing $g$ with $-g$).

We already noticed that $f\mapsto \tfrac12|\D(g+f)|^2\in L^1$ is convex (in the $\mm$-a.e.\ sense). It follows that its directional derivative $f\mapsto D^+f(\nabla g)$ is convex and positively 1-homogeneous for every $g\in W^{1,2}(\X)$ and that  $D^+f(\nabla g)\geq - D^-(-f)(\nabla g)$. The assumption of infinitesimal Hilbertianity tells that $f\mapsto \int |\D f|^2\,\d\mm$ is a quadratic form, i.e.\ satisfies the parallelogram rule: from this it easily follows that $\int D^+f(\nabla g)\,\d\mm=\int - D^-(-f)(\nabla g)\,\d\mm$ so that by what just said the integrands must coincide a.e.. Hence $f\mapsto D^+f(\nabla g)$ is also concave and  the linearity in $f$ of $\la \d f,\d g\ra:=D^+f(\nabla g)=D^-f(\nabla g)$ follows.

The `Cauchy-Schwarz' inequality \eqref{eq:csh} follows from the bound $|\D(g+\eps f)|\leq |\D g|+|\eps||\D f|$, while the locality property \eqref{eq:loch} comes  from \eqref{eq:dfloc}. The linearity in $f$ gives the chain rule \eqref{eq:chainh}  for $\varphi$ affine  and then the locality gives \eqref{eq:chainh} for $\varphi$ piecewise affine. The general case follows  approximating a general $\varphi\in C^1\cap\Lip(\R)$ with uniformly Lipschitz and piecewise affine functions using the continuity estimate 
\[
|\la\d f,\d g\ra-\la\d f',\d g\ra|\leq |\D(f-f')|\,|\D g|
\]
that in turn follows from \eqref{eq:csh} and the linearity in $f$. For the Leibniz rule \eqref{eq:leibh} we first add a constant to $f_1,f_2$ (and use linearity in $f$) to reduce to the case  $f_1,f_2\geq 1$ $\mm$-a.e., then we apply \eqref{eq:chainh} with $\varphi:=\log$ (that is Lipschitz on the image of $f_1,f_2$) to get
\[
\begin{split}
\frac{\la\d(f_1f_2),\d g\ra}{f_1f_2}=\la\d(\log(f_1f_2)),\d g\ra=\la\d(\log f_1),\d g\ra+\la\d(\log f_2),\d g\ra=\frac{\la\d f_1,\d g\ra}{f_1}+\frac{\la\d f_2,\d g\ra}{f_2}.
\end{split}
\]
It thus remains to prove symmetry of $\la\d f,\d g\ra$, that  requires a bit more of work. One first notices that $W^{1,2}(\X)\ni g\mapsto \frac{\ch(g+\eps f)-\ch(g)}{\eps}\in \R$ is continuous. Hence $g\mapsto \int \la\d f,\d g\ra\,\d\mm=\int D^+f(\nabla g)\,\d\mm=\inf_{\eps>0}\frac{\ch(g+\eps f)-\ch(g)}{2\eps}$ is upper semicontinuous (being the $\inf$ of continuous functions). Replacing $f$ with $-f$ and recalling linearity in $f$, we see that this map is also lower semicontinuous, thus continuous. 

Now notice that letting $\eps\downarrow0$ in the trivial identity $\frac{|\D(\alpha g+\eps f)|^2-|\D(\alpha g)|^2}{2\eps}=\alpha \frac{|\D(g+\frac\eps\alpha f)|^2-|\D g|^2}{2\frac\eps\alpha}$ and using the linearity in $f$ we deduce 1-homogeneity in $g$, i.e.\  $\la\d f,\d(\alpha g)\ra=\alpha\la\d f,\d g\ra$. Then, as before, by the locality \eqref{eq:loch} we deduce that $\varphi'\circ g\la \d f,\d g\ra=\la\d f,\d(\varphi\circ g)\ra$ holds for $\varphi$ piecewise affine and by approximation and thanks to the continuity just proved we conclude that
\begin{equation}
\label{eq:contg}
\int\varphi'\circ g\la \d f,\d g\ra\,\d \mm=\int\la\d f,\d(\varphi\circ g)\ra\,\d\mm\qquad\forall f,g\in W^{1,2}(\X),\ \varphi\in C^1\cap\Lip(\R).
\end{equation}
We also observe that we have $\frac{\ch(g+\eps f)-\ch(g)}{\eps}-\eps\ch(f)=\frac{\ch(f+\eps g)-\ch(f)}{\eps}-\eps\ch(g)$ (because $\ch$ is a quadratic form), hence letting $\eps\downarrow0$ we deduce the integrated symmetry identity
\begin{equation}
\label{eq:simmint}
\int\la\d f,\d g\ra\,\d\mm=\int\la\d g,\d f\ra\,\d\mm.
\end{equation}
Now for $h\in \Lip_{\bs}(\X)$, using \eqref{eq:chainh}, \eqref{eq:leibh}, \eqref{eq:contg} and \eqref{eq:simmint} we have
\[
\begin{split}
\int h|\D f|^2\,\d\mm&=\int \la\d(hf),\d f\ra -f\la \d h,\d f\ra\,\d\mm=\int \la\d(hf),\d f\ra -\la \d h,\d (\tfrac{f^2}2)\ra\,\d\mm.
\end{split}
\]
By the linearity in $f$ already established and the symmetry \eqref{eq:simmint}, both addends in the RHS are quadratic forms in $f$. Hence $f\mapsto \int h|\D f|^2\,\d\mm$ is a quadratic form and thus by polarization we get $\int h\la\d f,\d g\ra\,\d\mm=\int h\la\d g,\d f\ra\,\d\mm$. By the arbitrariness of $h$ we conclude.
\end{idea}

\subsubsection{The system of Evolution Variational Inequalities}
Let $M$ be a smooth Riemannian manifold and  $\E:M\to\R$ and $\gamma:\R^+\to M$ be smooth. Also, let $y\in M$, $\bar t\geq 0$, say that $y$ is not in the cut locus of $\gamma_{\bar t}$ and let $\eta:[0,1]\to M$ be the geodesic from $\gamma_t$ to $y$. Then
\[
\begin{split}
\tfrac12\frac\d{\d t}\sfd^2(\gamma_t,y)\restr{t=\bar t}=-\la \gamma'_{\bar t},\eta'_0\ra\qquad\text{and}\qquad \frac{\d}{\d s}\E(\eta_s)\restr{s=0}=\d \E_{\gamma_{\bar t}}(\eta'_0),
\end{split}
\]
and therefore  the identity $ \gamma'_{\bar t}=-\nabla\E_{\gamma_{\bar t}}$ can be characterized by\footnote{Compare the two derivatives in formula \eqref{eq:EVIpasso1} with the ones computed in Lemmas \ref{le:derw2}, \ref{le:derent}. Then see how \eqref{eq:defEVI} is derived and compare with the proof of the `only if' in Theorem \ref{thm:rcdevi}.}
\begin{equation}
\label{eq:EVIpasso1}
\frac\d{\d t}\sfd^2(\gamma_t,y)\restr{t=\bar t}= \frac{\d}{\d s}\E(\eta_s)\restr{s=0}.
\end{equation}
If $\E$ is $K$-convex this  in turn implies
\begin{equation}
\label{eq:defEVI}
\frac{\d}{\d t}\tfrac12\sfd^2(\gamma_t,y)+\E(\gamma_{ t})+\tfrac K2\sfd^2(\gamma_{ t},y)\leq \E(y)
\end{equation}
at $t=\bar t$.
This latter inequality can be written in abstract metric spaces and taken as alternative definition of gradient flow trajectory:
\begin{definition}[Evolution Variational Inequalities]\label{def:EVI}
Let $(\X,\sfd)$ be a metric space, $\E:\X\to\R\cup\{+\infty\}$ a lower semicontinuous functional and $K\in\R$. We say that $\gamma:\R^+\to\X$ satisfies the system of Evolution Variational Inequalities with parameter $K$, or simply that it is an $\EVI_K$ gradient flow trajectory, provided it is continuous on $\R^+$, absolutely continuous on $(0,\infty)$ and for any $y\in\X$ satisfies \eqref{eq:defEVI} for a.e.\ $t>0$.
\end{definition}
By the way inequality \eqref{eq:defEVI} has been derived, it encodes both information about the convexity of the functional and, in some sense, of the fact that the distance $\sfd$ is more ``Riemannian-like" than ``Finsler-like". At least the first of these claims can  be  made rigorous:
\begin{proposition}\label{prop:EVIconv}
Let $(\X,\sfd)$ be a geodesic space and $\E:\X\to\R\cup\{+\infty\}$ be such that every $x\in\X$ is the starting point of an $\EVI_K$-gradient flow trajectory. 

Then inequality \eqref{eq:defkconv} holds along \emph{every} geodesic; in particular $\E$ is $K$-geodesically convex.
\end{proposition}
\begin{idea}Let $\eta$ be a geodesic and $\gamma$ the $\EVI_K$-gradient flow trajectory starting from $\eta_{\frac 12}$. Write \eqref{eq:defEVI} for $t=0$ and the choices $y=\eta_0$ and $y=\eta_1$. Adding up the resulting inequalities we obtain
\[
\frac{\d}{\d t}\tfrac12\big(\sfd^2(\gamma_t,\eta_0)+\sfd^2(\gamma_t,\eta_1)\big)\restr{t=0}+2\E(\eta_{\frac12})+\tfrac K4\sfd^2(\eta_0,\eta_1)\leq \E(\eta_0)+\E(\eta_1).
\]
Conclude noticing that since $\eta_{\frac12}$ is a minimum for $\sfd^2(\cdot,\eta_0)+\sfd^2(\cdot,\eta_1)$, the derivative on the LHS is $\geq 0$.
\end{idea}
For two $\EVI_K$-gradient flow trajectories $\gamma,\eta$ it is easy to establish the contractivity estimate
\begin{equation}
\label{eq:contrEVI}
\sfd(\gamma_t,\eta_t)\leq e^{-Kt}\sfd(\gamma_0,\eta_0)\qquad\forall t\geq0,
\end{equation}
which in particular grants \emph{uniqueness} (compare with Example \ref{ex:nounique}). Roughly said, \eqref{eq:contrEVI} can be established as follows. Fix $t>0$, let $m$ be a midpoint of $\gamma_t$ and $\eta_t$ (say that the space has geodesics), then pick $y:=m$ in \eqref{eq:defEVI} written first for $\gamma$, then for $\eta$ and add up the resulting inequalities to get 
\[
\tfrac12\frac\d{\d s}\big(\sfd^2(\gamma_s,m)+\sfd^2(\eta_s,m)\big)\restr{s=t}+\E(\gamma_t)+\E(\eta_t)+ \tfrac K4\sfd^2(\gamma_t,\eta_t) \leq 2\E(m).
\]
By triangle inequality we have $\sfd^2(\gamma_s,m)+\sfd^2(\eta_s,m)\geq \frac12\sfd^2(\gamma_s,\eta_s)$ with equality for $s=t$, hence $\frac\d{\d s}\big(\sfd^2(\gamma_s,m)+\sfd^2(\eta_s,m)\big)\restr{s=t}\geq \tfrac12\frac\d{\d s}\sfd^2(\gamma_s,\eta_s)\restr{s=t}$ and taking also into account the $K$-convexity of $\E$ proved before, we get  $\frac14\frac\d{\d t}\sfd^2(\gamma_t,\eta_t)+\frac K2\sfd^2(\gamma_t,\eta_t)\leq 0$ and \eqref{eq:contrEVI} follows from Gronwall's lemma.

On the other hand, \emph{existence} of $\EVI_K$-gradient flows typically fails on Finsler-like spaces (explicitly: the space and functional considered in Example \ref{ex:nounique} do not admit an $\EVI_K$-gradient flow for any $K\in\R$).

The term `gradient flow' in connection with the $\EVI$ is justified by item $(i)$ below:
\begin{proposition}[Basic properties of $\EVI_K$-gradient flows]\label{prop:pasevi} Let $(x_t)$ be an $\EVI_K$-gradient flow trajectory for $\E$. Then:
\begin{itemize}
\item[i)] For any $\eps>0$ the curve $t\mapsto x_{t+\eps}$ is locally Lipschitz and also a gradient flow trajectory in the sense of Definition \ref{def:gfede}. 
\item[ii)] The maps $(0,\infty)\ni t\mapsto \E(x_t), e^{Kt}|\partial^-\E|(x_t)$ are non-increasing,
\item[iii)] For any $y\in\X$ with $\E(y)<\infty $ and $t>0$ the a priori estimate
\begin{equation}
\label{eq:aprioriEVI}
\tfrac{e^{Kt}}2\sfd^2(x_t,y)+ I_K(t)\big(\E(x_t)-\E(y)\big)+\tfrac{I_K(t)^2}2|\partial^-\E|^2(x_t)\leq\tfrac12\sfd^2(x_t,y),
\end{equation}
holds, where $I_K(t):=\int_0^te^{Kr}\,\d r$ is equal to $t$ if $K=0$.
\end{itemize}
\end{proposition}
\begin{idea} Say $K=0$. For any $h>0$ the curve $t\mapsto x_{t+h}$ is also an $\EVI_K$-gradient flow trajectory, hence \eqref{eq:contrEVI} gives $\sfd(x_s,x_{s+h})\leq\sfd(x_t,x_{t+h})$ for any $s>t>0$ and thus $|\dot x_s|\leq |\dot x_t|$, giving local Lipschitz regularity of $(x_t)$. Rearranging the terms in  \eqref{eq:defEVI} we get
\begin{equation}
\label{eq:2v}
\frac{\E(x_t)-\E(y)}{\sfd(x_t,y)}\leq\frac1{2\sfd(x_t,y)}\frac\d{\d t}\sfd^2(x_t,y)=\frac\d{\d t}\sfd(x_t,y)\leq|\dot x_t|
\end{equation}
and letting $y\to x_t$ we get $|\partial^-\E|(x_t)\leq|\dot x_t|$. Picking $y=x_s$ in \eqref{eq:2v}, then swapping the roles of $t,s$ and recalling that $|\dot x_t|$ is locally bounded, we see that $t\mapsto \E(x_t)$ is locally Lipschitz. Thus it is a.e.\ differentiable and for a.e.\ $t$ we have
\begin{equation}
\label{eq:perEDE}
-\frac\d{\d t}\E(x_t)=\lim_{h\downarrow0}\frac{\E(x_t)-\E(x_{t+h})}{\sfd(x_t,x_{t+h})}\frac{\sfd(x_t,x_{t+h})}h\leq |\partial^-\E|(x_t)|\dot x_t|\stackrel*\leq\tfrac12|\dot x_t|^2+\tfrac12|\partial^-\E|^2(x_t).
\end{equation}
Now integrate \eqref{eq:defEVI} from $t$ to $t+h$ and then pick $y:=x_{t}$ to get 
\[
\frac{\sfd^2(x_t,x_{t+h})}2\leq \int_t^{t+h}\E(x_{t})-\E(x_{s})\,\d s=h\int_0^1\E(x_t)-\E(x_{t+rh})\,\d r.
\]
Dividing by $h^2$ and letting $h\downarrow0$ we conclude that $|\dot x_t|^2\leq -\partial_t\E(x_t)$, that together with what proved above gives $(i)$. Now the fact that $ t\mapsto \E(x_t)$ is non-increasing is obvious, while for $ t\mapsto |\partial^-\E|(x_t)$  notice that we proved that  $t\mapsto |\dot x_t|$ is non-decreasing and that it is equal to $|\partial^-\E|(x_t)$ (since the above arguments give that the starred inequality in \eqref{eq:perEDE} is an equality).

For $(iii)$ we collect what proved so far and to get
\[
\begin{split}
\tfrac{t^2}2|\partial^-\E|^2(x_t)&\leq\int_0^ts|\partial^-\E|^2(x_s)\,\d s=-\int_0^ts\frac\d{\d s}(\E(x_s)-\E(x_t))\,\d s=\int_0^t\E(x_s)-\E(x_t)\,\d s\\
\text{(by \eqref{eq:defEVI})}\qquad&\leq \int_0^t\E(y)-\E(x_t)-\tfrac12\frac\d{\d s}\sfd^2(x_s,y)\,\d s\leq t\big(\E(y)-\E(x_t)\big)+\tfrac12\sfd^2(x_0,y),
\end{split}
\]
concluding the proof.
\end{idea}
For $\EVI_K$-gradient flows a stability result akin to, in fact stronger than, Theorem \ref{thm:stabede} holds: notice that, unlike Theorem \ref{thm:stabede}, it is not necessary to assume that the limit initial point has finite energy, nor that the converging sequence is a recovery sequence.
\begin{theorem}[Stability]\label{thm:stabevi} Let $(\X,\sfd)$ be compact and  $({\sf E_n})$ be a sequence of non-negative, lower semicontinuous functionals $\Gamma$-converging to a limit $\E_\infty$. For $n\in\N$ let $(x_{n,t})$  be an $\EVI_K$-gradient flow trajectory for $\E_n$ and assume that $x_{n,0}\to x_{\infty,0}\in\X$. Then $(x_{n,t})$ converge locally uniformly to a limit curve $x_{\infty,t}$ and such curve is an $\EVI_K$-gradient flow trajectory for $\E_\infty$.
\end{theorem}
\begin{idea}  Assume to know already that $x_{n,t}\to x_{\infty,t}$ locally uniformly for some limit curve $x_{\infty,\cdot}$, let $y\in D(\E_\infty)$  and find a recovery sequence $(y_n)$ for it. Integrate \eqref{eq:defEVI} written for $x_{n,t}$ from $t$ to $s$ to get
\[
\frac{\sfd^2(x_{n,s},y_n)-\sfd^2(x_{n,t},y_n)}{2}+\int_t^s\E_n(x_{n,r})+\tfrac K2\sfd^2(x_{n,r},y_n)\,\d r\leq (s-t)\E_n(y_n).
\]
Letting $n\to\infty$ and using Fatou's lemma to handle the integral of the energy, we easily see that an analogous inequality is in place for $x_{\infty,t}$. Then the arbitrariness of $y$ and of $t,s$ gives that $x_{\infty,\cdot}$ is an $\EVI_K$-gradient flow trajectory: here the a priori estimates \eqref{eq:aprioriEVI} and the fact that the $x_{n,t}$'s are EDE gradient flow trajectories can be used to prove that the limit curve is locally absolutely continuous.

We now  prove that a limit curve exists and notice that if $\lims_n\E_n(x_{n,0})<\infty$, then since we know that these curves are also $\EDE$-gradient flow trajectories,  the compactness result in Theorem \ref{thm:stabede} applies (notice that its proof  needs just the energies of the initial points to be bounded): up to subsequences the $(x_{n,t})$'s converge to a limit curve that, by what proved above, is an $\EVI_K$-gradient flow trajectory. As we know that these are unique, we get convergence of the full sequence.

For the general case we use the a priori estimates \eqref{eq:aprioriEVI} to prove that $x_{\infty,0}\in \overline{D(\E_\infty)}$. Then we let $(z_{\infty,k})\subset D(\E_\infty)$ be converging to $x_{\infty,0}$ and, for each $k$, $n\mapsto z_{n,k}$ be a recovery sequence for it and $t\mapsto z_{n,k,t}$ the $\EVI_K$-gradient flow trajectory starting from it. What we just proved tells that $(z_{n,k,t})$ converge locally uniformly to the $\EVI_K$-gradient flow trajectory $z_{\infty,k,t}$ of $\E_\infty$ starting from $z_{\infty,k}$, hence letting first $n,m\to\infty$ and then $k\to\infty$ in the bound
\[
\begin{split}
\sfd(x_{n,t},x_{m,t})&\leq \sfd(x_{n,t},z_{n,k,t})+\sfd(z_{n,k,t},z_{m,k,t})+\sfd(z_{m,k,t},x_{m,t})\\
\text{(by \eqref{eq:contrEVI})}\qquad &\leq e^{-Kt}( \sfd(x_{n},z_{n,k})+\sfd(x_{m},z_{m,k}))+\sfd(z_{n,k,t},z_{m,k,t})\\
\end{split}
\]
we conclude.
\end{idea}
Picking $\E_n\equiv \E$ in this last result we deduce:
\begin{equation}
\label{eq:exevi}
\begin{split}
&\text{if $\E$ admits $\EVI_K$-gradient flow trajectories starting from any $x\in D(\E)$}\\
&\text{then it admits $\EVI_K$-gradient flow trajectories starting from any $x\in \overline{D(\E)}$.}
\end{split}
\end{equation}
This marks an important difference from the $\EDE$ condition, for which there is no general existence result for initial data with infinite energy.

\subsubsection{The Heat Flow as gradient flow (again$^2$)}
\begin{definition}[The Riemannian Curvature Dimension condition]\label{def:rcd} We say that $(\X,\sfd,\mm)$ is an $\RCD(K,\infty)$ space, $K\in\R$, provided it is $\CD(K,\infty)$ and infinitesimally Hilbertian.
\end{definition}

\begin{theorem}\label{thm:rcdevi} Let $(\X,\sfd,\mm)$ be a normalized metric measure space. Then it is $\RCD(K,\infty)$ if and only if $\ent$ admits an $\EVI_K$-gradient flow trajectory starting from any $\mu\in \overline{D(\ent)}$.
\end{theorem}
\begin{idea} Say that $\X$ is compact.\\
{\sc Only if} Let $(\mu_t)=(f_t\mm)$ be an $\EDE$-gradient flow trajectory of the entropy, fix $t>0$, let $\nu\in\prd(\X)$ be a measure with bounded density, $\varphi_t$ a Kantorovich potential from $\mu_t$ to $\nu$ and $(\nu_s)$ the $W_2$-geodesic from $\mu_t$ to $\nu$ given by Lemma \ref{le:tapio}. Then 
\begin{equation}
\label{eq:bello}
\begin{split}
\frac\d{\d t}\tfrac12W_2^2(\mu_t,\nu)&\stackrel{\eqref{eq:derw2}}=\int\varphi_t\Delta f_t\,\d\mm\stackrel{\eqref{eq:intlaph}}=-\int\la\d \varphi_t,\d f_t\ra\,\d\mm\stackrel{\eqref{eq:chainh}}=-\int\la \d(\log f_t),\d\varphi_t\ra \,\d\mu_t\\
&\stackrel{\eqref{eq:derent}}\leq \frac\d{\d s}\ent(\nu_s)\restr{s=0} \stackrel{*}\leq \ent(\nu)-\ent(\mu_t)-\tfrac K2W_2^2(\mu_t,\nu),
\end{split}
\end{equation}
having used the $K$-convexity of the entropy along $(\nu_s)$ in the last step.\\
{\sc If} Let $(\mu^0_t),(\mu^1_t)\subset \prd(\X)$ be two $\EVI_K$-gradient flow trajectories. We shall prove that $t\mapsto\mu_t:=\frac{\mu^0_t+\mu^1_t}2$ is also an $\EVI_K$-gradient flow trajectory: since these are also (the only, by Theorem \ref{thm:uniqueheat}) $\EDE$-gradient flow trajectories for the entropy, by Theorem \ref{thm:hfgfagain} we see that the gradient flow of $\ch$ linearly depends on the initial datum. Since $L^2(\X)$ is Hilbert, it is readily verified that a functional has linear gradient flow iff it is a quadratic form, thus giving the conclusion. 

To see that $(\mu_t)$ satisfies \eqref{eq:defEVI}, fix $\nu=\eta\mm\in D(\ent)$, $t>0$ and an optimal plan $\ggamma$ from $\mu_t$ to $\nu$. Then define $\nu^i:=\ggamma_*\mu^i_t$ (recall \eqref{eq:gammapf}) and notice that \eqref{eq:w2conv} and the fact that $\ggamma_{\mu^i_t}$ is optimal (because its support is contained in the support of the optimal plan $\ggamma$) we have $W_2^2(\mu_s,\nu)\leq \tfrac12\sum_{i}W_2^2(\mu^i_s,\nu^i)$ for every $s$, with equality for $s=t$. Thus
\[
\begin{split}
\partial_s\tfrac12W_2^2(\mu_s,\nu)\restr{s=t}&\leq \tfrac12\Big(\partial_s\tfrac12W_2^2(\mu^0_s,\nu^0)\restr{s=t}+\partial_s\tfrac12W_2^2(\mu^1_s,\nu^1)\restr{s=t}\Big)\\
\text{(by the $\EVI_K$ for $(\mu^i_t)$)}\qquad &\leq \tfrac12\Big(\ent(\nu^0)-\ent(\mu^0_t)+\ent(\nu^1)-\ent(\mu^1_t)\Big),
\end{split}
\]
and the conclusion follows recalling Lemma \ref{le:convDE}.
\end{idea}
By \eqref{eq:exevi} we see that if $(\X,\sfd,\mm)$ is $\RCD(K,\infty)$, then for every $x\in\supp(\mm)$ there is a heat flow (i.e.\ an $\EVI_K$-gradient flow trajectory) $t\mapsto\h_t\delta_x$ starting from $\delta_x$. The a priori estimate \eqref{eq:aprioriEVI} tell in particular that $\h_t\delta_x\in D(\ent)$, and thus that $\h_t\delta_x\ll\mm$. We can certainly call  \emph{heat kernel} its density $\rho_t[x]:=\tfrac{\d\h_t\delta_x}{\d\mm}$.

The integration by parts formula \eqref{eq:intlaph} and the symmetry of $(f,g)\mapsto\la\d f,\d g\ra$ imply that $\Delta$ is self-adjoint in our setting, thus the heat flow is self-adjoint in $L^2$, i.e.\ satisfies $\int f\h_tg\,\d\mm=\int g\h_tf\,\d\mm$ (to see this differentiate $\int \h_sf\h_{t-s}g\,\d\mm$ in $s$). Then with little work one sees that
\[
\rho_t[x](y)=\rho_t[y](x)\qquad\mm\otimes\mm-a.e.\ x,y\in\X^2
\]
and then that the expected representation formula
\begin{equation}
\label{eq:reprform}
\h_tf(x)=\int f\,\d\h_t\delta_x\qquad\mm-a.e.\ x
\end{equation}
holds. One can also use the RHS of the above to extend the heat flow to a linear contraction semigroup in $L^p$ for any $p\in[1,\infty]$ (obtaining also a continuous version of $\h_tf$ if $f\in L^\infty$).

The heat flow is THE regularizing tool when working on $\RCD$ spaces. In some situation it is convenient to regularize it also in time, i.e.\ given $\eta\in C^\infty_c(0,\infty)$, to define $\h_\eta f:=\int \h_tf\eta(t)\,\d t$; then the formal computation $\Delta\h_\eta f=\int \eta(t)\Delta\h_tf\,\d t=\int \eta(t)\partial_t\h_tf\,\d t=-\int \eta'(t) \h_tf\,\d t$ shows
\begin{equation}
\label{eq:regheat}
f\in L^p(\X)\qquad\Rightarrow\qquad \Delta\h_\eta f\in L^p(\X)\qquad\forall p\in[1,\infty].
\end{equation}
Letting $\eta\weakto\delta_0$ it is not hard to check that the class of $f$ (resp.\ $g$) as in the next theorem is dense in $W^{1,2}(\X)$ (resp.\ in the class of non-negative $W^{1,2}$-functions).
\begin{theorem}[Bochner inequality]\label{thm:boch}
Let $(\X,\sfd,\mm)$ be an $\RCD(K,\infty)$ space. Then:
\begin{equation}
\label{eq:bochner}
\int \tfrac12|\d f|^2\Delta g\,\d\mm\geq \int g(\la\d f,\d\Delta f\ra+K|\d f|^2)\,\d\mm
\end{equation}
for any $g\in L^\infty(\X)\cap D(\Delta)$ non-negative with $\Delta g\in L^\infty(\X)$ and $f\in D(\Delta)$ with $\Delta f\in W^{1,2}(\X)$.
\end{theorem}
\begin{idea}
The plan is:
\[
\begin{array}{c}
\text{$\EVI_K$-property of the heat flow}\\
{\Downarrow}\\
\text{$W_2$-contractivity of heat flow: } W_2(\mu_t,\nu_t)\leq e^{-Kt}W_2(\mu_0,\nu_0)\\
\Downarrow\\
\text{Bakry-\'Emery contraction estimate: } |\d\h_tf|^2\leq e^{-2Kt}\h_t(|\d f|^2)\\
\Downarrow\\
\text{Bochner inequality } \tfrac12\Delta|\d f|^2\geq\la\d f,\d\Delta f\ra+K|\d f|^2 \text{ in the weak sense.}
\end{array}
\]
Here the first implication comes from the contractivity property \eqref{eq:contrEVI} and the last one, at least formally, by noticing that for $t=0$ the Bakry-\'Emery inequality is in fact an equality, so that by differentiating in $t$ at $t=0$ we deduce
\[
2 \la\d f\,\d\Delta f\ra=\partial_t |\d\h_tf|^2\restr{t=0}\leq\partial_t (e^{-2Kt}\h_t(|\d f|^2))\restr{t=0}=-2K|\d f|^2+\Delta |\d f|^2.
\]
For the second, let $f\in\Lip_{\bs}(\X)$, $y,z\in\supp(\mm)$ and $\bar \gamma$ a geodesic connecting them. Then by \eqref{eq:contrEVI} (and $W_2(\delta_{x_1},\delta_{x_2})=\sfd(x_1,x_2)$) the curve $r\mapsto\mu_r:=\h_t(\delta_{\bar\gamma_r})$ is $e^{-Kt}$-Lipschitz, thus it admits a lifting $\ppi$ as in Lemma \ref{le:superpp}. Hence recalling \eqref{eq:reprform} we get
\[
\begin{split}
|\h_tf(z)-\h_t f(y)|&\leq \int|f(\gamma_1)-f(\gamma_0)|\,\d\ppi(\gamma)\leq\iint_0^1\lip_af(\gamma_r)|\dot\gamma_r| \,\d r\,\d\ppi(\gamma)\\
&\leq\sqrt{\iint_0^1\lipa^2f \,\d\mu_r\,\d r\int_0^1|\dot\mu_r|^2\,\d r}\leq e^{-Kt} \sfd(z,y) \sqrt{\int_0^1\h_t(\lip_a^2f)(\bar \gamma_r)\,\d r}.
\end{split}
\]
Now notice that the upper semicontinuity of $\lip_af$, the weak continuity of $x\mapsto \h_t\delta_x$ (that follows from  \eqref{eq:contrEVI}) and the representation formula \eqref{eq:reprform} imply that $\h_t(\lip_a^2f)$ is upper semicontinuous, thus letting $y,z\to x$ in the above we deduce 
\[
\lip_a(\h_tf)^2\leq e^{-2Kt}\h_t(\lip_a^2f).
\]
Then the Sobolev estimate follows by relaxation: choose $(f_n)\subset\Lip(\X)$ optimal for $\ch(f)$, notice that $\h_tf_n\to\h_tf$ in $L^2$ and  that the above argument shows that  $\h_tf_n\subset\Lip(\X)$. The conclusion follows by the $\mm$-a.e.\ minimality of $|\D\h_tf|$.
\end{idea}
We have already noticed that $\Delta$ is a linear operator on $\RCD$ spaces; from this it is easy to see that $\Delta f$ is, whenever it exists, the only element in $-\partial^-\ch(f)$ and that the metric slope $|\partial^-\ch|(f)$ as in \eqref{eq:defsl} equals to $\|\Delta f\|_{L^2}$ (being intended that it is $+\infty$ if $f\notin D(\Delta)$).  Thus the a priori estimates \eqref{eq:aprioriEVI} for the convex functional $\ch$ and $y=0$ read as
\begin{equation}
\label{eq:aprioriheat}
\tfrac12\|\h_tf\|_{L^2}^2+t\ch(\h_tf)+\tfrac{t^2}2\|\Delta \h_tf\|^2_{L^2}\leq \tfrac12\|f\|_{L^2}^2\qquad\forall f\in L^2,\ t\geq 0.
\end{equation}
The characterization  $\partial^-\ch(f)=\{-\Delta f\}$ also implies the closure property along varying spaces
\begin{equation}
\label{eq:laplcl}
\left.\begin{array}{rl}
f_n&\to\ f_\infty\\
\Delta f_n&\weakto\ h
\end{array}\right\}\qquad\Rightarrow\qquad f_\infty\in D(\Delta)\quad\text{ and }\quad\Delta f_\infty=h,
\end{equation}
whenever the weak/strong convergence is intended in the $L^2$ sense as in Definition \ref{def:convl2var} and the underlying spaces $\X_n$ are all $\RCD(K,\infty)$, normalized and mGH-converging to a limit $\X_\infty$. To see \eqref{eq:laplcl}, let $g_\infty\in L^2(\X_\infty)$, find $g_n\to g_\infty$ recovery sequence for the Mosco convergence of the Cheeger energies and pass to the limit in $\ch_n(f_n)-\int\Delta f_n(g_n-f_n)\,\d\mm_n \leq \ch_n(g_n)$ to conclude that $-h\in\partial^-\ch(f_\infty)$, as desired.

\begin{corollary}[Stability of the heat flow (again)]\label{cor:stabhagain} Let $n\mapsto (\X_n,\sfd_n,\mm_n)$ be a sequence of normalized $\RCD(K,\infty)$ spaces mGH-converging to $(\X_\infty,\sfd_\infty,\mm_\infty)$. Let $n\mapsto f_n\in L^2(\X_n)$ be strongly $L^2$-converging to $f_\infty\in L^2(\X_\infty)$.

Then for every $t>0$ the sequences  $n\mapsto \h_tf_n,\Delta\h_nf_t$ converge $L^2$-strongly to $\h_tf_\infty,\Delta\h_tf_\infty$ respectively and  $n\mapsto \ch_n(\h_tf_n),\ch(\Delta\h_tf_n)$ converge to $\ch_n(\h_tf_\infty),\ch(\Delta\h_tf_\infty)$ respectively.
\end{corollary}
\begin{idea} If the $f_n$'s are also probability densities, then Theorems \ref{thm:stabheat}, \ref{thm:hfgfagain} and the a priori estimates \eqref{eq:aprioriheat} give that  $(\h_tf_n)$ weakly $L^2$-converges to $\h_tf_\infty$ for any $t>0$. The same conclusion for the $f_n$'s as in the statement follows by linearity and an approximation argument, e.g.\ by truncation. Thus Theorem \ref{thm:convslen} gives  $\limi_n\ch_n(\h_tf_n)\geq\ch_\infty(\h_tf_\infty)$.

Differentiating in $t$ the function $\|\h_tf_n\|_{L^2}^2$ we get $\|\h_tf_n\|_{L^2}^2+4\int_0^t\ch(\h_sf_n)\,\d s=\|f_n\|_{L^2}^2$, thus passing to the limit, by Fatou's lemma we get
\begin{equation}
\label{eq:perconvh}
\|\h_tf_\infty\|_{L^2}^2+4\int_0^t\ch(\h_sf_\infty)\,\d s\leq\limi_n\Big(\|\h_tf_n\|_{L^2}^2+4\int_0^t\ch(\h_sf_n)\,\d s\Big)\leq \limi_n\|f_n\|_{L^2}^2=\|f_\infty\|_{L^2}^2.
\end{equation}
Since the leftmost and rightmost side agree, we must have $\limi_n\|\h_tf_n\|_{L^2}=\|\h_tf_\infty\|_{L^2}$ and then it is easy to conclude that $\lim_n\|\h_tf_n\|_{L^2}=\|\h_tf_\infty\|_{L^2}$ so that $\h_tf_n\to\h_tf_\infty$ strongly in $L^2$. To prove  $\lims_n\ch_n(\h_tf_n)\leq\ch_\infty(\h_tf_\infty)$.  we notice that the above and the simple identity $\partial_t\ch(\h_tf)=-\|\Delta\h_tf\|^2_{L^2}$ yield convexity of $t\mapsto\|\h_tf\|_{L^2}^2$ and thus that $\ch(\h_tf)=\inf_{s>0}\tfrac1{4s}({\|\h_tf\|_{L^2}^2-\|\h_{t+s}f\|_{L^2}^2})$; hence $\lims_n\tfrac1{4s}({\|\h_tf_n\|_{L^2}^2-\|\h_{t+s}f_n\|_{L^2}^2})\geq\lims_n\ch_n(\h_tf_n)$ for any $s>0$ and letting $s\downarrow0$ we conclude.

For the claims about the Laplacian, notice that  \eqref{eq:laplcl} and the above  give $\Delta\h_tf_n\weakto\Delta\h_tf_\infty$. For  the $\lims$ inequality for the norms we argue as before noticing the convexity of $t\mapsto\ch(\h_tf)$ and then the identity $\|\Delta\h_tf\|^2_{L^2}=\inf_{s>0}\tfrac1s(\ch(\h_t f)-\ch(\h_{t+s}f))$. For $\ch(\Delta\h_tf_n)\to \ch(\Delta\h_tf_\infty)$ we use the trivial identity $\Delta\h_tf_n=\h_{t/2}\Delta\h_{t/2}f_n$ and what already proved.
\end{idea}

\subsubsection{A reconstruction theorem}
Geometric structures are often studied via the properties of suitable algebraic/analytic objects defined on it.  A reconstruction (or representation) theorem is a result that, roughly speaking, invert this procedure, so that any suitable abstract algebro-analytic structure must come, in a unique way, from a  geometric space of the type considered. Typical examples are:
\begin{itemize}
\item[i)] A Stone space, or profinite set, is a compact totally disconnected Hausdorff space. The collection of homeomorphism from such a space to the discrete space $\{0,1\}$ form in a natural way a Boolean algebra. The \emph{Stone representation theorem} asserts that any Boolean algebra is isomorphic to one built this way (the Stone space being that of ultrafilters on the algebra).
\item[ii)] Given a compact and Hausdorff topological space, the space of ${\mathbb C}$-valued continuous maps for a commutative $C^*$-algebra with unit. The  \emph{Gelfand-Naimark reconstruction theorem} asserts that any such  $C^*$-algebra is isomorphic to the one induced by a suitable compact and Hausdorff space (its spectrum).
\item[iii)] To a compact Riemannian manifold $M$ we can associate the so-called spectral triple $(C^\infty(M),L^2(M,\Lambda TM),\d+\d^*)$ made of the  Hilbert space $L^2(M,\Lambda TM)$, the commutative algebra $C^\infty(M)$ acting on it (by multiplication),  and the unbounded operator  $\d+\d^*$ on  $L^2(M,\Lambda TM)$ (whose square is the Hodge Laplacian).  The \emph{Connes reconstruction theorem} asserts that any spectral triple $(A,H,D)$ satisfying some natural, but non-trivial, properties arises this way. 
 \end{itemize}

In the setting of $\RCD$ spaces, a natural candidate for a theorem of this sort, very much in the spirit of Connes' result, is given by the couple $(L^2(\X),\Delta)$, possibly enriched with some additional data.

While no clean statement as the ones above is currently available, still a genuinely `functional' version of the $\RCD$ condition is available. To clarify the structure of both the statement and the proof, it is better to start with the following result, which has anyhow an intrinsic interest: notice indeed that it provides a characterization of $\RCD$ spaces independent on optimal transport and study of geodesics and for this reason is relevant in applications, as it makes it easier in some cases to detect $\RCD$ spaces. For instance, the fact that the product of two $\RCD$ spaces is still $\RCD$ is based on the result below.
\begin{theorem}\label{thm:reprinterm}
Let $K\in\R$ and $(\X,\sfd,\mm)$ be a metric measure space.  Assume that
\begin{itemize}
\item[a)] The space is infinitesimally Hilbertian.
\item[b)] For some $\bar x\in\X$ and $C>0$ we have $\mm(B_r(\bar x))\leq Ce^{Cr^2}$ for every $r>0$.
\item[c)] Every function $f\in W^{1,2}(\X)$ with $|\D f|\leq 1$ $\mm$-a.e.\ admits a 1-Lipschitz representative.
\item[d)] The Bochner inequality \eqref{eq:bochner} holds for any $f,g$ as in the statement of Theorem \ref{thm:boch}.
\end{itemize}
Then $(\X,\sfd,\mm)$ is an  $\RCD(K,\infty)$ metric measure space. 
\end{theorem}
\begin{idea} Say $\X$ is compact, $\mm\in\pr(\X)$ and $K=0$. Let $(\h_t)$ be the $L^2$-gradient flow of $\ch$.  Roughly said, one proves that:
\begin{subequations}
\begin{align}
i)&\text{ $(\supp(\mm),\sfd)$ is a geodesic space},\\
ii)&\text{ the adjoint semigroup $\h_t(\rho\mm):=(\h_t\rho)\mm$ $W_2$-continuously extends to $\pr(\X)$},\\
iii)&
\text{ for any $s\mapsto\mu_s:=\rho_s\mm$ sufficiently regular and $t>0$ we have }\\
\label{eq:acest}&\text{$W_2^2(\mu_0,\h_t\mu_1)+2t\ent(\h_t\mu_1)\leq\int_0^1|\dot\mu_s|^2\,\d s+2t\ent(\mu_0)$,}\\
iv)&\text{ any $W_2$-abs.cont.\ curve can be approximated by curves to which $(iii)$ applies.}
\end{align}
\end{subequations}
Once these are proved, we can apply a suitable approximation argument based on $(iv)$  to deduce that the estimate in $(iii)$  holds for a $W_2$-geodesic $(\mu_t)$, in which case it reads as   $\frac{W_2^2(\mu_0,\h_t\mu_1)-W_2^2(\mu_0,\mu_1)}{2t}\leq \ent(\mu_0)-\ent(\h_t\mu_1)$. Then arguing as for the proof of Proposition \ref{prop:EVIconv} we deduce that the entropy is convex along $W_2$-geodesics, as desired.

To prove $(i)$  it is enough to show that for any $x,y\in\supp(\mm)$ and $\eps>0$ there is $m\in\supp(\mm)$ with $\sfd(x,m),\sfd(y,m)\leq\frac{r}2+\eps$ for $r:=\sfd(x,y)$. Say not, thus there are $x,y,\eps$ so that $\mm(B_{\frac{r}2+\eps}(x)\cap B_{\frac{r}2+\eps}(y))=0$. Then the Lipschitz function 
\[
f(z):=\big(\tfrac{r+\eps}2-\sfd(z,x)\big)^+-\big(\tfrac{r+\eps}2-\sfd(z,y)\big)^+
\]
is so that $\lipa f\leq 1$ on $\supp(\mm)$, thus $|\d f|\leq 1$ $\mm$-a.e.\ and by $(c)$ and the continuity of $f$ we get that $f$ is 1-Lipschitz on $\supp(\mm)$. However $|f(x)-f(y)|=r+\eps>r=\sfd(x,y)$.

For $(ii)$  recall that the heat flow preserves the mass and the order (by linearity and weak maximum principle), thus it can be extended by continuity to $L^1$. Then let $f, g$ be as in \eqref{eq:bochner}, $t>0$ and put $F(s):=\int\h_s(|\d\h_{t-s}f|^2)g\,\d\mm$. Using \eqref{eq:bochner} we see that $F'(s)=\int |\d\h_{t-s}f|^2\Delta\h_sg-2\la\d\h_{t-s}f,\d\Delta\h_{t-s}f \ra g\,\d\mm\geq0$, so that $\int(\h_t(|\d f|^2)-|\d \h_tf|^2)g\,\d\mm=F(t)-F(0)\geq 0 $ and  the arbitrariness of $g\geq 0$ implies  the Bakry-\'Emery estimate
\begin{equation}
\label{eq:BElav}
|\d \h_tf|^2\leq\h_t(|\d f|^2)\qquad\mm-a.e.\ \forall t>0\qquad\forall f\in W^{1,2}(\X).
\end{equation}
In particular, if $f$ is 1-Lipschitz, this and  assumption $(c)$ ensure that $\h_tf$ has a 1-Lipschitz representative.  Thus for $\mu=\rho\mm$, $\nu=\eta\mm$ by duality we have
\[
W_1(\h_t\mu,\h_t\nu)=\sup_{\Lip(f)\leq 1}\int f(\h_t\rho-\h_t\eta)\,\d\mm\leq \sup_{\Lip(g)\leq 1}\int g(\rho-\eta)\,\d\mm=W_1(\mu,\nu)
\]
proving the claim (in compact spaces $W_1$ and $W_2$ both induce the weak topology). It is then easy, using that $\h_t$ is self adjoint in $L^2$, to prove that $\h_tf(x)=\int f\,\d\h_t\delta_x$ holds for a.e.\ $x$ and that if $f$ is continuous, so is $x\mapsto \int f\,\d\h_t\delta_x$, thus providing a continuous representative for $\h_tf$ that will be considered below. Even more, the estimate \eqref{eq:apriorical} that we shall prove later on shows that the functions $\frac{\d\h_t\delta_x}{\d\mm}$ are uniformly integrable, thus $x\mapsto \int f\,\d\h_t\delta_x $ is continuous for any $f\in L^\infty$.
 
We pass to $(iii)$ and  notice that   for $|\d f|\in L^\infty$ what just said, the estimate \eqref{eq:BElav}, assumption $(c)$ and a localization argument grant that  $\h_tf$ is Lipschitz with $\lipa(\h_tf)^2\leq\h_t(|\d f|^2)$ a.e.\ on $\X$. It follows that
\begin{equation}
\label{eq:uppreg}
|\d f|\in L^\infty\qquad\Rightarrow\qquad\text{$\lipa(\h_tf)\to|\d f|$ in $L^2\quad$ and $\quad\sup_{t<1}\|\lipa \h_tf\|_\infty<\infty$.}
\end{equation}
From this we get that for $\varphi$ Lipschitz  and $t\mapsto \mu_t=\rho_t\mm$ $W_2$-absolutely continuous we have
\begin{equation}
\label{eq:derpuntvarphi}
\partial_t\int\varphi\,\d\mu_t\leq\tfrac12\int|\d\varphi|^2\,\d\mu_t+\tfrac12|\dot\mu_t|^2
\end{equation}
 at every $t$ for which the metric speed $|\dot\mu_t|$ exists (integrating the trivial bound $\partial_t\varphi(\gamma_t)\leq\tfrac{\lipa^2\varphi}2(\gamma_t)+\tfrac12|\dot\gamma_t|^2$ over a lifting  of $(\mu_t)$ we get the estimate with $\lipa\varphi$ in place of $|\d\varphi|$, then we argue by approximation using \eqref{eq:uppreg}). Then for the estimate in $(iii)$  let $\eta_s:=\h_{ts}\rho_s$, for $\varphi$ Lipschitz set $\varphi_s:=Q_s\varphi$ as in formula \eqref{eq:defHL} and notice that at least formally we have
\[
\begin{split}
\partial_s\int\varphi_s\eta_s\,\d\mm&\leq\int-\tfrac12\lipa^2\varphi_s\eta_s+t\varphi_s\Delta\eta_s+\varphi_s\h_{ts}(\partial_s\rho_s)\,\d\mm\\
&\leq\int-\tfrac12|\d\varphi_s|^2\eta_s-t\la \d\varphi_s,\d \log\eta_s\ra \eta_s+\h_{ts}\varphi_s\partial_s\rho_s\,\d\mm,\\
\partial_s\int\eta_s\log\eta_s\,\d\mm&=\int\log\eta_s(t\Delta\eta_s+\h_{ts}(\partial_s\rho_s))\,\d\mm\\
&\leq\int-\tfrac t 2|\d\log\eta_s|^2\eta_s+\h_{ts}(\log\eta_s)\partial_s\rho_s\,\d\mm
\end{split}
\]
thus for $\psi_s:=\h_{ts}(\varphi_s+t\log\eta_s)$, keeping in mind also \eqref{eq:BElav} we have
\[
\begin{split}
\int\varphi_1\eta_1-\varphi_0\eta_0+t(\eta_1\log\eta_1-\eta_0\log\eta_0)\,\d\mm&\leq \iint_0^1 \psi_s\partial_s\rho_s-\tfrac12|\d\psi_s|^2\rho_s\,\d s\,\d\mm\!\!\!\stackrel{\eqref{eq:derpuntvarphi}}\leq\!\!\!\int_0^1\tfrac12|\dot\mu_t|^2\,\d t
\end{split}
\]
which by the arbitrariness of $\varphi$ and the duality formula \eqref{eq:dualw2} is the claim (here the use of \eqref{eq:derpuntvarphi} can be justified via a regularization procedure of the densities $\eta_s$ that in turn grants Lipschitz regularity for $\psi_s$).

For $(iv)$  we regularize via heat flow (that by arguments akin to those above does not increase the metric speed). With a bit of work one sees that the computations for the estimate in $(iii)$ are justifiable if we show that for $\h_t\mu=\rho_t\mm$ we have
\begin{equation}
\label{eq:apriorical}
\ent(\h_t\mu)+\int\tfrac{|\d\rho_t |^2}{\rho_t}\,\d\mm \leq C(t)\qquad\forall t>0
\end{equation}
for some $C(t)$ independent on $\mu$. Then  since  $\partial_t\ent(\h_t\mu)=-\int\tfrac{|\d\rho_t|^2}{\rho_t}\,\d\mm$ and 
\[
\frac\d{\d t} \int\frac{|\d\rho_t|^2}{\rho_t}\,\d\mm=-\int\rho_t(\Delta|\d\log\rho_t|^2-\la\d\log\rho_t,\d\Delta\log\rho_t\ra)\,\d\mm\stackrel{\eqref{eq:bochner}}\leq 0,
\]
we deduce $\int\tfrac{|\d\rho_t|^2}{\rho_t}\,\d\mm\leq {\ent(\h_t\mu)-\ent(\h_{t+1}\mu)}\leq \ent(\h_t\mu)$ and thus it is enough to prove  $\ent(\h_t\mu)\leq C(t)$. Also, by the convexity and lower semicontinuity of the entropy, it is enough to prove such uniform bound for $\mu:=\delta_x$. To this aim let $f:\X\to\R$ be regular enough and $\gamma$ a geodesic. Then for $t>0$ formally (but it is easy to justify the computations) we have
\[
\begin{split}
\partial_s\h_s(\log(\h_{t-s}f))(\gamma_s)&=\h_s(\Delta \log\h_{t-s}f-\tfrac{\Delta \h_{t-f}f}{\h_{t-s}f})(\gamma_s)+\d\h_s(\log(\h_{t-s}f)(\gamma'_s)\\
&\leq \h_s(-|\d\log\h_{t-s}f|^2)(\gamma_s)+|\d\h_s(\log(\h_{t-s}f)|^2(\gamma_s)+\tfrac14|\gamma_s'|^2.
\end{split}
\]
Using \eqref{eq:BElav} and integrating in $[0,t]$ we deduce the log-Harnack inequality
\[
\h_t(\log f)(x)\leq\log(\h_tf)(y)+\tfrac{\sfd^2(x,y)}{4t}\qquad\forall x,y\in\X,\ t>0.
\]
In particular, by approximation the above holds for the heat kernel $f:=\rho_t[x]$ in which case it gives  $\int\rho_t[x]\log\rho_t[x]\,\d\mm\leq \log(\rho_{2t}[x])(y)+\frac{{\sf D}^2}{4t}$ for ${\sf D}:=\diam(\X)$. Integrating w.r.t.\ $y$ we get
\[
\ent(\h_t\delta_x)\leq \int \log(\rho_{2t}[x])(y)\,\d\mm(y)+\frac{{\sf D}^2}{4t}\leq \log\Big(\int \rho_{2t}[x](y)\,\d\mm(y)\Big)+\frac{{\sf D}^2}{4t}=\frac{{\sf  D}^2}{4t} ,
\]
as desired.
\end{idea}
I now want to make a further abstraction step and formulate the notion of $\RCD$ space purely in terms of Dirichlet forms.  For this, some vocabulary has to be introduced. Recall that given a Polish space $(\X,\tau)$ equipped with a Radon measure $\mm$, a  Dirichet form on $\X$ is a lower semicontinuous quadratic form $\mathcal E:L^2(\X,\mm)\to[0,\infty]$ satisfying the Markov property
\begin{equation}
\label{eq:markov}
\mathcal E(\varphi\circ f)\leq \Lip(\varphi)^2\mathcal E(f)\qquad\forall f\in L^2(\X,\mm),\ \varphi:\R\to\R\ \text{Lipschitz with }\varphi(0)=0.
\end{equation}
By polarization we have a bilinear map, still denoted  $\mathcal E$, on $D(\mathcal E):=\{\mathcal E<\infty\}$, i.e.\ we put
\[
\mathcal E(f,g):=\tfrac12\big(\mathcal E(f+g)-\mathcal E(f)-\mathcal E(g)\big)\qquad\forall f,g\in D(\mathcal E).
\]
The form $\mathcal E$ is called strongly local if
\[
\mathcal E(f,g)=0\qquad\text{whenever $f=g+c$ for some $c\in\R$}.
\]
Picking $\varphi(z):=z^2\wedge C$ for $C\gg1$ in  \eqref{eq:markov} we see that $f^2\in D(\mathcal E)$ whenever $f\in \mathcal A:=D(\mathcal E)\cap L^\infty(\X)$, so that $\mathcal A$ is an algebra. Any $f\in \mathcal A$ induces a linear operator  $\Ggamma(f;\cdot): \mathcal A\to\R$, called Carr\'e du champ, by the formula
\[
\Ggamma(f;g):=\mathcal E(f,fg)-\tfrac12\mathcal E(f^2,g)\qquad\forall g\in \mathcal A.
\]
Building on top of the Markov property \eqref{eq:markov} one can prove that 
\begin{equation}
\label{eq:ineqcarre}
0\leq \Ggamma(\varphi\circ f;g)\leq \Ggamma(f;g)\leq \|g\|_{L^\infty}\mathcal E(f)\qquad\forall f,g\in\mathcal A,\ g\geq0,\ \Lip(\varphi)\leq 1,\ \varphi(0)=0
\end{equation}
and thus that $\Ggamma$ can be extended by continuity to a linear operator $\Ggamma(f):\mathcal A\to\R$ for any $f\in D(\mathcal E)$ that also satisfies  the inequalities in \eqref{eq:ineqcarre}. We denote by $D(\Gamma)\subset D(\mathcal E)$ the collection of those $f$'s for which there is $\Gamma(f)\in L^1(\X)$ such that $\Ggamma(f,g)=\int g\Gamma(f)\,\d\mm$ for every $g\in\mathcal A$ and then put   $\mathcal L:=\{f\in D(\Gamma) \ :\ \Gamma(f)\leq 1\ \mm-a.e.\}$. A non-trivial property one can prove starting from \eqref{eq:ineqcarre} is 
\begin{equation}
\label{eq:Lclosed}
\mathcal L\text{ is a closed subset of $L^2$,}
\end{equation}
see \cite[Chapter 1]{BH91}  for the proof.  Notice that in the smooth setting, the typical example of Dirichlet form is the Dirichlet energy $\mathcal E(f):=\frac12\int|\d f|^2\,\d\mm$ for which we have $\Ggamma(f)=|\d f|^2\mm$, hence functions in $\mathcal L$ are, in a sense, those with `differential $\leq 1$ from the perspective of $\mathcal E$'. Therefore, if we further impose a continuity condition on $f\in\mathcal L$, we can use these functions to induce, by duality, a distance on $\X$:
\begin{equation}
\label{eq:defde}
\sfd_\mathcal E(x,y):=\sup\{|f(y)-f(x)|\ :\ f\in \mathcal L_C\}\qquad\text{where}\qquad \mathcal L_C:=\mathcal L\cap C(\X).
\end{equation}
A Dirichlet form induces via integration by parts an operator $\Delta_{\mathcal E}:D(\Delta_{\mathcal E})\subset L^2(\X,\mm)\to L^2(\X,\mm)$, called \emph{infinitesimal generator}, defined as: 
\[
\text{$f\in D(\Delta_{\mathcal E})$ with $\Delta_{\mathcal E}f=h\qquad $ iff $\qquad \mathcal E(f,g)=-\int gh\,\d\mm\quad$ for every $g\in D(\mathcal E)$}.
\] 
To $\Delta_{\mathcal E}$ is associated  a diffusion process. Namely, and in analogy with \eqref{eq:gfch}, for every $f\in L^2(\X)$ there is  a unique curve $t\mapsto f_t\in L^2(\X) $ in $C([0,\infty))\cap C^1((0,\infty))$ with 
\[
\partial_tf_t=\Delta_{\mathcal E}f_t\quad\forall t>0\qquad \text{and}\qquad f_0=f.
\]
We put $P_tf:=f_t$. A way to prove such existence and uniqueness, very much in line with the spirit of this note, is to notice that $(f_t)$ is the gradient flow trajectory of the convex and lower semicontinuous functional $\mathcal E$ on $L^2(\X,\mm)$\footnote{this variational viewpoint is by far not the only one possible. An alternative approach is via the theory of linear semigroups, see e.g.\ \cite{Paz}.}.

We can now define the \emph{iterated Carre du champ operator} $\Ggamma_2$ as follows. We put 
\[
D(\Ggamma_2):=\{(f,\varphi)\in D(\Delta_{\mathcal E})\ :\ \Delta_{\mathcal E}f\in D(\mathcal E),\ \varphi,\Delta_{\mathcal E}\varphi\in L^\infty(\X)\},
\] 
notice that arguing as in \eqref{eq:regheat}  we see that $D(\Ggamma_2)$ is dense in $[D(\mathcal E)]^2$, and then define
\[
\Ggamma_2[f;\varphi]:=\tfrac12\Ggamma[f,\Delta_{\mathcal E}\varphi]-\tfrac12\big(\mathcal E(f,\varphi\Delta_{\mathcal E}f)+\mathcal E(\Delta_{\mathcal E}f,f\varphi)-\mathcal E(f\Delta_{\mathcal E}f,\varphi)\big)
\]
(in the smooth setting this reads as $\int\varphi(\Delta \tfrac12|\d f|^2-\la\d f,\d\Delta f\ra)\,\d\mm$ - the writing as above helps giving a meaning to such expression under minimal regularity assumptions on $f,\varphi$). Then:
\begin{theorem}\label{thm:reprfin}
Let $K\in\R$,  $(\X,\tau)$ be a Polish space, $\mm$ a Radon measure on it and $\mathcal E:L^2(\X,\mm)\to[0,\infty]$ a strongly local Dirichlet form. Assume that
\begin{itemize}
\item[a')] $\sfd_\mathcal E$ is a finite distance on $\X$ generating the topology $\tau$.
\item[b')] For some $\bar x\in\X$ and $C>0$ we have $\mm(B_r^{\sfd_\mathcal E}(\bar x))\leq Ce^{Cr^2}$ for every $r>0$.
\item[c')] every function $f\in D(\Gamma) $ with $\Gamma(f)\in L^\infty(\X,\mm)$ admits a continuous representative.
\item[d')] For every $(f,\varphi)\in D(\Ggamma_2)$ with $\varphi\geq 0$ we have
\begin{equation}
\label{eq:Gamma2cond}
\Ggamma_2[f;\varphi]\geq K\Ggamma[f;\varphi].
\end{equation}
\end{itemize}
Then the metric completion $\bar\X$ of $(\X,\sfd_{\mathcal E})$ equipped with the (naturally induced) distance $\sfd_{\mathcal E}$ and measure\footnote{in extending $\mm$ from $\X$ to its completion $\bar\X$ one uses the fact that $\X$ is a Borel subset of $\bar \X$. This in turn follows from  $(\X,\tau)$ being Polish and \cite[Theorem 6.8.6]{Bogachev07}. Alternatively, one can directly assume in the above that $\sfd_{\mathcal E}$ is complete.} $\mm$ is an $\RCD(K,\infty)$ metric measure space. 
\end{theorem}
\begin{idea}
By Theorem \ref{thm:reprinterm} it is sufficient to prove that the Cheeger energy built on $(\bar \X,\sfd_{\mathcal E},\mm)$ coincides, together with the notion of carr\'e du champ/minimal weak upper gradient, with $\mathcal E$. Say $K=0$ and assume that $\bar \X$ is compact, so that in particular $\mm$ is finite and $\sfd_{\mathcal E}$ is bounded. 
%
%
%
%
%
The proof will follow if we show that:
\begin{subequations}
\begin{align}
\label{eq:ineq1e}
D(\ch)&\subset D(\Gamma)&&\text{and }\qquad \Gamma(f)\leq |\d f|^2\quad\mm-a.e.\quad \forall f\in D(\ch), \\
\label{eq:dgamma}
D(\Gamma)&=D(\mathcal E),&&\\
\label{eq:ineq2}
D(\Gamma)&\subset D(\ch)&&\text{and }\qquad \ch(f)\leq\mathcal E(f)\quad  \forall f\in D(\Gamma).
\end{align}
\end{subequations}
To prove \eqref{eq:ineq1e} it suffices to prove, by relaxation, that if $f$ is $\sfd_{\mathcal E}$-Lipschitz, then it is in $D(\Gamma)$ with $\Gamma(f)\leq\lipa^2(f)$. Also, since for $L:=\Lip(f\restr U)$ with $U\subset\X$, the function $\tilde f(x):=\inf_{y\in U}f(y)+L\sfd_{\mathcal E}(x,y)$ coincides with $f$ in $U$, by locality and again a relaxation procedure, it suffices to show that for any $y\in\X$ the function $\sfd_{\mathcal E}(\cdot,y)$ is in $\mathcal L_C$ (recall \eqref{eq:defde}). Fix $y$. Continuity is obvious, because $\sfd_{\mathcal E}$ induces the topology.  Now observe that from \eqref{eq:markov} and strong locality  it follows that $\mathcal L$ is stable by the `max' operation, thus taking into account the separability of $\bar \X$ there is an increasing sequence  $(\varphi_n)\subset\mathcal L_C$ converging pointwise to $\sfd_{\mathcal E}(\cdot,y)$. The $\varphi_n$'s are uniformly bounded (by the diameter of $\bar \X$), thus the convergence is also in $L^2(\bar\X,\mm)$. By \eqref{eq:Lclosed} we just proved that   $\sfd_{\mathcal E}(\cdot,y)\in\mathcal L$, as desired.

To prove \eqref{eq:dgamma} the key fact to be proved is: $\forall f\in L^2$ and $t>0$ we have $P_tf\in D(\Gamma)$ and 
\begin{equation}
\label{eq:perdgamma}
t\Gamma(P_tf)\leq\tfrac12P_t(f^2)-\tfrac12(P_tf)^2\qquad\mm-a.e.,
\end{equation}
as then for generic $f\in D(\mathcal E)$ we can apply the above keeping in mind that the energy decreases along the flow to conclude by a limiting argument. To get \eqref{eq:perdgamma} we fix $\varphi\geq0$ with $\varphi,\Delta_{\mathcal E}\varphi\in L^\infty$, put $F(s):=\tfrac12\int (P_{t-s}f)^2 P_s\varphi\,\d\mm $ for $s\in[0,t]$ and notice that it is continuous on $[0,t] $ and that by direct computation, using \eqref{eq:Gamma2cond} we get 
\begin{equation}
\label{eq:Fprimo}
F'(s)=\Ggamma(P_{t-s}f;P_s\varphi)\qquad\text{and}\qquad F''(s)=2\Ggamma_2(P_{t-s}f;P_s\varphi)\geq 0\qquad\forall s\in(0,1).
\end{equation}
Thus $F$ is convex on $[0,t]$ and the inequality $tF'(0)\leq F(t)-F(0)$ reads as
\[
tF'(0)=t\Ggamma(P_tf;\varphi)\leq \int\varphi g\,\d\mm\qquad\text{where}\qquad g:=\tfrac12P_t(f^2)-\tfrac12P_t(f)^2.
\]
By \eqref{eq:ineqcarre}  we get $F'(0)\geq0$, thus
\begin{equation}
\label{eq:pertrovaregamma}
0\leq t\Ggamma(P_tf;\varphi)\leq \int\varphi g\,\d \mm\qquad\forall \varphi\geq 0,\ \varphi,\Delta_{\mathcal E}\varphi\in L^\infty.
\end{equation}
The arbitrariness of $\varphi$ implies $g\geq 0$. Also,  since clearly $g\in L^1(\X,\mm)$, the linear map $\varphi\mapsto \Ggamma(P_tf;\varphi)$ is bounded in $L^\infty$, and building upon the fact that \eqref{eq:pertrovaregamma} tells that $\Ggamma(P_tf;\varphi)\to 0$ as $\varphi\downarrow0$ $\mm$-a.e., one can prove that it is represented by an $L^1$ function, as desired.

To prove \eqref{eq:ineq2} start noticing that, trivially, any $f\in\mathcal L_C$ is 1-Lipschitz w.r.t.\ $\sfd_{\mathcal E}$. Then  considering, for fixed $x\in\X$, the function $(f(\cdot)-f(x))\varphi(\cdot)$ with $\varphi$ suitable Lipschitz cut-off function identically 1 on a neighbourhood of $x$, we can easily conclude by locality that for $f\in D(\Gamma)\cap C(\X)$ we have
\begin{equation}
\label{eq:perlilip}
\Gamma(f)\leq g\ \mm-a.e.\quad\text{for some $g$ upper semicontinuous}\quad\Rightarrow\quad \lipa^2 f\leq g\ \mm-a.e.
\end{equation}
Hence recalling the definition of $\ch$, to conclude it suffices to show that for any $f$ belonging to some subset of $D(\mathcal E)$ dense in energy there are $(f_n),(g_n)\subset C_\b(\X)$ with 
\begin{equation}
\label{eq:upperreg}
{\Gamma(f_n)}\leq g_n\quad \mm-a.e.,\qquad f_n\stackrel{L^2}\to f,\qquad\lims_n\tfrac12\int g_n\,\d\mm\leq \mathcal E(f).
\end{equation}
As dense subset we pick those $f$'s in $D(\mathcal E)\cap L^\infty(\X)$ with $\Gamma(f)\in L^\infty$ (by \eqref{eq:perdgamma}, for $\tilde f\in L^2\cap L^\infty$ the function $f:=P_t\tilde f$ has such property for any $t>0$). Applying   \eqref{eq:perdgamma} to ${\Gamma(f)}\in L^2\cap L^\infty(\X)$ we see that $\Gamma(P_t(\Gamma(f)))\in L^\infty(\X)$, thus $P_t(\Gamma(f))$ admits, by $(c')$,  a continuous representative (still denoted $P_t(\Gamma(f))$). Now notice that arguing as for \eqref{eq:pertrovaregamma}, the inequality $F'(0)\leq F'(t)$ reads as $\Gamma(P_tf)\leq P_t(\Gamma(f))$, thus \eqref{eq:upperreg} follows picking $f_n:=P_{\frac1n}f$ and $g_n:=P_{\frac1n}(\Gamma(f))$.
\end{idea}
Notice that combining \eqref{eq:perdgamma}, property $(c')$, estimate \eqref{eq:perlilip} and the fact that $(\X,\sfd)$ is a length space\footnote{strictly speaking, it is $(\supp(\mm),\sfd)$ that is a length space and in estimate \eqref{eq:linftylip} we should have $\Lip(\h_t f\restr{\supp(\mm)})$ at the left hand side. This is not so important, though, as we will never look at any function outside $\supp(\mm)$ and in any case we can extend them outside it using, e.g., Lemma \ref{le:locmcsh}. In this direction, notice also that unlike other results we discussed, the assumption `$\sfd_{\mathcal E}$ is a finite distance' in Theorem \ref{thm:reprfin} trivially forces $\supp(\mm)=\X$.} we get   the useful $L^\infty-\Lip$ regularization property of the heat flow, i.e.
\begin{equation}
\label{eq:linftylip}
\Lip(\h_t f)\leq C(K,t)\|f\|_{L^\infty}\qquad\forall t>0,\quad\text{ with }C(K,t)\sim\sqrt t\text{ for small }t
\end{equation}
on an $\RCD(K,\infty)$ space, being intended that $\h_tf$ is identified with its continuous representative. 

\begin{remark}{\rm
Theorems \ref{thm:reprinterm} and \ref{thm:reprfin} are less abstract that those recalled at the beginning of the section, as part of the data are a `space' containing `points' and a notion of `proximity' among these. It would be interesting to understand whether this sort of assumptions can be removed, in favour of a purely algebraic description of $\RCD$ spaces. In this direction, notice that the concept of closed quadratic forms is well defined on arbitrary Hilbert spaces and that the Markov condition \eqref{eq:markov} can be equivalently stated as $\mathcal E(0\wedge f\vee 1)\leq\mathcal E(f)$: this inequality can formulated in abstract terms if one assumes the existence of an algebra with unit (say $\mm(\X)<\infty$) that is also a lattice in the domain of the form, so perhaps this program is not out of reach. 
}\fr\end{remark}

\subsection{Bibliographical notes}\label{se:bibliorcd}
{\footnotesize{Sobolev  functions on an arbitrary metric measure space have been introduced in \cite{HajSob}; the approach in such paper produces a notion different from the one used here, and in particular lacks the key locality property \eqref{eq:dfloc} (read in the smooth setting, the `minimal upper gradient' in  \cite{HajSob} is the maximal function of $|\d f|$).

The first definition of Sobolev space equivalent to the ones presented here appeared in \cite{Cheeger00}, where a relaxation procedure very similar to the one in Section \ref{se:vertsob} has been presented (whence the choice of terminology `Cheeger energy'). Later in \cite{Shanmugalingam00} a different, but equivalent, definition has been proposed: here the idea is conceptually closer to the one presented in Section \ref{se:horsob}, although the concept of `Modulus of a family of curves', introduced by Fuglede in \cite{Fugl57},  is used in place of test plans  to describe `negligible set of curves' to ignore when imposing a `weak upper gradient property'. See the books \cite{HKST15}, \cite{Bjorn-Bjorn11} for more on this topic.

The definitions of Sobolev functions adopted here come from \cite{AmbrosioGigliSavare11} (see also \cite{AmbrosioGigliSavare11-3}), where the equivalence with these pre-existing approaches has also been proved (see also the more recent \cite{EB23} for a more direct proof of some of the equivalences that avoids the use of the heat flow). Extension of this approach to the $\BV$ case has been obtained in \cite{Ambrosio-DiMarino14} (see also the original  \cite{Mir03}). The inspiration for the relaxation procedure comes from the already mentioned \cite{Cheeger00}, while the notion of test plan is related to the studies of flows of Sobolev vector fields in the Euclidean spaces carried out in \cite{DiPerna-Lions89} and \cite{Ambrosio04}. Lemma \ref{le:locmcsh} is the main result in \cite{GDMP}.  

To the best of my knowledge, \cite{KS03conv} is the first paper where $\Gamma$-convergence is used in conjunction with mGH-convergence and also the first one where it appeared the concept of convergence of $L^2$ functions defined on a varying sequence of spaces; the motivations of the authors were very close to those of this manuscript. Theorem \ref{thm:glimsch} appears in this form here for this first time, but it is just a restatement of ideas presented in \cite{AmbrosioGigliSavare11-3}, where also the observation that the upper semicontinuity of $\lipa f$ leads to stability of `first order inequalities' as in Corollary \ref{cor:stablogsob} was made. Lemma \ref{le:equivsobw} comes from \cite[Theorem 2.1.12]{GP19}.

Proposition \ref{prop:HL} has been proven independently in \cite{AmbrosioGigliSavare11} and \cite{NCS14} under slightly different assumptions on the base space (in \cite{AmbrosioGigliSavare11}, no compactness is assumed). Notably, even in metric spaces, the Hopf-Lax formula produces `the viscosity solution' of the  Hamilton-Jacobi equation, see the independent papers \cite{AF14} and \cite{GS15}.

The superposition principle has been proved in the Euclidean setting in \cite{Ambrosio04} (see also \cite{AmbrosioGigliSavare08}). The metric version presented in Lemma \ref{le:superpp} has been obtained \cite{Lisini07}.

Lemma \ref{le:basegfch} and Theorem \ref{thm:hfgfagain} come from \cite{Gigli-Kuwada-Ohta10}, where they were obtained in the setting of compact Alexandrov spaces; their generalization to normalized spaces is straightforward once one knows that the Hopf-Lax formula produces subsolutions of the Hamilton-Jacobi equation, which was already known in $\CD(K,N)$ spaces by \cite{LVHJ}. In particular, Kuwada's lemma (here given as item $(iv)$ in Lemma \ref{le:basegfch}) comes from \cite{Gigli-Kuwada-Ohta10}, but the idea is so beautiful and crucial that  Ohta and I decided to name it after the person who actually got the idea. These principles have been used in \cite{AmbrosioGigliSavare11} to not only get  the non-trivial generalization of Theorem \ref{thm:hfgfagain} to spaces with infinite mass, and possibly infinite distances, but, perhaps more interestingly, to give new foundation to the theory of metric Sobolev spaces via Definitions \ref{def:ch} and \ref{def:defsobw} and Theorem \ref{thm:sobuguali}. Lemma \ref{le:inpartI} is also from \cite{AmbrosioGigliSavare11}, while Theorem \ref{thm:convslen} comes from the already mentioned \cite{GMS15}. In developing this program, inspiring studies (at least for me) have been those in  \cite{OhSt09} that showed that the Jordan-Kinderleher-Otto ideas, and related ones, are still available in a setting where the Laplacian is non-linear.

Lemma \ref{le:tapio} has been proved in \cite{Rajala12-2}. Proposition \ref{prop:sobtolip} appears here for the first time, the argument for the proof being taken  from \cite{Gigli13}. Earlier, a different proof based on the Bakry-\'Emery contraction rate was given in \cite{AmbrosioGigliSavare11-2} in the setting of $\RCD(K,\infty)$ spaces.

The Definitions \ref{def:planrepgr} and \ref{def:infhilb} of `plan representing a gradient' and of `infinitesimally Hilbertian spaces' as well as definition \eqref{eq:defdpm} of `duality between differentials and gradients', the calculus rules in Lemma \ref{le:horver} and Proposition \ref{prop:infhilbcalc} come from \cite{Gigli12}. Part of such calculus rules appeared earlier, at times with incomplete proofs (see \cite[Lemma 4.7 and the comments before it]{AmbrosioGigliSavare11-2}), in \cite{AmbrosioGigliSavare11-2} where they have been used to prove the key formulas in Lemmas \ref{le:derw2} and \ref{le:derent}.

The concept of gradient flow trajectory as encoded by Definition \ref{def:EVI} has been proposed in \cite{AmbrosioGigliSavare08}  and subsequently studied in detail by Savar\'e and collaborators. Proposition \ref{prop:pasevi} comes from \cite{MS20}, to which I also refer for a recent account of the theory. It has been soon understood that $\EVI$'s are linked to some sort of `Hilbert structure of tangent spaces' and while no general statement in this direction is available, it is a sort of meta-theorem of the theory that whenever the local geometry of the space under consideration resembles a Hilbertian one, $\EVI$-gradient flows exist; for results in this direction see for instance \cite{OhSt09},  \cite{OhSt12}, \cite{vRT12}, \cite{Gigli13},  \cite{OhPa17},  \cite{MS20}. Proposition \ref{prop:EVIconv} has been proved in \cite{DaneriSavare08} (partly inspired by \cite{OtWe05}). The stability properties in Theorem \ref{thm:stabevi} and Corollary \ref{cor:stabhagain} have been proved in \cite{AmbrosioGigliSavare11-2} (see also \cite{GMS15} and \cite{MS20} for generalizations to non-compact settings). In connection with Corollary \ref{cor:stabhagain} it is worth to recall that if $\Delta$ is a compact operator (more precisely: it has compact resolvents $(\Id-\tau\Delta)^{-1}:L^2\to L^2$) then it has discrete spectrum $0\leq \lambda_1\leq\lambda_2\leq\cdots$ and the $k$-th eigenvalue $\lambda_k$ admits the variational characterization
\[
\lambda_k=\min_{}\max_{0\neq f\in V}\frac{\int|\d f|^2\,\d\mm}{\int |f|^2\,\d\mm},
\]
where the $\min$ is taken among $k$-dimensional subspaces $V$ of $L^2$. This representation and the stability in formula \eqref{eq:laplcl} and Corollary \ref{cor:stabhagain} allow to deduce spectral convergence of the Laplacian under suitable uniform compactness assumptions on the resolvents, see \cite[Sections 2.6, 2.7]{KS03conv},  \cite[Theorem 7.8]{GMS15} and references therein.

The first definition of a `Riemannian' Curvature Dimension condition has been proposed in \cite{AmbrosioGigliSavare11-2}  specifically for the case $N=\infty$ and for normalized spaces in a form slightly different from the one given here in Definition \ref{def:rcd}: the main equivalence result proved in \cite{AmbrosioGigliSavare11-2}  is the one between the $\EVI_K$-condition for the heat flow and the coupling of infinitesimal Hilbertianity with a version of the curvature condition $\CD(K,\infty)$ stronger than the one encoded in Definition \ref{def:cd} (i.e.\ we proved a modified version of Theorem \ref{thm:rcdevi}). Then either of these two equivalent notions was taken as definition of $\RCD(K,\infty)$ space. The proposal to focus the attention only on infinitesimal Hilbertianity regardless of curvature/dimension/mass bounds come from \cite[Remark 4.20]{Gigli12}, where it was also suggested how to reconcile this definition with the one in \cite{AmbrosioGigliSavare11-2}: the suggestion was basically to rely on Lemma \ref{le:tapio} and has been put forward in \cite{AmbrosioGigliMondinoRajala12}, where Theorem \ref{thm:rcdevi} was proved on spaces with possibly infinite mass.

The idea behind the proposal in  \cite{Gigli12} to insist on properties of Sobolev functions, and solely on that, was that in order to prove geometric results in the $\RCD$ category akin to those available in the smooth Riemannian world, we need to have at disposal a calculus that resembles the Riemannian one (as opposed to the Finsler one - see also the introduction of \cite{Gigli13over}). Proofs of concept for this plan came initially  in \cite{Gigli-Mosconi12}, where the Abresch-Gromoll inequality has been proved for finite dimensional $\RCD$ spaces (answering an open criticism raised in \cite{Petrunin11} about the $\CD$ condition), and then in \cite{Gigli13} with the proof of the splitting theorem.

The Bochner inequality has been proved for the first time in the non-smooth setting in \cite{Gigli-Kuwada-Ohta10} on compact finite dimensional Alexandrov spaces. The structure of the argument given there is the one presented in the proof of Theorem \ref{thm:boch} and works with basically no modifications on $\RCD(K,\infty)$ spaces, once one has at disposal the $\EVI_K$-property of the heat flow (see also  \cite[Remark 6.3]{AmbrosioGigliSavare11-2}). The `hard' implication in the proof, namely the one linking a $W_2$-contraction property to a Bakry-\'Emery type of inequality was investigated in general setting by Kuwada in \cite{Kuwada10} and is known by the name `Kuwada's duality'.

\medskip

Coming to the finite dimensional world, after \cite{BacherSturm10} it became quite clear in the community that, thanks to its better local-to-global properties,  the $\CD^*(K,N)$ condition was better suited than $\CD(K,N)$ to be linked to the Bochner inequality, see for instance  the preprint \cite{Bochner-CD}, then merged in \cite{Gigli13} (the equivalence between $\CD$ and $\CD^*$ was not yet known). After these, the dimension-dependent Bochner inequality has been obtained in \cite{Erbar-Kuwada-Sturm13} with the help, as already mentioned, of the notion of $\CD^e(K,N)$ condition defined there - recall \eqref{eq:diffcdekn}, \eqref{eq:defcdekn} (see also the subsequent \cite{AmbrosioMondinoSavare13}). The first result  in \cite{Erbar-Kuwada-Sturm13} is a finite dimensional analogue of Theorem \ref{thm:rcdevi}, namely
\begin{equation}
\label{eq:cdknevikn}
\text{``$\CD^e(K,N)+{\rm Inf.\ Hilb.}$" is equivalent to ``the heat flow satisfies a $\EVI_{K,N}$ condition'',}
\end{equation}
the latter meaning that the stronger (see \cite[Lemma 2.15]{Erbar-Kuwada-Sturm13}) inequality
\begin{equation}
\label{eq:evikn}
\frac{\d}{\d t}{\sf S}_{\frac  KN}\big(\tfrac12\sfd^2(x_t,y)\big)+K\,{\sf S}_{\tfrac KN}\big(\tfrac12\sfd^2(x_t,y)\big)\leq \tfrac 12N\big(1-\tfrac{U_N(y)}{U_N(x_t)}\big)
\end{equation}
is in place of \eqref{eq:defEVI}. Here $U_N:=-\exp(-\tfrac 1N\E)$ (notice the different sign convention w.r.t.\ \cite{Erbar-Kuwada-Sturm13}) and 
\[
{\sf S}_\kappa(t):=\left\{
\begin{array}{ll}
\tfrac{1}{\sqrt \kappa}\sin(\sqrt\kappa t),&\quad\text{ if }\kappa>0,\\
t,&\quad\text{ if }\kappa=0,\\
\tfrac{1}{\sqrt {|\kappa|}}\sinh(\sqrt{|\kappa|} t),&\quad\text{ if }\kappa<0.\\
\end{array}
\right.
\]
The idea for proving \eqref{eq:cdknevikn} is that the stronger convexity assumption gives a better estimate on the derivate of the entropy than the one used in the starred inequality in \eqref{eq:bello} and, viceversa, the stronger $\EVI$ implies stronger convexity via the arguments in Proposition \ref{prop:EVIconv}.

Once `the correct' $\EVI$ has been found, the finite-dimensional Bochner inequality has been obtained, still in \cite{Erbar-Kuwada-Sturm13}, on the backbone of the proof of Theorem \ref{thm:boch}. More precisely, the $\EVI_{K,N}$-property of the heat flow implies the contraction estimate
\[
{\sf S}_{\frac  KN}\big(\tfrac12W_2^2(\mu_t,\nu_s)\big)\leq e^{-K(t+s)}{\sf S}_{\frac  KN}\big(\tfrac12W_2^2(\mu_0,\nu_0)\big)+\tfrac NK(1-e^{-K(t+s)})\tfrac{(\sqrt t-\sqrt s)^2}{2(t+s)}\qquad\forall t,s\geq 0
\]
that in turn, by a suitable version of Kuwada's duality, gives the Bakry–Ledoux (from \cite{BaLe06}) estimate
\begin{equation}
\label{eq:BL}
|\d\h_tf|^2+\frac{4Kt^2}{N(e^{2Kt}-1)}|\Delta\h_tf|^2\leq e^{-2Kt}\h_t(|\d f|^2)\qquad\forall t\geq 0.
\end{equation}
This is an equality at $t=0$, thus a (formal) differentiation at $t=0$ gives the desired  Bochner inequality
\begin{equation}
\label{eq:bochfindim}
\tfrac12\Delta |\d f|^2 \geq \tfrac1N |\Delta f|^2 +\la \d f,\d\Delta f\ra +K|\d f|^2.
\end{equation}
Alternative proofs of   these finite dimensional equivalences  have been obtained in the already mentioned \cite{AmbrosioMondinoSavare13}: their overall strategy is the same outlined here, but rather than working with the (linear) heat flow, they study  the (non-linear) porous media flow, that arises as $W_2$-gradient flow of the R\'enyi entropy (see \cite{Otto01}) used to define $\CD^*(K,N)$ spaces. Such non-linearity provides an additional technical challenge. 

With the Bochner inequality at disposal, it became standard in the field to work with the $\RCD^*=\RCD^e$ notion; as mentioned before it has been  \cite{CavMil16} that proved the equivalence \eqref{eq:equivcd} of these notions with $\RCD(K,N):=\CD(K,N)+\text{Inf. Hilb.}$ (at least in normalized spaces).

\medskip

The possibility of encoding a lower Ricci and an upper dimension curvature bound in terms of the relation between the $\Gamma$ and $\Gamma_2$ operators is an idea that goes back to \cite{BakryEmery85}, where the concept of `Curvature-Dimension condition' appeared for the first time. In the setting of smooth, weighted, Riemannian manifolds $(M,g,e^{-V}\vol)$ the $\CD(K,N)$ condition is equivalent to ${\rm Ric}_N\geq Kg$, where ${\rm Ric}_N$ is the Bakry-\'Emery $N$-Ricci tensor defined as
\begin{equation}
\label{eq:bern}
{\rm Ric}_N:=\left\{
\begin{array}{ll}
{\rm Ric}+\Hess V-\frac1{N-n}{\nabla V\otimes\nabla V},&\qquad\text{ if }N>n,\\
{\rm Ric},&\qquad\text{ if $N=n$ and $V$ is constant},\\
-\infty,&\qquad\text{ otherwise}.  
\end{array}
\right.
\end{equation}
The basic idea behind this definition is in the validity of the Bochner inequality
\[
\tfrac12\Delta|\d f|^2\geq \tfrac1N|\Delta f|^2+\la \d f,\d\Delta f\ra+{\rm Ric}_N(\d f,\d f),
\]
so that from lower bounds on ${\rm Ric}_N$ one can actually deduce informations about the interplay of `differential' and `Laplacian'. In terms of the calculus as developed by Bakry-\'Emery, such interplay is read as relation between $\Gamma$ and $\Gamma_2$ and such calculus   goes under the name of ``$\Gamma$-calculus"\footnote{this has nothing to do with $\Gamma$-convergence!}. THE reference for this topic is the monograph \cite{BakryGentilLedoux14}.

Standard references for the theory of Dirichlet forms are \cite{BH91}, \cite{MaRockner92}, \cite{Fuk80}. The fact  that  the intrinsic distance $\sfd_{\mathcal E}$ is geodesic and the key regularity property $\Gamma(\sfd_{\mathcal E}(x,\cdot))\leq 1$ have been obtained  in \cite{Sturm96I}. The rest of Theorems \ref{thm:reprinterm}, \ref{thm:reprfin} come from \cite{AmbrosioGigliSavare12}.  Building on these two, the finite dimensional analogue have been obtained in \cite{Erbar-Kuwada-Sturm13}  (and later in  \cite{AmbrosioMondinoSavare13} along similar lines): they proved that if \eqref{eq:bochfindim} is assumed in place of \eqref{eq:bochner} in Theorem \ref{thm:reprinterm}, then the space is $\CD^e(K,N)$. Here the point is that \eqref{eq:bochfindim} implies, by suitable integration in $t$, the estimate \eqref{eq:BL} and this in turn gives a better estimate in \eqref{eq:acest} than the one achievable by \eqref{eq:BElav}. Related results have been obtained in \cite{Koskela-Zhou12}, \cite{KSZ14}.

Theorems  \ref{thm:reprinterm}, \ref{thm:reprfin} and their finite dimensional counterpart are useful in checking when a given space is actually $\RCD$, as often computing differential operators is easier than understanding the behaviour of $W_2$-geodesics. For instance, the fact that Cartesian and suitable warped products of $\RCD$ spaces is still $\RCD$ are proved starting from results (in \cite{AmbrosioGigliSavare12} and \cite{Ketterer13}, respectively). The same for  stratified manifolds satisfying appropriate angle conditions, see   \cite{BKMR21}.

Speaking of examples of $\RCD$, and in fact of Ricci-limit spaces, let me mention the recent works  \cite{PanWei22} and \cite{Pan23} where it has been showed that the half-Grushin-plane and half-Grushin-sphere, that are subRiemannian manifolds, are also Ricci-limit spaces(!). These are also examples of Ricci-limit spaces whose Hausdorff dimension is non-integer (and thus bigger than the dimension of the regular set); notice that by Colding's volume convergence, examples of this sort must arise from a collapsing of dimension. The tools used in these papers have little to do with what has been exposed here. Earlier examples of `irregular' manifolds with a lower Ricci bound / Ricci-limits  spaces, in this case from the topological perspective, have been constructed in the  nineties in \cite{ShaYang89}, \cite{AndKroLeB89}, \cite{anderson1990}, \cite{Perelman97}, \cite{Menguy00}, see also references therein.

\medskip

In the recent \cite{CMTkatoI} it is proved stability of heat flow and lower Ricci bounds without relying on optimal transport at all: the authors impose an upper dimension bound, but weaken the uniform lower Ricci bound to a  so-called Kato lower Ricci bound (a suitable integrated lower Ricci bound). I think that this approach is extremely interesting, and believe it has more to say, so I  want to take few lines to describe the core of the idea in the simplified setting of a converging sequence of $\RCD(0,N)$ spaces. Thus say that the $(\X_n,\sfd_n,\mm_n)$'s are $\RCD(0,N)$ and are mGH-converging to some $(\X_\infty,\sfd_\infty,\mm_\infty)$. We aim at proving that $\X_\infty$ is also $\RCD(0,N)$. 

The basic idea  is to decouple the role of the distance and of the Cheeger/Dirichlet energy: we already have mGH-convergence, and by compactness in $\Gamma$-convergence, up to subsequences we can assume that $\ch_n\stackrel{\Gamma}\to \mathcal E$ for some functional $\mathcal E:L^2(\X_\infty,\mm_\infty)\to[0,\infty]$, that is clearly a Dirichlet form.  The criterium of strong compactness in Theorem \ref{thm:convslen} is still in place, so in fact $(\ch_n)$ Mosco-converges to $\mathcal E$ w.r.t.\ the weak/strong $L^2$-convergences of functions (I am not cheating here: the proof of the strong compactness does not require the limit space to be $\RCD$. Still, the proof of Theorem  \ref{thm:convslen} uses optimal transport: one can avoid this by using the local Poincar\'e inequality that come from the lower Ricci and upper dimension bounds - this is the approach in \cite{CMTkatoI}).  Now recall the  very general fact that if a sequence of convex and lower semicontinuous functionals Mosco-converge, then their Gradient Flows converge as well (because the minimum of $\E_n(\cdot)+\tfrac{\|\cdot-x\|^2}{2\tau}$ converge to that of $\E_\infty(\cdot)+\tfrac{\|\cdot-x\|^2}{2\tau}$ for any $\tau>0$, and iterating this minimization scheme we converge to the gradient flow with a rate known a priori and granted by the convexity). Thus the heat flows on the $\X_n$ converge to the flow associated to $\mathcal E$ and this allows to pass to the limit in virtually any inequality involving the heat flow and the distance. In particular, the Bochner inequality \eqref{eq:bochfindim} with $K=0$ holds in $\X_\infty$, the Laplacian being the one induced by $\mathcal E$ so that our task is `only' to prove that $\mathcal E$ coincides with the Cheeger energy built on $(\X_\infty,\sfd_\infty,\mm_\infty)$. In turn, this is more-or-less the same as to prove that the intrinsic distance $\sfd_{\mathcal E}$ induced by $\mathcal E$ by formula \eqref{eq:defde} coincides with $\sfd_\infty$.

To see this, notice that  since on $\RCD(0,N)$ spaces we have Gaussian estimates for the heat kernel (see \cite{Sturm96III}), these are also valid on $(\X_\infty,\sfd_\infty,\mm_\infty,\mathcal E)$:
\begin{equation}
\label{eq:gaussest}
\frac{1}{C \mm_\infty(B_{\sqrt t}(x))}e^{-C\frac{\sfd_\infty^2(x,y)}{4t}}\leq\rho_t(x,y)\leq \frac{C}{\sqrt{\mm_\infty(B_{\sqrt t}(x))\mm_\infty(B_{\sqrt t}(y))}}e^{-\frac{\sfd_\infty^2(x,y)}{4t}}\Big(1+\frac{\sfd_\infty^2(x,y)}{2}\Big)^{N+1}
\end{equation}
and the heat  kernel is a (family of) probability measure. For the same reason, the heat kernel satisfies suitable H\"older continuity estimates. These properties imply, by non-trivial general results about Dirichlet forms (\cite{Sturm96I}, \cite{Ram01}, \cite{ERS07},\cite{Gri10}, \cite{CarTew22} among others, see \cite[Appendix C]{CMTkatoI}), that Varadhan asymptotic formula 
\begin{equation}
\label{eq:varadhan}
\sfd^2_{\mathcal E}(x,y)=-4\lim_{t\downarrow0}t\log\rho_t(x,y)
\end{equation}
holds. Coupling this with the upper estimate in \eqref{eq:gaussest} we get $\sfd_{\mathcal E}\geq \sfd_\infty$. For the converse inequality I will argue differently from \cite{CMTkatoI} and pick $f\in C_\b(\X_\infty)$ with $\Gamma(f)\leq 1$ $\mm_\infty$-a.e.\ and notice that the Bakry-\'Emery estimate \eqref{eq:BElav} holds (as a consequence of Bochner inequality) thus $\Gamma(\h_t f)\leq 1$ $\mm_\infty$-a.e.\ as well. Now let   $(f_n)$ be a recovery sequence for $f$, so that $|\d f_n|\to \sqrt{\Gamma(f)}$ in $L^2$. By the Gaussian estimates, the norm of $\h_t$ as map from $L^1$ to $L^\infty$ is bounded in $t$ uniformly in $n$, thus $\lims_n\|\h_t(|\d f_n|^2)\|_\infty\leq1$ and again Bakry-\'Emery gives $\lims_n\Lip(\h_t f_n)\leq 1$. Thus for arbitrary sequences of points $x_n\to x_\infty$, $y_n\to y_\infty$ we get 
\[
|\h_tf(x_\infty)-\h_tf(y_\infty)|=\lim_n|\h_tf_n(x_n)-\h_tf_n(y_n)|\leq \limi_n\Lip(\h_tf_n)\sfd_n(x_n,y_n)=\sfd_\infty(x_\infty,y_\infty)
\]
so that letting $t\downarrow0$, by the arbitrariness of $f$ we conclude that $\sfd_{\mathcal E}\leq \sfd_\infty$, as desired.

Fore more about Kato lower bounds see \cite{Ket21}, \cite{ERST22}, \cite{CMTkatoII} and references therein. Other examples of situations where the decoupling of the roles of distance and Dirichlet energy has been useful are \cite{AmbrosioGigliSavare12} and \cite{BSR22},  in these cases to study  tensorization properties. 

}}

\section{Differential calculus on a non-differentiable environment}
\subsection{First order calculus on general metric measure spaces}\label{se:firstordercalc}
To develop tensor calculus on general metric measure spaces it is useful to introduce the following concept:
\begin{definition}[$L^2$-normed $L^\infty$-module]
An $L^2$-normed $L^\infty$-module over $(\X,\sfd,\mm)$ is a Banach space $(\M,\|\cdot\|_\M)$ that is also a module over (the commutative ring with unity) $L^\infty(\X)$ and which moreover admits a \emph{pointwise norm}, i.e.\ a map $|\cdot|:\M\to L^2(\X)$ satisfying
\[
\begin{split}
|v|&\geq 0\qquad\text{and}\qquad |fv|=|f|\,|v|\qquad\text{and}\qquad
\|v\|_\M=\sqrt{\int|v|^2\,\d\mm}
\end{split}
\]
for any $f\in L^\infty(\X)$, $v\in \M$, where the first two are intended $\mm$-a.e..
\end{definition}
It is not hard to check that the pointwise norm also satisfies
\[
|\lambda v|=|\lambda|\,|v|,\qquad\text{ and }\qquad |v+w|\leq |v|+|w|,
\]
$\mm$-a.e.\ for any $v,w\in \M$ and $\lambda\in\R$.

The prototypical example of $L^2$-normed $L^\infty$-module is the space of $L^2$ vector fields on a given Riemannian manifold or, more generally, the space of $L^2$ sections of a normed vector bundle. Intuitively, we shall  think at any such  module as the space of $L^2$-sections of some bundle, even if such bundle is not really given: from the technical perspective this offers several advantages, including a more direct link with the concept of integration (for which $\mm$-a.e.\ defined objects are `the' objects to handle) and with $W^{1,2}$-functions (whose differentials are, in a natural way, 1-forms in $L^2$).

In practice, $L^2$-normed $L^\infty$-modules enter the field via the following result:
\begin{thmdef}\label{thmdef:d}
Let $(\X,\sfd,\mm)$ be a metric measure space. Then there is a unique couple $(L^2(T^*\X),\d)$, with $L^2(T^*\X)$ being an $L^2$-normed $L^\infty$-module and $\d:W^{1,2}(\X)\to L^2(T^*\X)$ being linear and such that
\begin{itemize}
\item[i)] $|\d f|=|\D f|$ $\mm$-a.e.\ for any $f\in W^{1,2}(\X)$,
\item[ii)] $L^\infty(\X)$-linear combinations of  $\{\d f:f\in W^{1,2}(\X)\}$ are dense in $L^2(T^*\X)$.
\end{itemize}
Uniqueness is  up to unique isomoprhism: if $(\M,\d')$ is another couple with the same properties, then there is a unique $\Psi:L^2(T^*\X)\to \M$ isomorphism of modules such that $\Psi\circ\d=\d'$.\footnote{The reader familiar with category theory will realize that this has the flavour of a universal property: $(L^2(T^*\X),\d)$ can equivalently be described as universal among  couples $(\M,\d')$  such that $\d':W^{1,2}(\X)\to \M$ is linear with $|\d' f|\leq |\D f|$ $\mm$-a.e.\ for any $f$, meaning that it has this property and for any other such couple there is a unique morphism $\Psi:L^2(T^*\X)\to \M$  with $\Psi\circ\d=\d'$. Here `morphism' means $L^\infty$-linear and 1-Lipschitz. 

Similar considerations are in place for the pullback operation introduced in Theorem/Definition \ref{def:pb}.}
\end{thmdef}
\begin{idea}\ \\
{\sc Uniqueness} By $(ii)$  and the linearity of $\d$,  objects of the form $\sum_i\nchi_{A_i}\d f_i$ for $(A_i)$ finite Borel partition of $\X$ and $(f_i)\subset W^{1,2}(\X)$ are dense in $L^2(T^*\X)$. Then the requirement  $\Psi\circ\d=\d'$ forces the definition $\Psi(\sum_i\nchi_{A_i}\d f_i)=\sum_i\nchi_{A_i}\d' f_i$ and using $(i),(ii)$ it is easy to see that this is well posed and uniquely extends to the desired isomorphism.

\noindent{\sc Existence} comes from an explicit construction. We define `Pre-cotangent module' ${\sf Pcm}$ as the set of finite sequences $(A_i,f_i)$ with $(A_i)$ Borel partition of $\X$ and $(f_i)\subset W^{1,2}(\X)$. We then say $(A_i,f_i)\sim (B_j,g_j)$ if $|\D(f_i-g_j)|=0$ $\mm$-a.e.\ on $A_i\cap B_j$ for every $i,j$ and denote by $[A_i,f_i]\in{\sf Pcm}/\sim$ the equivalence class of $(A_i,f_i)\in {\sf Pcm}$ (intuitively: we shall think at $[A_i,f_i]$ as the 1-form that is equal to $\d f_i$ on $A_i$ - even if at this stage $\d f$ is not yet defined - notice that in the smooth setting these 1-forms are dense in the space of $L^2$ 1-forms). 

We endow ${\sf Pcm}/\sim$ with the following operations:
\[
\begin{array}{rll}
\alpha[A_i,f_i]+\beta[B_j,g_j]&:=\ [A_i\cap B_j,\alpha f_i+\beta g_j],\qquad&\text{\emph {(linear combination),}}\\
|[A_i,f_i]|&:=\ \sum_i\nchi_{A_i}|\D f_i|,\qquad&\text{\emph{(pointwise norm),}}\\
\big(\sum_j\nchi_{B_j}\alpha_j\big)\,[A_i,f_i]&:=\ [A_i\cap B_j,\alpha_j f_i],\qquad&\text{\emph{(product by simple functions),}}\\
\big\|[A_i,f_i]\big\|^2&:=\ {\int |[A_i,f_i]|^2\,\d\mm},\qquad&\text{\emph{(norm)}.}
\end{array}
\]
It is easy to see that these are well posed and uniquely extend, by continuity, to operations on the completion of $({\sf Pcm}/\sim,\|\cdot\|)$, thus giving such completion the structure of $L^2$-normed $L^\infty$-module. Now putting $\d f:=[\X,f]$ we see that $(i)$ holds by definition of pointwise norm and $(ii)$ by the identity $\sum_i\nchi_{A_i}\d f_i=[A_i,f_i]$ (that  follows from the above defined operations) and density of ${\sf Pcm}/\sim$ in the module just built.
\end{idea}
The uniqueness and property \eqref{eq:dffinsl} of minimal weak upper gradients show that if $\X$ is a smooth Finsler manifold, then the cotangent module so defined is canonically isomorphic to the space of $L^2$ 1-forms, the differential being (identified to) the one obtained via integration by parts.
\begin{proposition} The differential $\d$ is a closed operator when seen as unbounded operator from $L^2(\X)$ to $L^2(T^*\X)$ and satisfies
\[
\begin{array}{rlll}
\d f&=\ \d g,&\qquad\mm-a.e.\ on\ \{f=g\},\ \forall f,g\in W^{1,2}(\X)&\text{\emph {(locality),}}\\
\d(\varphi\circ f)&=\ \varphi'\circ f,&\qquad\forall\varphi\in C^1\cap\LIP(\R),\ f\in W^{1,2}(\X)&\text{\emph{(chain rule),}}\\
\d(fg)&=\ f\d g+g\d f,&\qquad\forall f,g\in L^\infty\cap W^{1,2}(\X)&\text{\emph{(Leibniz rule).}}\\
\end{array}
\]
\end{proposition}
\begin{idea} For closure, let $(f_n)\subset  W^{1,2}(\X)$ be such that $(f_n),(\d f_n)$ converge to $f,\omega$ in $L^2(\X)$ and $L^2(T^*\X)$ respectively. We aim at proving that $f\in W^{1,2}(\X)$ and $\d f=\omega$. The Sobolev regularity of $f$ follows from the $L^2$-lower semicontinuity of the Cheeger energy and  $2\ch(f_n)=\int|\d f_n|^2\,\d\mm\to \|\omega\|_{L^2(T^*\X)}^2<\infty$. The linearity of $\d$ and again such semicontinuity give, as in \eqref{eq:chlsc}, that
\[
\|\d f-\d f_n\|^2_{L^2(T^*\X)}=2\ch(f-f_n)\leq\limi_m 2\ch(f_m-f_n)=\limi_m \|\d f_m-\d f_n\|^2_{L^2(T^*\X)}
\]
so that sending $n\to\infty$ and using that $(\d f_n)$ is $L^2(T^*\X)$-Cauchy we conclude that such sequence converges to $\d f$ and thus that $\d f=\omega$.

For locality notice that $|\nchi_{\{f=0\}}\d f|=\nchi_{\{f=0\}}|\D f|=0$ by the locality property \eqref{eq:dfloc} and conclude via linearity of $\d$. For the chain rule say that $\mm(\X)<\infty$ so that constant functions are in $L^2(\X)$ (otherwise argue by locality). Then we already know by \eqref{eq:dfchain} that $\varphi\circ f\in W^{1,2}(\X)$ and, by the linearity of $\d$ and the trivial fact that constant functions have zero differential, that the chain rule holds if $\varphi$ is affine. Then by locality it holds if $\varphi$ is piecewise affine and the conclusion follows by approximation using the closure of $\d$. Finally, the Leibniz rule for $f,g\geq 1$ follows from the chain rule for $\varphi=\log$:
\[
\d(fg) =fg\,(\d(\log(fg))=fg\,(\d(\log f+\log g))=fg\,(\d(\log f)+\d (\log g))=g\,\d f+f\,\d g
\]
and the general case follows from this one applied to $f+C,g+C$ for $C\gg1$.
\end{idea}
We say that $\M$ is a {\bf Hilbert module} provided it is, when seen as a Banach space, a Hilbert space. From $\|\nchi_Ev\|^2_\M=\int_E|v|^2\,\d\mm$ it is easy to see that this is the case iff the pointwise norm satisfies the pointwise parallelogram identity
\[
|v+w|^2+|v-w|^2=2(|v|^2+|w|^2)\quad\mm-a.e.\qquad\forall v,w\in\M
\]
and in this case the formula
\[
\la v,w\ra:=\tfrac12(|v+w|^2-|v|^2-|w|^2)
\]
defines an $L^\infty$-bilinear and continuous map from $\M^2$ to $L^1$. 

From this discussion and Proposition \ref{prop:infhilbcalc} it is easy to see that $(\X,\sfd,\mm)$ is infinitesimally Hilbertian iff $L^2(T^*\X)$ is a Hilbert module and that in this case the scalar product of $\d f$ and $\d g$ as just defined coincides with the quantity introduced in Proposition \ref{prop:infhilbcalc}.

\medskip

The concept of $L^2$-normed module somehow interpolates the ones of Banach space and vector bundle and burrows constructions from both of these. In particular, both the {\bf dual} and the {\bf pullback} of a module can be introduced in a natural way. 

\medskip

The dual $\M^*$ of $\M$ is the space of $L^\infty(\X)$-linear and continuous maps $T:\M\to L^1(\X)$ equipped with the operator norm, the natural product $(fT(v)):=T(fv)$ and the pointwise norm:
\[
|T|_*:=\esssup_{|v|\leq 1\ \mm-a.e.} T(v)\qquad\forall T\in \M^*.
\]
A little bit of work shows that with this structure $\M^*$ is also an $L^2$-normed $L^\infty$-module and that
\begin{equation}
\label{eq:intmap}
\begin{split}
&\text{the natural `integration' map ${\sf Int}$ from $\M^*$ to the Banach dual $\M'$ of $\M$ }\\
&\text{sending $T$ to $\textstyle{v\mapsto \int T(v)\,\d\mm}$ is a bijective isometry.}\footnotemark
\end{split}
\end{equation}
\footnotetext{the only non trivial part is surjectivity of ${\sf Int}$. For this let $\ell\in \M'$ and notice that for $v\in \M$ the map taking a Borel set $E\subset\X$ and returning $\ell(\nchi_Ev)\in\R$ is a measure absolutely continuous w.r.t.\ $\mm$: its Radon-Nikodym density is the desired value of ${\sf Int}^{-1}(\ell)(v)$.}

Relying on this fact one can prove the duality formula
\[
|v|=\esssup_{|T|_*\leq 1\ \mm-a.e.} T(v)\qquad\forall v\in \M
\]
that shows that $\M$ isometrically embeds in its bidual $\M^{**}$ in a natural way. Then again the fact that ${\sf Int}$ is an isometric bijection shows that $\M$ is reflexive as a module (i.e.\ the embedding of $\M$ in $\M^{**}$ is surjective) iff it is reflexive as Banach space. 

If $\M$ is a Hilbert module, then the Riesz isomorphism of $\M$ with $\M'$ induces, by post-composition with ${\sf Int}^{-1}$, a Riesz isomorphism of modules $\mathscr R:\M\to\M^*$: for $v\in\M$ the element $\mathscr R(v)\in\M^*$ is characterized by the fact that 
\[
\la v,w \ra=\mathscr R(v)(w)\quad\mm-a.e.\qquad\forall w\in\M.
\]
In particular, on an infinitesimally Hilbertian space the tangent and cotangent modules are isomorphic via such isomorphism: we shall denote by $L^2(T\X)$ the dual of $L^2(T^*\X)$ and  by $\nabla f\in L^2(T\X)$  the image of $\d f\in L^2(T^*\X)$ under the Riesz isomorphism. Notice that   since, rather trivially from Theorem/Definition \ref{thmdef:d}, if $W^{1,2}(\X)$ is separable so is $L^2(T^*\X)$, recalling also \eqref{eq:w12reflsep} we have
\begin{equation}
\label{eq:infhilftangsep}
\text{$(\X,\sfd,\mm)$ infinitesimally Hilbertian}\qquad\Rightarrow\qquad\text{$L^2(T^*\X)$ and $L^2(T\X)$ are separable.}
\end{equation}
Having defined a differential and a concept of duality, by taking the adjoint we can define the divergence: for $v\in L^2(T\X)$ we say
\begin{equation}
\label{eq:defdiv}
v\in D(\div)\quad\text{ with }\quad\div (v)=f\qquad\text{iff}\qquad \int fg\,\d\mm=-\int \d g(v)\,\d\mm\quad\forall g\in W^{1,2}(\X).
\end{equation}
Starting from the analogous properties of $\d$ it is then easy to see that $\div:L^2(T\X)\to L^2(\X)$ is a closed operator and that
\begin{equation}
\label{eq:leibdiv}
\begin{array}{ll}
\div(fv)=\d f(v)+f\div(v)&\qquad\forall v\in D(\div),\ f\in\Lip_\b(\X),\\
f\in D(\Delta)\quad\Leftrightarrow\quad \nabla f\in D(\div)&\qquad\text{and in this case}\qquad\Delta f=\div(\nabla f).
\end{array}
\end{equation}

The pullback operation takes an $L^2(\X)$-normed $L^\infty(\X)$-module  $\M$ and a suitable map $\varphi:\Y\to\X$ and returns an $L^2(\Y)$-normed $L^\infty(\Y)$-module $\varphi^*\M$ together with a pullback map $\varphi:\M\to\varphi^*\M$. If we think at the module $\M$ as the space of sections of the bundle whose fibre at $x$ is the space $\M_x$, then the analogy with the classical construction in differential geometry would lead at thinking $\varphi^*\M$ as the space of sections of the bundle whose fibre at $y\in\Y$ is the space $\varphi^*\M_y:=\M_{\varphi(y)}$ and the pullback map would take the section $x\mapsto v_x\in\M_x$ and return $y\mapsto v_{\varphi(x)}\in \varphi^*\M_y$. This intuitive description is compatible with the following characterization,  given in terms of the axiomatization of modules as  above:
\begin{thmdef}\label{def:pb}  Let $(\X,\sfd_\X,\mm_\X)$ and $(\Y,\sfd_\Y,\mm_\Y)$ be two metric measure spaces and $\varphi:\Y\to\X$ Borel with $\varphi_*\mm_\Y\leq C\mm_\X$ for some $C>0$. Then for every $L^2(\X)$-normed $L^\infty(\X)$-module $\M$ there is a unique couple $(\varphi^*\M,\varphi^*)$ with $\varphi^*\M$ being an $L^2(\Y)$-normed $L^\infty(\Y)$-module and $\varphi^*:M\to\varphi^*\M$ linear and such that
\begin{itemize}
\item[i)] $|\varphi^*v|=|v|\circ\varphi$ $\mm_\Y$-a.e.\ for any $v\in\M$,
\item[ii)] $L^\infty(\Y)$-linear combinations of $\{\varphi^*v:v\in\M\}$ are dense in $\varphi^*\M$.
\end{itemize}
Uniqueness is  up to unique isomorphism: if $(\mathscr N,\tilde\varphi^*)$ is another couple with the same properties, then there is a unique $\Psi:\varphi^*\M\to\mathscr N$ isomorphism of modules such that $\Psi\circ\varphi^*=\tilde\varphi^*$.
\end{thmdef}
\begin{idea} Arguments close to those for Theorem \ref{thmdef:d} are valid also here. For existence, one starts with the `Prepullback module' ${\sf Ppb}$ defined as the set of finite sequences $(A_i,v_i)$ with $(A_i)$ Borel partition of $\Y$ and $(v_i)\subset\M$, puts on it the equivalence relation $(A_i,v_i)\sim(B_j,w_j)$ if $|v_i-w_j|\circ\varphi=0$ $\mm_\Y$-a.e.\ on $A_i\cap B_j$ and then defines the operations
\[
\begin{array}{rll}
\alpha[A_i,v_i]+\beta[B_j,w_j]&:=[A_i\cap B_j,\alpha v_i+\beta w_j],\qquad&\text{\emph {(linear combination),}}\\
|[A_i,v_i]|&:=\sum_i\nchi_{A_i}|v_i|\circ\varphi,\qquad&\text{\emph{(pointwise norm),}}\\
\big(\sum_j\nchi_{B_j}\alpha_j\big)\,[A_i,v_i]&:=[A_i\cap B_j,\alpha_j v_i],\qquad&\text{\emph{(product by simple functions),}}\\
\big\|[A_i,v_i]\big\|^2&:={\int |[A_i,v_i]|^2\,\d\mm_\Y},\qquad&\text{\emph{(norm)}.}
\end{array}
\]
As before, these can be extended by continuity to the completion of $({\sf Ppb}/\sim,\|\cdot\|)$, thus producing the desired pullback.
\end{idea}
Notable examples of pullbacks are:
\begin{itemize}
\item[-] $(\Y,\sfd_\Y,\mm_\Y):=(\X,\sfd,\mm\restr E)$ for $E\subset\X$ Borel and $\varphi:\Y\to\X$ the identity map. In this case the pullback  is given by the `restricted' module $\M\restr E:=\{\nchi_Ev:v\in\M\}$ and the restriction map $v\mapsto\nchi_Ev$.
\item[-] $\Y$ is the product of $\X$ and the Euclidean interval $[0,1]$ equipped with the product measures and distance,   and $\varphi:\Y\to\X$ is the projection map. In this case the pullback is given by the Lebesgue-Bochner space $L^2([0,1],\M)$ (that carries the structure of $L^\infty(\Y)$ module via the operations $|(t\mapsto v_t)|(t,x):=|v_t|(x)$ and $f(\cdot,\cdot)(t\mapsto v_t):=(t\mapsto f(t,\cdot)v_t)$) and the map sending $v\in\M$ to the constant function $(t\mapsto v)\in L^2([0,1],\M)$.
\end{itemize}
The latter example, the isometry ${\sf Int}:\M^*\to\M'$ and the well-established duality theory of Lebesgue-Bochner spaces (see e.g.\ \cite{DiestelUhl77}) show that in general the dual of the pullback is not the pullback of the dual. More precisely, the natural $L^\infty(\Y)$-linear and isometric immersion ${\sf I}:\varphi^*\M^*\to(\varphi^*\M)^*$ characterized by the validity of ${\sf I}(\varphi^*w)(\varphi^*v)=w(v)\circ\varphi$ is not always surjective. Surjectivity of ${\sf I}$ holds at least if $\M$ is Hilbert (with a direct proof based on Riesz' isomorphism - see \cite[Proposition 1.32]{Gigli17}) or if $\M^*$ is separable (via a more elaborate argument that reduces the problem to the classical study of $(L^2([0,1],\M))'$ - see \cite[Theorem 1.6.7]{Gigli14}).

\medskip

The next proposition shows that the pointwise norm in the tangent module can be used to compute speed of curves, much like in the classical smooth setting. Due to the irregularity of the underlying spaces and the nature of the definitions adopted, we cannot expect this statement to have a meaning for \emph{any} given curve: some sort of \emph{a.e.\ curve} should appear somewhere. It is then not surprising that the concept of test plan is called into play:    
\begin{theorem}\label{thm:speedplan}
Let $(\X,\sfd,\mm)$ be a metric measure space such that $L^2(T\X)$ is separable and $\ppi$ a test plan. Then for a.e.\ $t\in[0,1]$ there is $\ppi'_t\in \e_t^*L^2(T\X)$ such that for any $f\in W^{1,2}(\X)$ the derivative of $t\mapsto f\circ\e_t \in L^1(\ppi)$ (that exists by Lemma \ref{le:equivsobw}) is given by $(\e_t^*\d f)(\ppi'_t)$ $\ppi$-a.e..  Moreover, it holds
\begin{equation}
\label{eq:metsp}
|\ppi'_t|(\gamma)=|\dot\gamma_t|\qquad(\ppi\times\mathcal L^1)-a.e.\ (\gamma,t).
\end{equation}
\end{theorem}
\begin{idea} Recall that if the dual of a Banach space is separable so is the given Banach space. Hence $L^2(T^*\X)$ is separable and thus so is $W^{1,2}(\X)$ (that embeds isometrically in $L^2(\X)\times L^2(T^*\X)$ via $f\mapsto(f,\d f)$). 

By Lemma \ref{le:equivsobw} and an argument based on the separability of $W^{1,2}(\X)$ we deduce that there is a Borel exceptional set $N\subset[0,1]$ such that for any $f\in W^{1,2}(\X)$ the derivative ${\rm Der}_{\sppi,t}(f)$ of   $t\mapsto f\circ\e_t \in L^1(\X)$  exists for $t\notin N$ and satisfies 
\begin{equation}
\label{eq:derp}
|{\rm Der}_{\sppi,t}(f)|(\gamma)\leq |\D f|(\gamma_t)|\dot\gamma_t|\qquad for\ \ppi-a.e.\ \gamma.
\end{equation}
Now let $V\subset \e_t^*(L^2(T^*\X))$ be the space of elements of the form $\sum_i\nchi_{A_i}\e_t^*\d f_i$ where $(A_i)$ is a finite Borel partition of $C([0,1],\X)$ and $(f_i)\subset W^{1,2}(\X)$. Define $L_t:V\to L^1(\ppi)$ by putting $L_t(\sum_i\nchi_{A_i}\e_t^*\d f_i)=\sum_i\nchi_{A_i}{\rm Der}_{\sppi,t}(f_i)$ and notice that \eqref{eq:derp} ensures that the definition is well posed and that  $|L_t(w)|(\gamma)\leq |w||\dot\gamma_t|$ for any $w\in V$ and  for $\ppi$-a.e.\ $\gamma$. In particular, $L_t$ is continuous. Since $V$ is dense in $\e_t^*L^2(T^*\X)$, this tells that $L_t$ can be uniquely extended to a linear continuous map $\ppi_t'$ from $\e_t^*L^2(T^*\X)$ to $L^1(\ppi)$ and such extension still satisfies 
\begin{equation}
\label{eq:normpp}
|\ppi_t'(w)|\leq |w||\dot\gamma_t|\quad \forall  w\in \e_t^*L^2(T^*\X),\qquad \ppi-a.e.\ \gamma. 
\end{equation}
This bound ensures that $\ppi'_t$ respects the product with characteristic functions and thus, by linearity and continuity, the product with functions in $L^\infty(\ppi)$. In other words, $\ppi_t'\in (\e_t^*L^2(T^*\X))^*\sim \e_t^*L^2(T\X)$ (from the separability of $L^2(T\X)$  the discussion made before) and \eqref{eq:normpp} and the definition of dual pointwise norm give  $\leq$ in \eqref{eq:metsp}.

For the opposite inequality we notice that $|{\rm Der}_{\sppi,t}(f)|(\gamma)\leq |\D f|(\gamma_t)|\ppi'_t|(\gamma)$  for $\ppi$-a.e.\ $\gamma$ by definition of dual norm. Thus for $f$ 1-Lipschitz \eqref{eq:dflip} gives $\partial_t(f\circ\gamma)\leq |\ppi'_t|(\gamma)$ for $\ppi$-a.e.\ $\gamma$ and a.e.\ $t$, hence  to conclude it is sufficient to find a sequence $(f_n)\subset W^{1,2}(\X)$ of 1-Lipschitz functions such that $\sup_n\partial_t(f\circ\gamma)\geq|\dot\gamma_t|$ for a.e.\ $t$ and any absolutely continuous curve $\gamma$. Take $f_n:=(n-\sfd(\cdot,x_n))^+$ with $(x_n)\subset\X$ dense: since $\sfd(x,y)=\sup_nf_n(x)-f_n(y)$ we have
\[
\sfd(\gamma_s,\gamma_t)=\sup_nf_n(\gamma_s)-f_n(\gamma_t)=\sup_n\int_t^s\partial_rf_n(\gamma_r)\,\d r\leq\int_t^s \sup_n\partial_rf_n(\gamma_r)\,\d r
\]
and the conclusion follows from the minimality property of the metric speed.
\end{idea}

\subsection{Second order calculus on $\RCD$ spaces}\label{se:secondordercalc}
The language of $L^2$-normed modules offers a natural way to speak about higher order tensors: these are elements of suitable tensor products of $L^2(T^*\X)$ and $L^2(T\X)$. Here the tensor product of the Hilbert modules $\Hi_1,\Hi_2$ on $\X$ is obtained by:
\begin{itemize}
\item[-] first considering their algebraic tensor product $\Hi_1\otimes_{alg}\Hi_2$ as $L^\infty(\X)$-modules, i.e.\ the space of formal finite sums $\sum_iv_i\otimes w_i$ that are $L^\infty$-linear in each entry,
\item[-] then  defining on $\Hi_1\otimes_{alg}\Hi_2$ a symmetric map with values on $\mm$-a.e.\ defined Borel functions by putting
\[
\la v_1\otimes w_1,v_2\otimes w_2\ra:=\la v_1,v_2\ra_{1}\la w_1,w_2\ra_{2}
\]
where $\la\cdot,\cdot\ra_i$ is the pointwise scalar product on $\Hi_i$, and extend it by $L^\infty$-bilinearity. Standard manipulations show that this is well-posed and a.e.\ positively definite, thus $|A|:=\sqrt{\la A,A\ra}$ is $\mm$-a.e.\ well defined for any $A\in \Hi_1\otimes_{alg}\Hi_2$.
\item[-] Finally, the tensor product $\Hi_1\otimes\Hi_2$ is the completion of  the space $\{A\in \Hi_1\otimes_{alg}\Hi_2\ :\ |A|\in L^2(\X) \}$  w.r.t.\ the natural norm $\|A\|:=\||A|\|_{L^2}$.  It is not hard to see that this comes naturally with the structure of $L^2(\X)$-normed $L^\infty(\X)$-module and that if $\Hi_1,\Hi_2$ are separable, so is $\Hi_1\otimes\Hi_2$.
\end{itemize}
Thus, for instance, if we want to define the space $W^{2,2}(\X)$ and the Hessian of any such function, we know that such Hessian should belong to the tensor product  $L^2((T^*)^{\otimes 2}\X)$ of $L^2(T^*\X)$ with itself. We shall denote by $|\cdot|_{\HS}$ the pointwise norm in $L^2((T^*)^{\otimes 2}\X)$, as in the smooth setting the construction above produces the Hilbert-Schmidt norm.

In the smooth setting, the Hessian can be characterized via iterated gradients, i.e.\ via
\begin{equation}
\label{eq:Hesssmoooth}
{\rm Hess}f(\nabla g,\nabla h)=\tfrac12\big(\la\d\la\d f,\d g\ra,\d h\ra+\la\d \la\d f,\d h\ra,\d g\ra-\la\d f,\d\la\d g,\d h\ra\ra\big).
\end{equation}
We will use this formula to define the Hessian in $\RCD$ spaces, but in order to do so we need to have `many' functions $g,h$ for which $\d\la\d g,\d h\ra$ is well defined, i.e.\ for which $\la\d g,\d h\ra\in W^{1,2}(\X)$. This is one of the motivations that lead to the definition of the  space of \emph{test functions}
\[
\test(\X):=\big\{f\in L^\infty(\X)\cap\Lip(\X)\cap D(\Delta)\ :\ \Delta f\in L^\infty(\X)\cap D(\Delta)\big\}.
\]
Notice that using the mollified heat flow as in \eqref{eq:regheat} and recalling  the weak maximum principle and the $L^\infty-\Lip$ contraction property \eqref{eq:linftylip}, we can prove that   for $f\in W^{1,2}\cap L^\infty(\X)$ we have $\h_\eta f\in \test(\X)$, thus obtaining the density of $\test(\X)$ in $W^{1,2}(\X)$.

The relevance of the space $\test(\X)$ is due to the following result:
\begin{proposition}\label{prop:test} Let $(\X,\sfd,\mm)$ be $\RCD(K,\infty)$ and $f\in \test(\X)$. Then:
\begin{itemize}
\item[i)] $|\d f|^2\in W^{1,2}(\X)$ with
\begin{equation}
\label{eq:ddf}
\ch(|\d f|^2)\leq \Lip(f)^2\Big( \|\d f\|_{L^2}\|\d\Delta f\|_{L^2}+K^-\|\d f\|^2_{L^2}\Big).
\end{equation}
\item[ii)] $\test(\X)$ is an algebra.
\item[iii)] There exists a unique finite Borel (signed) measure $\Ggamma_2(f)$ on $\X$ satisfying
\begin{equation}
\label{eq:Ggamma2}
\int g\,\d\Ggamma_2(f)=\int \tfrac12{|\d f|^2}\Delta g-g\la \d f,\d\Delta f \ra\,\d\mm\qquad\forall g\in \test(\X)
\end{equation}
and we have
\begin{equation}
\label{eq:BG2}
\Ggamma_2(f)\geq K|\d f|^2\,\d\mm\qquad\text{(Bochner inequality)}.
\end{equation}
\end{itemize} 
\end{proposition}
\begin{idea}
Recall \eqref{eq:regheat}, pick $g:=\h_\eta(|\d f|^2)$ in \eqref{eq:bochner} and notice that $0\leq g\leq \Lip(f)^2$ to get
\[
\int \tfrac12\Delta \h_\eta(|\d f|^2) |\d f|^2\,\d\mm\geq -\Lip(f)^2\Big( \|\d f\|_{L^2}\|\d\Delta f\|_{L^2}+K^-\|\d f\|^2_{L^2}\Big).
\]
Then observe that $-\int \Delta \h_\eta(|\d f|^2) |\d f|^2\,\d\mm=\iint_0^1\eta_t | \d\h_{t/2}(|\d f|^2)|^2\,\d t\d\mm$, let $\eta\weakto \delta_0$ and use the lower semicontinuity of $\ch$ to get $(i)$. Then $(ii)$ follows by direct computation: the only non-trivial claim is the fact that $\Delta(fg)$ is in $W^{1,2}(\X)$ for $f,g\in\test(\X)$, but this follows from the - easy to check - Leibniz rule $\Delta(fg)=f\Delta g+g\Delta f+2\la\d f,\d g\ra$ and $(i)$.

For $\X$  compact,  $(iii)$ follows from  Riesz's theorem and the Bochner inequality \eqref{eq:bochner}. The general case is technically more involved: the representability as measure of the functional sending $g$ to the right hand side of \eqref{eq:Ggamma2} comes from the so-called \emph{quasi regularity} of the Dirichlet form $\ch$ (a sort of tightness of Dirichlet forms) and the  \emph{transfer method}.
\end{idea}
For $f\in\test(\X)$ we shall write
\[
\Ggamma_2(f)=\gamma_2(f)\mm+\Ggamma_2^s(f)\qquad{ where }\qquad\Ggamma_2^s(f)\perp\mm.
\]
Now that we know that $|\d f|^2\in W^{1,2}(\X)$ for $f\in\test(\X)$, the following definition makes sense (notice the analogy with \eqref{eq:Hesssmoooth}). Below we shall write $A(\nabla g,\nabla h)$ for the pointwise scalar product of $ A$ and $\d g\otimes \d h$ in $L^2((T^*)^{\otimes 2}\X)$.
\begin{definition}[The space $W^{2,2}(\X)$]\label{def:w22} Let $(\X,\sfd,\mm)$ be $\RCD(K,\infty)$. We say that $f\in W^{1,2}(\X)$ belongs to $W^{2,2}(\X)$ if there is $A\in L^2((T^*)^{\otimes 2}\X)$ such that
\[
\int hA(\nabla g_1,\nabla g_2)\,\d\mm=\tfrac12\int-\la\d f,\d g_1\ra\div(h\nabla g_2)-\la\d f,\d g_2\ra\div(h\nabla g_1)-h\la\d f,\d\la\d g_1,\d g_2\ra\ra\,\d\mm
\]
holds for any $g_1,g_2,h\in\test(\X)$. Such $A$ will be denoted ${\rm Hess}f$. We also define the norm
\[
\|f\|_{W^{2,2}(\X)}^2:=\|f\|_{L^2(\X)}^2+\|\d f\|_{L^2(T^*\X)}^2+\|{\rm Hess}\,f\|_{L^2((T^*)^{\otimes 2}\X)}^2
\] 
\end{definition}
From the density of $\test(\X)$ in $W^{1,2}(\X)$ one can prove that $A$ is unique, so that the definition of the Hessian ${\rm Hess}f$ is well posed. It is then clear that $W^{2,2}(\X)$ is a separable Hilbert space (for separability use the isometric embedding $f\mapsto (f,\d f,{\rm Hess}f)$ of $W^{2,2}$ in $L^2(\X)\times L^2(T^*\X)\times L^2((T^*)^{\otimes 2}\X)$ and \eqref{eq:infhilftangsep}).

Even after Proposition \ref{prop:test}, it is not obvious that $W^{2,2}(\X)$ contains any non-constant function. To prove this it is useful to introduce  the `proto-Hessian' $H[f](g,h)$ as
\[
H[f](g,h):=\tfrac12\big(\la\d\la\d f,\d g\ra,\d h\ra+\la\d \la\d f,\d h\ra,\d g\ra-\la\d f,\d\la\d g,\d h\ra\ra\big)\quad\in L^1\cap L^2(\mm) 
\]
for every $f,g,h\in\test(\X)$. Of course we expect that $H[f](g,h)={\rm Hess}[f](\nabla g,\nabla h)$ but we don't know this yet, because $H[f]$ is not a tensor and thus can't be chosen as $A$ in Definition \ref{def:w22}. Still, we have:
\begin{lemma}[Self improvement of  Bochner inequality]\label{le:sibo} Let $(\X,\sfd,\mm)$ be an $\RCD(K,\infty)$ space, $f\in\test(\X)$. Then $\Ggamma_2^s(f)\geq0$ and for every  choice $(h_j)\subset\test(\X)$, $j=1,\ldots,m$, we have
\begin{equation}
\label{eq:SIB}
\Big|\sum_jH[f](h_j,h_j)\Big|^2\leq  \Big(\sum_{j,j'}\big|\la\d h_j, \d h_{j'}\ra\big|^2\Big)\,\big(\gamma_2(f)-K|\d f|^2\big)\qquad\mm-a.e..
\end{equation}
\end{lemma}
\begin{idea}\footnotemark The inequality $\Ggamma_2^s(f)\geq0$  follows directly from \eqref{eq:BG2}. We prove \eqref{eq:SIB} for $m=1$. For $c,\lambda\in\R$ put $F_{\lambda,c}:=\lambda f+(h-c)^2\in\test(\X)$. Then direct computations give
\[
\Ggamma_2(F)-K|\d F|^2\mm=\lambda^2\big(\Ggamma_2(f)-K|\d f|^2\mm\big)+\big(4\lambda H[f](h,h)+4|\d h|^4\big)\mm+(h-c){\bf R}_{c}
\]
for some reminder term ${\bf R}_{c}={\bf R}_{c,\lambda,f,h}$ satisfying $|{\bf R}_{c}|\leq(1+c^2)\mu$ for some finite measure $\mu=\mu_{\lambda,f,h}$. Since the left hand side is $\geq0$ and $c\in\R$ is arbitrary, we can conclude that 
\[
\lambda^2\big(\Ggamma_2(f)-K|\d  f|^2\mm\big)+\lambda\,4 H[f](h,h)\mm+4|\d h|^4\mm\geq0\qquad \forall\lambda\in\R,
\]
so that inspecting the absolutely continuous part w.r.t.\ $\mm$ and using the arbitrariness of $\lambda$ it   follows that  $|H[f](h,h)|^2\leq |\d  h|^4\big(\gamma_2(f)-K|\d f|^2\big)$ $\mm$-a.e., which is \eqref{eq:SIB} for $m=1$. 
%
\end{idea}
\footnotetext{Read in the smooth setting, the statement tells that if the inequality $\Delta\frac{|\d f|^2}2-\la\d f,\d\Delta f\ra-K|\d f|^2\geq 0$ holds everywhere for any smooth function, then the stronger inequality $\Delta\frac{|\d f|^2}2-\la\d f,\d\Delta f\ra-K|\d f|^2\geq |{\rm Hess}f|_{\HS}^2$ also holds. On manifolds this is easy to prove, as recalling the Bochner-Weitzenb\"ock identity $$\Delta\frac{|\d f|^2}2= |{\rm Hess}f|_{\HS}^2+\la\d f,\d\Delta f\ra+{\rm Ric}(\nabla f,\nabla f)$$ the claim becomes: if $|{\rm Hess}f|_{\HS}^2+{\rm Ric}(\nabla f,\nabla f)\geq K|\d f|^2$ for any $f$, then in fact ${\rm Ric}(\nabla f,\nabla f)\geq K|\d f|^2$. This is clear, because for any point $x\in M$ and $v\in T_xM$ there is $f$ with $\nabla f(x)=v$ and ${\rm Hess}f(x)=0$: picking this $f$ the conclusion follows by the arbitrariness of $x,v$.

In the non-smooth setting this pointwise approach is not viable. The  argument is instead  analogue to that for the Cauchy-Schwarz inequality, that tells that $|\d f|^2\geq0$ `self improves' to $|\d g|^2|\d f|^2\geq|\la \d f,\d g\ra|^2$. A standard proof of it  applies the given bound to $\lambda f+g$ to get $0\leq|\d(\lambda f+g)|^2=\lambda^2|\d f|^2+2\lambda\la\d f,\d g\ra+|\d g|^2$ and  concludes by the arbitrariness of $\lambda\in\R$.}
We can now prove:
\begin{theorem}[There are many functions in $W^{2,2}(\X)$] Let $(\X,\sfd,\mm)$ be $\RCD(K,\infty)$. Then $D(\Delta)\subset W^{2,2}(\X)$ (in particular $W^{2,2}(\X)$ is dense in $W^{1,2}(\X)$) and
\begin{equation}
\label{eq:l2hl2d}
\int |{\rm Hess}f|_{\HS}^2\,\d\mm\leq \int |\Delta f|^2-K|\d f|^2\,\d\mm\qquad\forall f\in D(\Delta).
\end{equation}
Moreover, for any $f\in \test(\X)$ we have
\begin{subequations}
\begin{align}
\label{eq:hessugu}
H[f](g,h)&={\rm Hess}f(\nabla g,\nabla h)\quad\mm-a.e.\qquad\forall g,h\in \test(\X),\\
\label{eq:BI2}
\Ggamma_2(f)&\geq \big( |{\rm Hess}f|_{\HS}^2+K|\d f|^2\big)\mm.
\end{align}
\end{subequations}
\end{theorem}
\begin{idea} For $f\in\test(\X)$ inequality \eqref{eq:SIB} and the definition of pointwise norm in $L^2((T^*)^{\otimes 2}\X)$ tell
\begin{equation}
\label{eq:SIB2}
\big|\sum_jH[f](h_j,h_j)\big|\leq G(f)\Big|\sum_j\d h_j\otimes \d h_{j}\Big|_{\HS}\quad\mm-a.e.\quad\text{for}\quad G(f):=\sqrt{\gamma_2(f)-K|\d f|^2}\,
\end{equation}
and since we  have $ H[f](h_j,\tilde h_j)=\tfrac12( H[f](h_j+\tilde h_j,h_j+\tilde h_j)-H[f](h_j, h_j)-H[f](\tilde h_j,\tilde h_j))
$
and $\sum_j\frac{\d(h_j+\tilde h_j)\otimes\d(h_j+\tilde h_j)-\d h_j\otimes h_j-\tilde h_j\otimes \tilde h_j}2=\sum_j \tfrac{\d h_j\otimes \d \tilde h_{j}+\d\tilde h_j\otimes \d  h_{j}}2$, from \eqref{eq:SIB2} we deduce
\begin{equation}
\label{eq:SIB3}
\begin{split}
\big|\sum_jH[f](h_j,\tilde h_j)\big|&\leq  G(f)\,\Big|\sum_j\tfrac{\d h_j\otimes \d \tilde h_{j}+\d\tilde h_j\otimes \d  h_{j}}2\Big|_{\HS}\leq G(f)\,\Big|\sum_j\d h_j\otimes \d \tilde h_{j}\Big|_{\HS}
\end{split}
\end{equation}
$\mm$-a.e.. Now let $V\subset L^2((T^*)^{\otimes 2}\X)$ be the space of finite sums of the form $\sum_j\nchi_{A_j}\d h_j\otimes\d \tilde h_j$ for $(A_j)$ Borel partition of $\X$ and $(h_j),(\tilde h_j)\subset\test(\X)$. It is not hard to see that it is dense. Let $L:V\to L^1(\X)$ be given by $L(\sum_j\nchi_{A_j}\d h_j\otimes\d \tilde h_j):=\sum_j\nchi_{A_j}H[f](h_j,\tilde h_j)$. Then \eqref{eq:SIB3} grants that $L$  is well-defined, continuous and satisfies $L(\nchi_AT)=\nchi_AL(T)$ for every $T\in V$, $A\subset\X$ Borel. It is then clear that it uniquely extends to an $L^\infty$-linear and continuous map from $L^2((T^*)^{\otimes 2}\X)$ to $L^1(\X)$, i.e.\ to an element of the dual of $L^2((T^*)^{\otimes 2}\X)$. By the  Riesz isomorphism for Hilbert modules we conclude that there is a unique ${\rm Hess }f\in L^2((T^*)^{\otimes 2}\X)$ such that ${\rm Hess}f(\nabla g,\nabla g)=L(\d g\otimes\d h)$ (i.e.\ \eqref{eq:hessugu} holds) and it satisfies $|{\rm Hess} f|_{\HS}\leq G(f)$ (i.e.\ \eqref{eq:BI2} holds - recall that $\Ggamma_2^s(f)\geq 0$).

Now \eqref{eq:l2hl2d} for $f\in\test(\X)$ follows evaluating the measures in both sides of \eqref{eq:BI2} at $\X$ (picking $g\equiv1$ in \eqref{eq:Ggamma2} - possible at least if $\mm(\X)<\infty$  - we get $\Ggamma_2(f)(\X)=\int|\Delta f|^2\,\d\mm$). The general case can be proved by approximation.
\end{idea}
Notice that \eqref{eq:hessugu} in particular gives
\begin{equation}
\label{eq:dgrad}
f\in\test(\X)\qquad\Rightarrow\qquad |\d f|^2\in W^{1,2}(\X)\quad\text{ with }\quad\d(|\d f|^2)=2{\rm Hess}f(\nabla f,\cdot)\,.
\end{equation}

\begin{corollary}[Improved Bakry-\'Emery estimate] Let $(\X,\sfd,\mm)$ be $\RCD(K,\infty)$ and $f\in W^{1,2}(\X)$. Then
\begin{equation}
\label{eq:BEimpr}
|\d\h_tf|\leq e^{-Kt}\h_t(|\d f|)\qquad\mm-a.e.\qquad\forall t\geq 0.
\end{equation}
\end{corollary}
\begin{idea}
As for the proof of \eqref{eq:BElav}, fix $t,f$ and consider $[0,t]\ni s\mapsto \h_{s}(|\d\h_{t-s}f|)$. The derivative of this function is (at least formally, but computations can be justified via suitable approximations and integration by parts) given by
\[
\begin{split}
\h_s\big(\Delta|\d\h_{t-s}f|-\tfrac1{|\d \h_{t-s}f|}\la\d \h_{t-s}f,\d \Delta \h_{t-s}f\ra\big)=\h_s\big(\tfrac1{|\d \h_{t-s}f|}(\Ggamma_2(\h_{t-s}f) -|\d|\d\h_{t-s}f||^2)\big).
\end{split}
\]
Now observe that  \eqref{eq:dgrad} implies  $|\d|\d g||\leq |\Hess g|_{\HS}$. Hence \eqref{eq:BI2} tells that the above derivative is $\geq K\h_s(|\d \h_{t-s}f|)$ and the conclusion follows from Gronwall's lemma.
\end{idea}

The ideas used to define the Hessian of a function and the space $W^{2,2}(\X)$ can be tweaked to define the covariant derivative of a vector field and the space $W^{1,2}_C(T\X)$ of Sobolev vector fields (the subscript $C$ denotes Covariant differentiation). We shall denote by $L^2(T^{\otimes 2}\X)$ the tensor product of $L^2(T\X)$ with itself, by $|\cdot|_{\HS}$ the corresponding pointwise norm and by $A: B$ the pointwise scalar product. 

Noticing that in the smooth setting we have
\[
\la\nabla_{\nabla f}v,\nabla g\ra=\la\d\la v,\nabla g\ra,\d f\ra-{\rm Hess}\,g(v,\nabla f)
\]
we give the following:
\begin{definition}[The space $W^{1,2}_C(T\X)$]
Let $(\X,\sfd,\mm)$ be $\RCD(K,\infty)$. We say that $v\in L^2(T\X)$ belongs to the Sobolev space $W^{1,2}_C(T\X)$ provided there is $T\in L^2(T^{\otimes 2}\X)$ such that
\begin{equation}
\label{eq:defcov}
\int h T:(\nabla f\otimes\nabla g)\,\d\mm=\int- \la v,\nabla g\ra\div(h\nabla f)-h{\rm Hess}\,g(v,\nabla f)\,\d\mm 
\end{equation}
for every $f,g,h\in\test(\X)$. Such $T$ will be denoted $\nabla v$. On $W^{1,2}_C(T\X)$ we put the norm
\[
\|v\|_{W^{1,2}_C(T\X)}^2:=\|v\|^2_{L^2(T\X)}+\|\nabla v\|_{L^2(T^{\otimes 2}\X)}^2.
\] 
\end{definition}
It is easy to check that $T$ is unique, so that $\nabla v$ is well defined, that $W^{1,2}_C(T\X)$ is a separable Hilbert space and that if $f\in W^{2,2}(\X)$ then $\nabla f\in W^{1,2}_C(T\X)$ with ${\rm Hess}\,f\sim \nabla\nabla f$, where $\sim$ is the Riesz isomorphism. Also, as a consequence of the calculus rules for the differential, one can verify that  the covariant derivative  is \emph{compatible with the metric} and \emph{torsion free} in a natural way and satisfies
\[
f\in \LIP_\b(\X),\ v\in W^{1,2}_C(T\X)\quad\Rightarrow\quad fv\in W^{1,2}_C(T\X)\quad\text{with}\quad \nabla(fv)=\nabla f\otimes v+f\nabla v
\]
so that from the density of $W^{2,2}(\X)$ in $W^{1,2}(\X)$ we can deduce the one of $W^{1,2}_C(T\X)$ in $L^2(T\X)$. Also, with a bit of work and in analogy with \eqref{eq:dgrad} one can prove that for $v\in W^{1,2}(T\X)$ with $|v|\in L^\infty(\X)$ and $f\in\test(\X)$ we have
\begin{equation}
\label{eq:leibcov}
\text{$\d f(v)\in W^{1,2}(\X)\qquad$ with }\qquad\d (\d f(v)) ={\rm Hess}\,f(v)+\nabla v:(\,\cdot\,\otimes\nabla f ).
\end{equation}
\begin{remark}[$H\neq W$]\label{re:honda}{\rm It is natural to wonder whether test functions are dense in $W^{2,2}(\X)$ and, similarly, whether vector fields of the form $\sum_{i=1}^kg_i\nabla f_i$, $f_i,g_i\in \test(\X)$, $k\in\N$ are dense in $W^{1,2}_C(T\X)$. In both cases the answer is \emph{no}: the problem is the presence of `boundary' and can be seen already in the simple(st) case $\X=[0,1]$ equipped with the Euclidean distance and the Lebesgue measure. It is easy to see that  any $f\in D(\Delta)$ must have 0 Neumann boundary condition (otherwise Dirac masses appear in $0,1$, because in testing the definition of $\Delta f$ we are allowed to pick any $\varphi\in C^1([0,1])\subset W^{1,2}(\X)$). On the other hand, it is clear that $x\mapsto f(x):=x$  belongs to $W^{2,2}(\X)$ (with 0 Hessian) and that the map sending $g\in W^{2,2}(\X)$ to $g'(0)$ (more precisely: to the trace in 0 of $g'$) is continuous w.r.t.\ the $W^{2,2}$-norm.
}\fr\end{remark}
\subsection{Flow of vector fields}\label{se:flowrcd}
Once one has vector fields, it is natural to try to solve the ODE 
\begin{equation}
\label{eq:ODE}
\gamma'_t=v_t(\gamma_t)
\end{equation}
and to wonder whether existence/uniqueness/stability are in place. We shall see that this is the case under suitable Sobolev regularity for the $v_t$'s, once we refrain from studying these problems for a single curve, but rather consider all the solutions together and impose bounded compression for the resulting flow.

Let $\pr_\bc(\X)\subset \pr(\X)$ be the space of measures $\mu$ so that $\mu\leq C\mm$ for some $C>0$.
\begin{definition}[Regular Lagrangian Flows]
Let $(v_t)\in L^2([0,1],L^2(T\X))$. We say that a Borel map $F:[0,1]\times\X\to\X$ is a Regular Lagrangian Flow (RLF in short) for $(v_t)$ provided:
\begin{itemize}
\item[i)] $t\mapsto F_t(x)$ is continuous for $\mm$-a.e.\ $x\in\X$.
\item[ii)] There is $C>0$ such that
\begin{equation}
\label{eq:bcompr}
(F_t)_*\mm\leq C\mm\qquad\forall t\in[0,1].
\end{equation}
\item[iii)] For every $f\in W^{1,2}(\X)$ we have: for $\mm$-a.e.\ $x\in\X$ the function $t\mapsto f(F_t(x))$ is in $W^{1,1}(0,1)$ with derivative $\d f(v_t)(F_t(x))$.
\end{itemize}
We say that $(v_t)$ admits \emph{unique} RLF if for any two such flows $F,\tilde F$ we have: for  $\mm$-a.e.\ $x\in\X$ the identity  $F_t(x)=\tilde F_t(x)$ holds for every $t\in[0,1]$.
\end{definition}
Notice that by \eqref{eq:bcompr} the requirement in item $(iii)$ makes sense (since $\d f(v_t)(x)$ is well defined only up to equality for a.e.\ $(t,x)$, some regularity on $F_t(x)$ is needed to be sure that  $(t,x)\mapsto\d f(v_t)(F_t(x))$ is well defined up to equality  for a.e.\ $(t,x)$).  

Using Lemma  \ref{le:equivsobw}  for  $(v_t)\in L^2([0,1],L^2(T\X))$ it is not difficult to see that $F$ is a RLF for $(v_t)$ iff  for any $\mu\in\pr_\bc(\X)$ and $f\in W^{1,2}(\X)$ the map $t\mapsto f\circ F_t\in L^1(\mu)$ is absolutely continuous with derivative $\partial_t(f\circ F_t)=\d f(v_t)\circ F_t$. Thus recalling the definition of $\ppi'_t$ given in Theorem \ref{thm:speedplan} we see that
\[
\begin{split}
\text{$F$ is a RLF for $(v_t)$ iff $\forall\mu\in\pr_\bc(\X)$ the plan $\ppi:=(F_\cdot)_*\mu$ is  test with $\ppi'_t=\e_t^*v_t$ for a.e.\ $t$,}
\end{split}
\]
so that $\ppi'_t=\e_t^*v_t$ is our way to interpret \eqref{eq:ODE}. Then recalling \eqref{eq:metsp} we see that 
\begin{equation}
\label{eq:speedrlf}
\text{for $\mm$-a.e.\ $x$ the curve $t\mapsto F_t(x)$ is abs.\ cont.\ with speed  $|v_t|(F_t(x))$ for a.e.\ $t$.}
\end{equation}
The following measure-theoretic lemma shows that existence and uniqueness for RLFs follows from existence and uniqueness for solutions of such `ODE for test plans':
\begin{lemma}\label{le:measth}
Let $(v_t)\in L^2([0,1],L^2(T\X))$ and assume that for any $\mu\in\pr_\bc(\X)$ there exists a unique test plan $\ppi$ with $\ppi'_t=\e_t^*v_t$ for a.e.\ $t$ and $(\e_0)_*\ppi=\mu$.

 Then the RLF of $(v_t)$ exists and is unique.
  \end{lemma}
\begin{idea}\ \\ 
\noindent{\sc Uniqueness} Say that $F^1,F^2$ are two different RLF. 
Then by the continuity of $t\mapsto F^1_t(x),F^2_t(x)$ it is not hard to find $E\subset\X$ Borel with $\mm(E)\in(0,\infty)$, $r>0$ and $t\in[0,1]$ such that $\sfd(F^1_t(x),F^2_t(x))>4r$ for any $x\in E$. Then for some ball $B=B_r(y)$ the set $E':=E\cap (F^1_t)^{-1}(B)$ has positive measure and since the sets $F^1_t(E')$ and $F^2_t(E')$ are disjoint by construction, we have $(F^1_t)_*(\mm\restr{E'})\neq (F^2_t)_*(\mm\restr{E'})$. Hence for $\mu:=\mm(E')^{-1}\mm\restr{E'}$ the plans $\ppi^i:=(F^i_\cdot)_*\mu$, $i=1,2$, would be as in the statement and different.

\noindent{\sc Existence} We rely on the disintegration theorem and on the same basic principle just used. Say $\mm(\X)=1$, pick $\mu:=\mm$, let $\ppi$ be given by the statement and let $\{\ppi_x\}_{x\in\X}\subset\pr(C([0,1],\X))$ be its disintegration w.r.t.\ $\e_0$. If we prove that $\ppi_x$ is a Dirac mass, say at the curve $F_\cdot(x)$, for $\mm$-a.e.\ $x$, then $F$ is the RLF we are looking for. By contradiction, if not we can find $t\in[0,1]$ so that $(\e_t)_*\ppi_x$ is not a Dirac mass for a set of $x$'s of positive measure. Then for some $r>0$ the set of $x\in\X$ such that ${\rm diam}(\supp((\e_t)_*\ppi_x))>4r$ has positive measure and thus for some ball $B=B_r(y)$, $\delta>0$ and $E\subset\X$ with $\mm(E)>0$ we have $(\e_t)_*\ppi_x(B),(\e_t)_*\ppi_x(\X\setminus B)>\delta$ for every $x\in E$. Then for $\mm$-a.e.\ $x\in E$ we put
\[
\ppi^1_x:=\tfrac1{\sppi_x({\e_t^{-1}(B)})}\ppi_x\restr{\e_t^{-1}(B)}\qquad\text{ and }\qquad\ppi^2_x:=\tfrac1{\sppi_x({\e_t^{-1}(\X\setminus B)})}\ppi_x\restr{\e_t^{-1}(\X\setminus B)}
\]
and $\ppi^i:=\tfrac1{\mm(E)}\int_E\ppi^i_x\,\d\mm(x)$, $i=1,2$. It is easy to see that $\ppi^1,\ppi^2$ are two test plans absolutely continuous w.r.t.\ $\ppi$: this grants that $(\ppi^i)'_s=\e^*_sv_s$ (the meaning of this pullback depends on the measure $\ppi^i$ we are putting on $C([0,1],\X)$) for a.e.\ $s$, and $i=1,2$. The construction also gives $(\e_0)_*\ppi^1=(\e_0)_*\ppi^2=\mm(E)^{-1}\mm\restr E$, that  $(\e_t)_*\ppi^1$ is concentrated on $B$ and  $(\e_t)_*\ppi^2$ on $\X\setminus B$. Thus $\ppi^1\neq\ppi^2$, which is the desired contradiction. 
\end{idea}

\begin{definition}[Solution of the continuity equation]\label{def:solCE} Let $(v_t)\in L^2([0,1],L^2(T\X))$.
We say that $(\mu_t)\subset \pr(\X)$ solves the continuity equation
\begin{equation}
\label{eq:CE}
\partial_t\mu_t+\div(v_t\mu_t)=0
\end{equation}
provided $\mu_t\leq C\mm$ for every $t\in[0,1]$ and some $C>0$   and for every $f\in W^{1,2}(\X)$ the function $t\mapsto \int f\,\d\mu_t$ is absolutely continuous with derivative equal to $\int \d f(v_t)\,\d\mu_t$.   
\end{definition}
In particular, $t\mapsto\int f\,\d\mu_t$ is continuous for every $f\in \Lip_\bs(\X)$, from which it follows that $t\mapsto\mu_t$ is weakly continuous. 

It is a trivial consequence of the definitions that if $\ppi$ is a test plan with $\ppi'_t=\e_t^*v_t$ for a.e.\ $t$, then $t\mapsto\mu_t:=(\e_t)_*\ppi$ solves the continuity equation in the sense above.
It is much less trivial the fact that this operation can be inverted, so that solutions $(\mu_t)$ of the continuity equation can be appropriately  `lifted'. This is the content of the superposition principle we are going to discuss. The statement below is about the standard continuity equation in $\R^d$, whose solutions must be interpreted distributionally (i.e.\ Definition \ref{def:solCE} above plays no role).
\begin{lemma}[Superposition principle - the $\R^d$ case]\label{le:suprd}
Let $(\mu_t)\subset \pr(\R^d)$ and $(t,x)\mapsto v_t(x)\in \R^d$ Borel be such that
\[
\partial_t\mu_t+\div(v_t\mu_t)=0\qquad\text{and}\qquad \int_0^1\int|v_t|^2\,\d\mu_t\,\d t<\infty,
\]
where the continuity equation is intended in the sense of distributions. Then there is $\ppi\in\pr(C([0,1],\R^d))$   with  $(\e_t)_*\ppi=\mu_t$ for every $t\in[0,1]$ such that
\begin{equation}
\label{eq:ode}
\ppi-a.e.\ \gamma\text{ is absolutely continuous with }\qquad\gamma'_t=v_t(\gamma_t)\quad a.e.\ t\in[0,1].
\end{equation}
\end{lemma}
\begin{idea} Say that we are on the torus $\T^d$ rather than on $\R^d$. Let $(\rho^\eps)$ be a family of mollifiers and define
\[
\mu^\eps_t:=\mu_t\ast\rho^\eps\qquad\text{ and }\qquad v^\eps_t:=\tfrac{\d(v_t\mu_t)\ast\rho^\eps}{\d\mu^\eps_t}.
\]
Then clearly $(\mu^\eps_t,v^\eps_t)$ still solves the continuity equation and, by the convexity of $\R^d\times\R^+\ni (p,z)\mapsto\frac{|p|^2}z$ and Jensen's inequality, it is not hard to see that 
\begin{equation}
\label{eq:unifke}
 \int_0^1\int|v^\eps_t|^2\,\d\mu^\eps_t\,\d t\leq  \int_0^1\int|v_t|^2\,\d\mu_t\,\d t<\infty.
\end{equation}
By smoothness and the standard Cauchy-Lipschitz theory, for any $x\in\T^d$ and $\eps>0$ there is a unique curve $[0,1]\ni t\mapsto F^\eps_t(x)$ starting from $x$ and solving $\gamma'_t=v^\eps_t(\gamma_t)$. Put $\ppi^\eps:= (F^\eps_\cdot)_*\mu^\eps_0\in \pr(C([0,1],\T^d)).$ For $\varphi\in C^\infty_c((0,1)\times\T^d)$ the computation
\[
\begin{split}
\int_0^1\int \partial_t\varphi+\d_x\varphi(v_t^\eps)\,\d (\e_t)_*\ppi^\eps\,\d t=\iint_0^1 \partial_t(\varphi_t(\gamma_t))\,\d t\,\d\ppi^\eps(\gamma)=0
\end{split}
\]
shows that $t\mapsto(\e_t)_*\ppi^\eps_t$  solves the continuity equation with vector fields $v^\eps_t$. The construction ensures that the $v^\eps_t$'s are smooth in the space variable and bounded in the time variable, thus we know  (e.g.\ via the method of characteristics) that the associated continuity equation has unique solutions, i.e.\ that $\mu^\eps_t=(\e_t)_*\ppi^\eps$ for every $t,\eps$. 

The compactness of $\T^d$ and the uniform estimate \eqref{eq:unifke} imply the tightness of $(\ppi^\eps)$ (as in the proof of Lemma \ref{le:superpp}). Let $\ppi$ be the weak limit of $\ppi^{\eps_n}$: it is clear that $(\e_t)_*\ppi=\mu_t$ for any $t$, thus  to conclude it is sufficient to prove that $\int |\gamma_s-\gamma_t-\int_t^sv_r(\gamma_r)\,\d r|\,\d\ppi(\gamma)=0$ holds for any $t<s$. Let $\varphi\in C^\infty([0,1]\times\T^d;\R^d)$, put $\varphi^\eps_t:=\frac{\d(\varphi_t\mu_t)\ast\rho^\eps}{\d(\mu_t\ast\rho^\eps)}$, notice that $\varphi^\eps\to\varphi$ uniformly and let $\eps=\eps_n\downarrow0$ in 
\[
\begin{split}
\int|\gamma_s-\gamma_t-\int_t^s\varphi_r(\gamma_r)\,\d r|\,\d\ppi^\eps(\gamma)&\leq\iint_t^s| v^\eps_r(\gamma_r)-\varphi_r(\gamma_r)|\,\d r\,\d\ppi^\eps(\gamma)\\
&\leq \underbrace{\int_t^s\int|v^\eps_r-\varphi^\eps_r|\,\d\mu^\eps_r\,\d r}_{\text{(Jensen)\ }\ \leq\int_t^s\int|v_r-\varphi_r|\,\d\mu_r\,\d r}+ \int_t^s\int|\varphi^\eps_r-\varphi_r|\,\d\mu^\eps_r\,\d r
\end{split}
\] to deduce that 
\begin{equation}
\label{eq:col2}
\begin{split}
&\int|\gamma_s-\gamma_t-\!\!\int_t^s\!\! v_r(\gamma_r)\,\d r|\,\d\ppi(\gamma)\\
&\leq \int|\gamma_s-\gamma_t-\!\!\int_t^s\!\! \varphi_r(\gamma_r)\,\d r|\,\d\ppi(\gamma)+\int_t^s\!\!\int |v_r-  \varphi_r|\,\d\mu_r\,\d r
\leq2 \int_0^1\!\!\int|v_r-\varphi_r|\,\d\mu_r\,\d r.
\end{split}
\end{equation}
As $\varphi$ is arbitrary, by  approximation  in $L^1([0,1]\times\T^d,\d t\otimes\d\mu_t)$ we conclude.
\end{idea}

\begin{proposition}[Superposition principle  - the metric case]\label{prop:supmet}
Let $(\X,\sfd,\mm)$ be infinitesimally Hilbertian, $(v_t)\in L^2([0,1],L^2(T\X))$ and $(\mu_t)\subset  \pr(\X)$ be a solution of the continuity equation as in Definition \ref{def:solCE}. Then there is a test plan $\ppi$ with $(\e_t)_*\ppi=\mu_t$ and $\ppi'_t=\e_t^*v_t$ for a.e.\ $t$. 
\end{proposition}
\begin{idea} Say that $\X$ is compact. Let $(f_n)\subset W^{1,2}(\X)$ be a sequence of 1-Lipschitz  functions whose span is dense in $W^{1,2}(\X)$ (recall \eqref{eq:w12reflsep}). Possibly enlarging the set we can also assume that 
\begin{equation}
\label{eq:isoF}
\sfd(x,y)=\sup_n|f_n(x)-f_n(y)|\qquad\forall x,y\in\X.
\end{equation}
For $d\in\N$ let $F^d:\X\to\R^d$ be given by $F^d:=(f_1,\ldots,f_d)$ and $\{\mu_{t,p}\}_{p\in\R^d} $ be  the disintegration of $\mu_t$ w.r.t.\ $F^d$. Then put 
\begin{equation}
\label{eq:defnud}
\nu^d_t:=F^d_*\mu_t\qquad\text{ and }\qquad w^d_{i,t}:=\frac{\d(F^d_*(\d f_i(v_t)\mu_t))}{\d( F^d_*\mu_t)}= \int \d f_i(v_t)\,\d\mu_{t,\cdot}\quad i=1,\ldots,d.
\end{equation}
From $|\d f_i(v_t)|\leq |v_t|$ we get
\begin{equation}
\label{eq:pertigsup}
\int |w^d_{t}|_\infty^2\,\d\nu^d_t\leq \int|v_t|^2\,\d\mu_t,\qquad\forall t\in[0,1],
\end{equation}
where $w^d_t=(w^d_{1,t},\ldots,w^d_{d,t})$. Now notice  that for $g\in C^\infty_c(\R^d)$ the natural chain rule $\d(g\circ F^d)=\sum_i(\partial_i g)\circ F^d\d f_i$ holds (for $g$ polynomial this follows from  the Leibniz rule, the general case then comes by approximation recalling the closure of the differential), hence
\[
\begin{split}
\partial_t\int g\,\d\nu^d_t=\partial_t\int g\circ F^d\,\d\mu_t\stackrel{\eqref{eq:CE}}=\sum_i\int (\partial_i g)\circ F^d\,\d f_i(v_t)\,\d\mu_t=\int \d g(w^d_t)\,\d\nu^d_t
\end{split}
\]
and from this it follows that $(\nu^d_t,w^d_t)$ solves the continuity equation in $\R^d$; thus by Lemma \ref{le:suprd} we get a corresponding plan $\ppi^d$. Think $\R^d$ embedded in $\ell^\infty$ via the map sending $(p_1,\ldots,p_d)$ to $(p_1,\ldots,p_d,0,0,\ldots)$ and, for later use, notice that if $\varphi$ is a bounded function on $\ell^\infty$ continuously depending on the first $d'\leq d$  coordinates, from \eqref{eq:defnud} and Jensen's inequality we get
\begin{equation}
\label{eq:lateruse}
\int |w^d_{i,t} -\varphi |\,\d\nu^d_t\leq \int|\d f_i(v_t)\circ F^{-1}-\varphi |\,\d F_*\mu_t,\qquad\forall i\in\N,\ t\in[0,1]
\end{equation}
where $F:\X\to \ell^\infty$ is given by  $F(x):=(\ldots,f_i,\ldots)$. In particular, if $\varphi$ also depends continuously on time we have
\begin{equation}
\label{eq:later2}
\begin{split}
\lims_{d\to\infty}\int\big|\gamma_{s,i}-\gamma_{t,i}-\int_t^s\varphi_r(\gamma_r)\,\d r\big|\,\d\ppi^d(\gamma)&\leq \lims_{d\to \infty}\iint_t^s|w^d_{r,i}(\gamma_r)-\varphi_r(\gamma_r)|\,\d r\,\d\ppi^d(\gamma)\\
&=\lims_{d\to\infty} \int_t^s\!\!\int|w^d_{r,i}-\varphi_r|\,\d\nu^d_r\,\d r\\
\text{(by \eqref{eq:lateruse})}\qquad\qquad&\leq \int_t^s\!\!\int|\d f_i(v_r)\circ F^{-1}-\varphi_r|\,\d F_*\mu_r\,\d r
\end{split}
\end{equation}
for any $i\in\N$ and $t<s$. Now notice that  the $f_n$'s are uniformly bounded, thus  the measures $(\e_t)_*\ppi^d=\nu^d_t$ are all  concentrated in the same closed ball $K$ of $\ell^\infty$, which is  weakly$^*$-compact. This and the uniform estimate \eqref{eq:pertigsup} are sufficient to grant tightness of $\{\ppi^d\}$ (recall that on $K$ the weak$^*$ topology is Polish). Let $\ppi$ be  any weak limit and notice that $(\e_t)_*\ppi=F_*\mu_t$ for every $t\in[0,1]$ and, by \eqref{eq:later2} and arguing as in   \eqref{eq:col2} we get
\[
\int |\gamma_{s,i}-\gamma_{t,i}-\int \d f_i(v_r)(F^{-1}(\gamma_r))\,\d r|\,\d\ppi(\gamma)\leq 2 \int_0^1\int|\d f_i(v_r)\circ F^{-1}-\varphi_r|\,\d F_*\mu_r\,\d r
\]
for any $i\in\N$, $t<s$ and $\varphi$ as above. Letting $\varphi\to \d f_i(v_r)\circ F^{-1}$ in $L^1([0,1]\times\ell^\infty,\d t\otimes F_*\mu_t)$ we get
\begin{equation}
\label{eq:compi}
\gamma_{s,i}-\gamma_{t,i}=\int_t^s\d f_i(v_t)(F^{-1}(\gamma_r))\,\d r\qquad \ppi-a.e.\ \gamma\qquad\forall t<s,\ i\in\N.
\end{equation}
Then denoting by $F^{-1}(\gamma)$ the curve $t\mapsto F^{-1}(\gamma_t)\in\X$, it is easy to see that the plan $\bar\ppi:=F^{-1}_*\ppi$ is well defined, that $(\e_t)_*\bar\ppi=\mu_t$ for every $t$ and, since $F$ is an isometry (by \eqref{eq:isoF}) and by \eqref{eq:pertigsup}, that $\iint_0^1|\dot\gamma_t|^2\,\d t\,\d\bar\ppi(\gamma)<\infty$. Hence $\bar\ppi$ is a test plan. Also, \eqref{eq:compi} reads as
\[
f_i(\gamma_s)-f_i(\gamma_t)=\int_t^s\d f_i(v_r)\,\d r\qquad\bar\ppi-a.e.\ \gamma,\qquad\forall t<s,\ i\in\N
\]
and the arbitrariness of $t,s,i$ and the fact that the $f_i$'s generate a dense subspace of $W^{1,2}(\X)$ suffices to show that $\bar\ppi'_t=\e_t^*v_t$ for a.e.\ $t$, as desired.
\end{idea}
The superposition principle allows to prove the following result:
\begin{lemma}\label{le:ponte}
Let $(\X,\sfd,\mm)$ be infinitesimally Hilbertian and $(v_t)\in L^2([0,1],L^2(T\X))$. Assume that for every $\mu\in\pr_\bc(\X)$ there is a unique solution $(\mu_t)$ of the continuity equation in the sense of Definition \ref{def:solCE} with $\mu_0=\mu$. 

Then  $\forall\mu\in\pr_\bc(\X)$ there is a unique test plan $\ppi$ with $(\e_0)_*\ppi=\mu$ and $\ppi'_t=\e_t^*v_t$ for a.e.\ $t$.
\end{lemma}
\begin{idea}
Existence is the content of Proposition \ref{prop:supmet}. For uniqueness we argue by contradiction: if $\ppi^1\neq \ppi^2$ both have the required properties, then the disintegration of $\ppi:=\frac{\sppi^1+\sppi^2}2$ w.r.t.\ $\e_0$ is not always made of Dirac masses. Hence as in the proof of Lemma \ref{le:measth} we can find test plans $\tilde\ppi^1,\tilde\ppi^2$ with  $(\tilde\ppi^i)'_t=\e_t^*v_t$ for a.e.\ $t$, $(\e_0)_*\tilde\ppi^1=(\e_0)_*\tilde\ppi^2$ and so that $(\e_t)_*\tilde\ppi^1\neq(\e_t)_*\tilde\ppi^2$ for some $t$. Then $s\mapsto (\e_s)_*\tilde\ppi^1,(\e_s)_*\tilde\ppi^2$ are two different solutions of the continuity equation starting from the same measure, giving the desired contradiction.
\end{idea}
Thus existence and uniqueness of RLF is reduced to existence and uniqueness of solutions of the continuity equation as in Definition \ref{def:solCE}. Notice that in doing so we transformed the non-linear problem of studying solutions of the ODE \eqref{eq:ODE} into the linear one of studying solutions of the continuity equation \eqref{eq:CE} (in line with the classical use of Young's measures and with Kantorovich's approach to Optimal Transport).

The continuity equation is tractable on $\RCD$ spaces under suitable regularity assumptions on the vector fields:
\begin{proposition}[Existence and uniqueness for the continuity equation]\label{prop:exuniqCE} Let $(\X,\sfd,\mm)$ be $\RCD(K,\infty)$ and $(v_t)\in L^2([0,1],L^2(T\X))$ be with
\begin{equation}
\label{eq:assvt}
{\sf N}\big((v_t)\big):=\int_0^1\||v_t|\|_{L^\infty}+\|\div( v_t)\|_{L^2}+\|\div (v_t)\|_{L^\infty}+\||\nabla v_t|_{\HS}\|_{L^2}\,\d t<\infty.
\end{equation}
Then for every $\mu=\rho\mm\in\pr_\bc(\X)$ there exists a unique solution $(\mu_t)=(\rho_t\mm)$ of the continuity equation \eqref{eq:CE} with $\mu_0=\mu$ and it satisfies
\begin{equation}
\label{eq:expb}
\|\rho_t\|_{L^\infty}\leq \|\rho\|_{L^\infty}e^{\int_0^t\|(\div(v_r))^-\|_{L^\infty}\,\d r}\qquad\forall t\in[0,1].
\end{equation}
\end{proposition}
\begin{idea}\ \\
{\sc Existence} For $\eps>0$ we consider the viscous approximation 
\begin{equation}
\label{eq:visc}
\partial_t\rho_t+\div(v_t\rho_t)=\eps\Delta\rho_t.
\end{equation}
A rather standard  application of Lax-Milgram theorem (in its generalization given by J.\,L.\, Lions -  see \cite[Theorem III.2.1, Corollary III.2.3]{Show97}) ensures the existence of a solution $(\rho^\eps_t)$ of \eqref{eq:visc}  in the space $L^2([0,1],W^{1,2}(\X))$ with $\rho^\eps_0=\rho$ (in the suitable weak sense). For $u:\R\to\R$ convex smooth we consider the `pressure' ${\sf p}(z):=zu'(z)-u(z)$ and   the (formal, but justifiable) computation
\begin{equation}
\label{eq:formalu}
\begin{split}
\partial_t\int u\circ\rho_t^\eps\,\d\mm&=\int u'\circ\rho^\eps_t\partial_t\rho^\eps_t\,\d\mm=\int\d(u'\circ\rho^\eps_t)(v_t)\rho^\eps_t-\eps \la\d(u'\circ\rho^\eps_t),\d\rho^\eps_t\ra\,\d\mm\\
&=\int \d({\sf p}\circ\rho^\eps_t)(v_t)-\eps u''\circ\rho^\eps_t|\d\rho^\eps_t|^2\,\d\mm\leq -\int {\sf p}\circ\rho^\eps_t\,\div(v_t)\,\d\mm.
\end{split}
\end{equation}
The choice $u(z):=z^p$ gives ${\sf p}(z)=(p-1)z^p$ thus by Gronwall's lemma we get $\|\rho^\eps_t\|_{L^p}\leq \|\rho\|_{L^p}e^{(1-\frac1p)\int_0^t\|(\div(v_r))^-\|_{L^\infty}\,\d r}$ for every $t$.  This is sufficient both to get weak compactness of the solutions $(\rho^\eps_t)$, say in $L^2([0,1],L^2(\X))$, and the estimate \eqref{eq:expb} for any weak limit $(\rho_t)$. 

We claim that $\rho_t$ is a probability density for a.e.\ $t$. If $\mm(\X)<\infty$ the mass preservation is trivial because we can  test \eqref{eq:visc} with the function identically equal to 1 (in the general case the growth assumptions on the $v_t$'s play a role). For non-negativity  pick $u$ identically 0 on $\R^+$ and positive in $(-\infty,0)$ in \eqref{eq:formalu}  to deduce that $\int u(\rho^\eps_t)\,\d\mm\leq \int u(\rho)\,\d\mm=0$.

Now the fact that   $(\rho_t)$ solves the continuity equation easily follows passing to the limit in \eqref{eq:visc}: here it is convenient to test \eqref{eq:visc} with functions in the domain of the Laplacian to see that the RHS vanishes in the limit.

\noindent{\sc Uniqueness} Let $(\mu^1_t),(\mu^2_t)$ be two solutions with the same initial datum and set $\rho_t\mm=\mu^1_t-\mu^2_t$. Then  $(\rho_t)\subset L^\infty_t(L^\infty_x)$ and for any $f\in W^{1,2}(\X)$ the function $t\mapsto \int f\rho_t\,\d\mm$ is absolutely continuous with derivative $\int \d f(v_t)\rho_t\,\d\mm$, i.e.\ $(\rho_t)$ solves
\begin{equation}
\label{eq:CEdens}
\partial_t\rho_t+\div(v_t\rho_t)=0
\end{equation}
in a natural sense. Clearly, $\rho_0\equiv 0$: our goal is to prove that $\rho_t\equiv 0$ for every $t\in[0,1]$.

This is achieved by enforcing a non-linearity in the linear equation: one proves that 
\begin{equation}
\label{eq:CEb}
\partial_t\beta(\rho_t)+\div(v_t\beta(\rho_t))=\div(v_t)(\beta(\rho_t)-\rho_t\beta'(\rho_t))\qquad\forall \beta\in C^1(\R)
\end{equation}
(notice that in the smooth category, this follows from \eqref{eq:CEdens} by direct computation). If \eqref{eq:CEb} holds, picking $\beta(z)=z^2$   we get $\partial_t\int\rho_t^2\,\d\mm\leq {\|\div(v_t)\|_\infty}\int\rho_t^2\,\d\mm$, thus by Gronwall's lemma we conclude. To prove \eqref{eq:CEb} we regularize $(\rho_t)$ by defining $\rho^\alpha_t:=\h_\alpha\rho_t$ for $\alpha>0$. Then from   \eqref{eq:CEdens} it follows that  
\begin{equation}
\label{eq:CEreg}
\partial_t\rho^\alpha_t+\div(v_t\rho^\alpha_t)=\mathcal C_\alpha(v_t,\rho_t)\qquad\text{ where }\qquad \mathcal C_\alpha(v,\rho):=\div(v\h_\alpha\rho)-\h_\alpha^*(\div(v\rho)).
\end{equation}
Here we wrote $\h_\alpha^*(\div(v\rho))$ using the adjoint semigroup because $\div(v\rho)$ is not a function under our assumptions and thus must be treated via integration by parts. This is in line with the PDE in \eqref{eq:CEreg} whose solutions must be interpreted `distributionally', as for \eqref{eq:CEdens} above. The additional regularity on $(\rho^\alpha_t)$ allows rather easily to justify from \eqref{eq:CEreg} that
\begin{equation}
\label{eq:CEregb}
\partial_t\beta(\rho^\alpha_t)+\div(v_t\beta(\rho^\alpha_t))=\div(v_t)(\beta(\rho^\alpha_t)-\rho_t\beta'(\rho^\alpha_t))+\beta'(\rho^\alpha_t)\,\mathcal C_\alpha(v_t,\rho_t)
\end{equation}
and thus letting $\alpha\downarrow0$ we can deduce \eqref{eq:CEb} from \eqref{eq:CEregb} provided we can prove that
\begin{equation}
\label{eq:commint}
\lim_{\alpha \downarrow0}\int_0^1\|\mathcal C_\alpha(v_t,\rho_t)\|_{L^1}\,\d t=0.
\end{equation}
Notice that for $v\in D(\div)$ bounded and $\rho\in L^\infty\cap W^{1,2}(\X)$ it is trivial to check (using \eqref{eq:leibdiv} and that $\d\h_\alpha\rho\to\d\rho$ in $L^2(T^*\X)$) that $\|\mathcal C_\alpha(v,\rho)\|_{L^1}\to 0$ as $\alpha\downarrow0$. Thus \eqref{eq:commint} follows from \eqref{eq:assvt} by a density and equicontinuity argument if we prove that
\begin{equation}
\label{eq:commest}
\|\mathcal C_\alpha(v,\rho)\|_{L^1}\leq c\|\rho\|_{L^2}(\|\nabla v\|_{L^2}+\|\div (v)\|_{L^2})\qquad\forall \alpha\in(0,1)
\end{equation}
for some constant $c>0$. Let us thus prove \eqref{eq:commest} in the simplified case $\div(v)=0$. We have
\[
\begin{split}
\mathcal C_\alpha(v,\rho)&=\int_0^\alpha\partial_s \h^*_{\alpha-s}(\div(v\h_s\rho))\,\d s=\int_0^\alpha -\Delta^*\h_{\alpha-s}(\div(v\h_s\rho))+\h_{\alpha-s}(\div(v\Delta\h_\alpha\rho))\,\d s
\end{split}
\]
thus for  $f\in L^\infty(\X)$, using $\div(v)=0$, the definition of $\nabla v$ (with $h=1$) and \eqref{eq:dgrad}  we get
\[
\begin{split}
\la \mathcal C_\alpha(v,\rho),f\ra&=\iint_0^\alpha- \Delta\h_{\alpha-s}f\la v,\nabla\h_s\rho\ra -\la v,\nabla\h_{\alpha-s}f\ra \Delta\h_s\rho\,\d s\,\d\mm\\
&=\iint_0^\alpha \nabla v(\nabla\h_{\alpha-s}f,\nabla\h_\alpha\rho) + \nabla v(\nabla\h_\alpha\rho,\nabla\h_{\alpha-s}f) +\d(\la\d\h_{\alpha-s}f,\d\h_s\rho\ra)(v)\,\d s\,\d\mm.
\end{split}
\]
Since, again, $\div(v)=0$ the last term vanishes and we have
\[
\begin{split}
|\la \mathcal C_\alpha(v,\rho),f\ra|&\leq 2\||\nabla v|_{\HS}\|_{L^2}\int_0^\alpha \||\nabla\h_{\alpha-s}f |\|_{L^\infty} \||\nabla\h_{s}\rho |\|_{L^2}\,\d s\\
\text{(by \eqref{eq:linftylip} and \eqref{eq:aprioriheat})}\quad&\leq c\||\nabla v|_{\HS}\|_{L^2}\|\rho\|_{L^2}\|f\|_{L^\infty}\int_0^\alpha\tfrac1{\sqrt{s(\alpha-s)}}\,\d s
\end{split}
\]
for $c=c(K)$, and since the last integral is equal to $\int_0^1\tfrac1{\sqrt {s(1-s)}}\,\d s<\infty$ we proved that the norm of $\mathcal C_\alpha(v,\rho)$ as element of $(L^\infty(\X))'$ is bounded by the RHS in \eqref{eq:commest}. 
\end{idea}
Collecting Proposition \ref{prop:exuniqCE}, Lemma \ref{le:ponte} and Lemma \ref{le:measth} we  get:
\begin{theorem}[Existence and uniqueness of Regular Lagrangian Flows]\label{thm:RLF} Let $(\X,\sfd,\mm)$ be $\RCD(K,\infty)$ and $(v_t)\in L^2([0,1],L^2(T\X))$ be satisfying \eqref{eq:assvt}.

Then $(v_t)$ admits a unique Regular Lagrangian Flow.
\end{theorem}

\subsection{Bibliographical notes}

{\footnotesize The idea of basing differential calculus on non-smooth structures on the concept of $L^\infty$ modules goes back to   \cite{Sauvageot89,Sauvageot90} where it appeared in the context of Dirichlet forms. Later, \cite{Weaver01} used it in the setting of metric geometry to define differentials of Lipschitz functions in general metric measure spaces. The approach described in Section \ref{se:firstordercalc} is from \cite{Gigli14} and has been surveyed in  \cite{Gigli17}, \cite{GP19}: the main technical difference w.r.t.\ the approach in \cite{Weaver01}  is in the concept of pointwise norm. This sort of structure was studied earlier, for other reasons more related to functional analysis and probability theory, in \cite{HLR91}. A different, but equivalent (see \cite[Corollary 2.5.2]{Gigli14} and its proof), construction of cotangent module was made in \cite{Cheeger00} on doubling spaces supporting a local, weak Poincar\'e inequality: there a key result is the existence of suitable Lipschitz charts, so that the differential of a Lipschitz function can be defined in coordinates.

A crucial inspiration for Section \ref{se:secondordercalc}, and in particular for the proof of the key Lemma \ref{le:sibo}, is Bakry's work \cite{Bakry83}; the space of test functions was considered  in \cite{Savare13}, who was  adapting $\Gamma$-calculus techniques to the $\RCD$ setting.  Proposition \ref{prop:test} is also  from  \cite{Savare13} (for more about quasi-regolarity and the transfer method see \cite[Chapters IV, VI]{MaRockner92}). See also \cite{Sturm14} for related results under additional technical assumptions. The rest of Section  \ref{se:secondordercalc} comes from  \cite{Gigli14}, with the exception of the content of Remark \ref{re:honda} that is the result of a conversation between Braun and Honda of which I've been kindly informed. 

For more recent advances about second order calculus on nonsmooth spaces see \cite{Han14}, \cite{Br21}, \cite{Braun22}, \cite{HS22} and for more about the structure of $L^\infty$-modules see \cite{LP18}, \cite{P19}, \cite{DMLP21}, \cite{LPV22}, \cite{GLP22}.

The first studies about flows of Sobolev vector fields have been carried out   in  \cite{DiPerna-Lions89} in the Euclidean setting. Their approach has been extended  in \cite{Ambrosio04} to cover $\BV$ vector fields in $\R^d$. The content of Section \ref{se:flowrcd}, which generalizes these latter studies to the $\RCD$ category, comes from  \cite{Ambrosio-Trevisan14}, see also the survey \cite{AT15}. I haven't discussed at all the regularity of Regular Lagrangian Flows: formally linearizing the equation $\partial_t F_t=v_t(F_t)$ one gets 
\[
\partial_t\d F_t=\nabla v_t\circ F_t\,\d F_t
\]
that suggests that $\log(|\d F_t|)$ has the same integrability of $|\nabla v_t|$. Turning this to an actual estimate is tricky, though, because the limit $\log'(z)=\tfrac1z\to0$ as $z\to+\infty$ entails that an estimate on the integral of $\log(|\d F_t|)$ tells nothing about the regularity of $F_t$. Regularity results have been obtained in \cite{CDL08} (see also references therein) in the Euclidean case and in \cite{BS18} in the $\RCD$ setting, where they played a crucial role in proving constancy of dimension for $\RCD$ spaces (see also the more recent \cite{BDS21} for improved estimates).

}

\section{Examples of  stability/semicontinuity results}
It is well known that the lower semicontinuity of a Dirichlet form is equivalent to the closure of the underlying differentiation operator. In our setting, lower semicontinuity of the Cheeger energy is in place also along a converging sequence of $\CD(K,\infty)$ spaces (recall Theorem \ref{thm:convslen}), thus we can expect a suitable closure property of the differential even along such sequence of spaces: as we shall see in item \ref{it:cldiff} below, this actually happens. Even more so, given that we introduced all the other  operators (divergence, Laplacian,  covariant derivative, flows...)  building on top of the differential, we might expect that the closure of the latter implies some stability of the formers.

In this section we prove that this is the case  and give some example. Without exceptions, we shall be given a sequence $(\X_n,\sfd_n,\mm_n)$ of  $\RCD(K,\infty)$ spaces mGH-converging to a limit space $(\X_\infty,\sfd_\infty,\mm_\infty)$, all being normalized, via some isometric inclusions $\iota_n$, $n\in\N\cup\{\infty\}$, in a bigger space $\Y$. We shall identify the $\X_n$'s with their images in $\Y$.

As usual, none of the results presented really depend on the fact that the spaces are normalized, but restricting to this case makes the presentation easier.

\subsection{Convergence of tensor fields: definitions and basic properties}\label{se:defconvtens}
Theorem \ref{thm:convslen} tells that $f_n\stackrel{L^2}\weakto f_\infty$ implies $\ch_\infty(f_\infty)\leq \limi_n\ch_n(f_n)$ and that for any $f_\infty$ there is a sequence $(f_n)$ strongly $L^2$-converging to it with $\ch_\infty(f_\infty)\geq \lims_n\ch_n(f_n)$. It is thus natural to interpret this latter inequality by thinking that $(\d f_n)$ is converging to $\d f_\infty$ in some strong $L^2$ sense, and then use sequences of this kind to `test' convergence of other tensors. 

The key result that allows to turn this idea into practice is the following (notice that a posteriori the assumption on uniform Lipschitz and $L^\infty$ bounds can be removed - see items \ref{it:lscnorm}, \ref{it:cldiff} below):
\begin{lemma}\label{le:convnorme}
Let $f_n\stackrel{L^2}\to f_\infty$ be so that $\sup_n\Lip(f_n)+\|f_n\|_{L^\infty}<\infty$ and $\ch_n(f_n)\to\ch_\infty(f_\infty)$. 

Then $|\d f_n|\stackrel{L^2}\to |\d f_\infty|$.
\end{lemma}
\begin{idea}
Passing to a non-relabeled subsequence we can assume that  $|\d f_n|\stackrel{L^2}\weakto G$ for some $G\in L^2(\X_\infty)$. To conclude it is then clearly enough to prove that $|\d f_\infty|\leq G$ $\mm_\infty$-a.e..  

Let $(g_n)$ be  with $\sup_n\ch_n(g_n)<\infty$; then for $\eps\in\R$,  the inequality $\limi_n\ch_n(f_n+\eps g_n)\geq\ch_\infty(f_\infty+\eps g_\infty)$ and the arbitrariness of $\eps$  easily imply 
\begin{equation}
\label{eq:fnrecgncaso}
\int \la\d f_n,\d g_n\ra\,\d\mm_n\to \int \la\d f_\infty,\d g_\infty\ra\,\d\mm_\infty.
\end{equation}
Now let $(\psi_n)$ be another uniformly bounded and uniformly Lipschitz recovery sequence (notice that by Corollary \ref{cor:stabhagain}, the weak maximum principle and the estimate \eqref{eq:linftylip}, the choice $\psi_n:=\h_{n,t}(\psi)$ for some $t>0$ and $\psi\in\Lip_\bs(\Y)$ does the job). Then both $(\psi_nf_n)$ and $(|f_n|^2)$ have uniformly bounded Cheeger energies, thus by \eqref{eq:fnrecgncaso} we can pass to the limit in the identity $\int\psi_n|\d f_n|^2\,\d\mm_n=\int\la\d f_n,\d(\psi_n f_n)\ra-\la\d \psi_n,\d\tfrac{|f_n|^2}2\ra\,\d\mm_n$ to conclude (using also that $\sup_n\||\d f_n|^2\|_{L^2}\leq\sup_n\Lip(f_n)^2<\infty$ and recalling \eqref{eq:equivl2conv}) that $|\d f_n|^2\weakto |\d f_\infty|^2$. By polarization we then get
\begin{equation}
\label{eq:recrec}
\la\d f_n,\d\psi_n\ra\stackrel{L^2}\weakto \la\d f_\infty,\d\psi_\infty\ra. 
\end{equation}
Thus for $g\in\Lip_\bs(\Y)$, we can pass to the limit in the identity $\int \psi_n\la \d f_n,\d_ng\ra\,\d\mm_n=\int \la\d f_n,\d(g\psi_n)\ra- g\la\d f_n,\d\psi_n\ra\,\d\mm_n$ (here $\d_ng\in L^2(T^*\X_n)$ is the differential of $g$ as element of $W^{1,2}(\X_n)$) using \eqref{eq:fnrecgncaso} and \eqref{eq:recrec} to conclude that $\int \psi_n\la \d f_n,\d_ng\ra\,\d\mm_n\to \int \psi_\infty\la \d f_\infty,\d_\infty g\ra\,\d\mm_\infty$. Then the arbitrariness of $(\psi_n)$ and \eqref{eq:equivl2conv} give
\begin{equation}
\label{eq:fatto}
\la \d f_n,\d_ng\ra\stackrel{L^2}\weakto \la \d f_\infty,\d_\infty g\ra.
\end{equation}
Now approximate  $\frac{\d f_\infty}{|\d f_\infty|}\in L^2(T^*\X_\infty)$ with 1-forms of the kind $\sum_i\varphi_i\d g_i$ and then recall  \eqref{eq:perdenslip} to find functions  $\varphi_i^k,g_i^k\in \Lip_\bs(\Y)$ for $k\in\N$, $i=0,\ldots, N_k$ with  
\begin{equation}
\label{eq:diciamochece}
\varphi_i^k\geq0\quad\text{and}\qquad
\begin{array}{lll}
 \sum_i\varphi_{i}^k\lipa( g^k_i)&\to\ 1,&\quad\text{in $L^2(\X_\infty)$}\\  
 \sum_i\varphi_{i}^k\la\d f_\infty,\d g^k_i\ra&\to \ |\d f_\infty|&\quad\text{in $L^1(\X_\infty)$}
 \end{array}\quad \text{as $\quad k\to\infty$}.
\end{equation}
Then for any $\psi\in\Lip_\bs(\Y)$ non-negative and any $i,k$ we have
\[
\begin{split}
\int\psi \varphi_i^k\la \d f_\infty,\d_\infty g _i^k\ra\,\d\mm_\infty&\stackrel{\eqref{eq:fatto}}=\lim_n\int\psi  \varphi_i^k\la \d f_n,\d_ng_i^k \ra\,\d\mm_n\\
&\stackrel{\phantom{\eqref{eq:diciamocheevero}}}\leq\limi_n\int\psi | \d f_n|  \varphi_i^k\lipa(g_i^k) \,\d\mm_n\leq \int\psi \,G \, \varphi_i^k \lipa(g_i^k) \,\d\mm_\infty,
\end{split}
\]
having used the upper semicontinuity of $\lipa(g_i^k)$ in the last step. Summing over $i$ and then letting $k\to\infty$, recalling \eqref{eq:diciamochece} we conclude that $\int\psi |\d f_\infty|\,\d\mm_\infty\leq \int\psi G\,\d\mm_\infty$ which, by the arbitrariness of $\psi$,   gives the desired conclusion $|\d f_\infty|\leq G$ $\mm_\infty$-a.e..
\end{idea}

We now  give the following definitions:
\begin{definition}[Test sequences of functions and tensors]
We say that $n\mapsto f_n\in\test(\X_n)$, $n\in\N\cup\{\infty\}$, is a test sequence if  $\sup_n\|f_n\|_{L^\infty}+\Lip(f_n)+\|\Delta f_n\|_{L^\infty}<\infty$ and
\begin{equation}
\label{eq:deftestseq}
\begin{array}{rl}
f_n&\stackrel{L^2}\to\  f_\infty\\
\Delta f_n&\stackrel{L^2}\to\ \Delta f_\infty
\end{array}
\qquad\text{ and }\qquad
\begin{array}{rl}
&\ch_n(f_n)\to  \ch_\infty(f_\infty)\\
&\sup_n\ch_n(\Delta f_n)<\infty.
\end{array}
\end{equation}
Then, $n\mapsto z_n\in L^2(T\X_n)$ is a test sequence of vectors if it is of the form $z_n=\sum_{i=1}^mf_{i,n}\nabla g_{i,n}$ for $(f_{i,n}),(g_{i,n})$ test sequences.

More generally, for  $d> 1$ and $n\mapsto z_n\in L^2(T^{\otimes d}\X_n)$, $n\in\N\cup\{\infty\}$ we say that   $(z_n)$ is a test sequence of tensors if it is of the form $z_n=\sum_{i=1}^mf_{i,n}\nabla g_{1,i,n}\otimes\cdots\otimes\nabla g_{d,i,n}$ for $(g_{j,i,n})$ and $(f_{i,n})$ test sequences.
\end{definition}
\begin{definition}[$L^2$-convergence of tensors]\label{def:convtens} We say that $n\mapsto v_n\in L^2(T^{\otimes d}\X_n)$ converges weakly in $L^2$ to $v_\infty\in L^2(T^{\otimes d}\X_\infty)$ provided 
\[
\sup_n\|v_n\|_{L^2(T^{\otimes d}\X_n)}<\infty\qquad\text{and}\qquad \lim_n\int\la v_n,z_n\ra\,\d\mm_n=\int\la v_\infty,z_\infty\ra\,\d\mm_\infty
\]
for any test sequence $(z_n)$ of $d$-tensors. 

The convergence is strong if moreover $\|v_n\|_{L^2(T^{\otimes d}\X_n)}\to \|v_\infty\|_{L^2(T^{\otimes d}\X_\infty)}$.
\end{definition}

Lemma \ref{le:convnorme} easily implies the following properties of test sequences:
\begin{enumerate}[label=(\alph*)]
\item\label{it:prodtest2}  $(f_n),(g_n)$  test sequences of functions $\quad\Rightarrow\quad$ so is $(f_ng_n)$
\item $(f_n)$ test sequence of functions  $\quad\Rightarrow\quad$ $\d f_n\to\d f_\infty$ strongly as tensors
\item\label{it:convnormtest} $(z_n)$ test sequence of tensors  $\quad\Rightarrow\quad$ $|z_n|\to |z_\infty|$ as $L^2$ functions
\item $(z_n)$ test sequence of tensors  $\quad\Rightarrow\quad$ $z_n\to z_\infty$ strongly as tensors
\item\label{it:convdiv2} $(z_n)$ test sequence of vectors  $\quad\Rightarrow\quad$ $\div(z_n)\to \div(z_\infty)$ as $L^2$ functions\\
{\footnotesize (All these follow by polarization and other basic algebraic operations if we prove that for $(f_n),(g_n)$ test we have $f_n\d g_n\to f_\infty\d g_\infty$ strongly as tensors. Convergence of norms is a direct consequence of Lemma \ref{le:convnorme} which also implies, by polarization and uniform $L^\infty$ bounds, that for $(\tilde f_n),(\tilde g_n)$ test sequence of functions we have $\int f_n\tilde f_n\la\d g_n,\d\tilde g_n\ra\,\d\mm_n\to\int f_\infty\tilde f_\infty\la\d g_\infty,\d\tilde g_\infty\ra\,\d\mm_\infty$)}
\end{enumerate}

\medskip

\noindent It is also easy to establish that
\begin{enumerate}[label=(\alph*)]\setcounter{enumi}{5}
\item\label{it:testd} `Density of test objects':  $\{v_\infty:(v_n)\text{ is test}\}\subset L^2(T^{\otimes d}\X) $ is dense in $L^2(T^{\otimes d}\X)$ 

{\footnotesize {\vspace{-5pt}}(it suffices to prove that $\{f_\infty: (f_n)\text{ is test}\}$ is dense in $W^{1,2}(\X)$. For this let $f\in\Lip_\bs(\Y)$, recall \eqref{eq:regheat}, put   $f_n:=\h_{n,\eta}(f)\in\test(\X_n)$ and notice that $(f_n)$ is test: the uniform estimates follow by the maximum principle and the Bakry-\'Emery estimates, while the convergences in  \eqref{eq:deftestseq} follow from Corollary \ref{cor:stabhagain}. The arbitrariness of $\eta$ and the density of $\Lip_\bs(\Y)$ in $W^{1,2}(\X_\infty)$, recall \eqref{eq:w12reflsep}, give the claim).}

\end{enumerate}

\noindent Now the desired/expected properties of convergence of tensors easily follow:
\begin{enumerate}[label=(\roman*)]

\item {\sc Uniqueness of limits}: weak/strong limits of tensors are unique.{\footnotesize {\vspace{-5pt}}

 (from  item \ref{it:testd} above).}

\item\label{it:lscnorm}  $v_n\weakto v_\infty$ implies $|v_\infty|\leq G$ $\mm_\infty$-a.e.\ for any $L^2$-weak limit $G$ of $n\mapsto |v_n|$.

{\footnotesize {\vspace{-5pt}}(Let $(f_n),(z_n)$ be test sequences of functions and tensors, with $f_n\geq0$ and notice that $(f_nz_n)$ is also test (by item \ref{it:prodtest2}). Thus by definition of weak convergence and item \ref{it:convnormtest} we get 
\[
\begin{split}
 \int f_\infty\la v_\infty,z_\infty\ra-\tfrac12f_\infty|z_\infty|^2\,\d\mm_\infty&=\lim_n\int f_n\la v_n,z_n\ra-\tfrac12f_n|z_n|^2\,\d\mm_n\\
 &\leq\limi_n\int f_n|v_n|\,|z_n|-\tfrac12f_n|z_n|^2\,\d\mm_n=\int f_\infty G|z_\infty|-\tfrac12f_\infty|z_\infty|^2\,\d\mm_\infty.
 \end{split}
\]
The RHS is $\leq \int \tfrac12f_\infty G^2\,\d\mm_\infty$ and the sup of the LHS over $(z_n)$ equals   $\int \tfrac12f_\infty |v_\infty|^2\,\d\mm_\infty$. The conclusion follows from the arbitrariness of $(f_n)$.)}

\item\label{it:strpoint} $v_n\to v_\infty$ implies $|v_n|\to |v_\infty|$. 

{\footnotesize {\vspace{-5pt}} (Direct consequence of   item \ref{it:lscnorm} above.)}

\item\label{it:combinate} $v_n\weakto v_\infty$ and $w_n\to w_\infty$ imply $\la v_n,w_n\ra\mm_n\weakto \la v_\infty,w_\infty\ra\mm_\infty$ as measures.

{\footnotesize {\vspace{-5pt}}(passing to the limit in $\|w_n+tv_n\|^2_{L^2}=\|w_n\|^2_{L^2}+2t\int \la v_n,w_n\ra\,\d\mm_n+t^2\|v_n\|^2_{L^2}$ using  $w_n+tv_n\weakto w_\infty+tv_\infty$ and item \ref{it:lscnorm} we get $\limi_n t\int \la v_n,w_n\ra\,\d\mm_n\geq t \int \la v_\infty,w_\infty\ra\,\d\mm_\infty-2Ct^2$ for $C:=\sup_n\|v_n\|^2_{L^2}$. Being this true for any $t\in\R$ we conclude that $ \int\la v_n,w_n\ra\,\d\mm_n\to  \int\la v_\infty,w_\infty\ra\,\d\mm_\infty$. To localize this information  replace $(v_n)$ with $(f_nv_n)$ for $(f_n)$ arbitrary test sequence).}

\item\label{it:strlin} {\sc Linearity of strong conv.}: $v_n\to v_\infty$ and $w_n\to w_\infty$ imply $  v_n+  w_n\to  v_\infty+  w_\infty$.   

{\footnotesize {\vspace{-5pt}}(convergence of norms follows from $|  v_n+  w_n|^2=|v_n|^2+2\la v_n,w_n\ra+|w_n|^2$  using item \ref{it:combinate} above).}

\item\label{it:wcomp} {\sc Weak compactness}: if $\sup_n\|v_n\|_{L^2}<\infty$, then a subsequence is weakly converging. 

{\footnotesize {\vspace{-5pt}}(item \ref{it:convnormtest}  tell that for $(w_n),(w_n')$ test we have $\|w_n-w_n'\|_{L^2}\to \|w_\infty-w_\infty'\|_{L^2}$, thus by separability  it suffices to test weak convergence against a well chosen countable number of test sequences. The conclusion follows by standard means via diagonalization).}
\item\label{it:cldiff} {\sc\underline{Closure of the differential}}: if $f_n\weakto f_\infty $  with $\sup_n\ch_n(f_n)<\infty$, then $\d f_n\weakto \d f_\infty$. If  $\ch_n(f_n)\to \ch_\infty(f_\infty)$, then  $\d f_n\to \d f_\infty$. 

{\footnotesize {\vspace{-5pt}} (Theorem \ref{thm:convslen} gives $\ch_\infty(f_\infty)<\infty$ and by weak compactness a subsequence of $(\d f_n)$ has a weak limit. To identify such weak limit with $\d f_\infty$ pass to  the limit in $\int \d f_n(v_n)\,\d\mm_n=-\int f_n\div(v_n)\,\d\mm_n$ for $(v_n)$ test recalling items \ref{it:convdiv2}, \ref{it:testd}).}
\item\label{it:convHess} {\sc Closure of Hessian}: if $f_n\to f_\infty$ with $\ch_n(f_n)\to \ch_\infty(f_\infty)$ and $\sup_n\|{\rm Hess}f_n\|_{L^2}<\infty$, then $f_\infty\in W^{2,2}(\X_\infty)$ and ${\rm Hess}f_n\weakto {\rm Hess}f_\infty$. 

{\footnotesize {\vspace{-5pt}}(for $(g_n)$ test by \eqref{eq:ddf} we have $\sup_n\ch_n(|\d g_n|^2)<\infty$, thus using what already proved   we can pass to the limit in 
\[
\int h_n{\rm Hess}(f_n)(\nabla g_n,\nabla g_n)\,\d\mm_n=\int- \la\d f_n,\d g_n\ra\div(h_n\d g_n)-h_n\la \d f_n,\d\tfrac{|\d g_n|^2}2\ra\,\d\mm_n
\]
using item \ref{it:testd} and the weak compactness  of $({\rm Hess}f_n)$ to conclude. Notice that it is important to know that $\d f_n\to\d f_\infty$ strongly, due to the  weak convergence  $\d|\d g_n|^2\weakto \d|\d g_\infty|^2$).}
\item\label{it:closcovder} {\sc Closure of covariant derivative}: let  $v_n\to v_\infty$ be  with $\sup_n\|\nabla v_n\|_{L^2}<\infty$. Then $v_\infty\in W^{1,2}_C(T\X_\infty)$ and $\nabla v_n\weakto \nabla v_\infty$. 

{\footnotesize {\vspace{-5pt}}(for $(f_n),(g_n),(h_n)$ test sequences of functions we have $\Hess f_n\weakto \Hess f_\infty$ and $\nabla g_n\to \nabla g_\infty$ by the last two items and also $v_n\otimes \nabla g_n\to v_\infty\otimes \nabla g_\infty$  (by convergence of pointwise norms and uniform $L^\infty$ bounds on $|\d g_n|$). Thus using also item \ref{it:convdiv2} we can pass to the limit in the definition \eqref{eq:defcov} of covariant derivative of $v_n$ to deduce that any weak limit of $\nabla v_n$, that exist by weak compactness, must coincide with $\nabla v_\infty$. As above, here it matters that $v_n\to v_\infty$ strongly, due to the weak convergence of Hessians).}

\end{enumerate}
If we define the \emph{Covariant energy} $\mathcal E_C:L^2(T\X)\to[0,\infty]$ as $\mathcal E_C(v):=\tfrac12\|\nabla v\|^2_{L^2}$ if $v\in W^{1,2}_C(T\X)$ and $\mathcal E_C(v)=\infty$ otherwise, then  this last point implies 
\[
\glimi_{n\to\infty}\mathcal E_{C,n}\geq \mathcal E_{C,\infty}\qquad\text{w.r.t.\ strong $L^2$-convergence of vector fields.}
\]
A similar statement holds for the Hessian, provide we require strong $L^2$ convergence of both the functions and their differentials.

\begin{example}[Strong compactness for tensor fields]\label{eq:compvect}{\rm Let $\X_n$ be the circle $S^1$ of total length $\tfrac1n$ with the normalized volume measure: these are $\RCD(0,\infty)$ spaces converging to the one point space/manifold $\X_\infty$ (if thinking at a one point manifold is confusing, just replace $\X_n$ with $\X_n\times M$ for some fixed compact manifold). Let $v_n\in L^2(T\X_n)$ be the (only, up to a change of signs) unit parallel vector field. Then clearly $\nabla v_n\equiv 0$ and $(v_n)$ weakly converges to the 0 vector field $v_\infty$. However, since $\|v_\infty\|_{L^2}=0<1=\|v_n\|_{L^2}$, the convergence is not strong. This shows that a uniform bound on the $L^2$-norm of the covariant derivative is not sufficient to obtain strong $L^2$-compactness, unlike the case of functions discussed in Theorem \ref{thm:convslen}.

Here compactness is lost due to the drop in the dimension that, in some sense, prevents the existence of sufficiently many $L^2$-vector fields in the limit space. It can be shown compactness is restored if one assumes that, in a suitable sense made rigorous by the notion of non-collapsed convergence, there is no drop in the dimension (see \cite[Corollary 6.16]{Honda14}).\fr
}\end{example}
This list of properties emphasise those aspects of the theory of converging tensors that more closely resemble those valid for a fixed space. In some sense, no other property trivially generalizes to the current setting. A notable one that  is  missing  is Mazur's lemma, and this makes it hard to  promote weak to strong convergence in the statements above, e.g.:
\begin{open}[$\Gamma$-convergence of `covariant energy']{\rm
Let $f\in \test(\X_\infty)$. Can we find $v_n\in W^{1,2}_C(T\X_n)$ strongly $L^2$-converging to $\nabla f$ so that  $\nabla v_n\to{\rm Hess}f$? 

Notice that in \cite[Remark 1.10.6]{AH16} it is given an example of a (collapsing) sequence of surfaces with  curvature $\geq1$ converging to a (weighted) interval and functions $f_n\to f_\infty$ with $\Delta f_n\to\Delta f_\infty$ but with $\||{\rm Hess}f_\infty|_{\HS}\|_{L^2}<\limi_n\||{\rm Hess}f_n|_{\HS}\|_{L^2}$. In this example, the functions $f_n$ are eigenfunctions relative to the  first positive eigenvalue of $\Delta$ and one also has $\d|\d f_n|^2\to\d|\d f_\infty|^2$.

More explicitly, the metrics $g_n$ are given by spherical subspensions with cross section $\tfrac1nS^1$ and the functions $f_n$ are the projection on the underlying interval $[0,\pi]$. Then  it holds ${\rm Hess}f_n=f_ng_n$. Here the $f_n$'s are strongly converging, but the $g_n$'s are not, because $|g_n|_{\HS}^2\equiv2$ for every $n$, whereas $|g_\infty|_{\HS}^2\equiv1$\footnote{some readers may find easier to visualize the conical case (in which however compactness is lost and convergences should  in the appropriate `local'/`pointed' sense). Let $S_\alpha$ be the manifold $S^1$ of total length $\alpha$ and $(\X_\alpha,\sfd_\alpha)$ the cone built over it equipped with the 2 dimensional Hausdorff measure. Then (\cite{Ketterer13}) for $\alpha\in(0,2\pi]$ the $\X_\alpha$'s are $\RCD(0,2)$ spaces and the function $f_\alpha:=\tfrac12\sfd_\alpha^2(\cdot,o_\alpha)$, where $o_\alpha$ is the tip of the cone, satisfies $\Delta f_\alpha\equiv 2$ (this is the base  of the `volume cone to metric cone principle' in lower Ricci bounds). $\X_\alpha\setminus\{o_\alpha\}$ is a smooth manifold, say  with metric tensor $g_\alpha$; on such manifold $f_\alpha$ is smooth and satisfies ${\rm Hess}f_\alpha=g_\alpha$, hence $|{\rm Hess}f_\alpha|^2_{\HS}\equiv 2$ (notice that locally $\X_\alpha\setminus\{o_\alpha\}$ is  isometric to $\R^2$ with an isometry that sends $f_\alpha$ to $x\mapsto\frac12|x|^2$). As $\alpha\downarrow0$ the spaces converge to the Euclidean half line $\R^+$ equipped with the measure $x\d x$ and the functions $f_\alpha$ converge to $\R^+\ni x\mapsto f_0(x):=\frac12|x|^2$. Direct computations show that $\Delta f_0\equiv 2$, so that the $f_\alpha$'s and their Laplacian converge to $f_0,\Delta f_0$. However ${\rm Hess}f_0(x)\sim f''_0(x)=1$ (in the canonical base $\frac\d{\d x}$ of the tangent spaces) and thus $|{\rm Hess}f_0|_{\HS}^2\equiv 1$.}. 
}\fr
\end{open}
\subsection{Lower semicontinuity of  dimension}
\begin{theorem}\label{thm:lscdim}
Let $M_n$ be smooth (possibly weighted) normalized Riemannian manifolds with (Bakry-\'Emery, recall \eqref{eq:BEricci}) Ricci tensor uniformly bounded from below. Assume that they mGH-converge to a limit space that also happens to be a smooth  (possibly weighted) Riemannian manifold $M_\infty$. Then $\dim M_\infty\leq \limi_n\dim M_n$.
\end{theorem}
\begin{idea}
Recall that a finite set $v_1,\ldots,v_d$ of elements of a Hilbert space $H$ is linearly independent iff $\det(\la v_i,v_j\ra)\neq 0$. Then let $d:=\dim M_\infty$, let $v_1,\ldots,v_d\in L^2(T\X)$ be a pointwise orthonormal frame and use item  \ref{it:testd}  to find test sequences $(v_{1,n}),\ldots,(v_{d,n})$ such that $\det(\la v_{i,\infty},v_{j,\infty}\ra)\neq 0$ on a set of positive measure. Since $\det(\la v_{i,n},v_{j,n}\ra)\to \det(\la v_{i,\infty},v_{j,\infty}\ra)$ (by item \ref{it:convnormtest}) eventually we must have $\det(\la v_{i,n},v_{j,n}\ra)\neq0 $ on a set of positive measure.
\end{idea}
\subsection{Convergence of harmonic functions}\label{se:convharm}
The concepts of convergence of functions/tensors can easily be localized: for  $x_n\to x_\infty$ and $R>0$ we say that $v_n\weakto v_\infty$ (resp.\ $v_n\to v_\infty$) in $L^2(B_R(x_n))$ provided $\nchi_{B_R(x_n)}v_n\weakto \nchi_{B_R(x_\infty)}v_\infty$ (resp.\ $\nchi_{B_R(x_n)}v_n\to \nchi_{B_R(x_\infty)}v_\infty$) in the sense of Definition \ref{def:convtens}. 

Notice that this concept applies both to functions and tensors and that it makes sense even if these are not defined outside $B_R(x_n)$, in which case the object $\nchi_{B_R(x_n)}v_n$ is anyway intended to be 0 outside $B_R(x_n)$.

Now let $U$ be an open subset of an $\RCD(K,\infty)$ space $\X$. The space $W^{1,2}(U)$ is defined as the Sobolev space relative to the metric measure space $(\bar U,\sfd,\mm\restr U)$ (notice that we are not charging $\partial U$). Natural compatibility relations are in place between $W^{1,2}(U)$ and $W^{1,2}(\X)$, for instance it is not hard to check that for $f\in W^{1,2}(U)$ and $\varphi\in \Lip_\bs(\X)$ with $\supp(\varphi)\subset U$ the function $\varphi f$, intended to be 0 outside $U$, is in both $W^{1,2}(U)$ and $W^{1,2}(X)$ and 
\begin{equation}
\label{eq:samewug}
|\d f|_U=|\d(\varphi f)|_\X\qquad\mm-a.e.\ on\ \varphi^{-1}(1).
\end{equation}
Starting from this, one can then think at the differential of $f\in W^{1,2}(U)$ as an element of $L^2(T^*\X)$ that is defined only on $U$ or, more rigorously and relying on the locality of the differential, as the element of $L^2(T^*\X)$ that is 0 outside $U$ and coincides with $\d(\varphi f)$ $\mm$-a.e.\ on $\varphi^{-1}(1)$ for any $\varphi\in \Lip_\bs(\X)$ with $\supp(\varphi)\subset U$. Finally, with a bit of work (made somehow easier by the approach via test plans) one can show that
\begin{equation}
\label{eq:equivw12u}
f\in W^{1,2}(U)\quad\Leftrightarrow\quad\left\{\begin{array}{l} \forall\varphi\in \Lip_\bs(\X)\text{ with }\supp(\varphi)\subset U\text{ we have }\\
\text{$f\varphi\in W^{1,2}(\X)$ and $\int_U|f|^2+|\d f|^2\,\d\mm<\infty$} \end{array}\right.
\end{equation}
where here $|\d f|$ is defined via \eqref{eq:samewug}. 
\begin{definition}[Harmonic functions]
 $f\in W^{1,2}(U)$  is said harmonic provided 
\begin{equation}
\label{eq:defharm}
\int_U|\d f|^2\,\d\mm\leq \int_U|\d(f+\varphi)|^2\,\d\mm\qquad\forall \varphi\in W^{1,2}(\X)\ \text{ with }\supp(\varphi)\subset U.
\end{equation}
\end{definition} Writing the Euler-Lagrange equations for the minimizer and taking into account  the convexity of $\varphi\mapsto \int_U|\d(f+\varphi)|^2\,\d\mm$ we see that
\begin{equation}
\label{eq:equivharm}
f\text{ is harmonic }\qquad\Leftrightarrow\qquad \int_U\la\d f,\d\varphi\ra\,\d\mm=0\quad \forall \varphi\in W^{1,2}(\X)\ \text{ with }\supp(\varphi)\subset U.
\end{equation}
We then have the following stability result:
\begin{theorem}\label{thm:stabharm}
Let $x_n\to x_\infty$, $R>0$ and $f_n\in W^{1,2}(B_R(x_n))$ be harmonic. Assume that  $\sup_n\|f_n\|_{L^2(B_R(x_n))}<\infty$, $n\in\N$. Then there is a non-relabeled subsequence and $f_\infty\in L^2(B_R(x_\infty))$ such that for any $r\in(0,R)$ we have  strong $L^2(B_r(x_n))$-convergence of $(f_n),(\d f_n)$ to $f_\infty,\d f_\infty$ respectively. Also, $f_\infty$ is harmonic on $B_r(x_\infty)$.
\end{theorem}
\begin{idea} Pick $\varphi_n:=(1-\frac{\sfd(\cdot, B_r(x_n))}{R-r})^+$ in the classical estimate
\[
\begin{split}
\int\varphi_n^2|\d f_n|^2\,\d\mm_n=\int\underbrace{\la\d f_n,\d(\varphi^2_nf_n)\ra}_{\text{integral  0 by \eqref{eq:equivharm}}}-2 f_n\varphi_n\la\d f_n,\d\varphi_n\ra\,\d\mm_n \leq\int\tfrac12{\varphi_n^2}|\d f_n|^2+2 f_n^2|\d \varphi_n|^2\,\d\mm_n
\end{split}
\]
to get $\int_{B_r(x_n)}|\d f_n|^2\,\d\mm_n\leq\frac{4}{(R-r)^2} \int_{B_R(x_n)}|f_n|^2\,\d\mm_n$, hence the uniform bound on the $L^2$-norms and (a variant of) the compactness statement in Theorem \ref{thm:convslen} give strong $L^2$-compactness. Then the closure of the differential (item \ref{it:cldiff}), the density given by item \ref{it:testd}  and the characterization \eqref{eq:equivharm} give that any limit is harmonic. For convergence of energies the key observation is that, as above, we have $\int\varphi_n^2|\d f_n|^2\,\d\mm_n=-\int 2 f_n\varphi_n\la\d f_n,\d\varphi_n\ra\,\d\mm_n$ and the right hand side passes to the limit.
\end{idea}

\subsection{Convergence of flows}\label{se:convflow}

Let  $(\Z,\sfd_\Z)$ be a complete and separable space and $F^n:\X_n\to\Z$, $n\in\N\cup\{\infty\}$. We say that $F^n\to F^\infty$ in measure provided for some $(T_n)$ as in \eqref{eq:mappeTn}  we have 
\begin{equation}
\label{eq:convmes}
 \int 1\wedge\sfd_\Z(F^n\circ T_n\,,\,F^\infty )\,\d\mm_\infty\ \to\ 0\qquad\text{as $n\to\infty$}.
\end{equation}
The following equivalent characterization of such convergence will be useful (and shows that the particular choice of $(T_n)$ as in \eqref{eq:mappeTn} is not relevant):
\begin{lemma}\label{le:criterioL0} With the above notation,  $F^n\to F^\infty$  in measure if and only if we have:  for any $\mu_n\weakto\mu_\infty$ with $\mu_n\leq C\mm_n$ for some $C>0$ we have $F^n_*\mu_n\weakto F^\infty_*\mu_\infty$.  
\end{lemma}
\begin{idea}\ \\
\noindent{\sc If}
If not, there are $r,\eps>0$ so that putting $A_n:=\{\sfd_\Z(F^n\circ T_n,F^\infty)>4r\}\subset\X_\infty$ we have $\mm_\infty(A_n)>2\eps$ for infinitely many $n$'s. By the separability of $\Z$ we can find a finite number of points $z_1,\ldots,z_k\in\Z$ such that $\mm_\infty(C)>1-\eps$ for $C:=(F^\infty)^{-1}(\cup_iB_{r}(z_i))$. Hence $\mm_\infty(A_n\cap  C)>\eps$ for infinitely many $n$'s, and thus for some $i$ we have $\mm_\infty(A_n\cap D)\geq \frac\eps k$ for infinitely $n$'s, where $D:=(F^\infty)^{-1}(B_{r}(z_i))$. 

Then put $c^{-1}:=\mm_\infty(D)$,  $\mu_\infty:=c\mm_\infty\restr {D}$ and  consider $\mu_n:=(T_n)_*\mu_\infty\leq c\mm_n$. From \eqref{eq:mappeTn} it is clear that $\mu_n\weakto\mu_\infty$ and by construction  the measure $F^\infty_*\mu_\infty$ is concentrated on $B_r(z_i)$, while $F^n_*\mu_n(B_{2r}^c(z_i))=c\mm_\infty(\{x\in D:F^n(T_n(x))\in B_{2r}^c(z_i)\})\geq c\mm_\infty(A_n\cap D)$ tells that $F^n_*\mu_n(B_{2r}^c(z_i))\geq\tfrac{\eps c}k$ for infinitely many $n$'s. Hence  $(T_n)_*\mu_n\not\weakto (T_\infty)_*\mu_\infty$.

\noindent{\sc Only if} Let $\mu_n=\rho_n\mm_n$. Then $\rho_n\leq C$ and $\mu_n\weakto \mu_\infty$ give that $\rho_n\circ T_n\weakto \rho_\infty$ in the weak$^*$ topology of $L^\infty$. Also, for $\varphi\in C_b(\X)$ the assumption  \eqref{eq:convmes} implies that $(\varphi\circ F^n\circ T_n)$ converges in measure to $\varphi\circ F^\infty$ and since these are uniformly bounded we can conclude that 
\[
\int\varphi\,\d F^n_*\mu_n=\int \varphi\circ F^n\circ T_n\,\rho_n\circ T_n\,\d\mm_\infty\qquad\to\qquad \int \varphi\circ F^\infty\,\rho_\infty\,\d\mm_\infty =\int\varphi\,\d F^\infty_*\mu_\infty,
\]
as desired.
 \end{idea}
\begin{lemma}\label{le:2fl}
Let $(v^n_t)\in L^2([0,1],L^2(T\X_n))$ be with $\sup_{n}{\sf N}\big((v^n_t)\big)<\infty$ (recall \eqref{eq:assvt}) and let $F^n$ be the RLF of $(v^n_t)$ granted by Theorem \ref{thm:RLF}. Assume that for any $(\mu_n)$ as in Lemma \ref{le:criterioL0} we have $\mu_t^n\weakto \mu^\infty_t$ for any $t\in[0,1]$, where  $(\mu^n_t)$ is the solution of the continuity equation for $(v^n_t)$ starting from $\mu_n$ (recall Proposition \ref{prop:exuniqCE}). 

Then $F^n\to F^\infty$   in measure as maps from $\X_n$ to $C([0,1],\X_n)\subset C([0,1],\Y)$.
\end{lemma}
\begin{idea} The proof of Theorem \ref{thm:RLF} tells that $\mu^n_t=(F^n_t)_*(\mu_n)$, thus Lemma \ref{le:criterioL0} gives that $F^n_t\to F_t^\infty$ in measure for any $t\in[0,1]$. To conclude we notice that curves in the images of the $F_n$'s are uniformly Lipschitz (from $\sup_n\||v_t^n|\|_{L^\infty}<\infty$ and \eqref{eq:speedrlf}) and use the trivial bound $1\wedge\sup_t\sfd(\gamma_t,\eta_t)\leq \frac Lk+\sum_{i=0}^{k-1}1\wedge\sfd(\gamma_{\frac ik},\eta_{\frac ik})$ valid for every $k\in\N$ and any two $L$-Lipschitz curves $\gamma,\eta$. 
\end{idea}
\begin{definition} Let $(v^n_t)\in L^2([0,1],L^2(T\X_n))$. We say that $v^n\to v^\infty$ \emph{weakly in time and strongly in space} provided:
 \[
\begin{split}
\iint_0^1\varphi_t\la v^n_t, z^n\ra\,\d t\,\d\mm_n&\to\iint_0^1\varphi_t\la v^\infty_t, z^\infty\ra\,\d t\,\d\mm_\infty\quad\forall\varphi\in C([0,1]),\  (z^n)\ \text{\rm test sequence}\\
\iint_0^1|v^{n,\psi}_t|^2\,\d t\,\d\mm_n&\to \iint_0^1|v^{\infty,\psi}_t|^2\,\d t\,\d\mm_\infty\quad\forall \psi\in C_c(\R),\ where\ v^{n,\psi}_t:=\int_0^1v^n_s\psi_{t-s}\,\d s.
\end{split}
\]
\end{definition}
\begin{theorem}\label{thm:convflow}
Let $(v^n_t)\subset L^2([0,1],L^2(T\X_n))$ be with $\sup_{n}{\sf N}\big((v^n_t)\big)<\infty$ (recall \eqref{eq:assvt}) and converging weakly in time and strongly in space to $(v^\infty_t)\subset L^2([0,1],L^2(T\X_\infty))$. 

Then the Regular Lagrangian Flows  $F^n$ of $(v^n_t)$ converge in measure to that $F^\infty$ of $(v^\infty_t)$.
\end{theorem}
\begin{idea}
Let $(\mu_n)$ be arbitrary as in Lemma \ref{le:criterioL0} and $\mu^n_t:=(F^n_t)_*\mu_n$. By Lemma \ref{le:2fl}, to conclude it is sufficient to prove that $\mu^n_t\weakto\mu^\infty_t$ for every $t\in[0,1]$. The uniform bound on $\|\div(v^n_t)\|_\infty$ and \eqref{eq:expb} imply $\mu^n_t\leq \Comp\,\mm_n$ for some $\Comp>0$ and every $n,t$, that in turns gives the tightness of the $\mu^n_t$'s. Also, the uniform bound on $\||v^n_t|\|_\infty$  (and \eqref{eq:speedrlf}, \eqref{eq:superpfacile}) implies a uniform bound on the $W_2$-speed of $t\mapsto\mu^n_t$. By Ascoli-Arzel\`a, this suffices to get weak compactness    of the curves $t\mapsto \mu^n_t\in \pr(\X_n) \subset \pr(\Y)$. Let $t\mapsto\nu_t$ be a limit curve of some non-relabeled subsequence: we need to prove that $\nu_t=\mu_t^\infty$ for any $t\in[0,1]$. Since clearly this holds for $t=0$ and $(\nu_t)$ has bounded compression, it is sufficient to prove that it solves the continuity equation for the $v_t$'s.

Thus let  $(f_n)$ be an arbitrary test sequence: by item \ref{it:testd} it is easy to see that to conclude it is sufficient to prove that $\partial_t\int f_\infty\,\d\nu_t=\int\d f_\infty(v^\infty_t)\,\d\nu_t$. Since  $\int f_n\,\d\mu^n_t\to\int f_\infty\,\d\nu_t$ for any  $t\in[0,1]$, to conclude it suffices  to show that $\partial_t\int f_n\,\d\mu^n_t\stackrel{L^2([0,1])}{\weakto}\int \d f_\infty(v^\infty_t)\,\d\nu_t$  \footnote{if we strengthen our assumption in `the vector fields strongly converge in both space and time', then quite clearly we get that $\d f_n(v_t^n)\stackrel{}\to \d f_\infty(v_t^\infty)$ strongly in $L^2_{t,x}$; this  and  the weak convergence of measures allows to conclude. Thus from here on the proof is about handling the weak convergence in time.}. Put $I_n(s,t):=\int \d f_n(v^n_s)\,\d\mu^n_t$, $I_\infty(s,t):=\int \d f_\infty(v^\infty_s)\,\d\nu_t$, notice that $\partial_t\int f_n\,\d\mu^n_t=I_n(t,t)$  and fix $\varphi\in L^2(0,1)$ and an even mollification kernel  $\psi\in   C_c(\R)$.

The  assumptions on the $v^n_t$'s tell that, in a suitable natural sense, $v^{n,\psi}\to v^{\infty,\psi}$ strongly in $t,x$, thus arguing as in item \ref{it:combinate} keeping in mind the uniform bound on $|\d f_n|$ we get $\d f_n(v^{n,\psi}_\cdot)\to \d f_n(v^{\infty,\psi}_\cdot)$ strongly in $L^2_{t,x}$ that together with the weak convergence of the measures give $\iint\varphi_tI_n(s,t)\psi_{t-s}\,\d s\,\d t\to \iint\varphi_tI_\infty(s,t)\psi_{t-s}\,\d s\,\d t$. Hence  to conclude it is sufficient to prove that 
\begin{equation}
\label{eq:clstabfl}
\sup_nA_{n}\to 0\quad\text{as}\quad\psi\weakto\delta_0\qquad\text{where}\qquad A_n:=\int\varphi_t\Big(\int I_n(s,t)\psi_{t-s}\,\d s-I_n(t,t)\Big).
\end{equation}
Since $\int\varphi_tI_n(t,t)\,\d t=\int \varphi_sI_n(s,s)\psi_{t-s}\,\d t\,\d s$ and  $\psi$ is  even we have
\[
\begin{split}
A_{n}=\iint\varphi_t  (I_n(s,t)-I_n(s,s))\psi_{t-s}\,\d s \,\d t +\int (\varphi\ast\psi(s)-\varphi_s)I_n(s,s)\,\d s.
\end{split}
\]
Now recall that $(f_n)$ is a test sequence, that for a.e.\ $s$ we have  $v^n_s\in W^{1,2}_C(T\X)$ and use \eqref{eq:leibcov} to deduce that  $\d f_n(v^n_s)\in W^{1,2}(\X)$ with $\||\d (\d f_n(v^n_s))|\|_{L^2}\leq C$ for some $C$ depending only on the data in the statement and on $(f_n)$ (but independent on $n$). Thus $t\mapsto I_n(s,t)$ is absolutely continuous with $|\partial_tI_n(s,t)|=|\int \d(\d f_n(v^n_s))(v^n_t)\,\d\mu_t^n|\leq C\Comp \||v^n_t|\|_{L^2 }$ and we have
\[
\begin{split}
|A_{n}|&\leq \iint |\varphi_t|\int_s^t|\partial_rI_n(s,r)|\,\d r\,\psi_{t-s}\,\d s\,\d t+\int |\varphi\ast\psi(s)-\varphi_s|\,|I_n(s,s)|\,\d s\\
&\leq \Comp\Big(\int C |\varphi_t|\sqrt{|s-t|\int_0^1 \||v^n_r|\|^2_{L^2}\,\d r} \,\psi_{t-s}\,\d s\,\d t+\|\varphi-\varphi\ast\psi\|_{L^2}\sqrt{\int_0^1I_n^2(s,s)\,\d s}\Big)
\end{split}
\]
and since $\sup_n\int_0^1I_n^2(s,s)\,\d s<\infty$, the claim \eqref{eq:clstabfl} follows.
\end{idea}

\subsection{Mosco-convergence of Total Variation}\label{se:gtv}
In Section \ref{se:vertsob} we introduced the Cheeger energy functional by $L^2$-lower semicontinuous relaxation of the functional $\tfrac12\int \lipa^2(f)\,\d\mm$. There is nothing special about the choice of the exponent 2, as a similar construction can be done for any $p\in[1,\infty)$: we define $\ch_p:L^p(\X)\to[0,\infty]$ as
\[
\ch_p(f):=\inf\limi_n\tfrac1p\int\lipa^p(f_n)\,\d\mm_n 
\]
the $\inf$ being taken among all $(f_n)\subset\Lip_{\bs}(\X)$ $L^p$-converging to $f$. The set $\{\ch_p<\infty\}\subset L^p(\X)$ takes the name of $W^{1,p}(\X)$ Sobolev space (or $\BV(\X)$ space for $p=1$), and it is Banach when equipped with the norm
\[
\|f\|_{W^{1,p}}^p:=\|f\|^p_{L^p}+p\ch_p(f).
\]
Arguing as in Section \ref{se:vertsob} it is not hard to see that for $f\in W^{1,p}(\X)$ there is a unique $|\D f|_p\in L^p(\X)$  non-negative such that 
\[
\ch_p(f)=\tfrac1p\int|\D f|_p^p\,\d\mm\qquad\text{and}\qquad |\D f|_p\leq G\quad\mm-a.e.
\]
for any weak $L^p$-limit $G$ of some subsequence of $(\lipa f_n)$ for $(f_n)\subset\Lip_\bs(\X)$ converging to $f$ in $L^p$. Analogously, but with a bit more of work, for $f\in \BV(\X)$ there is a unique non-negative Radon measure $|\DD f|$ such that 
\[
\ch_1(f)=|\DD f|(\X)\qquad\text{and}\qquad |\DD f|\leq\wlimi_{n\to\infty}\lipa(f_n)\mm,
\]
the latter meaning that $\int\varphi\,\d|\DD f|\leq\limi_n\int\varphi \lipa(f_n)\,\d\mm $ for any $\varphi\in C_\b(\X)$ non-negative.
 
 Notice that there is an analogue `dual' construction based on the concept of test plans (with $\iint_0^1|\dot\gamma_t|^q\,\d t\,\d\ppi<\infty$  replacing \eqref{eq:finiteenergy}) and that the same arguments leading to Theorem \ref{thm:glimsch}  grant that for arbitrary spaces $\X_n$ we have
 \begin{equation}
\label{eq:glimischp}
 \X_n\stackrel{mGH}\to\X_\infty\qquad\Rightarrow\qquad \glims \ch_{n,p}\leq  \ch_{\infty,p}\quad\forall p\geq1.
\end{equation}
 We also point out that the object $|\D f|_p$ in general  \emph{depends} on $p$, and that typically only one inequality holds: for $p_1\leq p_2$ the fact that weak convergence in $L^{p_2}$ implies weak convergence in $L^{p_1}$ (say $\mm(\X)<\infty$, otherwise the claim should be intended locally) tells that
\begin{equation}
\label{eq:ineqp1p2}
f\in W^{1,p_2}(\X)\quad\Rightarrow\quad f\in W^{1,p_1}(\X)\qquad\text{with}\qquad |\D f|_{p_1}\leq |\D f|_{p_2}\quad\mm-a.e..
\end{equation}

In these matters, a uniform $\RCD$ condition has two roles: it grants that $|\D f|_p$ is actually independent on $p$ and ensures the weak-$\glimi$ inequality in \eqref{eq:glimischp}. 
\begin{theorem}[Independence on $p$ of  $p$-weak upper gradients]\label{thm:indp} Let $(\X,\sfd,\mm)$ be $\RCD(K,\infty)$, $p_1, p_2\in(1,\infty)$ and $f\in W^{1,p_1}(\X)$ be with $f,|\D f|_{p_1}\in L^{p_2}(\X)$. 

Then $f\in W^{1, p_2}(\X)$ with $|\D f|_{p_1}=|\D f|_{ p_2}$ $\mm$-a.e.. Similarly, if $f\in \BV\cap L^p(\X)$ is with $|\DD f|\ll\mm$ and $\frac{\d |\DD f|}{\d \mm}\in L^p(\X)$, then $f\in W^{1,p}(\X)$ and $|\DD f|=|\D f|_p\mm$. Also, we have
\begin{equation}
\label{eq:BEBV}
|\DD \h_tf|\leq e^{-Kt}\h_t(|\DD f|)\qquad\forall f\in \BV(\X),\ t\geq0. 
\end{equation}
\end{theorem}
\begin{idea} Say $p_1\leq  p_2$. By \eqref{eq:ineqp1p2} and a truncation argument it suffices to prove that $\ch_{ p_2}(f)\leq \tfrac1{p_2}\int|\D f|_{p_1}^{  p_2}\,\d\mm$ under the assumption that $f$ is also in $L^\infty$. From \eqref{eq:BEimpr} and Proposition \ref{prop:sobtolip} we deduce
\[
\lipa(\h_tg)\leq e^{-Kt}\h_t(\lipa g)\quad\text{ pointwise}\qquad\forall g\in \Lip_\bs(\X),\ t\geq0.
\]
Writing this for an optimal sequence $(g_n)\subset \Lip_\bs(\X)$ for the definition of $\ch_{p_1}(g)$  we deduce
\begin{equation}
\label{eq:BEp}
|\d\h_tg|_{p_1}\leq e^{-Kt}\h_t(|\d g|_{p_1})\qquad\mm-a.e.\qquad\forall g\in W^{1,{p_1}}(\X),\ t\geq0.
\end{equation}
Applying this to $g\in \Lip_\b(\X)$, noticing that $\h_t(|\d g|_{p_1})\in C_\b(\X)$ (e.g.\ by \eqref{eq:linftylip}) and then using again Proposition \ref{prop:sobtolip} we get
\[
|\d\h_t g|_{ p_2}\leq \lipa(\h_tg)\leq e^{-Kt}\h_t(|\d g|_{p_1})\qquad\mm-a.e.\qquad\forall g\in \Lip_\b(\X),\ t\geq0.
\]
having used also \eqref{eq:dflip}. Picking $g:=\h_tf$ (that is in $\Lip_\b(\X)$ by \eqref{eq:linftylip}) and using also \eqref{eq:BEp}   we get
\[
|\d\h_{2t}f |_{ p_2}\leq  e^{-Kt}\h_t(|\d \h_tf|_{p_1})\leq e^{-2Kt}\h_{2t}(|\d f|_{p_1})\qquad\mm-a.e.,
\]
so that raising to the $ p_2$-th power, integrating, using the trivial bound $\h_t(F)^p\leq \h_t(F^p)$ (by Jensen's inequality and \eqref{eq:reprform}), we conclude by letting $t\downarrow0$ using the lower semicontinuity of $\ch_{p_2}$. The BV case is handled analogously.
\end{idea}
I now discuss the weak-$\glimi$ for the $p$-Cheeger energies. For simplicity, I state the result only for the Total Variation (i.e.\ for $\ch_1$), but analogue results hold for any $p>1$.
\begin{theorem}\label{thm:stabtv} Let $\X_n\to\X_\infty$ be a sequence of normalized $\RCD(K,\infty)$ spaces, and $f_n\in L^1(\X_n)$, $n\in\N\cup\{\infty\}$, be such that $f_n\circ T_n\weakto f_\infty$ in $L^1(\X_\infty)$, where the $T_n$'s are as in \eqref{eq:mappeTn}.

Then $\ch_{1,\infty}(f_\infty)\leq\limi_n\ch_{1,n}(f_n)$. Also, if $\sup_n\ch_{1,n}(f_n)<\infty$ we have
\[
|\DD f|\leq\wlimi_{n\to\infty}|\DD f_n|
\]
meaning that for any $\varphi\in C_\b(\Y)$ non-negative we have $\int\varphi\,\d |\DD f_\infty|\leq\limi_n\int\varphi\,\d |\DD f_n|$.
\end{theorem}
\begin{idea} Assume at first that the $f_n$'s are equibounded in $L^\infty$. Then they are also equibounded in $L^2$ and thus for any $t>0$ we have $\sup_n\ch_{2,n}(\h_tf_n)<C(t)<\infty$ (by \eqref{eq:aprioriheat}). Also, quite clearly from the self adjointness of the heat flow and Corollary \ref{cor:stabhagain}, we have $\h_tf_n\stackrel{L^2}\weakto\h_tf_\infty$, thus  the closure of the differential tells $\d\h_tf_n\weakto \d\h_tf_\infty$. Up to subsequence we can assume that  $|\d\h_tf_n|\stackrel{L^2}\weakto G$, so that item \ref{it:lscnorm} in Section \ref{se:defconvtens} tells that $|\d \h_tf_\infty|\leq G$ $\mm_\infty$-a.e.\ and thus
\[
\begin{split}
\int \varphi\,\d|\DD(\h_tf_\infty)|&=\int \varphi|\d\h_tf_\infty|\,\d\mm_\infty\leq\int \varphi G\,\d\mm_\infty \\&= \lim_n\int \varphi|\d\h_tf_n|\,\d\mm_n= \lim_n\int \varphi\,\d|\DD(\h_tf_n)|\stackrel{}\leq e^{-Kt}\limi_n\int \varphi\,\d|\DD f_n|.
\end{split}
\]
Letting $t\downarrow0$ we conclude by the arbitrariness of $\varphi$.

For the general case, we truncate the $f_n$'s by putting $f_{n,k}:=(-k)\vee f_k\wedge k$ and, possibly passing to a non-relabeld subsequence, we let $f_{\infty,k}$ be the weak $L^1(\X_\infty)$-limit of $f_{n,k}\circ T_n$. Since quite clearly we have $|\DD f_{n,k}|\leq |\DD f_n|$, we see that 
\[
|\DD f_{\infty,k}|\leq \wlimi_{n\to\infty}|\DD f_{n,k}|\leq \wlimi_{n\to\infty}|\DD f_n|,
\]
so that to conclude is then enough to prove that $f_{\infty,k}\to f_\infty$ in $L^1(\X_\infty)$. To see this, use Dunford-Pettis theorem to get uniform integrability of $(f_n\circ T_n)\subset L^1(\X_\infty)$: this grants that $\lim_k\sup_n\|f_n-f_{n,k}\|_{L^1}=\lim_k\sup_n\int_{\{|f_n|>k\}}|f_n|\,\d\mm_n=0$, which suffices to conclude.
\end{idea}

\subsection{$\Gamma-\liminf$ of Willmore energy functional}\label{se:willmore}
Given a smooth open set $U$ inside a Riemannian manifold, its mean curvature is defined as $-\div(\nu)$, where $\nu$ is the outer pointing unit vector. Thus given a smooth function $u$, the mean curvature of $\{u\leq t\}$ at any point of the boundary with $|\d u|\neq 0$ is 
\[
H:=-\div\Big(\frac{\nabla u}{|\nabla u|}\Big)
\]
and the Willmore energy of the level set is $\int_{\{u=t\}}|\div(\frac{\nabla u}{|\nabla u|})|^2 $. Integrating in $t$ and using coarea formula we see that the \emph{Willmore energy functional} of $u$ is given by
\[
\WW(u):=\int\Big|\div\Big(\frac{\nabla u}{|\nabla u|}\Big)\Big|^2|\nabla u|\,\d\vol.
\]
Now let $v$ be a smooth vector field with compact support and notice that, by direct computation, it holds
\begin{equation}
\label{eq:partH}
\int H\la v,\tfrac{\nabla u}{|\nabla u|}\ra |\nabla u|\,\d\vol=-\int\big(\div(v)-\nabla v\big(\tfrac{\nabla u}{|\nabla u|},\tfrac{\nabla u}{|\nabla u|}\big)\big)|\nabla u|\,\d\vol
\end{equation}
where  we recognize the tangential divergence $\div_{\rm T}(v):=\div(v)-\nabla v\big(\tfrac{\nabla u}{|\nabla u|},\tfrac{\nabla u}{|\nabla u|})$ at the right hand side. With this in mind and the trivial  duality formula $\frac12H^2=\sup_v H\la v,\tfrac{\nabla u}{|\nabla u|} \ra-\frac12|v|^2$, it is easy to establish that
\[
\WW(u)=\sup\int \big(\div_{\rm T}(v)-\frac12|v|^2\big)|\nabla u|\,\d\vol,
\] 
where the sup is taken over all smooth compactly supported vector fields. The result of the integration by parts  \eqref{eq:partH} is, as usual, that now we have a formula for $\WW(u)$ where fewer  derivatives of $u$ appear (one rather than two), so that we can use such formula to define the functional under low(er) regularity of the function and the underlying space and also obtain natural lower semicontinuity properties: 
\begin{definition}
Let $(\X,\sfd,\mm)$ be $\RCD(K,\infty)$. Then $\WW:W^{1,2}(\X)\to[0,\infty]$ is defined as
\[
\WW(u)=\sup\int \Big(\div(v)-\nabla v\big(\tfrac{\nabla u}{|\nabla u|},\tfrac{\nabla u}{|\nabla u|})-\frac12|v|^2\Big)|\nabla u|\,\d\mm,
\]
 the sup being taken among all $v$'s  of the form $v=\sum_{i=1}^kf_i\nabla g_i$ with $f_i,g_i\in\test(\X)$, $k\in\N$. Here $\frac{\nabla u}{|\nabla u|}$ is intended to be 0 on $\{|\nabla u|=0\}$.
\end{definition}
Then a quite direct consequences of the results in Section \ref{se:defconvtens} is:
\begin{theorem}
Let $u_n\stackrel{L^2}\to u_\infty$ be with $\ch_n(u_n)\to \ch_\infty(u_\infty)$. Then $\WW(u_\infty)\leq \limi_n\WW(u_n)$.
\end{theorem}
\begin{idea} It is sufficient to prove that for $(v_n)$ test we have 
\[
\int \big(\div_{n,{\rm T}}(v_n)-\frac12|v_n|^2\big)|\nabla u_n|\,\d\mm_n\to \int \big(\div_{\infty,{\rm T}}(v_\infty)-\frac12|v_\infty|^2\big)|\nabla u|\,\d\mm_\infty
\]
where $\div_{n,{\rm T}}(v_n):=\div(v_n)-\nabla v_n(\tfrac{\nabla u_n}{|\nabla u_n|},\tfrac{\nabla u_n}{|\nabla u_n|})$. 

Here $\nabla u_n$ strongly converges to $\nabla u_\infty$ and since the sets $\{|\nabla u_n|=0\}$ give no contribution to the integral, it is not hard to check that $\frac{\nabla u_n}{|\nabla u_n|}\to \frac{\nabla u_\infty}{|\nabla u_\infty|}$ in a suitable locally strong sense. Then since $\div(v_n)\to \div(v_\infty)$  and $\nabla v_n\weakto \nabla v_\infty$, the conclusion follows.
\end{idea}
The interest of this discussion would be quite limited without some evidence of the relevance of the functional $\WW$ in the setting of $\RCD$ spaces. Without entering the details, let me just mention that  the study of $\WW$ makes it possible to prove a natural sharp version of the Willmore inequality on  $\RCD(0,N)$ spaces with Euclidean volume growth.

\subsection{Bibliographical notes}
{\footnotesize
Convergence of tensors along varying spaces has been first investigated in \cite{Honda11-2} in the context of Ricci-limit spaces, taking advantage of Cheeger-Colding charts to read tensors in coordinates. A different approach has been considered later in \cite{AST17}  and then in \cite{AH16}, where tensors were studied via their actions on a suitable algebra of Lipschitz functions on the space $\Y$ where all the $\X_n$'s are embedded. Definition \ref{def:convtens}  I am proposing here is more intrinsic, but equivalent to these previous ones and the ideas for proving the stated properties of weak/strong convergence of tensors  all come from these earlier works.   Example \ref{eq:compvect} comes from \cite[Remark 6.18]{Honda14}.

Theorem \ref{thm:lscdim} is a consequence, at least under a uniform upper bound on the dimension, of Colding's celebrated volume convergence theorem in \cite{Colding97}. The proof given here follows a different set of ideas that come from \cite{Honda14}. I gave the statement for manifolds, but in fact an equivalent version, with the same proof, holds for $\RCD$ spaces once one interprets `dimension' as `dimension of the tangent module'. It is an uneasy task to link the abstract concept of `dimension of the tangent module' to more concrete geometric quantities like the dimension of tangent spaces intended as pointed mGH-limits of rescaled spaces (see \cite{GP16} and \cite{IPS22}) and to accomplish this task one needs a very good understanding of the rectifiability properties of $\RCD$ spaces (which we do have  thanks to \cite{Gigli13}, \cite{Mondino-Naber14},\cite{BS18}, \cite{MK16}, \cite{GP16-2}, \cite{BPS21}). It is worth to notice that once one knows all this about the local geometry of $\RCD$ spaces, a proof of lower semicontinuity of dimension closer in spirit to Colding's one, and based on the notion of $\delta$-splitting maps, is available (see  \cite{BPS21} and references therein).

Section \ref{se:convharm} comes from \cite{AH16} (with the exception of the equivalence \eqref{eq:equivw12u}, that comes from \cite{AmbrosioGigliSavare11-2}).

Convergence of flows in $\RCD$ setting has been studied in \cite{AST17}, where the convergence of vector fields was `strong in both space and time'. The stability under the weaker `strong in space and weak in time' convergence as stated in Theorem \ref{thm:convflow} appears here for the first time and combines ideas from \cite{AST17} and \cite{GT18}.

Theorem \ref{thm:indp} is taken from \cite{GigliHan14} (see also the more recent \cite{GiNo21}). In the setting of doubling spaces supporting a Poincar\'e inequality (such as finite dimensional $\CD$ spaces \cite{Lott-Villani07}, \cite{Sturm06II}, \cite{Rajala12}), the same result was known from \cite[Lemma 5.1]{KKST12}. Theorem \ref{thm:stabtv} was stated in the already mentioned \cite{AH16} for $L^1$-strongly converging sequences of functions; the extension to weak convergence, i.e.\ `Mosco' convergence of Total Variation rather than `$\Gamma$' convergence, appears here for the first time.

Section \ref{se:willmore} comes from \cite{GV23}.
}

\section{Some  applications to the smooth category}
\subsection{Extracting information from  \hsout{Pythagora's} the  splitting theorem}\label{se:pitagora}
A \emph{line} in a metric space $(\X,\sfd)$ is a map  $\gamma:\R\to\X$ with $\sfd(\gamma_t,\gamma_s)=|s-t|$ for every $t,s\in\R$.

The splitting theorem (\cite{Gigli13}, \cite{Gigli13over}) in the $\RCD$ category reads as:
\begin{theorem}\label{thm:splittingrcd}
Let $(\X,\sfd,\mm)$ be an $\RCD(0,N)$ space, $N<\infty$, containing a line.

Then such space is isomorphic to the metric-measure product of $\R$ and an $\RCD(0,N-1)$ space $(\X,\sfd',\mm')$ (here $\X'$ is just a point if $N-1<1$).
\end{theorem}
Here `isomorphic to the metric-measure product' means that we can put coordinates $(x',t)$ on $\X$ with $x'\in\X'$ and $t\in\R$ in such a way that:
\begin{itemize}
\item[a)] Pythagora's theorem holds,
\item[b)] The area of a rectangle is the product of the sizes of its sides.
\end{itemize}
More rigorously, this means that
\[
\begin{split}
\sfd\big((x',t),(y',s)\big)^2=\sfd'(x',y')^2+|s-t|^2\qquad\text{ and }\qquad \mm(E\times F)=\mm'(E)\mathcal L^1(F)
\end{split}
\]
for any $x',y'\in\X'$, $t,s\in\R$ and $E\subset\X'$, $F\subset\R$ Borel.

A direct application of Gromov's compactness principle for pointed $\RCD(K,N)$ spaces gives the following `almost splitting' statement, so called because it tells that if a space almost satisfies the assumptions of the splitting, then it almost satisfies the conclusions:
\begin{corollary}\label{cor:almostspl}
For every $\eps>0$ and $N\geq 1$ there is $\delta>0$ such that the following holds. Let $(\X,\sfd,\mm,\bar x)$ be an $\RCD(-\delta,N)$ space such that $\bar x$ is the midpoint of some geodesic of length $\geq \delta^{-1}$.

Then there is a pointed $\RCD(0,N-1)$ space $(\X',\sfd',\mm',\bar x')$ (this  reduces to a singleton if $N-1<1$) such that $\mathbb D_{pmGH}(\X,\X'\times\R)\leq \eps$.
\end{corollary} 
We compare this last result to the almost splitting theorem as obtained in \cite{Cheeger-Colding96}: in the late nineties the  $\CD/\RCD$ technology was not yet available, but still Gromov's concepts of metric convergence Riemannian manifolds and his pre-compactness theorem were available. These were the grounds on which Cheeger-Colding built their theory of Ricci-limit spaces \cite{Cheeger-Colding97I} ,\cite{Cheeger-Colding97II}, \cite{Cheeger-Colding97III}, i.e.\ those spaces that could arise as limits of Riemannian manifolds with a uniform lower bound on the Ricci curvature and upper bound on the dimension. Clearly, by the stability of the $\RCD$ condition we know today that Ricci-limit spaces are a subclass of $\RCD$ ones (see below for further comments on this).

Results about Ricci-limit spaces are typically obtained by studying smooth Riemannian manifolds satisfying suitable curvature-dimension conditions and deriving estimates that are stable under Gromov-Hausdorff or measured-Gromov-Hausdorff convergence (there is little difference between these two concepts if the given measures are uniformly locally doubling, as is the case under a uniform $\CD(K,N)$ condition). A prototypical and key result in this direction obtained in \cite{Cheeger-Colding96}  (reformulated to emphasize the analogies with the above) is:
\begin{theorem}\label{thm:riccilim}  For every $\eps>0$ and $N\geq 1$ there is $\delta>0$ such that the following holds.  Let $M$ be a   Riemannian manifold with   Ricci $\geq -\delta$, dimension $\leq N$  and  a geodesic of length $\geq\delta^{-1}$. Let $p$ be the midpoint of any such geodesic and consider the pointed metric  space $(M,\sfd,p)$, where $\sfd$ is the distance induced by the metric tensor.

Then there is a pointed geodesic space $(\X',\sfd',\bar x')$ such that  $\mathbb D_{pGH}(M,\X'\times\R)\leq \eps$.
\end{theorem}
The key difference between Corollary \ref{cor:almostspl} and Theorem \ref{thm:riccilim} I want to highlight is the fact that the almost quotient space in Corollary \ref{cor:almostspl} can be chosen to be $\RCD(0,N-1)$, whereas in Theorem \ref{thm:riccilim} no geometric information, beside the geodesic condition, is given. Notice that such $\RCD$ structure is in place and non-trivial even if the space $\X$ is a smooth Riemannian manifold, so that from the study of the $\RCD$ category we derive new informations about the shape of smooth Riemannian manifolds\footnote{there is another important difference between  Corollary \ref{cor:almostspl} and Theorem \ref{thm:riccilim}, made invisible by the way I phrased the results: in  Theorem \ref{thm:riccilim} the dependance of $\delta$ on $\eps,n$ is \emph{quantitative}, whereas in Corollary \ref{cor:almostspl}, which is obtained via compactness, such dependance is only \emph{qualitative}. It is unclear if the two statements can be combined to deduce a quantitative information about the distance with an almost quotient space that is $\RCD$ (but see \cite[Theorem 4.1]{ColdMin14} for a possible approach in this direction).

Notice also that even if no information about the measure is given in Theorem \ref{thm:riccilim}, it would be easy to derive it along the lines of the proof given in \cite{Cheeger-Colding96}: the relevant `approximated Busemann function' is harmonic, and thus its gradient flow preserves the volume measure.}.

This is an instance of a more general phenomenon well known to experts in metric geometry: working in the non-smooth category gives more freedom in performing geometric constructions. For what concerns the class of $\RCD$ spaces, and in analogy with the case of Alexandrov spaces with curvature bounded from below, we know that it is stable under:
\begin{itemize}
\item[a)] Products (\cite{AmbrosioGigliSavare12}, \cite{Erbar-Kuwada-Sturm13}, \cite{AmbrosioMondinoSavare13}),
\item[b)] Factorization, i.e.\ if $\X\times\Y$ is $\RCD$ then so are $\X,\Y$ with natural bounds on curvature and dimension (the proof is easy - see \cite{Gigli13} for the case $\Y=\R$),
\item[a')] Cones and spherical suspensions (\cite{Ketterer13}),
\item[b')] Passage to cross sections of cones and spherical suspensions (\cite{Ketterer13}),
\item[c)] Metric measure submersions and, in particular, passage to the quotient under group actions (\cite{GKMS17}).
\end{itemize}
By contrast, the class of Ricci-limit spaces is only known to be stable under products. This sort of geometric stability of the $\RCD$ class helps in studying Riemannian manifolds and how they degenerate under a uniform lower Ricci bound. For instance, a typical singularity that can appear is that of a cone and studying it is ultimately reduced to that of its cross section: knowing that the latter is an $\RCD$ space rather than a more general doubling space supporting a Poincar\'e inequality is helpful, see for instance \cite[Sections 4.10 and 7.3]{ChJiNa21} for uses of spectral properties of the Laplacian and \cite[Claim 5]{ChNa15} for mean value estimates for harmonic functions. 

Concerning the relation between $\RCD$ and Ricci-limit spaces, it is natural to wonder whether the two classes coincide. This is unknown, but an example due to De Philippis-Mondino-Topping shows that there is an $\RCD(2,3)$ space that is not a non-collapsing Ricci-limit, i.e.\ it is not the limit of a sequence of 3 dimensional Riemannian manifolds with a uniform lower Ricci bound. The space $\X$ is the spherical suspension over $\R P^2$: since $\R P^2$ is smooth with curvature $\geq1$, it is clearly an $\RCD(1,2)$ space. Then the already mentioned result  \cite{Ketterer13} tells that $\X$  is $\RCD(2,3)$. From topological considerations it is not hard to see that $\X$ is not a manifold (e.g.\ if it were, `half' of it would be a manifold with $\R P^2$ as boundary, but $\R P^2$ has Euler characteristic 1, while boundaries have even characteristic). On the other hand, Simon proved in \cite{Simon12} that any non-collapsed 3 dimensional Ricci limit space must be topologically a manifold (this is achieved via suitable estimates on the Ricci flow starting from the manifolds converging to such Ricci-limit).

\subsection{Lack of quantitative $C^1$ estimates for harmonic functions}
Let $U\subset M$ be an open set in a (smooth, complete, connected, without boundary) Riemannian manifold $M$ and let $u:U\to\R$ be harmonic. It is a classical fact that $u$ is smooth and that the quantitative Lipschitz estimate
\begin{equation}
\label{eq:lipest}
\||\d u|\|_{L^\infty(B)}\leq C(K^-R^2,{\rm dim}(M))\sqrt{\int_{2B}|\d u|^2\,\d{\rm vol}}
\end{equation}
holds whenever $B=B_r(x)$ is with $r\leq R$, $2B=B_{2r}(x)\subset U$ and the Ricci curvature of $M$ is $\geq K$ uniformly. The bound \eqref{eq:lipest} can be proved noticing that the Bochner inequality for the harmonic function $u$  gives
\[
\Delta\tfrac12|\d u|^2\geq K|\d u|^2.
\]
Then a Moser iteration argument shows that the $L^\infty$ norm of $|\d u|^2$ in a given ball can be controlled with the $L^1$ norm in  a larger ball, where the constant appearing in the estimate depend on the doubling constant of the measure and constant appearing in the Sobolev inequality. Since both of these can be bounded in terms of the (scale invariant) lower Ricci bound and the upper dimension bound, \eqref{eq:lipest} follows.

\bigskip

One can wonder whether if it is possible to improve \eqref{eq:lipest} by getting quantitative $C^1$ estimates, possibly adding a `non-collapsing' assumption (i.e.\ imposing a lower bound on the volume of the unit ball with the same center of $B$\footnote{I haven't discussed at all the important concept of  non-collapsing sequence of Riemannian manifolds. Very briefly said, this comes from the already mentioned `volume convergence theorem' of Colding \cite{Colding97} (see also \cite{Cheeger-Colding97I}  and \cite{GDP17} for extensions to Ricci-limit and $\RCD$ spaces respectively). It tells the following: fix $K\in\R$, $n\in\N$ and consider the class $\B$ of unit balls in Riemannian manifolds with $\Ric\geq K$ and $\dim\leq n$. Equip $\B$ with the Gromov-Hausdorff (\emph{not} measured-Gromov-Hausdorff) distance. Then $\B\ni \X\mapsto \Hi^n(\X)$ is continuous. 

A GH-converging sequence is said to \emph{collapse} if the volume goes to 0, in which case the dimension of the limit space is $<n$, and to be \emph{non-collapsing} if it instead remains bounded away from 0, in which case the dimension of the limit is $=n$. Limit spaces of non-collapsing sequences are better behaved than arbitrary Ricci-limit spaces (in line with the general principle that whenever  equality holds in an inequality, we are in a somehow better position).

This phenomenon is evident in the theory of $\RCD(K,N)$ spaces, where there is a potential mismatch between the `analytic' upper bound $N$ on the dimension and the `geometric' dimension $n$ of the space, interpreted e.g.\ as generic dimension of the Euclidean blow-ups (\cite{BS18}). In general we have only the inequality $n\leq N$, and whenever equality occurs the space is better behaved. For instance, the reference measure $\mm$ must be a constant multiple of the $n$-dimensional Hausdorff measure, see \cite{H19}, \cite{HZ00}, \cite{BGHZ23}. This is very related to the fact that in order for the Bakry-\'Emery Ricci tensor in \eqref{eq:bern} to be bounded from below it is necessary that $N\geq n$ and if $N=n$ one needs $V$ to be constant.
}) and imposing a lower sectional instead of a lower Ricci bound. One of the  results in \cite{DPZ19} is that this is \emph{not} possible (in the statement below a \emph{modulus of continuity} is a function $\omega:(0,\infty)\to( 0,\infty)$ with $\omega(z)\downarrow0$ as $z\downarrow0$):
\begin{theorem}\label{thm:DPZ}
For any $k\in\R$,  $v\in(0,\pi)$ and modulus of continuity $\omega$ there is a 2-dimensional Riemannian manifold $M$, a unit ball $B=B_1(x)$ and an harmonic function $u:2B\to\R$ such that: the sectional curvature of $M$ is uniformly $\geq k$, the volume of $B$ is $\geq v$ and 
\begin{equation}
\label{eq:failurec1}
\sup_{B_r(x)}|\d u|-\inf_{B_r(x)}|\d u|\geq \omega(r)\qquad\forall r\ll1.
\end{equation}
\end{theorem}
This means that  one cannot hope to derive a uniform estimate, in terms of $k,v$, on the modulus of continuity of $|\d u|$ and a fortiori not even on $\d u$ (say that we put Sasaki's metric on the tangent bundle).  Notice that $\pi$ is the volume of the unit ball in $\R^2$ and thus, by Bishop-Gromov, it bounds from above the volume of any unit ball in a surface with non-negative curvature: since in the above we ca take $v$ as close to $\pi$ as we want, the result also shows that there is no $\eps$-regularity statement available. In other words, the ball can be as close as we want to the Euclidean one in terms of volume (and thus  also in a suitable mGH-sense by the `almost' version of the volume-cone-to-metric-cone principle) while still admitting harmonic functions with highly oscillating differentials.

It would be hard to prove Theorem \ref{thm:DPZ} by actually exhibiting a Riemannian manifold $M$ and an harmonic function $u$. Instead, what De Philippis-Zimbron did was to argue by contradiction along the following lines:
\begin{itemize}
\item[i)] By the stability result Theorem \ref{thm:stabharm}, harmonic maps converge to harmonic maps under a uniform lower Ricci bound. Conversely, it is not hard to see that any harmonic function $u$ in a Ricci-limit space can be realized as limit in this sense (because the Lipschitz estimate \eqref{eq:lipest} holds also in Ricci limit spaces, so it can be used  to assign boundary value to the Dirichlet problem along the converging sequence of spaces).
\item[ii)] Now suppose that the conclusion of Theorem \ref{thm:DPZ} were false, i.e.\ that some uniform continuity estimate for $|\d u|$ actually exists. Then by the previous point the same estimate would be valid for the limit harmonic function $u$.
\item[iii)] To conclude it is therefore enough to exhibit a compact space $\X$ that can arise as limit of a sequence of surfaces $M_n $ with curvature $\geq k$ and ${\rm Vol}_n(B_1(x_n))\geq v$ and an harmonic function $u:B_1(x)\to\R$ for which $|\d u|$ is discontinuous (more precisely: it does not have a continuous representative).
\end{itemize}
We thus see that the original problem of proving Theorem \ref{thm:DPZ} is ultimately reduced to the problem of building a suitable non-smooth space and a `sufficiently irregular' harmonic function on it. The advantage of this is that it is conceptually (and practically) easier to build  a bad-behaved harmonic function if the underlying space is non-smooth. The example given in \cite{DPZ19} is the boundary of an open convex subset of $\R^3$ with a dense set of conical singularities, i.e.\ points such that the blow-up of the set at that point is a cone strictly contained in a halfspace. This is coupled with the following considerations:
\begin{itemize}
\item[a)] The boundary of a smooth convex open subset of $\R^3$ (with the induced intrinsic metric) is a Riemannian manifold with non-negative curvature.
\item[b)] Any open convex subset of $\R^3$ is the union of an increasing sequence of smooth open convex subsets and the corresponding boundaries converge in the mGH-sense.
\item[c)]\label{it:c} There exists an open convex subset $U$ of $\R^3$ with a dense collection of conical singularities, e.g.: the epigraph of $\R^2\ni x\mapsto f(x):=\sum_{i\in\N}a_i|x-x_i|$ for $a_i>0$ sufficiently small and $(x_i)\subset\R^2$ dense. Also, taking the $a_i$'s very small, we can make the area of the unit ball centered at $(0,f(0))$ as close  as we wish to the area of the unit ball in $\R^2$.
\item[d)]\label{it:d} If $x\in \partial U$ is a conical singularity and $u$ is harmonic and defined on a neighbourhood of $x$, then $|\d u|(x)=0$ (more precisely: $\lim_{r\downarrow0}\fint_{B_r(x)}|\d u|^2\,\d\mathcal H^2=0$). This is the technically difficult part of the proof, carried out by a careful analysis of the asymptotic properties of $u$ written in polar coordinates in a neighbourhood of $x$ (the phenomenon according to which `harmonic functions must have 0 gradient at the tips of cones' was already well known -  see for instance \cite[Example 2.14]{ChNa15}).
\end{itemize}
We thus see that if $u$ is an harmonic function defined on an open set in $\partial U$ so that $|\d u|$ is a continuous, then by items c), d) we conclude that $|\d u|\equiv 0$ and thus that $u$ is constant. Since it is easy to construct non-constant harmonic functions (just solve the Dirichlet problem for a non-constant boundary condition), the proof of Theorem \ref{thm:DPZ} is complete.

\begin{remark}{\rm Notice that this line of though is extremely classical in modern analysis. As a matter of comparison, consider for instance the following problem: for $f\in C^\infty_c(\R^d)$, can we derive a quantitative information on its modulus of continuity if we know that $\int|\d f|^2\leq1$? Much like the problem above about regularity of harmonic functions, the question arises and makes perfect sense in the smooth category, but the answer can be best found looking at the non-smooth world: if such a modulus of continuity exists, it would pass to the limit and be in place also for Sobolev functions. However, it is easy to build  discontinuous functions in $W^{1,2}(\R^d)$.
}\fr\end{remark}

\subsection{Qualitative stability in Sobolev inequality on manifolds with non-negative Ricci}

Let $d\in \N$ and $p\in(1,d)$. The classical sharp Sobolev inequality in $\R^d$ reads as
\begin{equation}
\label{eq:sobRd}
\|u\|_{L^{p^*}}\leq S_{p,d}\|\d u\|_{L^p}\qquad\forall u\in \dot W^{1,p}(\R^d)\qquad p^*:=\frac{pd}{d-p}
\end{equation}
for some $S_{p,d}>0$, where $\dot W^{1,p}(\R^d)$ is the completion of $C^\infty_c(\R^d)$ w.r.t.\ the norm $\|u\|_{\dot W^{1,p}}:=\|\d u\|_{L^p}$ (here I am using $d$ for the dimension in place of the more common $n$ to avoid possible confusion with the weak upper bound on the dimension, for which the letter $N$ is customary) . The optimal value of $S_{p,d}$ has been computed in \cite{Aubin76-2} and \cite{Talenti76} and it is also known that the only extremizers for the above inequality are the functions
\begin{equation}
\label{eq:extremsob}
U_{a,b,y_0}:=\frac{a}{(1+b|\cdot-y_0|^{\frac p{p-1}})^{\frac{d-p}{p}}}\qquad\text{for } a\in\R,\ b>0,\ y_0\in\R^d,
\end{equation}
where the arbitrary parameters $a,b,y_0$ reflect the invariance of \eqref{eq:sobRd} under multiplication by a constant, homotheties (i.e.\ $u(x)\mapsto \lambda^{d}u(\lambda x)$) and translation, respectively (see also \cite{CENaVi04} for a proof based on optimal transport techniques).

Knowing the extremizers naturally lead to the following stability question: is a function that almost saturates \eqref{eq:sobRd} close to an extremizer? The answer is affirmative and given in the following statement:
\begin{theorem}[Qualitative stability of Sobolev inequality in $\R^d$]\label{thm:stabsobrd}
Let $p,d$ as above. Then for every $\eps>0$ there is $\delta=\delta(\eps,d,p)>0$ such that
\[
\frac{\|u\|_{L^{p^*}}}{\|\d u\|_{L^p}}\geq S_{p,d}-\delta\qquad\Rightarrow\qquad \inf_{a,b,y_0}\frac{\|\d(u-U_{a,b,y_0})\|_{L^p}}{\|\d u\|_{L^p}}\leq\eps
\]
\end{theorem}
In principle, one might try to prove such result by compactness: take a sequence $(u_n)\subset  \dot W^{1,p}(\R^d)$ such that $\delta_n:=\frac{\|u\|_{L^{p^*}}}{\|\d u\|_{L^p}}-S_{p,d}\downarrow0$. If for some reason $(u_n)$ converges to some limit non-zero function $u_\infty$  in $\dot W^{1,p}(\R^d)$, then it is easy to see that $u_\infty$ must realize the equality in \eqref{eq:sobRd}, and thus be one of the $U_{a,b,y_0}$, whence the conclusion follows. Unfortunately, this plan can't work this easily, precisely because of the invariances described above that create loss of compactness. 

Understanding how compactness might fail leads to the celebrated concentration compactness principle developed by Lions \cite{LionsLimCasI}, \cite{LionsLimCasII}, \cite{LionsLocComI}, \cite{LionsLocComII}. Say that we normalized the $u_n$'s so that $\|u_n\|_{L^{p^*}}=1$ for every $n$  and let $\rho_n:=|u_n|^{p^*}$ be the corresponding probability density telling `where the mass is located'. Then,  concentration-compactness tells that for some non-relabeled subsequence of $(\rho_n)$ exactly one of the following occurs:
\begin{itemize}
\item[i)] {\sc Compactness} $\rho_n$ weakly converges to some probability measure $\mu$ on $\R^d$.
\item[ii)] {\sc Escape at infinity} for some $(x_n)\subset \R^d$ with $|x_n|\to\infty$ the sequence $\rho_n(\cdot-x_n)$ weakly converges to some probability measure $\mu$ on $\R^d$.
\item[iii)] {\sc Dissipation} For any $(x_n)\subset \R^d$ and any $R>0$ we have $\lim_n\int_{B_R(x_n)}\rho_n= 0$
\item[iv)] {\sc Dichotomy} There is $\lambda\in(0,1)$, so that for every $\eps>0$ there are points $(x_n)\subset\R^d$ and $R>1$ so that
\[
\begin{split}
\lims_{n\to\infty}\Big|\int_{B_R(x_n)}\rho_n\,\d\mathcal L^d-\lambda\Big|&<\eps,\\
\lims_{n\to\infty}\Big|\int_{\R^d\setminus B_{R_n}(x_n)}\rho_n\,\d\mathcal L^d-(1-\lambda)\Big|&<\eps,\qquad\text{for some }R_n\uparrow+\infty.
\end{split}
\]
\end{itemize}
These four possibilities are in place for any sequence of probability measures on $\R^d$. Let us consider what happens in our case,  where the densities are given by $\rho_n:=|u_n|^{p^*}$ for $u_n$ as above. In this case it is easy to exclude dichotomy, because of strict subadditivity  of the ratio $\frac{\|u\|_{L^{p^*}}}{\|\d u\|_{L^p}}$: if, say, $u_1,u_2\in \dot W^{1,p}(\R^d)$ have disjoint support, then 
\begin{equation}
\label{eq:strictconv}
\frac{\|u_1+u_2\|_{L^{p^*}}}{\|\d (u_1+u_2)\|_{L^p}}<\frac{\|u_1\|_{L^{p^*}}}{\|\d u_1\|_{L^p}}+\frac{\|u_2\|_{L^{p^*}}}{\|\d u_2\|_{L^p}}
\end{equation}
and making this quantitative shows that an optimizing sequence for   \eqref{eq:sobRd} cannot be made by functions whose mass is almost split into two (or more) parts.

Dissipation can occur, but we can easily modify our given sequence $(u_n)$ into another for which it doesn't happen: put $u_{n,\lambda_n}:=c_{n} u_n(\lambda_n x)$ for $c_n$ normalizing constant and $\lambda_n\uparrow\infty$ chosen so that $\int_{B_1(0)} |u_{n,\lambda_n}|^{p^*}\geq \tfrac12$. Clearly $\frac{\|u_n\|_{L^{p^*}}}{\|\d u_n\|_{L^p}}=\frac{\|u_{n,\lambda_n}\|_{L^{p^*}}}{\|\d u_{n,\lambda_n}\|_{L^p}}$ and for $(u_{n,\lambda_n})$ dissipation does not occur.

Escape at infinity can also be easily ruled out by replacing $u_n$ with $u_n(\cdot-x_n)$.

Thus  we end up dealing with the `compactness case' only. In other words by studying how compactness might fail and relying on some strict subadditivity, we restored compactness. This is the heart of the concentrated compactness principle. Notice that reducing to such compact case  is not sufficient to conclude:  we wish  our sequence $(u_n)$ to converge to an extremal function for \eqref{eq:sobRd} but we only know that $|u_n|^{p^*}$ weakly converges to some probability measure. Such weak limit might very well be a Dirac mass (e.g.\ if $u_n=U_{a_n,b_n,0}$ for $b_n\uparrow\infty$), in which case we need `dilate' the functions by considering $u_{n,\lambda_n}:=c_{n} u_n(\lambda_n (x-x_n))$ for suitable $(x_n)\subset\R^d$ and $\lambda_n\downarrow0$. Once the correct scale has been found, with some further work one can show that the sequence truly converges to an extremizer for  \eqref{eq:sobRd}\footnote{I'm skipping few important steps of the proofs here, as in principle the limit can be any measure except a Dirac delta: it is the structure of the problem at hand that ultimately gives the desired convergence to an extremizer.}, and this must be   one of the $U_{a,b,y_0}$, whence the conclusion follows. 

The procedure just described is the archetypical application of Lions' concentration compactness principle. Before passing to the manifold case let me mention that estimate in Theorem \ref{thm:stabsobrd} has been made \emph{quantitative}, i.e.\ with an explicit dependance of $\delta$ on the data: see  \cite{BianchiEgnell91} for $p=2$ and \cite{CianchiFuscoMaggiPratelli09},  \cite{FigalliNeumayer19}, \cite{Neumayer19}, \cite{FigalliZhang22} for $p\neq 2$.

Now let us move from $\R^d$ to Riemannian manifolds. There are various Sobolev-type inequalities one might consider under various curvature bounds, here I shall focus on manifolds/spaces with non-negative Ricci curvature and Euclidean volume growth (see e.g.\ \cite{He99}, \cite{DrHe02} for more on the topic). For the latter, recall that the Bishop-Gromov inequality tells that on a $\CD(0,N)$ space $(\X,\sfd,\mm)$ the quantity $\frac{\mm(B_r( x))}{r^N}$ is non-increasing in $r$ for any $x\in\X$. Then one defines the Asymptotic Volume Ratio $\AVR(\X)$ of $\X$ as
\begin{equation}
\label{eq:defavr}
\AVR(\X):=\lim_{R\to\infty}\frac{\mm(B_R(x))}{\omega_NR^N}=\inf_{R>0}\frac{\mm(B_R(x))}{\omega_NR^N}\qquad\forall x\in \X,
\end{equation}
where $\omega_N=\frac{\pi^{ N/2}}{\int_0^\infty t^{ N/2}e^{-t}\,\d t}$ is, for $N\in\N$, the volume of the Euclidean $N$-dimensional ball. Then we say that $\X$ has Euclidean volume growth if $\AVR(\X)>0$. Several geometric inequalities on manifolds with ${\rm Ric}\geq0$ are related to the Asymptotic Volume Ratio. This is the case, for instance, of  the Sobolev inequality:
\begin{theorem}[Sharp Sobolev  inequality]\label{thm:sharpsobman} Let $M$ be a Riemannian manifold with ${\rm Ric}\geq 0$ and $\dim\leq N$. Then
\begin{equation}
\label{eq:sobvar}
\|u\|_{L^{p^*}}\leq S_{p,N}\AVR(M)^{-\frac1N} \|\d u\|_{L^p}  \qquad\forall u\in \dot W^{1,p}(M),\qquad p^*:=\tfrac{Np}{N-p}
\end{equation}
and the inequality is sharp, meaning that we can always find $(u_n)\subset \dot W^{1,p}(M)$ such that $\frac{\|u\|_{L^{p^*}}}{\|\d u\|_{L^p}}\to S_{p,N}\AVR(M)^{-\frac1N}$.
\end{theorem}
Here $S_{p,N}$ is, for $N$ integer, the Euclidean constant appearing also in \eqref{eq:sobRd}. Notice that since $\AVR(M)\leq 1$ (by Bishop-Gromov monotonicity and the trivial limit $\lim_{R\downarrow0}\frac{\mm(B_R(x))}{\omega_NR^N}=1$), inequality \eqref{eq:sobvar} is worse than the analogue Euclidean one.

Theorem \ref{thm:sharpsobman}  has been proved in \cite{BalKri23} and generalized to the stable class of $\CD(0,N)$ spaces in \cite{NoVi22} (see also the earlier contributions \cite{Ledoux99}, \cite{Xia01}). The proof is based on the sharp isoperimetric inequality, also established there (see Theorem \ref{thm:sharpisoavr}) that in turn is proved starting from the Bishop-Gromov inequality. In \cite{BalKri23} is has also been proved that if $M$ is smooth, then no smooth non-negative function $u$ satisfies the equality in \eqref{eq:sobvar}, unless $M=\R^N$. 

The next result, proved in \cite{NobiliViolo22}, shows that extremizers exist in the non-smooth category (Theorems \ref{thm:extrsobman} and \ref{thm:qualstabsob} below are stated in the more tractable case $p=2$, but it seems reasonable to expect that the techniques used for proving them can be pushed to obtain similar results for general $p$). 

Below $\dot W^{1,2}(\X)$ is the completion of $\Lip_\bs(\X)$ w.r.t.\ the norm $u\mapsto\|\d u\|_{L^2}$:
\begin{theorem}[Space-Function extremizers for the Sobolev inequality]\label{thm:extrsobman}
Let $\X$ be $\RCD(0,N)$ and $u\in\dot W^{1,2}(\X)$ not identically zero. Then $u$ realizes the equality in \eqref{eq:sobvar} for $p=2$ if and only if $\X$ is a $N$-cone and $u=U_{a,b,\bar x}$, where $\bar x$ is  one of the tips of $\X$ and
\begin{equation}
\label{eq:extrem}
U_{a,b,\bar x}:=a(1+b\sfd^2(\cdot,\bar x))^{\frac{2-N}2}.
\end{equation}
\end{theorem}
Notice that here, as customary, it is important to assume $\RCD$ in place of $\CD$ to get metric informations. I haven't defined the notion of $N$-cone: informally, the cone construction is purely metric and it is the one that, when performed on $S^{N-1}$ produces $\R^N$. In writing $N$-cone we also give information on how the measure on the cross section produces the one on the cone: it scales with ``(distance from the tip)$^{N-1}$" (as in the from-$S^{N-1}$-to-$\R^N$ case). Theorem \ref{thm:extrsobman} is proved by `symmetrization' \`a la P\'olya-Szeg\H{o} by comparing the original space/function with a weighted half line and a suitably defined function on it. 

Knowing the shape of the extremizers leads, as in the Euclidean setting, to the stability question, which is answered by the following result, also proved in \cite{NobiliViolo22}:
\begin{theorem}\label{thm:qualstabsob}
For every $N>2$, $\bar v>0$ and $\eps>0$ there is $\delta=\delta(\eps,N,\bar v)$ such that the following holds. Let $M$ be a smooth Riemannian manifold with ${\rm Ric}\geq0$, dimension $\leq N$ and $\AVR(M)\geq\bar v$. Assume that $u\in\dot W^{1,2}(M)$ is non-zero with $\frac{\|u\|_{L^{2^*}}}{\|\d u\|_{L^2}}\geq S_{2,N}\AVR(M)^{-\frac1N}-\delta$. 

Then there are $a,b\in\R$, $b>0$ and $\bar x\in M$ such that 
\[
\frac{\|\d(u-U_{a,b,\bar x})\|_{L^2}}{\|\d u\|_{L^2}}\leq \eps,
\]
where $U_{a,b,\bar x}$ is as in \eqref{eq:extrem}.
\end{theorem}
Notice that we cannot conclude that $M$ is $\eps$-close to a cone because the statement is invariant under scaling of the distance, while (pointed) mGH-distance is not. This last result is actually proved in the class of $\RCD(0,N)$ spaces, but in the spirit of this chapter I'm stating it for smooth manifolds.

In any case, the proof passes necessarily from the non-smooth category and goes as follows. Ideally, we would like to apply the principles of concentration-compactness as in the Euclidean case: it is not hard to check that for any sequence $(\rho_n)$ of probability measures on $M$, up to subsequences exactly one of $(i),(ii),(iii),(iv)$ holds. The problem is that the underlying space is not anymore as symmetric as $\R^d$, so that we don't have dilations/translations at disposal to get the mass `in the correct place at the correct scale'. What we can do, however, is to dilate/translate the space itself; in other words, for given $(\X,\sfd,\mm)$ we define:

\noindent \emph{Tangent cone at $\bar x\in\X$} any pointed-mGH limit of  $(\X,\lambda_n\sfd,c_n\mm,\bar x)$ for some $\lambda_n\uparrow\infty$,

\noindent\emph{Asymptotic cone} any pointed-mGH limit of  $(\X,\lambda_n\sfd,c_n\mm,\bar x)$ for some $\lambda_n\downarrow 0$,

\noindent \emph{Limit space} any pointed-mGH limit of $(\X,\sfd,c_n\mm,x_n)$ for some $(x_n)\subset\X$ with $\sfd(x_0,x_n)\to\infty$.

Here the constant $c_n$ is chosen so that the unit ball centered at the given point has mass 1. I haven't given the precise definition of pointed convergence. A non-precise one is: normalized balls centered in the given points mGH-converge. Obviously, even if our starting space was a smooth manifold, asymptotic cones and limit space can be non-smooth, and clearly if the manifold had non-negative Ricci, such cone/space is going to be $\RCD(0,N)$. Also, relevantly for what comes next, notice that the Asymptotic Volume Ratio is upper semicontinuous under pointed mGH convergence (being an infimum of volume of balls), and thus in these limit spaces the same Sobolev inequality as in the original $\RCD(0,N)$ space holds (and possibly a better one if $\AVR$ increased in the limit).

Now for the proof of Theorem \ref{thm:qualstabsob} take a sequence $(M_n)$ of smooth Riemannian manifolds with ${\rm Ric}\geq 0$, ${\rm dim}\leq N$ and $\AVR(M_n)\geq \bar v>0$. Let  $u_n\in\dot W^{1,2}(M_n)$ be with $\int|u_n|^{2^*}=1$ for every $n\in\N$ and $
S_{2,N}\AVR(M_n)^{-\frac1N}-\frac{1}{\|\d u_n\|_{L^2}}\downarrow0$: if we can prove that the $u_n$'s converge in energy to an extremizer for the Sobolev inequality in a $\RCD(0,N)$ space with $\AVR\geq \bar v$, then by Theorem \ref{thm:extrsobman} we are done. Thus let  $\rho_n:=|u_n|^{2^*}$, embed all the $M_n$'s in a common space and notice that:
\begin{itemize}
\item[-] Dichothomy can be excluded exactly as before via the strict subadditivity in \eqref{eq:strictconv}.
 
\item[-]If dissipation occurs, then we can consider the $u_n$'s along a properly chosen scalings $(M_n,\lambda_n\sfd,c_n\mm,\bar x)$ for $\lambda_n\downarrow 0$: by Gromov's (pre)compactness we can find a limit spaces $\X$ (that we can think of as an asymptotic cone for the sequence $(M_n)$) and, by $\Gamma$-convergence of the Cheeger energies, an extremizer $u$ for the Sobolev inequality in $\X$, which is exactly what we wanted. Notice that $\AVR(\X)\geq\bar v$ because $\AVR$ is upper semicontinuous w.r.t.\ pointed mGH-convergence (as it is the inf of volume of balls, and these pass to the limit) and that $\X$  is a space obtained as limit, hence will typically be non-smooth, whence the necessity of being able to deal with non-smooth geometries.

\item[-] Escaping of the mass to infinity can be handled like the dissipation. The only difference is that one looks at limit spaces rather than asymptotic cones.

\item[-] It remains the compactness case. This also can be treated along similar lines, possibly after considering a sequence of dilations - and thus up to pass to a tangent space - to find an extremizer. Here it is conceptually relevant to notice that in practice, since we know that extremizers do not exist in the smooth category outside $\R^d$, we are \emph{never} in this `compact scenario' if our starting manifold is not $\R^d$.
\end{itemize}
For the stability of other Sobolev inequalities (e.g.\ on compact manifolds) see \cite{NoVi22} and \cite{NobiliViolo22}.
\subsection{Sharp concavity of isometric profile}

The isoperimetric inequality is arguably one of the most important geometric inequalities. As this note is titled after De Giorgi and Gromov, let me mention that the former gave the first complete proof on $\R^d$ (by investigating the compactness properties of Caccioppoli's sets and lower semicontinuity of the perimeter functional, see \cite{DeGiorgiSelected} for more informations and detailed references), while the latter realized the relevance of lower Ricci curvature bounds in this matter (by generalizing an earlier work of Levy valid for convex hypersurfaces in $\R^d$, see \cite{Gromov07}) and proving what is now know as Levy-Gromov isoperimetric inequality:  if ${\rm Ric}_M\geq K>0$, then
\begin{equation}
\label{eq:LGiso}
\frac{{\rm Per}_M(E)}{{\rm vol}(M)}\geq \frac{{\rm Per}_S(B)}{{\rm vol}(S)}\qquad\forall E\subset M,
\end{equation}
where $S$ is the sphere with ${\rm dim}(S)={\rm dim}(M)$ and ${\rm Ric}_S\equiv K$ and $B\subset S$ is a ball with $\frac{{\rm vol}(B)}{{\rm vol}(S)}=\frac{{\rm vol}(E)}{{\rm vol}(M)}$. In other words, the sphere satisfies the `best normalized isoperimetric inequality' among manifolds of same dimension and bigger Ricci curvature. 

Notably, inequality \eqref{eq:LGiso} has been extended to a large class of $\CD(K,N)$ spaces, including $\RCD$ ones, in  \cite{CavMon15}. The technique used has little to do with those I presented in this manuscript, and is rather related to the so-called `needle decomposition' or `localization technique', that in some sense allows to reduce the study of some relevant geometric quantity from the original metric measure space to a suitable family of 1-dimensional metric measure spaces, where things are more tractable. For an overview on this and  detailed bibliography I refer to  the survey \cite{Cavalletti17}, here I just mention that even these tools have been useful in deriving new informations about the smooth Riemannian world, see for instance \cite{CMM19}.

\medskip

Studying the isoperimetric inequality on a space $(\X,\sfd,\mm)$ amounts in understanding the shape of the so-called isoperimetric profile function ${\sf I}_\X:\R^+\to\R^+$ defined as 
\[
{\sf I}_\X(v):=\inf\{{\rm Per}_\X(E)\ :\ E\subset\X,\ \mm(E)=v\},
\]
where here and below the perimeter ${\rm Per}_\X(E)$ of the set $E\subset\X$ is defined as   the Total Variation of the characteristic function, i.e.\ 
\begin{equation}
\label{eq:perch1}
{\rm Per}(E):=\ch_1(\nchi_E)
\end{equation}
in the notation introduced in Section \ref{se:gtv}. The result I want to present here is:
\begin{theorem}\label{thm:concI}
Let $M$ be a Riemannian manifold with ${\rm Ric}\geq K$ and ${\rm dim}\leq N$ with $\inf_{x\in M}{\rm vol}(B_1(x))>0$. Then
\begin{equation}
\label{eq:concI}
\psi''\leq-\frac{KN}{N-1}\psi^{\frac{2-N}{N}}\qquad\text{ for }\qquad\psi:=({\sf I}_M)^{\frac N{N-1}}.
\end{equation}
\end{theorem}
Here the function $\psi$ is not necessarily $C^2$, thus inequality \eqref{eq:concI} should be intended in the weak sense of barrier/viscosity/distributions (these are equivalent in this setting).

Theorem \ref{thm:concI} has been proved for compact manifolds in \cite{MorJoh00}, \cite{Bay04}, \cite{NiWan16}, see also the earlier contributions \cite{BavPan86}, \cite{Gal88}. The extension to the non-compact setting requires the $\RCD$ theory and has been obtained  in \cite{APPS22b} in the more general setting of  non-collapsed $\RCD$ spaces.

To get an idea of why  \eqref{eq:concI} holds, let  $E$ be a smooth minimizer for the perimeter among sets with given volume $v$ and put $E^r:=\{x:\sfd(x,E)\leq r\}$ for $r\geq 0$ and $E^r:=\{x:\sfd(x,M\setminus E)>|r|\}$ for $r<0$. The minimality of $E$ ensures that its mean curvature $H$ is constant, and by direct computation we have
\begin{equation}
\label{eq:der1pe}
\frac{\d}{\d r}{\rm Per}(E^r)\restr{r=0}=H\,{\rm Per}(E).
\end{equation}
A further differentiation gives
\[
\frac{\d^2}{\d r^2}{\rm Per}(E^r)\restr{r=0}=\int_{\partial E}H^2-\|{\rm II}\|^2-{\rm Ric}(\nu_E,\nu_E)\,\d\mathcal H^{N-1},
\]
where ${\rm II},\nu_E$  are  the second fundamental form and the outer unit normal of $E$. Thus ${\rm Ric}\geq K$ and ${\rm dim}\leq N$ gives
\[
\frac{\d^2}{\d r^2}{\rm Per}(E^r)\restr{r=0}\leq \big(\frac{N-2}{N-1}H^2-K\big){\rm Per}(E)
\]
and since ${\rm Per}(E^r)\geq {\sf I}_M(\mathcal H^N(E^r))$ for any $r$, with equality at $r=0$, \eqref{eq:concI} follows by comparison.

By nature of the argument,  two ingredients play a key role: existence of the minimizer and the possibility of estimating the first and second derivative of ${\rm Per}(E^r)$ at $r=0$ (for this, regularity of the minimizer plays a role).

Existence of minimizers is easy to obtain if $M$ is compact, and this is why Theorem \ref{thm:concI} was first obtained on compact manifolds, but in general is false (see \cite{Rit01} and \cite{AFP21} for examples). The problem is that a minimizing sequence $(E_n)$ of sets can have part of the mass `escaping at infinity'. Thus, much like in the previous section, the authors of \cite{APPS22b}  tackled this problem via concentration-compactness (developed in the metric setting independently from \cite{NobiliViolo22}, see also \cite{AntonelliNardulliPozzetta22}, based on \cite{AFP21} and \cite{Nard14}). Let us analyze the possibilities $(i),\ldots,(iv)$ mentioned in the previous section. If we are in the `compact' situation $(i)$ we find a minimizer in our manifold, which is the best scenario. Escaping of the mass at infinity can certainly occur, in which case we end up with an isoperimetric set (by the $\glims$ inequality  \eqref{eq:glimischp} for the Total Variation) in a limit space. 

Dicothomy can also occur, because the strict concavity in \eqref{eq:strictconv} actually fails for $p=p^*=1$ (but see below), so we must be ready to deal with multiple `limit sets in limit spaces'.  Still, at least   dissipation can be excluded thanks to a non-sharp isoperimetric inequality valid for sets of small volumes in $\CD(K,N)$ spaces provided $\inf_{x\in\X}\mm(B_1(x))>0$ (see \cite{APP22} and the original argument in \cite{CouSal93} valid in smooth manifolds) that tells that for any such space $\X$ we have 
\begin{equation}
\label{eq:smallvol}
{\rm Per}(E)\geq C \mm(E)^{1-\frac1N}+o(\mm(E)^{1-\frac1N})\qquad\forall E\subset\X.
\end{equation}
This grants  that dividing the mass into too many small pieces is for sure not convenient and thus that there is at most a countable number of `limit sets in limit spaces', whose total mass is that of the sets in the original sequence. By the $\glims$ inequality  \eqref{eq:glimischp} for the Total Variation, any such limit set must be isoperimetric in its limit space. Thus by a regularity result proved in \cite{APP22} (see also \cite{APPV23}) we see that each of these limit sets is open. We can then  use this information  to conclude that in fact there is only a finite number of  `limit sets in limit spaces':  by \eqref{eq:smallvol} one pays a perimeter of at least $Cv^{1-\frac1N}$ for a set of volume $v$, but if we suitably enlarge a given open set adding volume $v$, we only add $cv$ to the perimeter, thus it is not convenient to have limit sets of arbitrary small volume.

We thus replaced the unknown existence of isoperimetric set in the given manifold with a finite number of isoperimetric sets in limit spaces. Now, if we are able to suitably bound from above the second derivative of the isoperimetric profile function in these limit spaces at the given volumes, we could add up the resulting bounds to get the desired estimate \eqref{eq:concI}.  This last step is quite technical (after all, we are speaking about taking two derivatives of a functional depending on a co-dimension one object, in a setting where neither the underlying the space nor the set considered are  smooth!). As in the outline above, one starts from an isoperimetric set $E$ and wants to estimate  the derivatives of ${\rm Per}(E^r)$: without entering in the details, let me just mention that this is achieved via a careful study of the Laplacian of the signed distance function from $E$ (see \cite{APPS22b}, inspired by \cite{CafCor95}, \cite{Petrunin03} and  based on the earlier \cite{MS21}, \cite{CavMon20}, \cite{Gigli12}).

\medskip

Theorem \ref{thm:concI} has a further application to the smooth category when put in conjunction with the following sharp isoperimetric inequality:
\begin{theorem}\label{thm:sharpisoavr}
Let $\X$ be $\CD(0,N)$. Then 
\begin{equation}
\label{eq:isoavr}
{\sf I}_\X(v)\geq N(\AVR(\X)\omega_N)^{\frac1N} \,v^{\frac{N-1}N}\quad\forall v>0
\end{equation}
and this inequality is sharp for large volumes, i.e.
\begin{equation}
\label{eq:sharpisoavr}
\lim_{v\to\infty}\ {{\sf I}_\X(v)}\,v^{-\frac{N-1}N} =N(\AVR(\X)\omega_N)^{\frac1N}.
\end{equation}
\end{theorem}
Here the $\lims$ inequality in \eqref{eq:sharpisoavr} is a trivial consequence of the Bishop-Gromov inequality in its spherical version
\[
r\quad\mapsto\quad\frac{{\rm Per}(B_r(x))}{r^{N-1}}\quad\text{ agrees a.e.\ with a non-increasing function}
\]
(this was stated in   \cite[Theorem 2.3]{Sturm06II} with the outer Minkowski content in place of the perimeter; the version above can then be easily deduced via coarea). The non-trivial, even in the smooth category, bound is \eqref{eq:isoavr}: this has been obtained in \cite{BalKri23} under the stated assumptions (with a careful study of the Bishop-Gromov inequality), see also \cite{CavMan22},  \cite{AFM18}, \cite{APPS22} and the earlier contributions \cite{CouSal93}, \cite{Bre23}, \cite{Joh21} for results in the smooth setting.

\medskip

Using Theorems \ref{thm:concI} and \ref{thm:sharpisoavr} one can get the following interesting existence result. In the statement below we say that $\X$ splits a line if it is isomorphic to the product of the Euclidean line $\R$ and some other space $(\X',\sfd',\mm')$ as in the splitting Theorem \ref{thm:splittingrcd}.
\begin{theorem}\label{thm:ABFP22}
Let $M$ be a smooth Riemannian manifold with ${\rm Ric}\geq 0$ and strictly positive {asymptotic volume ratio}. Assume that neither $M$ nor any asymptotic cone split a line.

Then there is $V>0$ such that for any $v\geq V$ there is an isoperimetric set of volume $v$.
\end{theorem}
This theorem  actually holds on non-collapsed $\RCD(0,N)$ spaces (see  \cite[Theorem 1.2]{APPS22}), but in the spirit of this chapter I'm stating it in the smooth case. The proof goes along the following lines:
\begin{itemize}
\item The key role of Theorem \ref{thm:concI}  is to deduce that the isoperimetric profile ${\sf I}_M$ is strictly subadditive, i.e.\ 
 \begin{equation}
\label{eq:Isubadd}
{\sf I}_M(v_1+v_2)<{\sf I}_M(v_1)+{\sf I}_M(v_2)\quad\forall v_1,v_2>0.
\end{equation}
Indeed,  since ${\sf I}_M(v)\to 0$ as $v\to 0$, the above  follow from the strict concavity of  ${\sf I}_M$ that in turn is a consequence of \eqref{eq:concI} with $K=0$.
\item Such strict subadditivity and the concentration-compactness principle already described ensure that there is no dichothomy for a minimizing sequence $(E_n)$ for the perimeter: either it converges in $L^1$ to a set, which therefore will be isoperimetric, or they converge to some limit set in a limit space $\X$ of $M$. Thus if we exclude this latter possibility we are done.
\item Here is where the assumptions about lines comes into play. Notice that possibly  factoring out lines in $M$, we can always reduce to the case in which $M$ does not split any line. The assumption that the same holds for asymptotic cones is used to deduce that 
\begin{equation}
\label{eq:avrlimiti}
\text{there is $\eps>0$ such that any limit $\X$ of $M$ at infinity satisfies $\AVR(\X)\geq\AVR(M)+\eps$.}
\end{equation}
Indeed, we have already noticed that $\AVR$ is upper semicontinuous w.r.t.\ pointed mGH convergence (because it is an $\inf$ of normalized volume of balls), whence   $\AVR(\X)\geq\AVR(M)$ for any limit space $\X$ of $M$. Also,  by Gromov's (pre)compactness the collection of limit spaces is compact. Now assume that $\X_i$  are limit spaces with $\AVR(\X_i)\leq\AVR(M)+\tfrac1i$: by diagonalization and a suitable blow-down we can find  an asymptotic cone $(\Y,\sfd_\Y,\mm_\Y,\bar y)$ and $\bar y'\neq\bar y$ with $\frac{\mm_\Y(B_r(\bar y'))}{\mm_\Y(B_R(\bar y'))}=\big(\tfrac rR\big)^N$ for any $r,R>0$. By the `volume-cone-to-metric-cone' principle, $\Y$ is also a cone over $\bar y'$, and thus contains a line (the one passing through $\bar y$ and $\bar y'$) and therefore, by the splitting theorem, it splits a line. This is the desired contradiction.
 
 Notice that if $M$ has non-negative sectional curvature, then by Toponogov's theorem (see e.g.\ \cite[Theorem 12.2.2]{Petersen16}) it admits a unique asymptotic cone and it splits a line iff $M$ does, so in this case (after taking out the maximal Euclidean factor from $M$) there is no need of  any assumption about lines on $M$ to deduce \eqref{eq:avrlimiti}  (and Theorem  \ref{thm:ABFP22} is new and relevant also in this case, see \cite[Corollary 1.3]{ABFP22}).

\item Thus let $(E_n)$ be a minimizing sequence on $M$ for the perimeter among sets with volume $v>0$. Assume that it converges to some set $E\subset\X$ with $\X$ limit space of $M$.  Then by lower semicontinuity of the perimeter along a sequence of $\RCD$ spaces (Theorem \ref{thm:stabtv}) we get
\[
\frac{{\sf I}_M(v)}{v^{\frac{N-1}N}}=\lim_{n\to\infty}\frac{{\rm Per}_M(E_n)}{v^{\frac{N-1}N}}\geq \frac{{\rm Per}_\X(E)}{v^{\frac{N-1}N}}\stackrel{\eqref{eq:isoavr}}\geq N(\AVR(\X)\omega_N)^{\frac1N}\stackrel{\eqref{eq:avrlimiti}}\geq N((\AVR(M)+\eps)\omega_N)^{\frac1N}.
\]
This, however, cannot be true for $v\gg1$ by \eqref{eq:sharpisoavr}.
\end{itemize}

For much more on this topic and  detailed references I refer to the recent survey \cite{PozzettaSurvey}.

\subsection{A version of the Bonnet-Myers estimate for ${\rm Ric}> 0$}\label{se:bmricci0}
The classical Bonnet-Myers estimate tells that if a manifold $M$ has Ricci $\geq K>0$ and dimension $\leq N$, then its diameter is bounded above by $\pi\sqrt{\tfrac{N-1}K}$.  In particular any such manifold is compact.

Clearly, replacing the assumption  $\geq K>0$ with $>0$ cannot lead to any diameter bound and the simple example of a paraboloid shows that compactness might fail as well. Still, in the example of the paraboloid the Ricci curvature in the radial direction goes to 0 faster than the one in the tangential direction. One might therefore wonder if a pinching condition on the Ricci curvature forces compactness, as conjectured by Hamilton (see also \cite{Lott19}); recall that one says that the Ricci curvature of $(M,g)$ is pinched if it is non-negative and there is $c>0$ such that at any point its  highest eigenvalue is bounded above by $c$ times the lowest one. Equivalently, if for some $c>0$ we have
\begin{equation}
\label{eq:pinch}
{\rm Ric}\geq c\,{\rm Scal}\,g.
\end{equation}
We then have:
\begin{theorem}\label{thm:Riccipinch}
Let $M$ be a three dimensional manifold with Ricci curvature strictly positive  and pinched.

Then $M$ is  compact.
\end{theorem}
Notice that the assumption is scale invariant (a $c$-pinched manifold remains $c$-pinched after scaling), thus the conclusion has to be scale invariant as well, whence the lack of any diameter bound. This result has an extrinsic analogue, valid in any dimension, proved by Hamilton in \cite{Ham94}: a convex hypersurface in $\R^d$ with pinched second fundamental form is either compact or flat. In principle, one might  wonder whether an $n$-dimensional analogue of \ref{thm:Riccipinch} holds as well, but this is not clear (and possibly requires an assumption different from pinching of the Ricci); in any case the strategy for proving the above is strictly related to dimension 3.

To get an idea on why Theorem \ref{thm:Riccipinch} should hold, assume that $M$ is a cone outside some compact set (here I more or less follow the introduction of \cite{Lott19}). Then the Ricci curvature vanishes in the radial direction, hence $M$ is Ricci flat by the pinching condition. Since the dimension is 3, the manifold is actually flat. Then the cross section of such cone is a 2-dimension manifold with constant curvature $\equiv 1$ (by direct computation) and connected (because the truncated cone is connected). Hence it is either $S^2$ or $\R P^2$. However,  the latter does not bound any manifold (it has odd Euler characteristic, while boundaries have even characteristic), while clearly by construction our cross section bounds the `truncated part' of the manifold. Hence the cross section is $S^2$. Then the flatness of $M$ forces it to be a cone and its smoothness such cone to be $\R^3$.

Very roughly said, the actual proof of Theorem  \ref{thm:Riccipinch} follows this principle and works by contradiction: one assumes that $M$ is non-compact and non-flat and derives an absurdum. Here one uses Ricci flow to replace the original metric with a `better one' that is more cone-like, and Alexandrov/$\RCD$ theory to study the asymptotic of such cone.

\medskip
   
The first proof of Theorem \ref{thm:Riccipinch} has been given  in \cite{CheZhu00} under the additional assumption that the sectional curvature is non-negative as well. Here no non-smooth theory appears; instead the authors study the long time asymptotic of the Ricci flow and show, using also some estimates due to Hamilton, that $M$ cannot be both non-compact and non-flat.

The assumption on non-negativity of sectional curvature has been weakened, and finally removed, in the more recent series of papers \cite{Lott19}, \cite{DSS22}, \cite{LT22}  with the last one having  the proof of  Theorem \ref{thm:Riccipinch} as stated. The rough outline of the proof given below comes from \cite{DSS22} and builds upon ideas in \cite{Lott19}.
\begin{enumerate}[label=(\roman*)]
\item Let $(M,g)$ be non-compact and non-flat and $c$-Ricci pinched. Then there is an immortal, i.e.\ defined for any $t\geq0$, Ricci flow $(g_t)$ starting from $g$, it is $c$-Ricci pinched and satisfies $|{\rm Riem}_{g_t}|\leq\frac Ct$ for some $C$ depending only on $c$. This existence result together with the curvature bound is the main result in \cite{LT22} (previous works gave existence under additional curvature assumptions) and the one that ultimately led to the proof of Theorem \ref{thm:Riccipinch}.
\item\label{it:AVRpos} By \cite[Propositions 1.5, 3.1 and 4.1]{Lott19}  we see that $g_t$ has  uniformly positive Asymptotic Volume Ratio, i.e.\ $\AVR(g_t):=\lim_{r\to\infty}\frac{{\rm Vol}_{g_t}(B_r(x))}{\omega_3r^3}\geq \bar v$ for every $t>0$ and some $\bar v>0$ independent on $t$. 
\item Scale down the flow by putting $g^s_t:=\tfrac1sg_{st}$ and notice that both the $\AVR$, the curvature bound and the pinching remain unchanged. Then by mGH-compactness we  can find $s_n\to\infty$ and a limit flow $(g^\infty_t)_{t>0}$ defined on some smooth manifold $M^\infty$ which satisfies the same curvature bound. Here it is used the information that the $\AVR$ remains positive (to avoid collapsing, but the  proof is still non-trivial  - see \cite[Propositions 1.6, Section 3]{Lott19}).
\item Building upon \cite{SimTop21} one can prove that as $t\downarrow0$ the distances $\sfd_t$ induced by $g^\infty_t$ on $M^\infty$ converge in the pointed-Gromov-Hausdorff sense to an asymptotic cone $C$ of the original manifold $M$.  Since $M$ has positive $\AVR$ (by item \ref{it:AVRpos} and upper semicontinuity of $\AVR$ under mGH convergence), $C$ is a volume cone, and thus a metric cone \cite{Cheeger-Colding96}.
\item Clearly, $C$ is also an $\RCD(0,3)$ space, thus by  \cite{Ketterer13} its cross section $F$ is an $\RCD(1,2)$ space.
\item By  \cite{LS18} the cross section $F$ is also an Alexandrov space with curvature $\geq 1$. Notice  that the result in \cite{LS18} is the non-smooth analogue of the consideration, trivial in the smooth category but highly not so in the non-smooth one, that in dimension 2 we have Ricci=sectional (up to suitable identification of tensors).
\item By structural results about low dimensional Alexandrov spaces (see \cite{IRV15}) and topological considerations akin to those given above, one deduces that $F$ is homeomorphic to $S^2$ and that there is a sequence $g_n$ of smooth Riemannian metrics on $S^2$ with curvature $\geq 1$ that is GH-converging to the metric on $F$.
\item Building cones over $(S^2,g_n)$ one sees that $C$ is the limit of a sequence of smooth Riemannian 3-manifolds with non-negative sectional curvature. Using results from \cite{Der16} or alternatively \cite{SchSim13} we have  expanding gradient Ricci solitons  having non-negative curvature starting from each of these approximations, and by passing to the limit we can find an expanding gradient Ricci solitons $(\tilde g_t)_{t>0}$ converging to $C$ as $t\downarrow0$.
\item If we knew that  $\tilde g_t$ was Ricci pinched for some $t>0$, then using also that it is an expanding gradient Ricci solitons we could conclude that it is flat following   \cite{CheZhu00}. A variant of this argument, provided in  \cite{DSS22}, works also if $\tilde g_t$ is almost Ricci-pinched in a suitable sense. Such `almost Ricci pinched' for  $\tilde g_t$ will be obtained by comparing it to $g^\infty_t$.
\item Coming back to the study of $(g^\infty_t)$, by relying on a splitting result by \cite{HochardThesis}, it is proved in \cite[Lemma 81]{DSS22} that outside of the tip, $C$ has only $\R^3$ as tangent spaces.
\item This regularity informations, and the existence of $(1+\eps)$-biLipschitz charts in $C$ (see \cite{BGP92}, \cite{BBI01} and recall that $C$  is an Alexandrov space) is the key ingredient that allows to use the main result in \cite{DSS22}, namely  a stability result which allows one to locally compare solutions to Ricci flow which start locally from the same initial data, in the case that the initial data only has trivial cones as blow ups. This allows to conclude that $(\tilde g_t)$ is almost Ricci pinched and thus, as mentioned above, to produce a contradiction following \cite{CheZhu00}.
\end{enumerate}

{\small
\def\cprime{$'$} \def\cprime{$'$}


\begin{thebibliography}{GGKMS18}

\bibitem[AB84]{homofinsl}
Emilio Acerbi and Giuseppe Buttazzo.
\newblock On the limits of periodic {R}iemannian metrics.
\newblock {\em J. Analyse Math.}, 43:183--201, 1983/84.

\bibitem[ABFP22]{ABFP22}
Gioacchino Antonelli, Elia Bru\`e, Mattia Fogagnolo, and Marco Pozzetta.
\newblock On the existence of isoperimetric regions in manifolds with
  nonnegative {R}icci curvature and {E}uclidean volume growth.
\newblock {\em Calc. Var. Partial Differential Equations}, 61(2):Paper No. 77,
  40, 2022.

\bibitem[ABS21]{AmbBruSem21}
Luigi Ambrosio, Elia Bru\'{e}, and Daniele Semola.
\newblock {\em Lectures on optimal transport}, volume 130 of {\em Unitext}.
\newblock Springer, Cham, [2021] \copyright 2021.
\newblock La Matematica per il 3+2.

\bibitem[ADM14]{Ambrosio-DiMarino14}
Luigi Ambrosio and Simone Di~Marino.
\newblock Equivalent definitions of {$BV$} space and of total variation on
  metric measure spaces.
\newblock {\em J. Funct. Anal.}, 266(7):4150--4188, 2014.

\bibitem[AF14]{AF14}
Luigi Ambrosio and Jin Feng.
\newblock On a class of first order {H}amilton-{J}acobi equations in metric
  spaces.
\newblock {\em J. Differential Equations}, 256(7):2194--2245, 2014.

\bibitem[AFM20]{AFM18}
Virginia Agostiniani, Mattia Fogagnolo, and Lorenzo Mazzieri.
\newblock Sharp geometric inequalities for closed hypersurfaces in manifolds
  with nonnegative {R}icci curvature.
\newblock {\em Invent. Math.}, 222(3):1033--1101, 2020.

\bibitem[AFP21]{AFP21}
Gioacchino Antonelli, Mattia Fogagnolo, and Marco Pozzetta.
\newblock The isoperimetric problem on {R}iemannian manifolds via
  {G}romov--{H}ausdorff asymptotic analysis.
\newblock Preprint, arXiv:2101.12711, 2021.

\bibitem[AGMR12]{AmbrosioGigliMondinoRajala12}
Luigi Ambrosio, Nicola Gigli, Andrea Mondino, and Tapio Rajala.
\newblock Riemannian {R}icci curvature lower bounds in metric measure spaces
  with $\sigma$-finite measure.
\newblock {\em Trans. Amer. Math. Soc.}, 367(7):4661--4701, 2012.

\bibitem[AGS08]{AmbrosioGigliSavare08}
Luigi Ambrosio, Nicola Gigli, and Giuseppe Savar{\'e}.
\newblock {\em Gradient flows in metric spaces and in the space of probability
  measures}.
\newblock Lectures in Mathematics ETH Z\"urich. Birkh\"auser Verlag, Basel,
  second edition, 2008.

\bibitem[AGS12]{AmbrosioGigliSavare-compact}
Luigi Ambrosio, Nicola Gigli, and Giuseppe Savar{\'e}.
\newblock Heat flow and calculus on metric measure spaces with {R}icci
  curvature bounded below---the compact case.
\newblock {\em Boll. Unione Mat. Ital. (9)}, 5(3):575--629, 2012.

\bibitem[AGS13]{AmbrosioGigliSavare11-3}
Luigi Ambrosio, Nicola Gigli, and Giuseppe Savar{\'e}.
\newblock Density of {L}ipschitz functions and equivalence of weak gradients in
  metric measure spaces.
\newblock {\em Rev. Mat. Iberoam.}, 29(3):969--996, 2013.

\bibitem[AGS14a]{AmbrosioGigliSavare11}
Luigi Ambrosio, Nicola Gigli, and Giuseppe Savar{\'e}.
\newblock Calculus and heat flow in metric measure spaces and applications to
  spaces with {R}icci bounds from below.
\newblock {\em Invent. Math.}, 195(2):289--391, 2014.

\bibitem[AGS14b]{AmbrosioGigliSavare11-2}
Luigi Ambrosio, Nicola Gigli, and Giuseppe Savar{\'e}.
\newblock Metric measure spaces with {R}iemannian {R}icci curvature bounded
  from below.
\newblock {\em Duke Math. J.}, 163(7):1405--1490, 2014.

\bibitem[AGS15]{AmbrosioGigliSavare12}
Luigi Ambrosio, Nicola Gigli, and Giuseppe Savar{\'e}.
\newblock Bakry-\'{E}mery curvature-dimension condition and {R}iemannian
  {R}icci curvature bounds.
\newblock {\em The Annals of Probability}, 43(1):339--404, 2015.

\bibitem[AGS17]{savareEMS}
Luigi Ambrosio, Nicola Gigli, and Giuseppe Savar{\'e}.
\newblock Diffusion, optimal transport and {R}icci curvature for metric measure
  spaces.
\newblock {\em Eur. Math. Soc. Newsl.}, (103):19--28, 2017.

\bibitem[AH17]{AH16}
Luigi Ambrosio and Shouhei Honda.
\newblock New stability results for sequences of metric measure spaces with
  uniform {R}icci bounds from below.
\newblock In {\em Measure theory in non-smooth spaces}, Partial Differ. Equ.
  Meas. Theory, pages 1--51. De Gruyter Open, Warsaw, 2017.

\bibitem[AKL89]{AndKroLeB89}
Michael~T. Anderson, Peter~B. Kronheimer, and Claude LeBrun.
\newblock Complete {R}icci-flat {K}\"{a}hler manifolds of infinite topological
  type.
\newblock {\em Comm. Math. Phys.}, 125(4):637--642, 1989.

\bibitem[Amb90]{Ambr90}
Luigi Ambrosio.
\newblock Metric space valued functions of bounded variation.
\newblock {\em Ann. Scuola Norm. Sup. Pisa Cl. Sci. (4)}, 17(3):439--478, 1990.

\bibitem[Amb04]{Ambrosio04}
Luigi Ambrosio.
\newblock Transport equation and {C}auchy problem for {$BV$} vector fields.
\newblock {\em Invent. Math.}, 158(2):227--260, 2004.

\bibitem[Amb18]{AmbrosioICM}
Luigi Ambrosio.
\newblock Calculus, heat flow and curvature-dimension bounds in metric measure
  spaces.
\newblock In {\em Proceedings of the {I}nternational {C}ongress of
  {M}athematicians---{R}io de {J}aneiro 2018. {V}ol. {I}. {P}lenary lectures},
  pages 301--340. World Sci. Publ., Hackensack, NJ, 2018.

\bibitem[AMS15]{AmbrosioMondinoSavare13}
Luigi Ambrosio, Andrea Mondino, and Giuseppe Savar{\'e}.
\newblock Nonlinear diffusion equations and curvature conditions in metric
  measure spaces.
\newblock {\em Mem. Amer. Math. Soc.}, 262(1270):0, 2015.

\bibitem[And90]{anderson1990}
Michael~T. Anderson.
\newblock Short geodesics and gravitational instantons.
\newblock {\em J. Differential Geom.}, 31(1):265--275, 1990.

\bibitem[ANP22]{AntonelliNardulliPozzetta22}
Gioacchino Antonelli, Stefano Nardulli, and Marco Pozzetta.
\newblock The isoperimetric problem {$via$} direct method in noncompact metric
  measure spaces with lower {R}icci bounds.
\newblock {\em ESAIM Control Optim. Calc. Var.}, 28:Paper No. 57, 32, 2022.

\bibitem[APP22]{APP22}
Gioacchino Antonelli, Enrico Pasqualetto, and Marco Pozzetta.
\newblock Isoperimetric sets in spaces with lower bounds on the {R}icci
  curvature.
\newblock {\em Nonlinear Anal.}, 220:Paper No. 112839, 59, 2022.

\bibitem[APPS22a]{APPS22}
Gioacchino Antonelli, Enrico Pasqualetto, Marco Pozzetta, and Daniele Semola.
\newblock Asymptotic isoperimetry on non collapsed spaces with lower {R}icci
  bounds.
\newblock Preprint, arXiv:2208.03739, 2022.

\bibitem[APPS22b]{APPS22b}
Gioacchino Antonelli, Enrico Pasqualetto, Marco Pozzetta, and Daniele Semola.
\newblock Sharp isoperimetric comparison on non collapsed spaces with lower
  {R}icci bounds.
\newblock Preprint, arXiv:2201.04916, 2022.

\bibitem[APPV23]{APPV23}
Gioacchino Antonelli, Enrico Pasqualetto, Marco Pozzetta, and Ivan~Yuri Violo.
\newblock Topological regularity of isoperimetric sets in {PI} spaces having a
  deformation property.
\newblock Preprint, arXiv:2303.01280, 2023.

\bibitem[AST17]{AST17}
Luigi Ambrosio, Federico Stra, and Dario Trevisan.
\newblock Weak and strong convergence of derivations and stability of flows
  with respect to {MGH} convergence.
\newblock {\em J. Funct. Anal.}, 272(3):1182--1229, 2017.

\bibitem[AT14]{Ambrosio-Trevisan14}
Luigi Ambrosio and Dario Trevisan.
\newblock Well posedness of {L}agrangian flows and continuity equations in
  metric measure spaces.
\newblock {\em Anal. PDE}, 7(5):1179--1234, 2014.

\bibitem[AT17]{AT15}
Luigi Ambrosio and Dario Trevisan.
\newblock Lecture notes on the {D}i{P}erna-{L}ions theory in abstract measure
  spaces.
\newblock {\em Ann. Fac. Sci. Toulouse Math. (6)}, 26(4):729--766, 2017.

\bibitem[Aub76]{Aubin76-2}
Thierry Aubin.
\newblock Probl\`emes isop\'{e}rim\'{e}triques et espaces de {S}obolev.
\newblock {\em J. Differential Geometry}, 11(4):573--598, 1976.

\bibitem[Bak85]{Bakry83}
Dominique Bakry.
\newblock Transformations de {R}iesz pour les semi-groupes sym\'etriques. {II}.
  \'{E}tude sous la condition {$\Gamma_2\geq 0$}.
\newblock In {\em S\'eminaire de probabilit\'es, {XIX}, 1983/84}, volume 1123
  of {\em Lecture Notes in Math.}, pages 145--174. Springer, Berlin, 1985.

\bibitem[Bay04]{Bay04}
Vincent Bayle.
\newblock A differential inequality for the isoperimetric profile.
\newblock {\em Int. Math. Res. Not.}, (7):311--342, 2004.

\bibitem[BB11]{Bjorn-Bjorn11}
Anders Bj{\"o}rn and Jana Bj{\"o}rn.
\newblock {\em Nonlinear potential theory on metric spaces}, volume~17 of {\em
  EMS Tracts in Mathematics}.
\newblock European Mathematical Society (EMS), Z\"urich, 2011.

\bibitem[BBI01]{BBI01}
Dmitri Burago, Yuri Burago, and Sergei Ivanov.
\newblock {\em A course in metric geometry}, volume~33 of {\em Graduate Studies
  in Mathematics}.
\newblock American Mathematical Society, Providence, RI, 2001.

\bibitem[BDS22]{BDS21}
Elia Bru\'{e}, Qin Deng, and Daniele Semola.
\newblock Improved regularity estimates for {L}agrangian flows on {${\rm
  RCD}(K,N)$} spaces.
\newblock {\em Nonlinear Anal.}, 214:Paper No. 112609, 26, 2022.

\bibitem[B{\'E}85]{BakryEmery85}
Dominique Bakry and Michel {\'E}mery.
\newblock Diffusions hypercontractives.
\newblock In {\em S\'eminaire de probabilit\'es, {XIX}, 1983/84}, volume 1123
  of {\em Lecture Notes in Math.}, pages 177--206. Springer, Berlin, 1985.

\bibitem[BE91]{BianchiEgnell91}
Gabriele Bianchi and Henrik Egnell.
\newblock A note on the {S}obolev inequality.
\newblock {\em J. Funct. Anal.}, 100(1):18--24, 1991.

\bibitem[BGHZ23]{BGHZ23}
Camillo Brena, Nicola Gigli, Shouhei Honda, and Xingyu Zhu.
\newblock Weakly non-collapsed {$\sf{RCD}$} spaces are strongly non-collapsed.
\newblock {\em J. Reine Angew. Math.}, 794:215--252, 2023.

\bibitem[BGL14]{BakryGentilLedoux14}
Dominique Bakry, Ivan Gentil, and Michel Ledoux.
\newblock {\em Analysis and geometry of {M}arkov diffusion operators}, volume
  348 of {\em Grundlehren der Mathematischen Wissenschaften [Fundamental
  Principles of Mathematical Sciences]}.
\newblock Springer, Cham, 2014.

\bibitem[BGP92]{BGP92}
Yu~Burago, Mikhael Gromov, and Grigori Perelman.
\newblock A. {D}. {A}leksandrov spaces with curvatures bounded below.
\newblock {\em Uspekhi Mat. Nauk}, 47(2(284)):3--51, 222, 1992.

\bibitem[BH91]{BH91}
Nicolas Bouleau and Francis Hirsch.
\newblock {\em Dirichlet forms and analysis on {W}iener space}, volume~14 of
  {\em De Gruyter Studies in Mathematics}.
\newblock Walter de Gruyter \& Co., Berlin, 1991.

\bibitem[BK23]{BalKri23}
Zolt\'{a}n~M. Balogh and Alexandru Krist\'{a}ly.
\newblock Sharp isoperimetric and {S}obolev inequalities in spaces with
  nonnegative {R}icci curvature.
\newblock {\em Math. Ann.}, 385(3-4):1747--1773, 2023.

\bibitem[BKMR21]{BKMR21}
J\'{e}r\^{o}me Bertrand, Christian Ketterer, Ilaria Mondello, and Thomas
  Richard.
\newblock Stratified spaces and synthetic {R}icci curvature bounds.
\newblock {\em Ann. Inst. Fourier (Grenoble)}, 71(1):123--173, 2021.

\bibitem[BL06]{BaLe06}
Dominique Bakry and Michel Ledoux.
\newblock A logarithmic {S}obolev form of the {L}i-{Y}au parabolic inequality.
\newblock {\em Rev. Mat. Iberoam.}, 22(2):683--702, 2006.

\bibitem[Bog07]{Bogachev07}
Vladimir Bogachev.
\newblock {\em Measure theory. {V}ol. {I}, {II}}.
\newblock Springer-Verlag, Berlin, 2007.

\bibitem[BP86]{BavPan86}
Christophe Bavard and Pierre Pansu.
\newblock Sur le volume minimal de {${\bf R}^2$}.
\newblock {\em Ann. Sci. \'{E}cole Norm. Sup. (4)}, 19(4):479--490, 1986.

\bibitem[BPS21]{BPS21}
Elia Bru\`e, Enrico Pasqualetto, and Daniele Semola.
\newblock Rectifiability of {${\rm RCD}(K,N)$}spaces via {$\delta$}-splitting
  maps.
\newblock {\em Ann. Fenn. Math.}, 46(1):465--482, 2021.

\bibitem[Bra02]{Braides02}
Andrea Braides.
\newblock {\em {$\Gamma$}-convergence for beginners}, volume~22 of {\em Oxford
  Lecture Series in Mathematics and its Applications}.
\newblock Oxford University Press, Oxford, 2002.

\bibitem[Bra21]{Br21}
Mathias Braun.
\newblock Vector calculus for tamed {D}irichlet spaces.
\newblock Preprint, arXiv:2108.12374, 2021.

\bibitem[Bra22]{Braun22}
Mathias Braun.
\newblock Heat flow on 1-forms under lower {R}icci bounds. {F}unctional
  inequalities, spectral theory, and heat kernel.
\newblock {\em J. Funct. Anal.}, 283(7):Paper No. 109599, 65, 2022.

\bibitem[Bre11]{Brezis11}
Haim Brezis.
\newblock {\em Functional analysis, {S}obolev spaces and partial differential
  equations}.
\newblock Universitext. Springer, New York, 2011.

\bibitem[Bre20]{Bre23}
Simon Brendle.
\newblock Sobolev inequalities in manifolds with nonnegative curvature.
\newblock Preprint, arXiv:2009.13717, 2020.

\bibitem[BS10]{BacherSturm10}
Kathrin Bacher and Karl-Theodor Sturm.
\newblock Localization and tensorization properties of the curvature-dimension
  condition for metric measure spaces.
\newblock {\em J. Funct. Anal.}, 259(1):28--56, 2010.

\bibitem[BS20]{BS18}
Elia Bru\'{e} and Daniele Semola.
\newblock Constancy of the dimension for {${\rm RCD}(K,N)$} spaces via
  regularity of {L}agrangian flows.
\newblock {\em Comm. Pure Appl. Math.}, 73(6):1141--1204, 2020.

\bibitem[Cav17]{Cavalletti17}
Fabio Cavalletti.
\newblock An overview of {$L^1$} optimal transportation on metric measure
  spaces.
\newblock In {\em Measure theory in non-smooth spaces}, Partial Differ. Equ.
  Meas. Theory, pages 98--144. De Gruyter Open, Warsaw, 2017.

\bibitem[CC95]{CafCor95}
Luis~A. Caffarelli and Antonio C\'{o}rdoba.
\newblock Correction: ``{A}n elementary regularity theory of minimal surfaces''
  [{D}ifferential {I}ntegral {E}quations {\bf 6} (1993), no. 1, 1--13;
  {MR}1190161 (94c:49042)].
\newblock {\em Differential Integral Equations}, 8(1):223, 1995.

\bibitem[CC96]{Cheeger-Colding96}
Jeff Cheeger and Tobias~Holck Colding.
\newblock Lower bounds on {R}icci curvature and the almost rigidity of warped
  products.
\newblock {\em Ann. of Math. (2)}, 144(1):189--237, 1996.

\bibitem[CC97]{Cheeger-Colding97I}
Jeff Cheeger and Tobias~Holck Colding.
\newblock On the structure of spaces with {R}icci curvature bounded below. {I}.
\newblock {\em J. Differential Geom.}, 46(3):406--480, 1997.

\bibitem[CC00a]{Cheeger-Colding97II}
Jeff Cheeger and Tobias~Holck Colding.
\newblock On the structure of spaces with {R}icci curvature bounded below.
  {II}.
\newblock {\em J. Differential Geom.}, 54(1):13--35, 2000.

\bibitem[CC00b]{Cheeger-Colding97III}
Jeff Cheeger and Tobias~Holck Colding.
\newblock On the structure of spaces with {R}icci curvature bounded below.
  {III}.
\newblock {\em J. Differential Geom.}, 54(1):37--74, 2000.

\bibitem[CDL08]{CDL08}
Gianluca Crippa and Camillo De~Lellis.
\newblock Estimates and regularity results for the {D}i{P}erna-{L}ions flow.
\newblock {\em J. Reine Angew. Math.}, 616:15--46, 2008.

\bibitem[CE08]{ChEb08}
Jeff Cheeger and David~G. Ebin.
\newblock {\em Comparison theorems in {R}iemannian geometry}.
\newblock AMS Chelsea Publishing, Providence, RI, 2008.
\newblock Revised reprint of the 1975 original.

\bibitem[CEMS01]{CEMCS01}
Dario Cordero-Erausquin, Robert~J. McCann, and Michael Schmuckenschl\"{a}ger.
\newblock A {R}iemannian interpolation inequality \`a la {B}orell, {B}rascamp
  and {L}ieb.
\newblock {\em Invent. Math.}, 146(2):219--257, 2001.

\bibitem[CEMS06]{CEMS06}
Dario Cordero-Erausquin, Robert~J. McCann, and Michael Schmuckenschl\"{a}ger.
\newblock Pr\'{e}kopa-{L}eindler type inequalities on {R}iemannian manifolds,
  {J}acobi fields, and optimal transport.
\newblock {\em Ann. Fac. Sci. Toulouse Math. (6)}, 15(4):613--635, 2006.

\bibitem[CENV04]{CENaVi04}
Dario Cordero-Erausquin, Bruno Nazaret, and C{\'e}dric Villani.
\newblock A mass-transportation approach to sharp {S}obolev and
  {G}agliardo-{N}irenberg inequalities.
\newblock {\em Adv. Math.}, 182(2):307--332, 2004.

\bibitem[CFMP09]{CianchiFuscoMaggiPratelli09}
A.~Cianchi, N.~Fusco, F.~Maggi, and A.~Pratelli.
\newblock The sharp {S}obolev inequality in quantitative form.
\newblock {\em J. Eur. Math. Soc. (JEMS)}, 11(5):1105--1139, 2009.

\bibitem[Cha06]{Chavel06}
Isaac Chavel.
\newblock {\em Riemannian geometry}, volume~98 of {\em Cambridge Studies in
  Advanced Mathematics}.
\newblock Cambridge University Press, Cambridge, second edition, 2006.
\newblock A modern introduction.

\bibitem[Che99]{Cheeger00}
Jeff Cheeger.
\newblock Differentiability of {L}ipschitz functions on metric measure spaces.
\newblock {\em Geom. Funct. Anal.}, 9(3):428--517, 1999.

\bibitem[Che01]{Cheegersurvey}
Jeff Cheeger.
\newblock {\em Degeneration of {R}iemannian metrics under {R}icci curvature
  bounds}.
\newblock Lezioni Fermiane. [Fermi Lectures]. Scuola Normale Superiore, Pisa,
  2001.

\bibitem[CJN21]{ChJiNa21}
Jeff Cheeger, Wenshuai Jiang, and Aaron Naber.
\newblock Rectifiability of singular sets of noncollapsed limit spaces with
  {R}icci curvature bounded below.
\newblock {\em Ann. of Math. (2)}, 193(2):407--538, 2021.

\bibitem[CM14]{ColdMin14}
Tobias~Holck Colding and William~P. Minicozzi, II.
\newblock On uniqueness of tangent cones for {E}instein manifolds.
\newblock {\em Invent. Math.}, 196(3):515--588, 2014.

\bibitem[CM17]{CavMon15}
Fabio Cavalletti and Andrea Mondino.
\newblock Sharp and rigid isoperimetric inequalities in metric-measure spaces
  with lower ricci curvature bounds.
\newblock {\em Invent. Math.}, 208(3):803--849, 2017.

\bibitem[CM20]{CavMon20}
Fabio Cavalletti and Andrea Mondino.
\newblock New formulas for the {L}aplacian of distance functions and
  applications.
\newblock {\em Anal. PDE}, 13(7):2091--2147, 2020.

\bibitem[CM21]{CavMil16}
Fabio Cavalletti and Emanuel Milman.
\newblock The globalization theorem for the curvature-dimension condition.
\newblock {\em Invent. Math.}, 226(1):1--137, 2021.

\bibitem[CM22]{CavMan22}
Fabio Cavalletti and Davide Manini.
\newblock Isoperimetric inequality in noncompact {$\rm{MCP}$} spaces.
\newblock {\em Proc. Amer. Math. Soc.}, 150(8):3537--3548, 2022.

\bibitem[CMM19]{CMM19}
Fabio Cavalletti, Francesco Maggi, and Andrea Mondino.
\newblock Quantitative isoperimetry \`a la {L}evy-{G}romov.
\newblock {\em Comm. Pure Appl. Math.}, 72(8):1631--1677, 2019.

\bibitem[CMT21]{CMTkatoI}
Gilles Carron, Ilaria Mondello, and David Tewodrose.
\newblock Limits of manifolds with a {K}ato bound on the {R}icci curvature.
\newblock Preprint, arXiv:2102.05940, 2021.

\bibitem[CMT22]{CMTkatoII}
Gilles Carron, Ilaria Mondello, and David Tewodrose.
\newblock Limits of manifolds with a {K}ato bound on the {R}icci curvature.
  {II}.
\newblock Preprint, arXiv:2205.01956, 2022.

\bibitem[CN12]{ColdingNaber12}
Tobias~Holck Colding and Aaron Naber.
\newblock Sharp {H}\"older continuity of tangent cones for spaces with a lower
  {R}icci curvature bound and applications.
\newblock {\em Ann. of Math. (2)}, 176(2):1173--1229, 2012.

\bibitem[CN15]{ChNa15}
Jeff Cheeger and Aaron Naber.
\newblock Regularity of {E}instein manifolds and the codimension 4 conjecture.
\newblock {\em Ann. of Math. (2)}, 182(3):1093--1165, 2015.

\bibitem[Col96]{Colding96b}
Tobias~Holck Colding.
\newblock Shape of manifolds with positive {R}icci curvature.
\newblock {\em Invent. Math.}, 124(1-3):175--191, 1996.

\bibitem[Col97]{Colding97}
Tobias~Holck Colding.
\newblock Ricci curvature and volume convergence.
\newblock {\em Ann. of Math. (2)}, 145(3):477--501, 1997.

\bibitem[CSC93]{CouSal93}
Thierry Coulhon and Laurent Saloff-Coste.
\newblock Isop\'{e}rim\'{e}trie pour les groupes et les vari\'{e}t\'{e}s.
\newblock {\em Rev. Mat. Iberoamericana}, 9(2):293--314, 1993.

\bibitem[CT22]{CarTew22}
Gilles Carron and David Tewodrose.
\newblock A rigidity result for metric measure spaces with {E}uclidean heat
  kernel.
\newblock {\em J. \'{E}c. polytech. Math.}, 9:101--154, 2022.

\bibitem[CZ00]{CheZhu00}
Bing-Long Chen and Xi-Ping Zhu.
\newblock Complete {R}iemannian manifolds with pointwise pinched curvature.
\newblock {\em Invent. Math.}, 140(2):423--452, 2000.

\bibitem[Den20]{Deng20}
Qin Deng.
\newblock H\"older continuity of tangent cones in {${\rm RCD}(K,N)$} spaces and
  applications to non-branching.
\newblock Preprint, arXiv:2009.07956, 2020.

\bibitem[Der16]{Der16}
Alix Deruelle.
\newblock Smoothing out positively curved metric cones by {R}icci expanders.
\newblock {\em Geom. Funct. Anal.}, 26(1):188--249, 2016.

\bibitem[DG92]{DeGiorgi92}
Ennio De~Giorgi.
\newblock Movimenti minimizzanti.
\newblock In {\em talk given at the meeting ``Aspetti e problemi della
  matematica oggi'', {L}ecce}, October 20-22, 1992.

\bibitem[DG93]{DeGiorgi93}
Ennio De~Giorgi.
\newblock New problems on minimizing movements.
\newblock In Claudio Baiocchi and Jacques~Louis Lions, editors, {\em Boundary
  Value Problems for PDE and Applications}, pages 81--98. Masson, 1993.

\bibitem[DG13]{DeGiorgiSelected}
Ennio De~Giorgi.
\newblock {\em Selected papers}.
\newblock Springer Collected Works in Mathematics. Springer, Heidelberg, 2013.
\newblock [Author name on title page: Ennio Giorgi], Edited by Luigi Ambrosio,
  Gianni Dal Maso, Marco Forti, Mario Miranda and Sergio Spagnolo, Reprint of
  the 2006 edition [MR2229237].

\bibitem[DGF75]{DeGiorgiGamma}
Ennio De~Giorgi and Tullio Franzoni.
\newblock Su un tipo di convergenza variazionale.
\newblock {\em Atti Accad. Naz. Lincei Rend. Cl. Sci. Fis. Mat. Nat. (8)},
  58(6):842--850, 1975.

\bibitem[DGMT80]{DeGiorgiMarinoTosques80}
Ennio De~Giorgi, A.~Marino, and M.~Tosques.
\newblock Problems of evolution in metric spaces and maximal decreasing curve.
\newblock {\em Atti Accad. Naz. Lincei Rend. Cl. Sci. Fis. Mat. Natur. (8)},
  68(3):180--187, 1980.

\bibitem[DH02]{DrHe02}
Olivier Druet and Emmanuel Hebey.
\newblock The {$AB$} program in geometric analysis: sharp {S}obolev
  inequalities and related problems.
\newblock {\em Mem. Amer. Math. Soc.}, 160(761):viii+98, 2002.

\bibitem[DL89]{DiPerna-Lions89}
Ronald~J. DiPerna and Pierre-Louis Lions.
\newblock Ordinary differential equations, transport theory and {S}obolev
  spaces.
\newblock {\em Invent. Math.}, 98(3):511--547, 1989.

\bibitem[DM93]{DalMaso93}
Gianni Dal~Maso.
\newblock {\em An introduction to {$\Gamma$}-convergence}, volume~8 of {\em
  Progress in Nonlinear Differential Equations and their Applications}.
\newblock Birkh\"{a}user Boston, Inc., Boston, MA, 1993.

\bibitem[DMGP20]{GDMP}
Simone Di~Marino, Nicola Gigli, and Aldo Pratelli.
\newblock Global {L}ipschitz extension preserving local constants.
\newblock {\em Atti Accad. Naz. Lincei Rend. Lincei Mat. Appl.},
  31(4):757--765, 2020.

\bibitem[DMLP21]{DMLP21}
Simone Di~Marino, Danka Lu\v{c}i\'{c}, and Enrico Pasqualetto.
\newblock Representation theorems for normed modules.
\newblock Preprint, arXiv:2109.03509, 2021.

\bibitem[DPNnZ23]{DPZ19}
Guido De~Philippis and Jes\'{u}s N\'{u}\~{n}ez Zimbr\'{o}n.
\newblock The behavior of harmonic functions at singular points of {$\rm
  {RCD}$} spaces.
\newblock {\em Manuscripta Math.}, 171(1-2):155--168, 2023.

\bibitem[DS08]{DaneriSavare08}
Sara Daneri and Giuseppe Savar{\'e}.
\newblock Eulerian calculus for the displacement convexity in the {W}asserstein
  distance.
\newblock {\em SIAM J. Math. Anal.}, 40(3):1104--1122, 2008.

\bibitem[DSS22]{DSS22}
Alix Deruelle, Felix Schulze, and Miles Simon.
\newblock Initial stability estimates for {R}icci flow and three dimensional
  {R}icci-pinched manifolds.
\newblock Preprint, arXiv:2203.15313v1, 2022.

\bibitem[DU77]{DiestelUhl77}
Joseph Diestel and J.~Jerry Uhl, Jr.
\newblock {\em Vector measures}.
\newblock American Mathematical Society, Providence, R.I., 1977.
\newblock With a foreword by B. J. Pettis, Mathematical Surveys, No. 15.

\bibitem[EB23]{EB23}
Sylvester Eriksson-Bique.
\newblock Density of {L}ipschitz functions in energy.
\newblock {\em Calc. Var. Partial Differential Equations}, 62(2):Paper No. 60,
  23, 2023.

\bibitem[EBSR22]{BSR22}
Sylvester Eriksson-Bique, Elefterios Soultanis, and Tapio Rajala.
\newblock Tensorization of quasi-{H}ilbertian {S}obolev spaces.
\newblock Preprint, arXiv:2209.03040, 2022.

\bibitem[EKS14]{Erbar-Kuwada-Sturm13}
Matthias Erbar, Kazumasa Kuwada, and Karl-Theodor Sturm.
\newblock On the equivalence of the entropic curvature-dimension condition and
  {B}ochner's inequality on metric measure spaces.
\newblock {\em Invent. Math.}, 201(3):1--79, 2014.

\bibitem[ERST22]{ERST22}
Matthias Erbar, Chiara Rigoni, Karl-Theodor Sturm, and Luca Tamanini.
\newblock Tamed spaces---{D}irichlet spaces with distribution-valued {R}icci
  bounds.
\newblock {\em J. Math. Pures Appl. (9)}, 161:1--69, 2022.

\bibitem[FG21]{FigGla21}
Alessio Figalli and Federico Glaudo.
\newblock {\em An invitation to optimal transport, {W}asserstein distances, and
  gradient flows}.
\newblock EMS Textbooks in Mathematics. EMS Press, Berlin, [2021] \copyright
  2021.

\bibitem[FN19]{FigalliNeumayer19}
Alessio Figalli and Robin Neumayer.
\newblock Gradient stability for the {S}obolev inequality: the case {$p\geq
  2$}.
\newblock {\em J. Eur. Math. Soc. (JEMS)}, 21(2):319--354, 2019.

\bibitem[Foc12]{Foc12}
Matteo Focardi.
\newblock {$\Gamma$}-convergence: a tool to investigate physical phenomena
  across scales.
\newblock {\em Math. Methods Appl. Sci.}, 35(14):1613--1658, 2012.

\bibitem[Fug57]{Fugl57}
Bent Fuglede.
\newblock Extremal length and functional completion.
\newblock {\em Acta Math.}, 98:171--219, 1957.

\bibitem[Fuk80]{Fuk80}
Masatoshi Fukushima.
\newblock {\em Dirichlet forms and {M}arkov processes}, volume~23 of {\em
  North-Holland Mathematical Library}.
\newblock North-Holland Publishing Co., Amsterdam-New York; Kodansha, Ltd.,
  Tokyo, 1980.

\bibitem[Fuk87]{fukaya}
Kenji Fukaya.
\newblock Collapsing of {R}iemannian manifolds and eigenvalues of {L}aplace
  operator.
\newblock {\em Invent. Math.}, 87(3):517--547, 1987.

\bibitem[FZ22]{FigalliZhang22}
Alessio Figalli and Yi~Ru-Ya Zhang.
\newblock Sharp gradient stability for the {S}obolev inequality.
\newblock {\em Duke Math. J.}, 171(12):2407--2459, 2022.

\bibitem[Gal88]{Gal88}
Sylvestre Gallot.
\newblock In\'{e}galit\'{e}s isop\'{e}rim\'{e}triques et analytiques sur les
  vari\'{e}t\'{e}s riemanniennes.
\newblock Number 163-164, pages 5--6, 31--91, 281 (1989). 1988.
\newblock On the geometry of differentiable manifolds (Rome, 1986).

\bibitem[GDP18]{GDP17}
Nicola Gigli and Guido De~Philippis.
\newblock Non-collapsed spaces with {R}icci curvature bounded from below.
\newblock {\em J. \'{E}c. polytech. Math.}, 5:613--650, 2018.

\bibitem[GGKMS18]{GKMS17}
Fernando Galaz-Garc\'{\i}a, Martin Kell, Andrea Mondino, and Gerardo Sosa.
\newblock On quotients of spaces with {R}icci curvature bounded below.
\newblock {\em J. Funct. Anal.}, 275(6):1368--1446, 2018.

\bibitem[GH16]{GigliHan14}
Nicola Gigli and Bang-Xian Han.
\newblock Independence on {$p$} of weak upper gradients on {$\sf{RCD}$} spaces.
\newblock {\em J. Funct. Anal.}, 271(1):1--11, 2016.

\bibitem[Gig10]{Gigli10}
Nicola Gigli.
\newblock On the heat flow on metric measure spaces: existence, uniqueness and
  stability.
\newblock {\em Calc. Var. PDE}, 39(1-2):101--120, 2010.

\bibitem[Gig11]{G11}
Nicola Gigli.
\newblock {\em Introduction to optimal transport: theory and applications}.
\newblock Publica\c{c}\~{o}es Matem\'{a}ticas do IMPA. [IMPA Mathematical
  Publications]. Instituto Nacional de Matem\'{a}tica Pura e Aplicada (IMPA),
  Rio de Janeiro, 2011.
\newblock 28${^{{}}{\rm{o}}}$ Col\'{o}quio Brasileiro de Matem\'{a}tica. [28th
  Brazilian Mathematics Colloquium].

\bibitem[Gig12]{Bochner-CD}
Nicola Gigli.
\newblock On the relation between the curvature dimension condition and
  {B}ochner inequality.
\newblock Unpublished paper, available at
  \texttt{http://cvgmt.sns.it/person/226/}, 2012.

\bibitem[Gig13]{Gigli13}
Nicola Gigli.
\newblock The splitting theorem in non-smooth context.
\newblock Preprint, arXiv:1302.5555, 2013.

\bibitem[Gig14]{Gigli13over}
Nicola Gigli.
\newblock An overview of the proof of the splitting theorem in spaces with
  non-negative {R}icci curvature.
\newblock {\em Analysis and Geometry in Metric Spaces}, 2:169--213, 2014.

\bibitem[Gig15]{Gigli12}
Nicola Gigli.
\newblock On the differential structure of metric measure spaces and
  applications.
\newblock {\em Mem. Amer. Math. Soc.}, 236(1113):vi+91, 2015.

\bibitem[Gig18a]{Gigli17}
Nicola Gigli.
\newblock Lecture notes on differential calculus on {$\sf{RCD}$} spaces.
\newblock {\em Publ. Res. Inst. Math. Sci.}, 54(4):855--918, 2018.

\bibitem[Gig18b]{Gigli14}
Nicola Gigli.
\newblock Nonsmooth differential geometry---an approach tailored for spaces
  with {R}icci curvature bounded from below.
\newblock {\em Mem. Amer. Math. Soc.}, 251(1196):v+161, 2018.

\bibitem[GKO13]{Gigli-Kuwada-Ohta10}
Nicola Gigli, Kazumasa Kuwada, and Shin-ichi Ohta.
\newblock Heat flow on {A}lexandrov spaces.
\newblock {\em Communications on Pure and Applied Mathematics}, 66(3):307--331,
  2013.

\bibitem[GLP]{GLP22}
Nicola Gigli, Danka Lu\v{c}i\'c, and Enrico Pasqualetto.
\newblock Duals and pullbacks of normed modules.
\newblock Preprint.

\bibitem[GM14]{Gigli-Mosconi12}
Nicola Gigli and Sunra Mosconi.
\newblock The {A}bresch-{G}romoll inequality in a non-smooth setting.
\newblock {\em Discrete Contin. Dyn. Syst.}, 34(4):1481--1509, 2014.

\bibitem[GMS15]{GMS15}
Nicola Gigli, Andrea Mondino, and Giuseppe Savar\'e.
\newblock Convergence of pointed non-compact metric measure spaces and
  stability of {R}icci curvature bounds and heat flows.
\newblock {\em Proc. Lond. Math. Soc. (3)}, 111(5):1071--1129, 2015.

\bibitem[GN22]{GiNo21}
Nicola Gigli and Francesco Nobili.
\newblock A first-order condition for the independence on {$p$} of weak
  gradients.
\newblock {\em J. Funct. Anal.}, 283(11):Paper No. 109686, 52, 2022.

\bibitem[GP20]{GP19}
Nicola Gigli and Enrico Pasqualetto.
\newblock {\em Lectures on nonsmooth differential geometry}, volume~2 of {\em
  SISSA Springer Series}.
\newblock Springer, Cham, [2020] \copyright 2020.

\bibitem[GP21]{GP16-2}
Nicola Gigli and Enrico Pasqualetto.
\newblock Behaviour of the reference measure on {$\rm{RCD}$} spaces under
  charts.
\newblock {\em Comm. Anal. Geom.}, 29(6):1391--1414, 2021.

\bibitem[GP22]{GP16}
Nicola Gigli and Enrico Pasqualetto.
\newblock Equivalence of two different notions of tangent bundle on rectifiable
  metric measure spaces.
\newblock {\em Comm. Anal. Geom.}, 30(1):1--51, 2022.

\bibitem[Gri10]{Gri10}
Alexander Grigor'yan.
\newblock Heat kernels on metric measure spaces with regular volume growth.
\newblock In {\em Handbook of geometric analysis, {N}o. 2}, volume~13 of {\em
  Adv. Lect. Math. (ALM)}, pages 1--60. Int. Press, Somerville, MA, 2010.

\bibitem[Gro81]{Gro81}
Mikhael Gromov.
\newblock {\em Structures m\'{e}triques pour les vari\'{e}t\'{e}s
  riemanniennes}, volume~1 of {\em Textes Math\'{e}matiques [Mathematical
  Texts]}.
\newblock CEDIC, Paris, 1981.
\newblock Edited by J. Lafontaine and P. Pansu.

\bibitem[Gro07]{Gromov07}
Misha Gromov.
\newblock {\em Metric structures for {R}iemannian and non-{R}iemannian spaces}.
\newblock Modern Birkh\"auser Classics. Birkh\"auser Boston Inc., Boston, MA,
  english edition, 2007.
\newblock Based on the 1981 French original, With appendices by M. Katz, P.
  Pansu and S. Semmes, Translated from the French by Sean Michael Bates.

\bibitem[GRS14]{NCS14}
Nathael Gozlan, Cyril Roberto, and Paul-Marie Samson.
\newblock Hamilton {J}acobi equations on metric spaces and transport entropy
  inequalities.
\newblock {\em Rev. Mat. Iberoam.}, 30(1):133--163, 2014.

\bibitem[G{\'S}15]{GS15}
Wilfrid Gangbo and Andrzej {\'S}wiech.
\newblock Metric viscosity solutions of {H}amilton-{J}acobi equations depending
  on local slopes.
\newblock {\em Calc. Var. Partial Differential Equations}, 54(1):1183--1218,
  2015.

\bibitem[GT21]{GT18}
Nicola Gigli and Alexander Tyulenev.
\newblock Korevaar-{S}choen's directional energy and {A}mbrosio's regular
  {L}agrangian flows.
\newblock {\em Math. Z.}, 298(3-4):1221--1261, 2021.

\bibitem[GV23]{GV23}
Nicola Gigli and Ivan~Yuri Violo.
\newblock Work in progress.
\newblock 2023.

\bibitem[Ha96]{HajSob}
Piotr Haj\l~asz.
\newblock Sobolev spaces on an arbitrary metric space.
\newblock {\em Potential Anal.}, 5(4):403--415, 1996.

\bibitem[Ham94]{Ham94}
Richard~S. Hamilton.
\newblock Convex hypersurfaces with pinched second fundamental form.
\newblock {\em Comm. Anal. Geom.}, 2(1):167--172, 1994.

\bibitem[Han18]{Han14}
Bang-Xian Han.
\newblock Ricci tensor on {${\rm RCD}^*(K,N)$} {spaces}.
\newblock {\em J. Geom. Anal.}, 28(2):1295--1314, 2018.

\bibitem[Heb99]{He99}
Emmanuel Hebey.
\newblock {\em Nonlinear analysis on manifolds: {S}obolev spaces and
  inequalities}, volume~5 of {\em Courant Lecture Notes in Mathematics}.
\newblock New York University, Courant Institute of Mathematical Sciences, New
  York; American Mathematical Society, Providence, RI, 1999.

\bibitem[HKST15]{HKST15}
Juha Heinonen, Pekka Koskela, Nageswari Shanmugalingam, and Jeremy~T. Tyson.
\newblock {\em Sobolev spaces on metric measure spaces}, volume~27 of {\em New
  Mathematical Monographs}.
\newblock Cambridge University Press, Cambridge, 2015.
\newblock An approach based on upper gradients.

\bibitem[HLR91]{HLR91}
Richard Haydon, Mireille Levy, and Yves Raynaud.
\newblock {\em Randomly normed spaces}, volume~41 of {\em Travaux en Cours
  [Works in Progress]}.
\newblock Hermann, Paris, 1991.

\bibitem[Hoc19]{HochardThesis}
Raphael Hochard.
\newblock {\em Th\'eor\`emes d'existence en temps court du flot de {R}icci pour
  des vari\'et\'es non-compl\`etes, non-\'effondr\'ees, \`a courbure
  minor\'ee}.
\newblock PhD thesis, Universit\'e de Bordeaux, 2019.

\bibitem[Hon11]{Honda11-2}
Shouhei Honda.
\newblock Ricci curvature and convergence of {L}ipschitz functions.
\newblock {\em Comm. Anal. Geom.}, 19(1):79--158, 2011.

\bibitem[Hon18]{Honda14}
Shouhei Honda.
\newblock Elliptic {PDE}s on compact {R}icci limit spaces and applications.
\newblock {\em Mem. Amer. Math. Soc.}, 253(1211):v+92, 2018.

\bibitem[Hon20]{H19}
Shouhei Honda.
\newblock New differential operator and noncollapsed {RCD} spaces.
\newblock {\em Geom. Topol.}, 24(4):2127--2148, 2020.

\bibitem[HS22]{HS22}
Bang-Xian Han and Karl-Theodor Sturm.
\newblock Curvature-dimension conditions under time change.
\newblock {\em Ann. Mat. Pura Appl. (4)}, 201(2):801--822, 2022.

\bibitem[HZ20]{HZ00}
Shouhei Honda and Xingyu Zhu.
\newblock A characterization of non-collapsed {$\sf{RCD}(K,N)$} spaces via
  {E}instein tensors.
\newblock Preprint, arXiv:2010.02530, 2020.

\bibitem[IPS22]{IPS22}
Toni Ikonen, Enrico Pasqualetto, and Elefterios Soultanis.
\newblock Abstract and concrete tangent modules on {L}ipschitz
  differentiability spaces.
\newblock {\em Proc. Amer. Math. Soc.}, 150(1):327--343, 2022.

\bibitem[IRV15]{IRV15}
Jin-ichi Itoh, Jo\"{e}l Rouyer, and Costin V\^{i}lcu.
\newblock Moderate smoothness of most {A}lexandrov surfaces.
\newblock {\em Internat. J. Math.}, 26(4):1540004, 14, 2015.

\bibitem[JKO98]{JKO98}
Richard Jordan, David Kinderlehrer, and Felix Otto.
\newblock The variational formulation of the {F}okker-{P}lanck equation.
\newblock {\em SIAM J. Math. Anal.}, 29(1):1--17, 1998.

\bibitem[Joh21]{Joh21}
Florian Johne.
\newblock Sobolev inequalities on manifolds with nonnegative {B}akry-\'{E}mery
  {R}icci curvature.
\newblock Preprint, arXiv:2103.08496, 2021.

\bibitem[Jos17]{Jost17}
J\"{u}rgen Jost.
\newblock {\em Riemannian geometry and geometric analysis}.
\newblock Universitext. Springer, Cham, seventh edition, 2017.

\bibitem[Ket15]{Ketterer13}
Christian Ketterer.
\newblock Cones over metric measure spaces and the maximal diameter theorem.
\newblock {\em J. Math. Pures Appl. (9)}, 103(5):1228--1275, 2015.

\bibitem[Ket21]{Ket21}
Christian Ketterer.
\newblock Stability of metric measure spaces with integral {R}icci curvature
  bounds.
\newblock {\em J. Funct. Anal.}, 281(8):Paper No. 109142, 48, 2021.

\bibitem[KKST12]{KKST12}
Juha Kinnunen, Riikka Korte, Nageswari Shanmugalingam, and Heli Tuominen.
\newblock A characterization of {N}ewtonian functions with zero boundary
  values.
\newblock {\em Calc. Var. Partial Differential Equations}, 43(3-4):507--528,
  2012.

\bibitem[KM18]{MK16}
Martin Kell and Andrea Mondino.
\newblock On the volume measure of non-smooth spaces with {R}icci curvature
  bounded below.
\newblock {\em Ann. Sc. Norm. Super. Pisa Cl. Sci. (5)}, 18(2):593--610, 2018.

\bibitem[KS03]{KS03conv}
Kazuhiro Kuwae and Takashi Shioya.
\newblock Convergence of spectral structures: a functional analytic theory and
  its applications to spectral geometry.
\newblock {\em Comm. Anal. Geom.}, 11(4):599--673, 2003.

\bibitem[KS20]{KazShi20}
Daisuke Kazukawa and Takashi Shioya.
\newblock High-dimensional ellipsoids converge to {G}aussian spaces.
\newblock Preprint, arXiv:2003.05105, 2020.

\bibitem[KSZ14]{KSZ14}
Pekka Koskela, Nageswari Shanmugalingam, and Yuan Zhou.
\newblock Geometry and analysis of {D}irichlet forms ({II}).
\newblock {\em J. Funct. Anal.}, 267(7):2437--2477, 2014.

\bibitem[Kuw10]{Kuwada10}
Kazumasa Kuwada.
\newblock Duality on gradient estimates and {W}asserstein controls.
\newblock {\em J. Funct. Anal.}, 258(11):3758--3774, 2010.

\bibitem[KY21]{KazYok21}
Daisuke Kazukawa and Takumi Yokota.
\newblock Boundedness of precompact sets of metric measure spaces.
\newblock {\em Geom. Dedicata}, 215:229--242, 2021.

\bibitem[KY22]{KazYok22}
Daisuke Kazukawa and Takumi Yokota.
\newblock Boundedness of measured {G}romov-{H}ausdorff precompact sets of
  metric measure spaces in pyramids.
\newblock Preprint, arXiv:2210.00687, 2022.

\bibitem[KZ12]{Koskela-Zhou12}
Pekka Koskela and Yuan Zhou.
\newblock Geometry and analysis of {D}irichlet forms.
\newblock {\em Adv. Math.}, 231(5):2755--2801, 2012.

\bibitem[Led99]{Ledoux99}
Michel Ledoux.
\newblock On manifolds with non-negative {R}icci curvature and {S}obolev
  inequalities.
\newblock {\em Comm. Anal. Geom.}, 7(2):347--353, 1999.

\bibitem[Li12]{Li12}
Peter Li.
\newblock {\em Geometric Analysis}.
\newblock Cambridge Studies in Advanced Mathematics. Cambridge University
  Press, 2012.

\bibitem[Lio84a]{LionsLocComI}
Pierre-Louis Lions.
\newblock The concentration-compactness principle in the calculus of
  variations. {T}he locally compact case. {I}.
\newblock {\em Ann. Inst. H. Poincar\'{e} Anal. Non Lin\'{e}aire},
  1(2):109--145, 1984.

\bibitem[Lio84b]{LionsLocComII}
Pierre-Louis Lions.
\newblock The concentration-compactness principle in the calculus of
  variations. {T}he locally compact case. {II}.
\newblock {\em Ann. Inst. H. Poincar\'{e} Anal. Non Lin\'{e}aire},
  1(4):223--283, 1984.

\bibitem[Lio85a]{LionsLimCasI}
Pierre-Louis Lions.
\newblock The concentration-compactness principle in the calculus of
  variations. {T}he limit case. {I}.
\newblock {\em Rev. Mat. Iberoamericana}, 1(1):145--201, 1985.

\bibitem[Lio85b]{LionsLimCasII}
Pierre-Louis Lions.
\newblock The concentration-compactness principle in the calculus of
  variations. {T}he limit case. {II}.
\newblock {\em Rev. Mat. Iberoamericana}, 1(2):45--121, 1985.

\bibitem[Lis07]{Lisini07}
Stefano Lisini.
\newblock Characterization of absolutely continuous curves in {W}asserstein
  spaces.
\newblock {\em Calc. Var. Partial Differential Equations}, 28(1):85--120, 2007.

\bibitem[Lot19]{Lott19}
John Lott.
\newblock On 3-manifolds with pointwise pinched nonnegative {R}icci curvature.
\newblock Preprint, arXiv:1908.04715, 2019.

\bibitem[LP19]{LP18}
Danka Lu\v{c}i\'{c} and Enrico Pasqualetto.
\newblock The {S}erre-{S}wan theorem for normed modules.
\newblock {\em Rendiconti del Circolo Matematico di Palermo Series 2},
  68:385--404, 2019.

\bibitem[LPV22]{LPV22}
Milica Lu\v{c}i\'{c}, Enrico Pasqualetto, and Ivana Vojnovi\'{c}.
\newblock On the reflexivity properties of {B}anach bundles and {B}anach
  modules.
\newblock Preprint, arXiv:2205.11608, 2022.

\bibitem[LS23]{LS18}
Alexander Lytchak and Stephan Stadler.
\newblock {R}icci curvature in dimension 2.
\newblock {\em J. Eur. Math. Soc. (JEMS)}, 25(3):845--867, 2023.

\bibitem[LT22]{LT22}
Man-Chun Lee and Peter~M. Topping.
\newblock Three-manifolds with non-negatively pinched {R}icci curvature.
\newblock Preprint, arXiv:2204.00504, 2022.

\bibitem[LV07a]{LVHJ}
John Lott and C{\'e}dric Villani.
\newblock Hamilton-{J}acobi semigroup on length spaces and applications.
\newblock {\em J. Math. Pures Appl. (9)}, 88(3):219--229, 2007.

\bibitem[LV07b]{Lott-Villani07}
John Lott and C{\'e}dric Villani.
\newblock Weak curvature conditions and functional inequalities.
\newblock {\em J. Funct. Anal.}, 245(1):311--333, 2007.

\bibitem[LV09]{Lott-Villani09}
John Lott and C{\'e}dric Villani.
\newblock Ricci curvature for metric-measure spaces via optimal transport.
\newblock {\em Ann. of Math. (2)}, 169(3):903--991, 2009.

\bibitem[Mar18]{mariani}
Mauro Mariani.
\newblock A {$\Gamma$}-convergence approach to large deviations.
\newblock {\em Ann. Sc. Norm. Super. Pisa Cl. Sci. (5)}, 18(3):951--976, 2018.

\bibitem[McC97]{McCann97}
Robert~J. McCann.
\newblock A convexity principle for interacting gases.
\newblock {\em Adv. Math.}, 128(1):153--179, 1997.

\bibitem[Men00]{Menguy00}
X.~Menguy.
\newblock Noncollapsing examples with positive {R}icci curvature and infinite
  topological type.
\newblock {\em Geom. Funct. Anal.}, 10(3):600--627, 2000.

\bibitem[Mir03]{Mir03}
Michele Miranda, Jr.
\newblock Functions of bounded variation on ``good'' metric spaces.
\newblock {\em J. Math. Pures Appl. (9)}, 82(8):975--1004, 2003.

\bibitem[MJ00]{MorJoh00}
Frank Morgan and David~L. Johnson.
\newblock Some sharp isoperimetric theorems for {R}iemannian manifolds.
\newblock {\em Indiana Univ. Math. J.}, 49(3):1017--1041, 2000.

\bibitem[MN19]{Mondino-Naber14}
Andrea Mondino and Aaron Naber.
\newblock Structure theory of metric measure spaces with lower {R}icci
  curvature bounds.
\newblock {\em J. Eur. Math. Soc. (JEMS)}, 21(6):1809--1854, 2019.

\bibitem[MR92]{MaRockner92}
Zhi~Ming Ma and Michael R{\"o}ckner.
\newblock {\em Introduction to the theory of (nonsymmetric) {D}irichlet forms}.
\newblock Universitext. Springer-Verlag, Berlin, 1992.

\bibitem[MS20]{MS20}
Matteo Muratori and Giuseppe Savar\'{e}.
\newblock Gradient flows and evolution variational inequalities in metric
  spaces. {I}: {S}tructural properties.
\newblock {\em J. Funct. Anal.}, 278(4):108347, 67, 2020.

\bibitem[MS21]{MS21}
Andrea Mondino and Daniele Semola.
\newblock Weak laplacian bounds and minimal boundaries in non-smooth spaces
  with {R}icci curvature lower bounds.
\newblock Preprint, arXiv:2107.12344, 2021.

\bibitem[Nar14]{Nard14}
Stefano Nardulli.
\newblock Generalized existence of isoperimetric regions in non-compact
  {R}iemannian manifolds and applications to the isoperimetric profile.
\newblock {\em Asian J. Math.}, 18(1):1--28, 2014.

\bibitem[Neu20]{Neumayer19}
Robin Neumayer.
\newblock A note on strong-form stability for the {S}obolev inequality.
\newblock {\em Calc. Var. Partial Differential Equations}, 59(1):Paper No. 25,
  8, 2020.

\bibitem[NS21a]{NakShi21-2}
Hiroki Nakajima and Takashi Shioya.
\newblock Convergence of group actions in metric measure geometry.
\newblock Preprint, arXiv:2104.00187, 2021.

\bibitem[NS21b]{NakShi21}
Hiroki Nakajima and Takashi Shioya.
\newblock A natural compactification of the {G}romov-{H}ausdorff space.
\newblock Preprint, arXiv:2109.14853, 2021.

\bibitem[NV22a]{NoVi22}
Francesco Nobili and Ivan~Yuri Violo.
\newblock Rigidity and almost rigidity of {S}obolev inequalities on compact
  spaces with lower {R}icci curvature bounds.
\newblock {\em Calc. Var. Partial Differential Equations}, 61(5):Paper No. 180,
  65, 2022.

\bibitem[NV22b]{NobiliViolo22}
Francesco Nobili and Ivan~Yuri Violo.
\newblock Stability of {S}obolev inequalities on {R}iemannian manifolds with
  {R}icci curvature lower bounds.
\newblock Preprint, arXiv:2210.00636, 2022.

\bibitem[NW16]{NiWan16}
Lei Ni and Kui Wang.
\newblock Isoperimetric comparisons via viscosity.
\newblock {\em J. Geom. Anal.}, 26(4):2831--2841, 2016.

\bibitem[OP17]{OhPa17}
Shin-ichi Ohta and Mikl\'{o}s P\'{a}lfia.
\newblock Gradient flows and a {T}rotter-{K}ato formula of semi-convex
  functions on {${\rm CAT}(1)$}-spaces.
\newblock {\em Amer. J. Math.}, 139(4):937--965, 2017.

\bibitem[OS09]{OhSt09}
Shin-ichi Ohta and Karl-Theodor Sturm.
\newblock Heat flow on {F}insler manifolds.
\newblock {\em Comm. Pure Appl. Math.}, 62(10):1386--1433, 2009.

\bibitem[OS12]{OhSt12}
Shin-ichi Ohta and Karl-Theodor Sturm.
\newblock Non-contraction of heat flow on {M}inkowski spaces.
\newblock {\em Arch. Ration. Mech. Anal.}, 204(3):917--944, 2012.

\bibitem[OS15]{OzaShi14}
Ryunosuke Ozawa and Takashi Shioya.
\newblock Limit formulas for metric measure invariants and phase transition
  property.
\newblock {\em Math. Z.}, 280(3-4):759--782, 2015.

\bibitem[Ott01]{Otto01}
Felix Otto.
\newblock The geometry of dissipative evolution equations: the porous medium
  equation.
\newblock {\em Comm. Partial Differential Equations}, 26(1-2):101--174, 2001.

\bibitem[OV00]{OV00}
Felix Otto and C{\'e}dric Villani.
\newblock Generalization of an inequality by {T}alagrand and links with the
  logarithmic {S}obolev inequality.
\newblock {\em J. Funct. Anal.}, 173(2):361--400, 2000.

\bibitem[OW05]{OtWe05}
Felix Otto and Michael Westdickenberg.
\newblock Eulerian calculus for the contraction in the {W}asserstein distance.
\newblock {\em SIAM J. Math. Anal.}, 37(4):1227--1255, 2005.

\bibitem[OY19]{OzaYok19}
Ryunosuke Ozawa and Takumi Yokota.
\newblock Stability of {RCD} condition under concentration topology.
\newblock {\em Calc. Var. Partial Differential Equations}, 58(4):Paper No. 151,
  30, 2019.

\bibitem[Pan23]{Pan23}
Jiayin Pan.
\newblock The {G}rushin hemisphere as a {R}icci limit space with curvature
  {$\ge1$}.
\newblock {\em Proc. Amer. Math. Soc. Ser. B}, 10:71--75, 2023.

\bibitem[Pas19]{P19}
Enrico Pasqualetto.
\newblock Direct and inverse limits of normed modules.
\newblock Preprint, arXiv:1902.04126, 2019.

\bibitem[Paz83]{Paz}
A.~Pazy.
\newblock {\em Semigroups of linear operators and applications to partial
  differential equations}, volume~44 of {\em Applied Mathematical Sciences}.
\newblock Springer-Verlag, New York, 1983.

\bibitem[Per97]{Perelman97}
Grigori Perelman.
\newblock Construction of manifolds of positive {R}icci curvature with big
  volume and large {B}etti numbers.
\newblock In {\em Comparison geometry ({B}erkeley, {CA}, 1993--94)}, volume~30
  of {\em Math. Sci. Res. Inst. Publ.}, pages 157--163. Cambridge Univ. Press,
  Cambridge, 1997.

\bibitem[Pet03]{Petrunin03}
Anton Petrunin.
\newblock Harmonic functions on {A}lexandrov spaces and their applications.
\newblock {\em Electron. Res. Announc. Amer. Math. Soc.}, 9:135--141, 2003.

\bibitem[Pet11]{Petrunin11}
Anton Petrunin.
\newblock Alexandrov meets {L}ott-{V}illani-{S}turm.
\newblock {\em M\"unster J. Math.}, 4:53--64, 2011.

\bibitem[Pet16]{Petersen16}
Peter Petersen.
\newblock {\em Riemannian geometry}, volume 171 of {\em Graduate Texts in
  Mathematics}.
\newblock Springer, Cham, third edition, 2016.

\bibitem[Poz23]{PozzettaSurvey}
Marco Pozzetta.
\newblock Isoperimetry on manifolds with {R}icci bounded below: overview of
  recent results and methods.
\newblock Preprint, arXiv:2303.11925, 2023.

\bibitem[Pra07]{Pratelli07}
Aldo Pratelli.
\newblock On the equality between {M}onge's infimum and {K}antorovich's minimum
  in optimal mass transportation.
\newblock {\em Ann. Inst. H. Poincar\'e Probab. Statist.}, 43(1):1--13, 2007.

\bibitem[PW22]{PanWei22}
Jiayin Pan and Guofang Wei.
\newblock Examples of {R}icci limit spaces with non-integer {H}ausdorff
  dimension.
\newblock {\em Geom. Funct. Anal.}, 32(3):676--685, 2022.

\bibitem[Raj12a]{Rajala12-2}
Tapio Rajala.
\newblock Interpolated measures with bounded density in metric spaces
  satisfying the curvature-dimension conditions of {S}turm.
\newblock {\em J. Funct. Anal.}, 263(4):896--924, 2012.

\bibitem[Raj12b]{Rajala12}
Tapio Rajala.
\newblock Local {P}oincar\'e inequalities from stable curvature conditions on
  metric spaces.
\newblock {\em Calc. Var. Partial Differential Equations}, 44(3-4):477--494,
  2012.

\bibitem[Ram01]{Ram01}
Jos\'{e}~A. Ram\'{\i}rez.
\newblock Short-time asymptotics in {D}irichlet spaces.
\newblock {\em Comm. Pure Appl. Math.}, 54(3):259--293, 2001.

\bibitem[Rit01]{Rit01}
Manuel Ritor\'{e}.
\newblock Constant geodesic curvature curves and isoperimetric domains in
  rotationally symmetric surfaces.
\newblock {\em Comm. Anal. Geom.}, 9(5):1093--1138, 2001.

\bibitem[RS12]{RajalaSturm12}
Tapio Rajala and Karl-Theodor Sturm.
\newblock Non-branching geodesics and optimal maps in strong
  ${CD(K,{\infty})}$-spaces.
\newblock {\em Calc. Var. Partial Differential Equations}, 50(3-4):831--846,
  2012.

\bibitem[San15]{Santambrogio15}
Filippo Santambrogio.
\newblock {\em Optimal transport for applied mathematicians}, volume~87 of {\em
  Progress in Nonlinear Differential Equations and their Applications}.
\newblock Birkh\"auser/Springer, Cham, 2015.
\newblock Calculus of variations, PDEs, and modeling.

\bibitem[Sau89]{Sauvageot89}
Jean-Luc Sauvageot.
\newblock Tangent bimodule and locality for dissipative operators on
  {$C^*$}-algebras.
\newblock In {\em Quantum probability and applications, {IV} ({R}ome, 1987)},
  volume 1396 of {\em Lecture Notes in Math.}, pages 322--338. Springer,
  Berlin, 1989.

\bibitem[Sau90]{Sauvageot90}
Jean-Luc Sauvageot.
\newblock Quantum {D}irichlet forms, differential calculus and semigroups.
\newblock In {\em Quantum probability and applications, {V} ({H}eidelberg,
  1988)}, volume 1442 of {\em Lecture Notes in Math.}, pages 334--346.
  Springer, Berlin, 1990.

\bibitem[Sav14]{Savare13}
Giuseppe Savar{\'e}.
\newblock Self-improvement of the {B}akry-\'{E}mery condition and {W}asserstein
  contraction of the heat flow in {${\rm RCD}(K,\infty)$} metric measure
  spaces.
\newblock {\em Discrete Contin. Dyn. Syst.}, 34(4):1641--1661, 2014.

\bibitem[Sha00]{Shanmugalingam00}
Nageswari Shanmugalingam.
\newblock Newtonian spaces: an extension of {S}obolev spaces to metric measure
  spaces.
\newblock {\em Rev. Mat. Iberoamericana}, 16(2):243--279, 2000.

\bibitem[Shi16]{Shi16}
Takashi Shioya.
\newblock {\em Metric measure geometry}, volume~25 of {\em IRMA Lectures in
  Mathematics and Theoretical Physics}.
\newblock EMS Publishing House, Z\"{u}rich, 2016.
\newblock Gromov's theory of convergence and concentration of metrics and
  measures.

\bibitem[Shi17]{Shi17}
Takashi Shioya.
\newblock Metric measure limits of spheres and complex projective spaces.
\newblock In {\em Measure theory in non-smooth spaces}, Partial Differ. Equ.
  Meas. Theory, pages 261--287. De Gruyter Open, Warsaw, 2017.

\bibitem[Sho97]{Show97}
R.~E. Showalter.
\newblock {\em Monotone operators in {B}anach space and nonlinear partial
  differential equations}, volume~49 of {\em Mathematical Surveys and
  Monographs}.
\newblock American Mathematical Society, Providence, RI, 1997.

\bibitem[Sim12]{Simon12}
Miles Simon.
\newblock Ricci flow of non-collapsed three manifolds whose {R}icci curvature
  is bounded from below.
\newblock {\em J. Reine Angew. Math.}, 662:59--94, 2012.

\bibitem[SS04]{SandSerf}
Etienne Sandier and Sylvia Serfaty.
\newblock Gamma-convergence of gradient flows with applications to
  {G}inzburg-{L}andau.
\newblock {\em Comm. Pure Appl. Math.}, 57(12):1627--1672, 2004.

\bibitem[SS13]{SchSim13}
Felix Schulze and Miles Simon.
\newblock Expanding solitons with non-negative curvature operator coming out of
  cones.
\newblock {\em Math. Z.}, 275(1-2):625--639, 2013.

\bibitem[ST21]{SimTop21}
Miles Simon and Peter~M. Topping.
\newblock Local mollification of {R}iemannian metrics using {R}icci flow, and
  {R}icci limit spaces.
\newblock {\em Geom. Topol.}, 25(2):913--948, 2021.

\bibitem[Stu94]{Sturm96I}
Karl-Theodor Sturm.
\newblock Analysis on local {D}irichlet spaces. {I}. {R}ecurrence,
  conservativeness and {$L^p$}-{L}iouville properties.
\newblock {\em J. Reine Angew. Math.}, 456:173--196, 1994.

\bibitem[Stu96]{Sturm96III}
Karl-Theodor Sturm.
\newblock Analysis on local {D}irichlet spaces. {III}. {T}he parabolic
  {H}arnack inequality.
\newblock {\em J. Math. Pures Appl. (9)}, 75(3):273--297, 1996.

\bibitem[Stu06a]{Sturm06I}
Karl-Theodor Sturm.
\newblock On the geometry of metric measure spaces. {I}.
\newblock {\em Acta Math.}, 196(1):65--131, 2006.

\bibitem[Stu06b]{Sturm06II}
Karl-Theodor Sturm.
\newblock On the geometry of metric measure spaces. {II}.
\newblock {\em Acta Math.}, 196(1):133--177, 2006.

\bibitem[Stu18]{Sturm14}
Karl-Theodor Sturm.
\newblock Ricci tensor for diffusion operators and curvature-dimension
  inequalities under conformal transformations and time changes.
\newblock {\em J. Funct. Anal.}, 275(4):793--829, 2018.

\bibitem[SY89]{ShaYang89}
Ji-Ping Sha and DaGang Yang.
\newblock Examples of manifolds of positive {R}icci curvature.
\newblock {\em J. Differential Geom.}, 29(1):95--103, 1989.

\bibitem[Tal76]{Talenti76}
Giorgio Talenti.
\newblock Best constant in {S}obolev inequality.
\newblock {\em Ann. Mat. Pura Appl. (4)}, 110:353--372, 1976.

\bibitem[tERS07]{ERS07}
Antonius Frederik~Maria ter Elst, Derek~W. Robinson, and Adam Sikora.
\newblock Small time asymptotics of diffusion processes.
\newblock {\em J. Evol. Equ.}, 7(1):79--112, 2007.

\bibitem[Vil]{Villani2017}
C{\'e}dric Villani.
\newblock {I}n\'egalit\'es isop\'erim\'etriques dans les espaces m\'etriques
  mesur\'es [d'apr\`es {F}. {C}avalletti \& {A}. {M}ondino].
\newblock S\'eminaire {B}ourbaki, available at:
  http://www.bourbaki.ens.fr/TEXTES/1127.pdf.

\bibitem[Vil09]{Villani09}
C{\'e}dric Villani.
\newblock {\em Optimal transport. Old and new}, volume 338 of {\em Grundlehren
  der Mathematischen Wissenschaften}.
\newblock Springer-Verlag, Berlin, 2009.

\bibitem[vRS05]{vRS05}
Max-K. von Renesse and Karl-Theodor Sturm.
\newblock Transport inequalities, gradient estimates, entropy, and {R}icci
  curvature.
\newblock {\em Comm. Pure Appl. Math.}, 58(7):923--940, 2005.

\bibitem[vRT12]{vRT12}
Max-K. von Renesse and Jonas~M. T\"{o}lle.
\newblock On an {EVI} curve characterization of {H}ilbert spaces.
\newblock {\em J. Math. Anal. Appl.}, 385(1):589--598, 2012.

\bibitem[Wea00]{Weaver01}
Nik Weaver.
\newblock Lipschitz algebras and derivations. {II}. {E}xterior differentiation.
\newblock {\em J. Funct. Anal.}, 178(1):64--112, 2000.

\bibitem[Wei07]{Wei07}
Guofang Wei.
\newblock Manifolds with a lower {R}icci curvature bound.
\newblock In {\em Surveys in differential geometry. {V}ol. {XI}}, volume~11 of
  {\em Surv. Differ. Geom.}, pages 203--227. Int. Press, Somerville, MA, 2007.

\bibitem[Xia01]{Xia01}
Changyu Xia.
\newblock Complete manifolds with nonnegative {R}icci curvature and almost best
  {S}obolev constant.
\newblock {\em Illinois J. Math.}, 45(4):1253--1259, 2001.

\end{thebibliography}
}
\end{document}